\documentclass[draft,10pt]{article}%
\usepackage{amscd}
\usepackage{amsmath}
\usepackage{amsfonts}
\usepackage{amssymb}
\usepackage{graphicx}
\usepackage[T1]{fontenc}
\usepackage[margin=0.5in]{geometry}
\usepackage[square,sort&compress,comma,numbers]{natbib}%
\providecommand{\U}[1]{\protect\rule{.1in}{.1in}}
\providecommand{\U}[1]{\protect\rule{.1in}{.1in}}
\providecommand{\U}[1]{\protect\rule{.1in}{.1in}}

\numberwithin{equation}{section}
\newtheorem{theorem}{Theorem}[section]
\newtheorem{corollary}[theorem]{Corollary}
\newtheorem{definition}[theorem]{Definition}
\newtheorem{lemma}[theorem]{Lemma}

\newtheorem{condition}[theorem]{Condition}

\newtheorem{remark}[theorem]{Remark}

\newenvironment{proof}[1][Proof]{\noindent\textbf{#1.} }{\ \rule{0.5em}{0.5em}}

\def\p2{\mathcal A_{\Phi,2\pi}(B)}
\def\0p2{\mathcal A_{\Phi,2\pi}(0)}
\def\sp2{\mathcal A_{\Phi,2\pi,\hbox{\rm SR}}(B)}
\def\beq{\begin{equation}}
\def\ene{\end{equation}}

\def\qed{\ifhmode\unskip\nobreak\fi\ifmmode\ifinner
\else\hskip5pt\fi\fi\hbox{\hskip5pt\vrule width4pt height6pt
depth1.5pt\hskip1pt}}

\def\+out{x^{\rm out}}
\begin{document}

\title{On travelling waves of the non linear Schr\"odinger equation escaping a potential well}
\author{ Ivan Naumkin\thanks{Electronic Mail: ivan.naumkin@unice.fr}\\Laboratoire J.A. Dieudonn\'e, Universit\'e de la C\^ote d'Azur, France \\and \\ Pierre Rapha\"el\thanks{Electronic Mail: praphael@unice.fr}\\Laboratoire J.A. Dieudonn\'e, Universit\'e de la C\^ote d'Azur, France.}
\date{}
\maketitle
\begin{abstract}
In this paper we consider the NLS equation with focusing nonlinearities in the presence of a potential. We investigate the compact soliton motions that correspond to a free soliton escaping the well created by the potential. We exhibit the dynamical system driving the exiting trajectory and construct associated nonlinear dynamics for untrapped motions. We show that the nature of the potential/soliton is fundamental, and two regimes may exist: one where the tail of the potential is fat and dictates the motion, one where the tail is weak and the soliton self interacts with the potential defects, hence leading to different motions.
\end{abstract}

\section{Introduction}



\subsection{Setting of the problem}


We consider in this paper the focusing nonlinear Schr\"{o}dinger equation in the presence of a potential
\begin{equation}
i\partial_{t}u+\Delta u+\mathcal{V}\left(  x\right)  u+\left\vert u\right\vert
^{p-1}u=0,\text{ \ }t\in\mathbb{R}\text{, }x\in\mathbb{R}^{d}, \label{NLS}%
\end{equation}
in the $L^{2}$ sub-critical range $p<1+\frac{4}{d}$. This equation appears in a variety of physical models like the Ginzburg-Landau theory of superconductivity
\cite{deGennes}, the one-dimensional self-modulation of a
monochromatic wave \cite{Taniuti}, stationary two-dimensional self-focusing of
a plane wave \cite{Bespalov}, propagation of a heat pulse in a solid Langmuir
waves in plasmas \cite{Shimizu} and the self-trapping phenomena of nonlinear
optics \cite{Karpman}, see
\cite{Berge, Scott, Eboli, Combes, Benney, Sulem} for more complete references.\\

The complete qualitative description of solutions to \eqref{NLS} is far from being complete. In the case of trivial external potential $\mathcal{V}=0$, various classes of solutions are known. First low
energy scattering solutions which asymptotically in time $t\pm \infty$ behave as linear or modified linear waves are investigated in \cite{Strauss2, Strauss, Barab,
GinibreV1, GinibreV2, CazenaveW1, Ozawa1, HayashiNau1, GinibreOV, NakanishiO,
CN, CN1}. The ground state nonlinear soliton is a time periodic solution $$u(t,x)=Q(x)e^{it}$$ where $Q$ is the ground state solution to
\begin{equation}
\Delta Q-Q+Q^{p}=0. \label{eqsoliton}%
\end{equation}
Multisoliton like solutions which asymptotically behave like decoupled non trivial trains of solitons were first explicitly computed in the completely integrable case $d=1$, $p=3$,  and then systematically constructed as {\em compact} solutions of the flow in \cite{Merle, Martel1,MartelMerle}. As discovered in \cite{Krieger} for Hartree type nonlinearities and systematically computed in \cite{MaRa,Nguyen}, the interaction between two solitary waves can be computed through a two body problem dynamical system which describes their exchange of energy: untrapped hyperbolic like motion leading to asymptotically free non interacting bubbles, untrapped parabolic like regimes which is a threshold dynamic were solitary remain logarithmically close, and a priori trapped  elliptic like motion. Note that the resonant parabolic motion can lead to spectacular exchanges of energy which may dramatically modify the size of the solitary waves and lead to finite or infinite time blow up mechanisms \cite{MaRa}.\\

A fundamental open problem is to understand the possibility of {\em trapped} multitudes. While these are formally predicted by the two body problem, they are sometimes believed to exist, sometimes not, the instability mechanism being  a subtle radiation phenomenon. We refer for example to \cite{Si} for a beautiful introduction to these problems for the Gross Pitaevski model with non vanishing density at $+\infty$.


\subsection{Potential/soliton interaction}


We propose in this paper to start the investigation of such mechanisms for \eqref{NLS} with non trivial potential $\mathcal V$. Let us first recall that in $L^{2}$ sub-critical case, (\ref{NLS}) is
globally well-posed in $H^{1}$ for a wide class of potentials $\mathcal{V}$
(see, for example, Corollary 6.1.2 of \cite{Cazenave}), but little is known on the long-time
behaviour of solutions. The existence of standing wave
solutions to (\ref{NLS}) was investigated in \cite{Floer}. The existence of
low energy scattering for (\ref{NLS}) in dimension $d=1$ was studied in
\cite{Weder2000, Cuccagna, Ivan, Ivan1}. The dynamics of solitons for NLS
equation with certain external potentials were studied in \cite{Gang}. In
particular, it was shown that the dynamical law of motion of the soliton is
close to Newton's equation with some dissipation due to radiation of the
energy. Long-time behaviour of solution for the perturbed NLS equation was
also studied in \cite{Soffer1, Soffer2, Soffer3, Soffer4, Tai, Tai1, Tai2}. In
\cite{Frolich} it was shown the existence of solutions for (\ref{NLS}) that
behaves as solitary waves for the free NLS\ equation for large but finite times.\\

Our aim in this paper is to start the investigation of {\em compact} soliton motions corresponding to a free soliton ($V=0)$ escaping the well created by $V$. Our first task is to exhibit the dynamical system driving the exiting trajectory and construct associated nonlinear dynamics for untrapped motions. Interestingly enough, we shall see that the nature of the potential/soliton is fundamental, and two regimes may exist: one where the tail of the potential is fat and dictates the motion, one where the tail is weak and the soliton self interacts with the potential defects, hence leading to different motions.


\subsection{Assumptions on the potential}


 The sharp structure of the potential is essential for the qualitative description of solutions. We exhibit below a large class of potentials for which we can construct compact non trapped solutions. \

\begin{condition}
\label{ConidtionPotential}We assume the following on $\mathcal V$:\\
\noindent{\em 1. Regularity and decay}: $\mathcal{V}\in C^{\infty}\left(
\mathbb{R}^{d}\right)  $ is a real-valued, symmetric function that satisfies
the decay estimate
\[
\left\vert \mathcal{V}^{\left(  k\right)  }\left(  \left\vert x\right\vert
\right)  \right\vert \leq C\left(  1+\left\vert x\right\vert \right)  ^{-\rho
},\text{ }x\in\mathbb{R}^{d},\text{\ }\rho>0,
\]
for all derivatives of $\mathcal{V}^{\left(  k\right)  },$ $k\geq0.$\\
\noindent{\em 2. Structure and monotonicity of the tail}: there is $r_{0}\geq0$ such that the following is true for all $r\geq
r_{0}.$ The potential and its derivatives $\mathcal{V}^{\left(  k\right)
}\left(  r\right)  ,$ $k\geq0,$ are monotone. For any $\lambda>0,$
$\mathcal{V}$ has one of the forms
\[
\mathcal{V}\left(  r\right)  =\mathcal{V}_{+}\left(  r\right)  \text{ or
}\mathcal{V}\left(  r\right)  =\mathcal{V}_{-}\left(  r\right)  ,
\]
where
\[
\mathcal{V}_{\pm}\left(  r\right)  =\kappa e^{-2\frac{r}{\lambda}-\left(
d-1\right)  \ln\frac{r}{\lambda}\pm\mathbf{H}\left(  r\right)  },
\]
and $\mathbf{H}\geq0$ is such that $\mathbf{H},\mathbf{H}^{\prime}$ are
monotone and $\mathbf{H}^{\prime\prime}\left(  r\right)  $ is
bounded.\ Moreover, the estimate
\[
\left\vert \mathcal{V}^{\left(  k\right)  }\left(  r\right)  \right\vert \leq
C_{k}e^{-2\frac{r}{\lambda}-\left(  d-1\right)  \ln\frac{r}{\lambda}},\text{
}k\geq0,
\]
with $C_{k}>0$ is satisfied. In addition, in the case $\mathcal{V}%
=\mathcal{V}_{+},$ $\mathbf{H}$ is either $\mathbf{H}\left(  r\right)
=o\left(  r\right)  $ or $0<cr\leq\mathbf{H}\left(  r\right)  \leq2r+o\left(
r\right)  .$ Furthermore, in the case when $\mathbf{H}\left(  r\right)
=2\frac{r}{\lambda}+\left(  d-1\right)  \ln r+o\left(  r\right)  ,$ suppose
that the potential $\mathcal{V=}$ $e^{-\mathbf{h}\left(  r\right)  }$ with
$\mathbf{h}\left(  r\right)  \geq0$ satisfies $\left\vert \mathcal{V}\left(
r\right)  \right\vert ^{N_{0}}\leq C\left\vert \mathcal{V}^{\prime\prime
}\left(  r\right)  \right\vert $ for some $N_{0}>0.$ The function $\mathbf{h}$
is such that $\mathbf{h}\left(  r\right)  =o\left(  r\right)  ,$
$\mathbf{h}^{\left(  k\right)  }\left(  r\right)  $ are monotone for all
$k\geq0$ and
\[
\left\vert \mathbf{h}^{\left(  k\right)  }\left(  \frac{r}{2}\right)
\right\vert \leq C_{k}\left\vert \mathbf{h}^{\left(  k\right)  }\left(
r\right)  \right\vert ,\text{ }k\geq0,
\]
and%
\[
\left\vert \mathbf{h}^{\left(  k\right)  }\left(  r\right)  \right\vert
\leq\frac{C}{r}\left\vert \mathbf{h}^{\left(  k-1\right)  }\left(  r\right)
\right\vert ,\text{ }k\geq1,
\]
are valid.
\end{condition}

In order to introduce the leading order dynamical system describing the motion of the center of the solitary wave escaping the well, we need to compare the tail of $\mathcal V$ with the one of the solitary wave. Recall that there is a unique positive, radial
symmetric solution $Q\left(  x\right)  $ to
$$
-\Delta Q-Q+Q^{p}=0.
$$
(see Chapter 8 of \cite{Cazenave}). Moreover, $Q\left(  x\right)  =q\left(
\left\vert x\right\vert \right)  $ where $q$ satisfies for some $\mathcal{A}%
>0$ and all $\left\vert x\right\vert \geq1$
\begin{equation}
\left\vert q\left(  \left\vert x\right\vert \right)  -\mathcal{A}\left\vert
x\right\vert ^{-\frac{d-1}{2}}e^{-\left\vert x\right\vert }\right\vert
+\left\vert q^{\prime}\left(  \left\vert x\right\vert \right)  +\mathcal{A}%
\left\vert x\right\vert ^{-\frac{d-1}{2}}e^{-\left\vert x\right\vert
}\right\vert \leq C\left\vert x\right\vert ^{-1-\frac{d-1}{2}}e^{-\left\vert
x\right\vert } \label{ap2}%
\end{equation}
and, moreover%
\begin{equation}
\left\vert \nabla^{k}Q\left(  x\right)  \right\vert \leq C\left(  1+\left\vert
x\right\vert \right)  ^{-\frac{d-1}{2}}e^{-\left\vert x\right\vert },\text{
}k\geq0. \label{ap29}%
\end{equation}

\begin{definition}
\label{Def1}Let $\lambda>0.$ Let $$\upsilon\left(  d\right)  ={\textstyle\int_{0}^{\infty}}
e^{-2\eta^{2}}\eta^{d-2}d\eta$$ and $\mathbf{C}_{\pm}\left(  \xi\right)
=\xi^{\frac{d-1}{2}}%
{\textstyle\int_{1}^{\xi-1}}
\left(  r\left(  \xi-r\right)  \right)  ^{-\frac{d-1}{2}}e^{\pm\mathbf{H}%
\left(  r\right)  }dr.$
We define the function $\mathbf{U}_{\lambda
,\mathcal{V}}\left(  \xi\right)  ,$ $\xi\geq0$ as follows. If $\mathcal{V}%
\left(  r\right)  =\mathcal{V}_{+}\left(  r\right)  $ we put%
\[
\mathbf{U}_{\lambda,\mathcal{V}}\left(  \xi\right)  =\left\{
\begin{array}
[c]{c}%
\dfrac{\kappa}{2}\mathcal{A}^{2}\upsilon\left(  d\right)  \mathbf{C}%
_{+}\left(  \xi\right)  e^{-2\xi}\xi^{-\left(  d-1\right)  },\text{ if
}\mathbf{H}\left(  r\right)  =o\left(  r\right)  ,\\
\mathcal{KV}\left(  \lambda\xi\right)  ,\text{ if }0<c\lambda^{-1}%
r\leq\mathbf{H}\left(  r\right)  \leq2\lambda^{-1}r,
\end{array}
\right.
\]
where $\mathcal{K=}%
{\displaystyle\int}
e^{\left(  2-a\right)  \frac{\chi\cdot z}{\left\vert \chi\right\vert }}%
Q^{2}\left(  z\right)  dz,$ $a=\lim_{r\rightarrow\infty}\mathbf{H}^{\prime
}\left(  r\right)  .$ If $\mathcal{V}\left(  r\right)  =\mathcal{V}_{-}\left(
r\right)  ,$ we set%
\[
\mathbf{U}_{\lambda,\mathcal{V}}\left(  \xi\right)  =\left\{
\begin{array}
[c]{c}%
\dfrac{1}{2}\left(
{\displaystyle\int}
\left(  \mathcal{A}^{2}\mathcal{V}\left(  \left\vert z\right\vert \right)
+\kappa\mathcal{\tilde{K}}Q^{2}\left(  z\right)  \right)  e^{2\frac{\chi\cdot
z}{\left\vert \chi\right\vert }}dz\right)  e^{-2\xi}\xi^{-\left(  d-1\right)
},\text{ if }d\geq4,\\
\dfrac{1}{2}\left(
{\displaystyle\int}
V\left(  \left\vert z\right\vert \right)  e^{2\frac{\chi\cdot z}{\left\vert
\chi\right\vert }}dz\right)  e^{-2\xi}\xi^{-\left(  d-1\right)  },\text{ if
}d=2,3,\text{ }r^{-\frac{d-1}{2}}e^{-\mathbf{H}\left(  r\right)  }\in
L^{1}\left(  [1,\infty)\right)  ,\\
\dfrac{\kappa}{2}\mathcal{A}^{2}\upsilon\left(  d\right)  \mathbf{C}%
_{-}\left(  \xi\right)  e^{-2\xi}\xi^{-\left(  d-1\right)  },\text{ if
}d=2,3,\text{ }r^{-\frac{d-1}{2}}e^{-\mathbf{H}\left(  r\right)  }\notin
L^{1}\left(  [1,\infty)\right)  ,
\end{array}
\right.
\]
with $\mathcal{\tilde{K}}=\lim_{r\rightarrow\infty}e^{-\mathbf{H}\left(
r\right)  }.$
\end{definition}

\begin{condition}
\label{C1}Let $\lambda>0.$ Suppose that there are constants $R_{1}\left(
\lambda\right)  ,R_{2}\left(  \lambda\right)  \in\mathbb{R}$ such that
\begin{equation}
\mathbf{U}_{\lambda,\mathcal{V}^{\prime}}^{\prime}\left(  r\right)
\mathcal{Y}^{2}\left(  r,\lambda\right)  =R_{1}\left(  \lambda\right)
+o\left(  1\right)  \text{ and }\frac{\mathbf{U}_{\lambda,\mathcal{V}}%
^{\prime}\left(  r\right)  }{r}\mathcal{Y}^{2}\left(  r,\lambda\right)
=R_{2}\left(  \lambda\right)  +o\left(  1\right)  , \label{condpotin}%
\end{equation}
as $r\rightarrow\infty,$ where we denote $\mathcal{Y}\left(  r,\lambda\right)
=\int_{r_{0}}^{r}\mathbf{U}_{\lambda,\mathcal{V}}^{-1/2}\left(  \frac{\tau
}{\lambda}\right)  d\tau.$
\end{condition}

The main assumption on the potential is the first item in Condition
\ref{ConidtionPotential}. The rest of the assumptions are made in order to
determine the asymptotics in Section \ref{AppendixA}, and to obtain the
approximate dynamics in Section \ \ref{Sec1} (see Lemma \ref{approxsyst}) and
the a priori estimates of Section \ref{Secproof}. In order to simplify the
reading, one can follow the proofs by taking in account the examples $V\left(
r\right)  =Ce^{-c\sqrt{1+r^{2}}}$ or $V\left(  r\right)  =C\left(
1+r^{2}\right)  ^{-\rho/2},$ $\rho>0,$ with $C,c>0,$ for all $r>0,$ which
clearly satisfy the above assumptions.

\subsection{Statement of the result}

For $\lambda^{\infty}\in\mathbb{R}^{+}$ let us consider the problem of the motion in the central field
\begin{equation}
\left\{
\begin{array}
[c]{c}%
\dot{\chi}^{\infty}=2\beta^{\infty},\\
\dot{\beta}^{\infty}=\dfrac{1}{2\left\Vert Q\right\Vert _{L^{2}}^{2}}%
\nabla\left(  \mathbf{U}_{\lambda^{\infty},\mathcal{V}}\left(  \frac
{\left\vert \chi^{\infty}\right\vert }{\lambda^{\infty}}\right)  \right)  .
\end{array}
\right.  \label{s1}%
\end{equation}

\begin{remark}
We observe that the force in (\ref{s1}) is given in terms of $\mathbf{U}%
_{\lambda^{\infty},\mathcal{V}}$ and not of the potential $\mathcal{V}$
directly. This is due to the fact that for potentials decaying faster than
$Q^{2},$ the soliton dictates the behaviour of the speed parameter
$\beta^{\infty}$ (see (\ref{ap17}) and Lemma \ref{L2 1}). In the case when the
potential decays slower than $Q^{2},$ that is $0<c\left(  \lambda^{\infty
}\right)  ^{-1}r\leq\mathbf{H}\left(  r\right)  \leq2\left(  \lambda^{\infty
}\right)  ^{-1}r,$ we have $\mathbf{U}_{\lambda^{\infty},\mathcal{V}}^{\prime
}\left(  \frac{\left\vert \chi^{\infty}\right\vert }{\lambda^{\infty}}\right)
=\mathcal{KV}^{\prime}\left(  \left\vert \chi^{\infty}\right\vert \right)  $
and system (\ref{s1}) is reduced to%
\[
\left\{
\begin{array}
[c]{c}%
\dot{\chi}^{\infty}=2\beta^{\infty},\\
\dot{\beta}^{\infty}=\dfrac{\mathcal{K}}{2\left\Vert Q\right\Vert _{L^{2}}%
^{2}}\nabla\mathcal{V}\left(  \left\vert \chi^{\infty}\right\vert \right)  .
\end{array}
\right.
\]
Therefore, in this case the force is given in terms of $\mathcal{V}$ itself.
\end{remark}

The energy of the system (\ref{s1}) is given by
\[
E_{0}=\frac{\left\vert \dot{\chi}^{\infty}\right\vert ^{2}}{2}-\frac
{\mathbf{U}_{\lambda^{\infty},\mathcal{V}}\left(  \frac{r^{\infty}}%
{\lambda^{\infty}}\right)  }{\left\Vert Q\right\Vert _{L^{2}}^{2}}%
\]
where $r^{\infty}=\left\vert \chi^{\infty}\right\vert .$ The asymptotic
behaviour for large $t$ of the unbounded solutions $\chi^{\infty}$ to
(\ref{s1}) depends on the regime:\\

\noindent\underline{$E_{0}>0$}: hyperbolic motion with untrapped trajectory
\[
r^{\infty}=K_{\operatorname*{hyp}}t+o\left(  t\right)  \text{ and }\left\vert
\beta^{\infty}\right\vert =K_{\operatorname*{hyp}}^{\prime}+o\left(  1\right)
.
\]
for some $K_{\operatorname*{hyp}},K_{\operatorname*{hyp}}^{\prime}>0.$\\
\noindent\underline{$E_{0}=0$}: parabolic motion with untrapped trajectory
\[
\frac{\left\Vert Q\right\Vert _{L^{2}}}{\sqrt{2}}\int_{r_{0}}^{\left\vert
\chi^{\infty}\right\vert }\tfrac{dr}{\sqrt{\mathbf{U}_{\lambda^{\infty
},\mathcal{V}}\left(  \frac{r}{\lambda^{\infty}}\right)  -\left(  \frac{\mu
}{r}\right)  ^{2}}}=t+t_{0},\text{ }\left\vert \beta^{\infty}\right\vert
=\tfrac{\sqrt{\mathbf{U}_{\lambda^{\infty},\mathcal{V}}\left(  \frac
{r^{\infty}}{\lambda^{\infty}}\right)  -\left(  \frac{\mu}{r}\right)  ^{2}}%
}{\sqrt{2}\left\Vert Q\right\Vert _{L^{2}}},
\]
where $\mu\geq0$ is the angular momentum.\\
\noindent\underline{$E_{0}<0$}: all the solutions to (\ref{s1}) remain bounded in space, this is the trapped regime.\\

Our main result is that both untrapped hyperbolic and parabolic regimes of the two body problem \eqref{s1} can be reproduced as the leading order dynamics of the soliton centers for {\em compact} solutions to the full problem \eqref{NLS}.

\begin{theorem}
\label{T1}Let the potential $\mathcal{V}$ satisfy Conditions
\ref{ConidtionPotential} and \ref{C1}. Suppose that $\lambda^{\infty}%
\in\mathbb{R}^{+}$ is such that $\left(  \lambda^{\infty}\right)  ^{2}%
\sup_{r\in\mathbb{R}}\mathcal{V}\left(  r\right)  <1.\ $ Let $\Xi^{\infty
}\left(  t\right)  =\left(  \chi^{\infty}\left(  t\right)  ,\beta^{\infty
}\left(  t\right)  \right)  $ be a solution to (\ref{s1}).

(i) Positive energy. If $E_{0}>0,$ then there is a solution $u\in H^{1}$ for
the perturbed NLS equation (\ref{NLS}) and $\gamma\left(  t\right)
\in\mathbb{R}$, such that%
\[
\lim_{t\rightarrow+\infty}\left\Vert u\left(  t,x\right)  -\left(
\lambda^{\infty}\right)  ^{-\frac{2}{p-1}}Q\left(  \frac{x-\chi^{\infty
}\left(  t\right)  }{\lambda^{\infty}}\right)  e^{-i\gamma\left(  t\right)
}e^{i\beta^{\infty}\left(  t\right)  \cdot x}\right\Vert _{H^{1}}=0.
\]

(ii) Zero energy. Suppose in addition that $\mathcal{V}\left(  r\right)
\geq0,$ for all $r\geq r_{0},$ with some $r_{0}>0.$ If $E_{0}=0,$ then there
is a solution $u\in H^{1}$ for the perturbed NLS equation (\ref{NLS}) and
$\chi\left(  t\right)  \in\mathbb{R}^{d},$ $\gamma\left(  t\right)
\in\mathbb{R}$, such that
\[
\lim_{t\rightarrow+\infty}\left\Vert u\left(  t,x\right)  -\left(
\lambda^{\infty}\right)  ^{-\frac{2}{p-1}}Q\left(  \frac{x-\chi\left(
t\right)  }{\lambda^{\infty}}\right)  e^{-i\gamma\left(  t\right)  }%
e^{i\beta^{\infty}\left(  t\right)  \cdot x}\right\Vert _{H^{1}}=0
\]
and%
\[
\lim_{t\rightarrow+\infty}\left\vert \frac{\left\vert \chi\left(  t\right)
\right\vert }{\left\vert \chi^{\infty}\left(  t\right)  \right\vert
}-1\right\vert =0.
\]

\end{theorem}

\noindent{\em Comments on the result}.\\

\noindent{\em 1. Positivity condition on $\mathcal{V}$} In the case $E_0=0$, the positivity condition in $\mathcal V$ is made in order to ensure the existence of unbounded solutions for(\ref{s1}). Since by Condition \ref{ConidtionPotential} $\mathcal{V}$ is
monotone, there are no such solutions if this positivity assumption is not
satisfied on $\mathcal{V}$. Let us stress that as in \cite{Krieger}, the parabolic regime is particularly difficult to close due to uncertainties on the trajectory of the centers and degeneracies in the control of the infinite dimensional part of the solution.\\

\noindent{\em 2. Soliton/potential regimes}. Let us stress onto the fact that Theorem \ref{T1} covers both cases which are very different: one where the potential is "fat" at $\infty$ and where the potential tail drives the untrapped dynamics of the centers, one where the potential tail is "weak" and the soliton dynamics is driven by self interaction with the potential. The existence of such dynamics driven by a suitable leading order like two body problems was first predicted in \cite{Frolich}, but for suitable transient times, while Theorem \ref{T1} ensures the existence of such {\em global in time} compact flows.\\

\noindent{\em 3. Overview of the proof}. We adapt the method developed in
\cite{Krieger}. We modulate
the solitary wave by letting act the symmetries of the free NLS\ equation (see
(\ref{modulation})). Then, we translate (\ref{NLS}) into the stationary
equation
\begin{equation}
\triangle W-W+\left\vert W\right\vert ^{p-1}W=F, \label{*}%
\end{equation}
where $F=F\left(  \chi,\beta,\lambda\,\right)  $, with $\chi$-the translation
parameter, $\beta-$the Galilean drift and $\lambda-$the scaling parameter (see
(\ref{E2})). By adjusting the modulation parameters $\chi,\beta,\lambda$, we
construct approximate solutions to (\ref{*}) in such way that the error is
uniformly bounded by a small enough constant. This is achieved in Lemmas
\ref{Lemmaapp} and \ref{Lemmaapp1}. We separate the potentials in "fast" and
"slow" decaying: $V\left(  r\right)  =O\left(  e^{-cr}\right)  ,$ $c>0,$ or
$V\left(  r\right)  =O\left(  e^{-h\left(  r\right)  }\right)  ,$ $h\left(
r\right)  =o\left(  r\right)  \geq0,$ as $r\rightarrow\infty,$
respectively.\ The both cases are delicate. On the one hand, this is due to
the possible slow decay of the potential. On the other hand, if the potential
decays very fast, the solitary waves dominates and it becomes complicated to
control the error and to extract the leading order terms in the expansion for
the speed parameter $\beta$. The construction of these approximate solutions
for (\ref{*}) yields modulation equations for $\chi,\beta,\lambda\,.$ Then, by
energy estimates applied in a neighborhood of the solitary wave, we obtain a
priori bounds for the modulation parameters $\chi,\beta,\lambda$ and the error
$\varepsilon\left(  t,x\right)  $ (see Lemma \ref{L6}). Theorem \ref{T1} then
follows from a compactness argument in Section \ref{Sec2} and the asymptotic
behaviour of the approximate modulation parameters (see Lemma \ref{approxsyst}).\\

There are two main open problems after this work. First to address the question of {\em stability} of the corresponding compact dynamics. This has been proved for two bubbles KdV like flows using remarkable dine monotonicity properties \cite{Martel1}, but it is still an open problem for two bubbles in the Sch\"odinger case, and the case of the potential interaction seems a nice intermediate problem to investigate. The second main open problem  is to address the case of {\em trapped dynamics} where bubbles are predicted to stay close to one another. Here new radiation mechanisms are expected which is a fundamental open problem in the field, see again\cite{Si} for beautiful related problems.\\

The rest of the paper is organized as follows. In Section 2, we construct
approximate free solitary wave solutions for the perturbed NLS equation which
we use to obtain approximate equations for the modulation parameters. Section
3 is devoted to the proof of Theorem \ref{T1}. This proof depends on Lemma
\ref{L6}, which is stated in Section~3, but whose proof is deferred until
Section 4. The asymptotics for modulation equation for the speed parameter are
obtained in Section 5. In Section 6 we establish an invertibility result for
the perturbation $\mathcal{L}_{V}$ of the linearized operator for the equation
(\ref{eqsoliton}) around $Q.\ $Finally, in Section 7 we prove a lemma that is
used in the construction procedure of Section 3.

\subsection*{Notations} We denote by $L^{p}\left(  {\mathbb{R}}^{d}\right)  $, for $1\leq
p\leq\infty$, the usual (complex valued) Lebesgue spaces. $H^{s}({\mathbb{R}%
}^{d})$, $s\in{\mathbb{R}}$, is the usual (complex valued) Sobolev space. (See
e.g.~\cite{AdamsF} for the definitions and properties of these spaces.) For
any $f,g\in L^{2}$ we define the scalar product by
\[
\left(  f,g\right)  :=\operatorname{Re}\int_{{\mathbb{R}}^{d}}f\left(
x\right)  \overline{g\left(  x\right)  }dx.
\]
We adapt the Japanese brackets notation
\[
\left\langle x\right\rangle =\left(  1+x^{2}\right)  ^{1/2}.
\]
Finally, the same letter $C$ may denote different positive constants which
particular value is irrelevant.

\subsection*{Acknowledgements}  Both authors are supported by the ERC-2014-CoG 646650 SingWave.

\section{Approximate solutions.}

\subsection{First order approximation.}

\bigskip We recall that initial value problem for the free NLS\ equation%
\begin{equation}
\left\{
\begin{array}
[c]{c}%
i\partial_{t}u+\Delta u+\left\vert u\right\vert ^{p-1}u=0,\\
u\left(  0,x\right)  =u_{0}\left(  x\right)  .
\end{array}
\right.  \label{NLSfree}%
\end{equation}
admits the following symmetries. Let $\Xi=\left(  \chi,\beta,\lambda\right)
\in\mathbb{R}^{d}\times\mathbb{R}^{d}\times\mathbb{R}^{+}$ be a vector of
parameters and $\gamma\in\mathbb{R}$. Then, if $u_{0}\left(  x\right)
\rightarrow\lambda^{\frac{2}{p-1}}u_{0}\left(  \lambda\left(  x+\chi\right)
\right)  e^{-i\gamma}e^{i\beta\cdot\left(  x+\chi\right)  },$ the solution to
(\ref{NLSfree}) is transformed as%
\begin{equation}
u\left(  t,x\right)  \rightarrow\lambda^{-\frac{2}{p-1}}u_{0}\left(
\lambda^{-2}t,\lambda^{-1}\left(  x+\chi-\beta t\right)  \right)  e^{-i\gamma
}e^{i\tfrac{\beta}{2}\cdot\left(  x+\chi-\tfrac{\beta}{2}t\right)  }
\label{sym}%
\end{equation}

\bigskip Let $\Xi\left(  t\right)  $ encode the vector of parameters
$\Xi\left(  t\right)  =\left(  \chi\left(  t\right)  ,\beta\left(  t\right)
,\lambda\left(  t\right)  \right)  .$ We translate the solution $u$ by using a
combination of the symmetries for the free NLS\ equation. We let
\begin{equation}
u\left(  t,x\right)  =\lambda^{-\frac{2}{p-1}}\left(  t\right)  v\left(
t,\tfrac{x-\chi\left(  t\right)  }{\lambda\left(  t\right)  }\right)
e^{-i\gamma\left(  t\right)  }e^{i\beta\left(  t\right)  \cdot x},\text{
\ }v=v\left(  t,y\right)  ,\text{ }y=\tfrac{x-\chi\left(  t\right)  }%
{\lambda\left(  t\right)  }. \label{modulation}%
\end{equation}
It is convenient for us to rescale the potential $\mathcal{V}$. We define%
\begin{equation}
V\left(  x\right)  =V_{\lambda}\left(  x\right)  =\mathcal{V}\left(  \lambda
x\right)  . \label{vcall}%
\end{equation}
Also, we let%
\[
\Lambda=\frac{2}{p-1}+y\cdot\nabla.
\]
Then%
\begin{equation}
i\partial_{t}u+\Delta u+\left\vert u\right\vert ^{p-1}u+\mathcal{V}\left(
x\right)  u=\frac{1}{\lambda^{\frac{2p}{p-1}}}\mathcal{E}\left(  v\right)
\left(  t,\tfrac{x-\chi}{\lambda}\right)  e^{-i\gamma\left(  t\right)
}e^{i\beta\left(  t\right)  \cdot x}, \label{calc1}%
\end{equation}
where%
\[
\left.
\begin{array}
[c]{c}%
\mathcal{E}\left(  v\right)  =i\lambda^{2}\partial_{t}v+\Delta v-v+\left\vert
v\right\vert ^{p-1}v-i\lambda\dot{\lambda}\Lambda v-i\lambda\left(  \dot{\chi
}-2\beta\right)  \cdot\nabla v-\lambda^{3}\left(  \dot{\beta}\cdot y\right)
v\\
+\lambda^{2}\left(  \dot{\gamma}+\frac{1}{\lambda^{2}}-\left\vert
\beta\right\vert ^{2}-\dot{\beta}\cdot\chi\right)  v+\lambda^{2}V\left(
\left\vert y+\frac{\chi}{\lambda}\right\vert \right)  v.
\end{array}
\right.
\]
We aim to show that there is a vector $\Xi\left(  t\right)  $ such that
$v=Q+\varepsilon$ with $\varepsilon\left(  t,x\right)  \in C\left(  \left(
0,\infty\right)  ;\mathbf{H}^{1}\right)  $ solves the equation
\begin{equation}
\mathcal{E}\left(  v\right)  =\mathcal{E}\left(  Q+\varepsilon\right)  =0,
\label{E1}%
\end{equation}
in such way that
\[
\left\Vert \varepsilon\left(  t,x\right)  \right\Vert _{\mathbf{H}^{1}%
}\rightarrow0\text{ and\ }\left\vert \chi\left(  t\right)  \right\vert
\rightarrow\infty,
\]
as $t\rightarrow\infty.$ We also want to precise the asymptotic behaviour of
$\chi\left(  t\right)  .$

As a first step, we aim to construct approximate (in a suitable way) solution
to the equation
\begin{equation}
\mathcal{E}\left(  W\right)  =0. \label{appeq}%
\end{equation}
Let us denote the approximate modulation equation for the speed parameter
$\dot{\beta}$ by $B.$ We decompose $\mathcal{E}\left(  v\right)  $ as
\begin{equation}
\left.  \mathcal{E}\left(  v\right)  =\mathcal{\tilde{E}}_{\operatorname*{apr}%
}\left(  v\right)  +\tilde{R}\left(  v\right)  ,\right.  \label{E3}%
\end{equation}
with%
\begin{equation}
\mathcal{\tilde{E}}_{\operatorname*{apr}}\left(  v\right)  =\Delta
v-v+\left\vert v\right\vert ^{p-1}v+\lambda^{2}V\left(  \left\vert
y+\frac{\chi}{\lambda}\right\vert \right)  v-\lambda^{3}\left(  B\cdot
y\right)  v, \label{E4}%
\end{equation}
and%
\begin{equation}
\left.  \tilde{R}\left(  v\right)  =-i\lambda\left(  \dot{\chi}-2\beta\right)
\cdot\nabla v-i\lambda\dot{\lambda}\Lambda v+\lambda^{2}\left(  \dot{\gamma
}+\frac{1}{\lambda^{2}}-\left\vert \beta\right\vert ^{2}-\dot{\beta}\cdot
\chi\right)  v-\lambda^{3}\left(  \left(  \dot{\beta}-B\right)  \cdot
y\right)  v.\right.  \label{ap20}%
\end{equation}
If we modulate the parameters $\chi,$ $\beta,$ $\gamma$ in such a way that $R$
is small, the main part in (\ref{E3}) comes from $\mathcal{E}%
_{\operatorname*{apr}}\left(  v\right)  .$ As $Q$ solves (\ref{eqsoliton}),
introducing $v=Q+\varepsilon$ into (\ref{E4}) we have%
\begin{equation}
\mathcal{\tilde{E}}_{\operatorname*{apr}}\left(  v\right)  =\lambda
^{2}V\left(  \left\vert y+\frac{\chi}{\lambda}\right\vert \right)
Q-\lambda^{3}\left(  B\cdot y\right)  Q+r\left(  \varepsilon\right)  ,
\label{E5}%
\end{equation}
where
\[
r\left(  \varepsilon\right)  =\Delta\varepsilon-\varepsilon+V\left(
\left\vert y+\frac{\chi}{\lambda}\right\vert \right)  \varepsilon-\left(
B\cdot y\right)  \varepsilon+\left\vert Q+\varepsilon\right\vert ^{p-1}\left(
Q+\varepsilon\right)  -\left\vert Q\right\vert ^{p-1}Q.
\]
The stability problem for one travelling wave solution suggests
(\cite{Weisntein},\cite{Martel}) to adjust the modulation parameters in such
way that $\left(  \mathcal{\tilde{E}}_{\operatorname*{apr}}\left(  v\right)
,\nabla Q\right)  \sim0.$ Taking the scalar product of (\ref{E5}) with $\nabla
Q$ we get%
\[
\left(  \mathcal{\tilde{E}}_{\operatorname*{apr}}\left(  v\right)  ,\nabla
Q\right)  =\left(  \lambda^{2}V\left(  \left\vert y+\frac{\chi}{\lambda
}\right\vert \right)  Q-\lambda^{3}\left(  B\cdot y\right)  Q,\nabla Q\right)
+\left(  r\left(  \varepsilon\right)  ,\nabla Q\right)  .
\]
Then%
\[
\left\vert \left(  \lambda^{2}V\left(  \left\vert y+\frac{\chi}{\lambda
}\right\vert \right)  Q-\lambda^{3}\left(  B\cdot y\right)  Q,\nabla Q\right)
\right\vert \leq C\left(  \left\vert \left(  \mathcal{\tilde{E}}%
_{\operatorname*{apr}}\left(  v\right)  ,\nabla Q\right)  \right\vert
+\left\vert \left(  r\left(  \varepsilon\right)  ,\nabla Q\right)  \right\vert
\right)  .
\]
Therefore, we put%
\[
\left(  \lambda^{2}V\left(  \left\vert y+\frac{\chi}{\lambda}\right\vert
\right)  Q-\lambda^{3}\left(  B\cdot y\right)  Q,\nabla Q\right)  =0.
\]
Noting that~$-\left(  \left(  B\cdot y\right)  Q,\nabla Q\right)  =B\left\Vert
Q\right\Vert _{L^{2}}^{2}$ we arrive to%
\begin{equation}
B=B\left(  \chi,\lambda\right)  =-\frac{1}{2\lambda\left\Vert Q\right\Vert
_{L^{2}}^{2}}\int V\left(  \left\vert y+\frac{\chi}{\lambda}\right\vert
\right)  \nabla Q^{2}\left(  y\right)  dy. \label{ap17}%
\end{equation}
We want $R$ to be small as $t\rightarrow\infty$, approximately (\ref{ap20})
and (\ref{ap17}) yield the system
\begin{equation}
\left\{
\begin{array}
[c]{c}%
\dot{\chi}=2\beta,\\
\dot{\beta}=B\left(  \chi,\lambda\right)  ,
\end{array}
\right.  \label{syst}%
\end{equation}
and
\begin{equation}
\dot{\gamma}=-\frac{1}{\lambda^{2}}+\left\vert \beta\right\vert ^{2}%
+\dot{\beta}\cdot\chi. \label{gamma}%
\end{equation}

In order to understand the behaviour as $t\rightarrow\infty$ of the solutions
to (\ref{syst}), we need to study the asymptotics of the integral
\[
\mathcal{J}\left(  \chi\right)  =\mathcal{J}_{V}\left(  \chi\right)  =\int
V\left(  \left\vert y+\tilde{\chi}\right\vert \right)  \nabla Q^{2}\left(
y\right)  dy.
\]
This means that we need to compare the decay of the potential with the soliton
$Q.$ We consider a bounded function $V\in C^{\infty}$ satisfying
\begin{equation}
\left\vert \frac{d^{k}}{dr^{k}}V\left(  r\right)  \right\vert \leq C\left(
1+r\right)  ^{-\rho},\text{ }\rho>0, \label{p1}%
\end{equation}
for all $k\geq0.$ Suppose that $V^{\left(  k\right)  }\left(  r\right)  ,$
$k\geq0,$ are monotone for all $r\geq r_{0}>0.$ We ask $V$ to have one of the
forms
\begin{equation}
V\left(  r\right)  =V_{+}\left(  r\right)  \text{ or }V\left(  r\right)
=V_{-}\left(  r\right)  , \label{potential}%
\end{equation}
where $V_{\pm}\left(  r\right)  =\kappa e^{-2r-\left(  d-1\right)  \ln r\pm
H\left(  r\right)  },$ and $H\geq0$ is such that $H,H^{\prime}$ are monotone
and $H^{\prime\prime}\left(  r\right)  $ is bounded, for all $r\geq r_{0}%
.$\ Moreover, we assume some control on the derivatives
\[
\left\vert \left(  \frac{d}{dr}\right)  ^{k}V_{-}\left(  r\right)  \right\vert
\leq C_{k}e^{-2r-\left(  d-1\right)  \ln r},\text{ }k\geq0,
\]
with $C_{k}>0.$ In addition, in the case $V=V_{+},$ $H$ is either $H\left(
r\right)  =o\left(  r\right)  $ or $0<cr\leq H\left(  r\right)  \leq
2r+o\left(  r\right)  .$

\bigskip We also need to estimate $V\left(  \left\vert y+\tilde{\chi
}\right\vert \right)  Q\left(  y\right)  ,$ as $\left\vert \tilde{\chi
}\right\vert $ $\rightarrow\infty.$ Therefore, we require some information
about the behavior of the potential compared to the soliton $Q.$ We suppose
either one of the following asymptotics. Let $V_{\pm}^{\left(  1\right)
}=\kappa_{1}e^{-r-\frac{1}{2}\left(  d-1\right)  \ln r\pm H_{1}\left(
r\right)  }$, for $r\geq r_{0},$ with $H_{1}\geq0,$ $H_{1},$ $H_{1}^{\prime}$
are monotone, for all $r\geq r_{0},$ and
\begin{equation}
\left\vert \left(  \frac{d}{dr}\right)  ^{k}V_{-}^{\left(  1\right)  }\left(
r\right)  \right\vert \leq C_{k}e^{-r-\frac{1}{2}\left(  d-1\right)  \ln
r},\text{ }k\geq0, \label{ap256}%
\end{equation}
with $C_{k}>0.$ Moreover, in the case $V_{+}^{\left(  1\right)  }$ we suppose
that $H_{1}^{\left(  j\right)  },$ $j\geq1,$ are bounded. Then, we consider
potentials $V$ of the form (\ref{potential}) that can be represented as
\begin{equation}
V=V_{+}^{\left(  1\right)  }\text{ or }V=V_{-}^{\left(  1\right)  }.
\label{potential1}%
\end{equation}

Finally, in the case when $H_{1}\left(  r\right)  =r+\frac{1}{2}\left(
d-1\right)  \ln r+o\left(  r\right)  ,$ we suppose that the potential $V$ is
given by
\[
V\left(  r\right)  =V_{+}\left(  r\right)  =V_{+}^{\left(  1\right)  }\left(
r\right)  =V^{\left(  2\right)  }\left(  r\right)  ,
\]
where
\begin{equation}
V^{\left(  2\right)  }\left(  r\right)  :=e^{-h_{1}\left(  r\right)  },
\label{v2}%
\end{equation}
with $h_{1}\left(  r\right)  \geq0$ is such that $h_{1}\left(  r\right)
=o\left(  r\right)  ,$ $h_{1}^{\left(  k\right)  }$ are monotone,\ $k\geq0$
and
\begin{equation}
\left\vert h_{1}^{\left(  k\right)  }\left(  \frac{r}{2}\right)  \right\vert
\leq C_{k}\left\vert h_{1}^{\left(  k\right)  }\left(  r\right)  \right\vert
,\text{ }k\geq0, \label{prop}%
\end{equation}
for all $r\geq r_{0}.$ We also suppose that
\[
\left\vert h_{1}^{\left(  k\right)  }\left(  r\right)  \right\vert \leq
Cr^{-1}h_{1}^{\left(  k-1\right)  }\left(  r\right)  ,\text{ }k\geq1,
\]
and that for some $N_{0}>0,$
\begin{equation}
\left\vert V\left(  r\right)  \right\vert ^{N_{0}}\leq C\left\vert
V^{\prime\prime}\left(  r\right)  \right\vert \label{cond5}%
\end{equation}

\begin{remark}
We observe that since $V$ and $\mathcal{V}$ are related by (\ref{vcall}), the
above assumptions on the potential are satisfied if Condition
\ref{ConidtionPotential} holds.
\end{remark}

We define%
\begin{equation}
\upsilon\left(  d\right)  =%
{\displaystyle\int_{0}^{\infty}}
e^{-2\eta^{2}}\eta^{d-2}d\eta, \label{ap142}%
\end{equation}
and%
\begin{equation}
C_{\pm}\left(  \xi\right)  =\xi^{\frac{d-1}{2}}%
{\displaystyle\int_{1}^{\xi-1}}
\left(  r\left(  \xi-r\right)  \right)  ^{-\frac{d-1}{2}}e^{\pm H\left(
r\right)  }dr. \label{ap167}%
\end{equation}
If $K^{\prime}=\lim_{r\rightarrow\infty}H^{\prime}\left(  r\right)  $ exists,
we set
\begin{equation}
\mathcal{I}=%
{\displaystyle\int}
e^{\left(  2-K^{\prime}\right)  \frac{\chi\cdot z}{\left\vert \chi\right\vert
}}Q^{2}\left(  z\right)  dz. \label{ap184}%
\end{equation}
If $V\left(  r\right)  =V_{+}\left(  r\right)  $ we define
\begin{equation}
U_{V}\left(  \xi\right)  =\left\{
\begin{array}
[c]{c}%
\dfrac{\kappa}{2}\mathcal{A}^{2}\upsilon\left(  d\right)  C_{+}\left(
\xi\right)  e^{-2\xi}\xi^{-\left(  d-1\right)  },\text{ if }H\left(  r\right)
=o\left(  r\right)  ,\\
\mathcal{I}V\left(  \xi\right)  ,\text{ if }0<cr\leq H\left(  r\right)
\leq2r.
\end{array}
\right.  \label{ap190}%
\end{equation}
If $V\left(  r\right)  =V_{-}\left(  r\right)  $ we set%
\begin{equation}
U_{V}\left(  \xi\right)  =\left\{
\begin{array}
[c]{c}%
\dfrac{1}{2}\left(
{\displaystyle\int}
\left(  \mathcal{A}^{2}V\left(  \left\vert z\right\vert \right)  +K\kappa
Q^{2}\left(  z\right)  \right)  e^{2\frac{\chi\cdot z}{\left\vert
\chi\right\vert }}dz\right)  e^{-2\xi}\xi^{-\left(  d-1\right)  },\text{ if
}d\geq4,\\
\dfrac{1}{2}\left(
{\displaystyle\int}
V\left(  \left\vert z\right\vert \right)  e^{2\frac{\chi\cdot z}{\left\vert
\chi\right\vert }}dz\right)  e^{-2\xi}\xi^{-\left(  d-1\right)  },\text{ if
}d=2,3,\text{ }r^{-\frac{d-1}{2}}e^{-H\left(  r\right)  }\in L^{1}\left(
[1,\infty)\right)  ,\\
\dfrac{\kappa}{2}\mathcal{A}^{2}\upsilon\left(  d\right)  C_{-}\left(
\xi\right)  e^{-2\xi}\xi^{-\left(  d-1\right)  },\text{ if }d=2,3,\text{
}r^{-\frac{d-1}{2}}e^{-H\left(  r\right)  }\notin L^{1}\left(  [1,\infty
)\right)  ,
\end{array}
\right.  \label{ap191}%
\end{equation}
where we denote by $K=\lim_{r\rightarrow\infty}e^{-H\left(  r\right)  }.$

\begin{lemma}
\label{L2 1}Let $V\in C^{\infty}$ have the form (\ref{potential}) where
$H,H^{\prime}$ are monotone and $H^{\prime\prime}\left(  r\right)  $ is
bounded. If $V\left(  r\right)  =V_{+}\left(  r\right)  \ $let $H\left(
r\right)  <2r.$ Then, the asymptotics%
\begin{equation}
\mathcal{J}\left(  \chi\right)  =-\frac{\tilde{\chi}}{\left\vert \tilde{\chi
}\right\vert }\left(  1+o\left(  1\right)  \right)  U_{V}^{\prime}\left(
\left\vert \tilde{\chi}\right\vert \right)  , \label{l2}%
\end{equation}
as $\left\vert \tilde{\chi}\right\vert \rightarrow\infty$ is true. If
$V=V^{\left(  2\right)  },$ then%
\begin{equation}
\mathcal{J}\left(  \chi\right)  =-\frac{\chi}{\left\vert \chi\right\vert
}\left(  \left(  \int Q^{2}\left(  z\right)  dz\right)  V^{\prime}\left(
\left\vert \chi\right\vert \right)  +\mathbf{r}\left(  \left\vert
\chi\right\vert \right)  +e^{-\frac{\left\vert \chi\right\vert }{2}}\right)  ,
\label{l2bis}%
\end{equation}
where
\[
\mathbf{r}\left(  \left\vert \chi\right\vert \right)  =O\left(  \left(
\left\vert h^{\prime}\left(  \left\vert \chi\right\vert \right)  \right\vert
\left(  \left\vert h^{\prime}\left(  \left\vert \chi\right\vert \right)
\right\vert ^{2}+\tfrac{\left\vert h^{\prime}\left(  \left\vert \chi
\right\vert \right)  \right\vert }{\left\vert \chi\right\vert }+\left\vert
h^{\prime\prime}\left(  \left\vert \chi\right\vert \right)  \right\vert
+\left\vert \chi\right\vert ^{-2}\right)  +\tfrac{\left\vert h^{\prime\prime
}\left(  \left\vert \chi\right\vert \right)  \right\vert }{\left\vert
\chi\right\vert }+\left\vert h^{\prime\prime\prime}\left(  \left\vert
\chi\right\vert \right)  \right\vert \right)  V\left(  \left\vert
\chi\right\vert \right)  \right)  .
\]

\end{lemma}

\begin{proof}
See Lemmas \ref{L3} and \ref{L4} in Section \ref{AppendixA}.
\end{proof}

In the next lemma we compare the potential with the solitary wave $Q.$ We
denote
\begin{equation}
\Theta\left(  \left\vert \chi\right\vert \right)  =\Theta_{V}\left(
\left\vert \chi\right\vert \right)  =\left\{
\begin{array}
[c]{c}%
e^{-\left\vert \chi\right\vert }\left(  1+\left\vert \chi\right\vert \right)
^{-\frac{\left(  d-1\right)  }{2}},\text{ if }V\left(  r\right)
=V_{-}^{\left(  1\right)  }\\
\left\vert V\left(  \left\vert \chi\right\vert \right)  \right\vert ,\text{
\ if }V\left(  r\right)  =V_{+}^{\left(  1\right)  }.
\end{array}
\right.  \label{theta}%
\end{equation}
Observe that%
\begin{equation}
\left\vert U_{V}\left(  \left\vert \chi\right\vert \right)  \right\vert \leq
C\Theta\left(  \left\vert \chi\right\vert \right)  . \label{ap246}%
\end{equation}
We have the following result.

\begin{lemma}
\label{L2}Let $V\in C^{\infty}$ be the form (\ref{potential}) and in addition
may be represented as (\ref{potential1}). Then, there is $C\left(  V\right)
>0$ such that for any $\left\vert \tilde{\chi}\right\vert \geq C\left(
V\right)  $ the following hold. The estimate
\begin{equation}
\left\Vert \left(  \left(  \frac{d}{dr}\right)  ^{k}V\right)  \left(
\left\vert y+\tilde{\chi}\right\vert \right)  \tilde{Q}\left(  y\right)
\right\Vert _{L^{\infty}}\leq C_{k}\Theta\left(  \left\vert \tilde{\chi
}\right\vert \right)  ,\text{ }k\geq0,\text{ }C_{k}>0, \label{estp}%
\end{equation}
is valid for any $\tilde{Q}$ satisfying (\ref{ap29}). Moreover, in the case
when $V\left(  r\right)  =V^{\left(  2\right)  }\left(  r\right)  $ we have%
\begin{equation}
\left\Vert \left(  \left(  \frac{d}{dr}\right)  ^{k}V\right)  \left(
\left\vert y+\tilde{\chi}\right\vert \right)  e^{-\delta\left\vert
y\right\vert }\right\Vert _{L^{\infty}}\leq C_{k}\left(  \left\vert \left(
\left(  \frac{d}{dr}\right)  ^{k}V\right)  \left(  \left\vert \tilde{\chi
}\right\vert \right)  \right\vert +e^{-\frac{\delta\left\vert \tilde{\chi
}\right\vert }{2}}\right)  ,\text{ }\delta>0, \label{estp1}%
\end{equation}
for any $k\geq0.$
\end{lemma}

\begin{proof}
Let us prove (\ref{estp}). If $V\left(  r\right)  =V_{-}^{\left(  1\right)
},$ using (\ref{ap29}) and (\ref{ap256})\ we have%
\begin{equation}
\left\vert \left(  \left(  \frac{d}{dr}\right)  ^{k}V\right)  \left(
\left\vert y+\tilde{\chi}\right\vert \right)  \tilde{Q}\left(  y\right)
\right\vert \leq C_{k}e^{-\left\vert \tilde{\chi}\right\vert }\left(
1+\left\vert y+\tilde{\chi}\right\vert \right)  ^{-\frac{1}{2}\left(
d-1\right)  }\left(  1+\left\vert y\right\vert \right)  ^{-\frac{1}{2}\left(
d-1\right)  }e^{-\left(  \left\vert y+\tilde{\chi}\right\vert -\left\vert
\tilde{\chi}\right\vert +\left\vert y\right\vert \right)  } \label{ap180}%
\end{equation}
and then%
\begin{equation}
\left\vert \left(  \left(  \frac{d}{dr}\right)  ^{k}V\right)  \left(
\left\vert y+\tilde{\chi}\right\vert \right)  \tilde{Q}\left(  y\right)
\right\vert \leq C_{k}e^{-\left\vert \tilde{\chi}\right\vert }\left(
1+\left\vert \tilde{\chi}\right\vert \right)  ^{-\frac{1}{2}\left(
d-1\right)  },\text{ }C_{k}>0\text{.} \label{ap259}%
\end{equation}
Suppose now that $V\left(  r\right)  =V_{+}^{\left(  1\right)  }.$ First, as
all the derivatives $H_{1}^{\left(  j\right)  },$ $j\geq1,$ are bounded we
have%
\begin{equation}
\left\vert \left(  \left(  \frac{d}{dr}\right)  ^{k}V\right)  \left(
\left\vert y+\tilde{\chi}\right\vert \right)  \tilde{Q}\left(  y\right)
\right\vert \leq C_{k}\left\vert V\left(  \left\vert y+\tilde{\chi}\right\vert
\right)  \tilde{Q}\left(  y\right)  \right\vert . \label{ap257}%
\end{equation}
If $\left\vert y+\tilde{\chi}\right\vert \leq\left\vert \tilde{\chi
}\right\vert ,$ since $H_{1}$ is monotone similarly to (\ref{ap180}) we get
\begin{equation}
\left\vert V\left(  \left\vert y+\tilde{\chi}\right\vert \right)  \tilde
{Q}\left(  y\right)  \right\vert \leq C\left\vert \left(  1+\left\vert
y+\tilde{\chi}\right\vert \right)  ^{-\frac{1}{2}\left(  d-1\right)
}e^{-\left\vert y+\tilde{\chi}\right\vert }\tilde{Q}\left(  y\right)
\right\vert e^{H_{1}\left(  \left\vert \tilde{\chi}\right\vert \right)  }\leq
C\left\vert V\left(  \left\vert \tilde{\chi}\right\vert \right)  \right\vert .
\label{ap258}%
\end{equation}
If $\left\vert y+\tilde{\chi}\right\vert \geq\left\vert \tilde{\chi
}\right\vert ,$ we have $\left\vert y\right\vert ^{2}+2\tilde{\chi}\cdot
y\geq0.$ Then
\[
-\left(  \left\vert y+\tilde{\chi}\right\vert -\left\vert \tilde{\chi
}\right\vert \right)  +\left(  H_{1}\left(  \left\vert y+\tilde{\chi
}\right\vert \right)  -H_{1}\left(  \left\vert \tilde{\chi}\right\vert
\right)  \right)  =-\left(  1-H^{\prime}\left(  \xi\right)  \right)
\frac{\left\vert y\right\vert ^{2}+2\tilde{\chi}\cdot y}{\left\vert
y+\tilde{\chi}\right\vert +\left\vert \tilde{\chi}\right\vert },
\]
for some $\xi\in\left[  \left\vert \tilde{\chi}\right\vert ,\left\vert
y+\tilde{\chi}\right\vert \right]  .$ As $V$ is bounded, $H_{1}\geq0$ and
$H_{1}^{\prime}$ is monotone, in particular $H_{1}^{\prime}\leq1+o\left(
1\right)  ,$ as $r\rightarrow\infty.$ Then $-\left(  1-H_{1}^{\prime}\left(
\xi\right)  \right)  \frac{\left\vert y\right\vert ^{2}+2\tilde{\chi}\cdot
y}{\left\vert y+\tilde{\chi}\right\vert +\left\vert \tilde{\chi}\right\vert
}\leq o\left(  1\right)  \left\vert y\right\vert ,$ as $\left\vert \tilde
{\chi}\right\vert \rightarrow\infty,$ and hence by (\ref{ap180})%
\begin{equation}
\left\vert V\left(  \left\vert y+\tilde{\chi}\right\vert \right)  \tilde
{Q}\left(  y\right)  \right\vert \leq C\left\vert V\left(  \left\vert
\tilde{\chi}\right\vert \right)  \right\vert . \label{ap4}%
\end{equation}
Therefore, using (\ref{ap258}) and (\ref{ap4}) in (\ref{ap257}) we deduce%
\begin{equation}
\left\vert \left(  \left(  \frac{d}{dr}\right)  ^{k}V\right)  \left(
\left\vert y+\tilde{\chi}\right\vert \right)  \tilde{Q}\left(  y\right)
\right\vert \leq C_{k}\left\vert V\left(  \left\vert \tilde{\chi}\right\vert
\right)  \right\vert . \label{ap260}%
\end{equation}
Relation (\ref{estp}) follows from (\ref{ap259}) and (\ref{ap260})$.$

We now prove (\ref{estp1})$.$ For $\left\vert y\right\vert \geq\frac
{\left\vert \tilde{\chi}\right\vert }{2},$ since $h_{1}\left(  r\right)
=o\left(  r\right)  $ we estimate%
\begin{equation}
\left\vert V^{\prime}\left(  \left\vert y+\tilde{\chi}\right\vert \right)
e^{-\delta\left\vert y\right\vert }\right\vert \leq C\left\Vert V^{\prime
}\right\Vert _{L^{\infty}}e^{-\frac{\delta\left\vert \tilde{\chi}\right\vert
}{2}}. \label{ap252}%
\end{equation}
Using that $h_{1}^{\prime\prime}$ is bounded we get\qquad%
\[
\left\vert h_{1}\left(  \left\vert \tilde{\chi}+y\right\vert \right)
-h_{1}\left(  \left\vert \tilde{\chi}\right\vert \right)  -h_{1}^{\prime
}\left(  \left\vert \tilde{\chi}\right\vert \right)  \left(  \frac{\tilde
{\chi}\cdot y}{\left\vert \tilde{\chi}\right\vert }\right)  \right\vert \leq
C\frac{1+\left\vert y\right\vert ^{4}}{\left\vert \tilde{\chi}\right\vert }%
\]
for $\left\vert y\right\vert \leq\frac{\left\vert \tilde{\chi}\right\vert }%
{2}.$ This implies
\begin{equation}
\left\vert V\left(  \left\vert \tilde{\chi}+y\right\vert \right)  -V\left(
\left\vert \tilde{\chi}\right\vert \right)  e^{-h_{1}^{\prime}\left(
\left\vert \tilde{\chi}\right\vert \right)  \frac{\tilde{\chi}\cdot
y}{\left\vert \tilde{\chi}\right\vert }}\right\vert \leq\left\vert V\left(
\left\vert \tilde{\chi}\right\vert \right)  \right\vert e^{-h_{1}^{\prime
}\left(  \left\vert \tilde{\chi}\right\vert \right)  \frac{\tilde{\chi}\cdot
y}{\left\vert \tilde{\chi}\right\vert }}\frac{1+\left\vert z\right\vert ^{4}%
}{\left\vert \tilde{\chi}\right\vert },\text{ as }\left\vert \tilde{\chi
}\right\vert \rightarrow\infty. \label{ap248}%
\end{equation}
Since $h_{1}\left(  r\right)  =o\left(  r\right)  $ and $h_{1}^{\prime}$ is
monotone $\left\vert h_{1}^{\prime}\left(  r\right)  \right\vert \leq
\frac{\delta}{2},$ for all $r\ $sufficiently large$.$ As for $\left\vert
y\right\vert \leq\frac{\left\vert \tilde{\chi}\right\vert }{2},$ $\left\vert
\tilde{\chi}+y\right\vert \geq\frac{\left\vert \tilde{\chi}\right\vert }{2},$
using that $h_{1}^{\prime}$ is monotone we get
\begin{equation}
\left\vert V\left(  \left\vert \tilde{\chi}+y\right\vert \right)
e^{-\delta\left\vert y\right\vert }\right\vert \leq C\left\vert V\left(
\left\vert \tilde{\chi}\right\vert \right)  e^{-\frac{\delta\left\vert
y\right\vert }{4}}\right\vert . \label{ap261}%
\end{equation}
Then, as for $\left\vert y\right\vert \leq\frac{\left\vert \tilde{\chi
}\right\vert }{2},$ $\left\vert \tilde{\chi}+y\right\vert \geq\frac{\left\vert
\tilde{\chi}\right\vert }{2},$ using that $h_{1}^{\left(  k\right)  }$ are
monotone and relations (\ref{ap261}), (\ref{prop}) we deduce%
\[
\left\vert \left(  \left(  \frac{d}{dr}\right)  ^{k}V\right)  \left(
\left\vert y+\tilde{\chi}\right\vert \right)  e^{-\delta\left\vert
y\right\vert }\right\vert \leq C_{k}\left\vert \left(  \left(  \frac{d}%
{dr}\right)  ^{k}V\right)  \left(  \frac{\left\vert \tilde{\chi}\right\vert
}{2}\right)  e^{h_{1}\left(  \frac{\left\vert \tilde{\chi}\right\vert }%
{2}\right)  }V\left(  \left\vert \tilde{\chi}\right\vert \right)  \right\vert
\leq C_{k}\left\vert \left(  \left(  \frac{d}{dr}\right)  ^{k}V\right)
\left(  \left\vert \tilde{\chi}\right\vert \right)  \right\vert .
\]
Combining the last estimate with (\ref{ap252}), we get (\ref{estp1}).
\end{proof}

\subsection{Refined approximation.}

Observe that $\left\vert \mathcal{\tilde{E}}_{\operatorname*{apr}}\left(
Q\right)  \right\vert \leq C\Theta\left(  \left\vert \chi\right\vert \right)
.$ This bound on the error is not good enough to close the estimates in
Section \ref{Secproof}. We can improve the estimate on $\mathcal{\tilde{E}%
}_{\operatorname*{apr}}$ if we consider a refined approximate solution $W=Q+T$
and adjust the function $T$ and the modulation parameters $\Xi=\left(
\chi,\beta,\lambda\right)  \in\mathbb{R}^{d}\times\mathbb{R}^{d}%
\times\mathbb{R}$ and $\gamma\in\mathbb{R}$ in a suitable way. Indeed,
motivated by \cite{Krieger}, we search for a stationary function $W=W\left(
y,\Xi\left(  t\right)  \right)  .$ Since $W$ is stationary, the dependence of
$W$ on time is through the modulation vector $\Xi\left(  t\right)  .$ This
yields the relation
\[
\partial_{t}W=\dot{\chi}\cdot\nabla_{\chi}W+\dot{\beta}\cdot\nabla_{\beta
}W+\frac{\partial W}{\partial\lambda}\dot{\lambda}.
\]
Using the last identity we get%
\begin{equation}
\left.
\begin{array}
[c]{c}%
\mathcal{E}\left(  W\right)  =\Delta W-W+\left\vert W\right\vert
^{p-1}W+i\lambda^{2}\left(  \dot{\chi}\cdot\nabla_{\chi}W+\dot{\beta}%
\cdot\nabla_{\beta}W+\frac{\partial W}{\partial\lambda}\dot{\lambda}\right)
-i\lambda\dot{\lambda}\Lambda W-i\lambda\left(  \dot{\chi}-2\beta\right)
\cdot\nabla W\\
-\lambda^{3}\left(  \dot{\beta}\cdot y\right)  W+\lambda^{2}\left(
\dot{\gamma}+\frac{1}{\lambda^{2}}-\left\vert \beta\right\vert ^{2}-\dot
{\beta}\cdot\chi\right)  W+\lambda^{2}V\left(  \left\vert y+\frac{\chi
}{\lambda}\right\vert \right)  W.
\end{array}
\right.  \label{E2}%
\end{equation}
Let $M=M\left(  \Xi\left(  t\right)  \right)  $ and $B=B\left(  \Xi\left(
t\right)  \right)  $ be modulation equations (to be defined) for the scaling
and speed parameters $\dot{\lambda}$ and $\dot{\beta}.$ Then, equation
(\ref{E2}) takes the form%
\begin{equation}
\left.  \mathcal{E}\left(  W\right)  =\mathcal{E}_{\operatorname*{apr}}\left(
W\right)  +R\left(  W\right)  \right.  \label{E6}%
\end{equation}
where%
\begin{equation}
\left.
\begin{array}
[c]{c}%
\mathcal{E}_{\operatorname*{apr}}\left(  W\right)  =\Delta W-W+\left\vert
W\right\vert ^{p-1}W+\lambda^{2}V\left(  \left\vert y+\frac{\chi}{\lambda
}\right\vert \right)  W\\
+i\lambda^{2}\left(  2\beta\cdot\nabla_{\chi}W+B\cdot\nabla_{\beta}%
W+\frac{\partial W}{\partial\lambda}\dot{\lambda}\right)  -i\lambda M\Lambda
W-\lambda^{3}\left(  B\cdot y\right)  W
\end{array}
\right.  \label{E8}%
\end{equation}
and%
\begin{equation}
\left.
\begin{array}
[c]{c}%
R\left(  W\right)  =i\lambda^{2}\left(  \left(  \dot{\chi}-2\beta\right)
\cdot\nabla_{\chi}W+\left(  \dot{\beta}-B\right)  \cdot\nabla_{\beta}W+\left(
\dot{\lambda}-M\right)  \frac{\partial W}{\partial\lambda}\right)
-i\lambda\left(  \dot{\lambda}-M\right)  \Lambda W-i\lambda\left(  \dot{\chi
}-2\beta\right)  \cdot\nabla W\\
-\lambda^{3}\left(  \left(  \dot{\beta}-B\right)  \cdot y\right)
W+\lambda^{2}\left(  \dot{\gamma}+\frac{1}{\lambda^{2}}-\left\vert
\beta\right\vert ^{2}-\dot{\beta}\cdot\chi\right)  W.
\end{array}
\right.  \label{E9}%
\end{equation}
Observe that $\mathcal{E}_{\operatorname*{apr}}\left(  Q\right)
=\mathcal{\tilde{E}}_{\operatorname*{apr}}\left(  Q\right)  -i\lambda M\Lambda
Q.$ Introducing $W=Q+T,$ into (\ref{E8}) we get%
\begin{equation}
\left.
\begin{array}
[c]{c}%
\mathcal{E}_{\operatorname*{apr}}\left(  Q+T\right)  =\Delta T-T+\frac{p+1}%
{2}Q^{p-1}T+\frac{p-1}{2}Q^{p-1}\overline{T}\\
+\lambda^{2}V\left(  \left\vert y+\frac{\chi}{\lambda}\right\vert \right)
Q-i\lambda M\Lambda Q-\lambda^{3}\left(  B\cdot y\right)  Q\\
+\lambda^{2}V\left(  \left\vert y+\frac{\chi}{\lambda}\right\vert \right)
T-i\lambda M\Lambda T-\lambda^{3}\left(  B\cdot y\right)  T\\
+i\lambda^{2}\left(  2\beta\cdot\nabla_{\chi}T+B\cdot\nabla_{\beta}%
T+M\frac{\partial T}{\partial\lambda}\right)  +\mathcal{N}\left(  T\right)  ,
\end{array}
\right.  \label{E7}%
\end{equation}
with%
\[
\mathcal{N}\left(  T\right)  =\left\vert Q+T\right\vert ^{p-1}\left(
Q+T\right)  -\left(  Q^{p}+\frac{p+1}{2}Q^{p-1}T+\frac{p-1}{2}Q^{p-1}%
\overline{T}\right)  .
\]
Let
\[
L_{+}u=-\Delta u+u-pQ^{p-1}u
\]
and%
\[
L_{-}u=-\Delta u+u-Q^{p-1}u.
\]
Decomposing $T=$ $T_{1}+iT_{2},$ with $T_{1},T_{2}$ real we get%
\begin{equation}
\left.
\begin{array}
[c]{c}%
\mathcal{E}_{\operatorname*{apr}}\left(  Q+T\right)  =-L_{+}T_{1}+\lambda
^{2}V\left(  \left\vert y+\frac{\chi}{\lambda}\right\vert \right)
T_{1}+\lambda^{2}V\left(  \left\vert y+\frac{\chi}{\lambda}\right\vert
\right)  Q-\lambda^{3}\left(  B\cdot y\right)  Q\\
-i\left(  L_{-}T_{2}\right)  +i\lambda^{2}V\left(  \left\vert y+\frac{\chi
}{\lambda}\right\vert \right)  T_{2}-i\lambda M\Lambda Q+2i\lambda^{2}%
\beta\cdot\nabla_{\chi}T_{1}+R_{1},
\end{array}
\right.  \label{eq3}%
\end{equation}
where%
\begin{equation}
R_{1}=-i\lambda M\Lambda T-\lambda^{3}\left(  B\cdot y\right)  T+i\lambda
^{2}\left(  2i\beta\cdot\nabla_{\chi}T_{2}+B\cdot\nabla_{\beta}T+M\frac
{\partial T}{\partial\lambda}\right)  +\mathcal{N}\left(  T\right)  .
\label{eq4}%
\end{equation}
As a first step, we want to adjust the approximate modulation parameters $B$
and $M$ in such way that we can solve the equations%
\begin{equation}
\left(  L_{+}-\lambda^{2}V\left(  \left\vert y+\frac{\chi}{\lambda}\right\vert
\right)  \right)  T_{1}=-\lambda^{3}\left(  B\cdot y\right)  Q+\lambda
^{2}V\left(  \left\vert y+\frac{\chi}{\lambda}\right\vert \right)  Q,
\label{eq5}%
\end{equation}%
\begin{equation}
\left(  L_{-}-\lambda^{2}V\left(  \left\vert y+\frac{\chi}{\lambda}\right\vert
\right)  \right)  T_{2}=-\lambda M\Lambda Q+2\lambda^{2}\beta\cdot\nabla
_{\chi}T_{1} \label{eq6}%
\end{equation}
and $R_{1}$ has a better decay than $\mathcal{E}_{\operatorname*{apr}}\left(
Q\right)  ,$ as $\left\vert \chi\right\vert \rightarrow\infty.$ In this way,
we get a first order approximation. Then, from (\ref{E7}), we recursively
construct higher order approximations. This will be done separately for fast
and slow decaying potentials in Lemmas \ref{Lemmaapp} and \ref{Lemmaapp1}
below, respectively. Before we present this construction, we prepare two
results that are involved in the construction. We denote by $\mathcal{L}%
:H^{1}\rightarrow H^{-1}$ the linearized operator for the equation
(\ref{eqsoliton}) around $Q:$%
\begin{equation}
\mathcal{L}f:=-\Delta f+f-\frac{p+1}{2}Q^{p-1}f-\frac{p-1}{2}Q^{p-1}%
{\overline{f}},\text{ \ \ }f\in H^{1}. \label{L}%
\end{equation}
Representing $f\in H^{1}$ as $f=h+ig,$ with real $h\ $and $g$, we have%
\[
\mathcal{L}f=L_{+}h+iL_{-}g
\]
Then
\begin{equation}
\left(  \mathcal{L}f,f\right)  =\left(  L_{+}h,h\right)  +\left(
L_{-}g,g\right)  , \label{tw78}%
\end{equation}
for real functions $h,g\in H^{1}.$ Let us denote the perturbed operator
\begin{equation}
\mathcal{L}_{V}=\mathcal{L}-\lambda^{2}V\left(  \left\vert y+\tilde{\chi
}\right\vert \right)  . \label{Lv}%
\end{equation}
In order to solve (\ref{eq5}) and (\ref{eq6}), we need the following
invertibility result for the operators $\mathcal{L}_{V}$.

\begin{lemma}
\label{Linv}Suppose that $1<p<1+\frac{4}{d}$ and $V\in C^{\infty},$
$\left\vert V^{\left(  k\right)  }\left(  x\right)  \right\vert \rightarrow0,$
as $\left\vert x\right\vert \rightarrow\infty,$ for any $k\geq0.$ Let
$\lambda\geq\lambda_{0}>0\ $be such that $\lambda^{2}\sup_{r\in\mathbb{R}%
}V\left(  r\right)  <1.$ Then, there is $C\left(  V\right)  >0$ such that for
any $\left\vert \chi\right\vert \geq C\left(  V\right)  $
\begin{equation}
\left\Vert \mathcal{L}_{V}f\right\Vert _{H^{-1}}^{2}\geq c\left\Vert
f\right\Vert _{H^{1}}^{2}-\frac{1}{c}\left(  \left\vert \left(  f,iQ\right)
\right\vert ^{2}+\left\vert \left(  f,\nabla Q\right)  \right\vert
^{2}\right)  , \label{ineqpert}%
\end{equation}
with some $c>0$ and%
\begin{equation}
\left(  \mathcal{L}_{V}f,f\right)  \geq c\left\Vert f\right\Vert _{H^{1}}%
^{2}-\frac{1}{c}\left(  \left(  f,Q\right)  ^{2}+\left\vert \left(
f,xQ\right)  \right\vert ^{2}+\left(  f,i\Lambda Q\right)  ^{2}\right)  .
\label{ineqpert1}%
\end{equation}
Moreover, for any $\left\vert \chi\right\vert \geq C\left(  V\right)  $ and
$F\in H^{1}\cap C^{\infty}\ $such that $\left(  F,\nabla Q\right)  =0$ the
equation $L_{+}u-\lambda^{2}V\left(  y+\tilde{\chi}\right)  u=F,$ $\left(
u,\nabla Q\right)  =0,$ has a real-valued solution $u\in H^{1}\cap C^{\infty}%
$.\ If $\left(  F,Q\right)  =0,$ there is a solution $w\in H^{1}\cap
C^{\infty}$ to $L_{-}w-\lambda^{2}V\left(  y+\tilde{\chi}\right)  w=F$,
$\left(  w,Q\right)  =0.$ Furthermore, if $\left\vert F\left(  x\right)
\right\vert \leq Ce^{-\eta\left\vert x\right\vert },$ for some $0<\eta<1,$
then
\begin{equation}
\left\vert u\left(  x\right)  \right\vert +\left\vert w\left(  x\right)
\right\vert \leq Ce^{-\eta\left\vert x\right\vert }. \label{exp}%
\end{equation}

\end{lemma}

\begin{proof}
See Section \ref{Invertibility}.
\end{proof}

We also present a lemma that follows from the properties of the Bessel
potential $\left(  1-\Delta\right)  ^{-1}.$

\begin{lemma}
\label{L8}Suppose that $T$ solves
\begin{equation}
\left(  L_{+}-\lambda^{2}V\left(  \left\vert \cdot+\frac{\chi}{\lambda
}\right\vert \right)  \right)  T=f,\text{ }\left(  T,\nabla Q\right)  =0,
\label{eq18}%
\end{equation}
with $f$ satisfying
\begin{equation}
\left\vert e^{-\delta\left\vert z\right\vert }f\left(  z\right)  \right\vert
\leq A_{0}\left(  \left\vert \chi\right\vert \right)  , \label{ap290}%
\end{equation}
for some $A_{0}\left(  \left\vert \chi\right\vert \right)  >0$ and
$0<\delta<1.$ Then, the estimate
\begin{equation}
\left\Vert e^{-\delta\left\vert \cdot\right\vert }T\left(  \cdot\right)
\right\Vert _{L^{\infty}}\leq C\left(  \delta,\delta^{\prime}\right)  \left(
A_{0}\left(  \left\vert \chi\right\vert \right)  +\left\Vert T\right\Vert
_{L^{\infty}}\Theta\left(  \left\vert \tilde{\chi}\right\vert \right)
^{\delta^{\prime}}\right)  , \label{ap291}%
\end{equation}
with some $C\left(  \delta,\delta^{\prime}\right)  >0$ and $\delta^{\prime
}<\delta$ holds.
\end{lemma}

\begin{proof}
See Appendix.
\end{proof}

We are in position to prove an approximation result for (\ref{appeq}). For a
vector of multi-indices $K=\left(  k,k_{1},k_{2},k_{3}\right)  ,$ let us
denote
\[
\mathcal{D}=\partial_{\chi}^{k}\partial_{\lambda}^{k_{1}}\partial_{y}^{k_{2}%
}\partial_{\beta}^{k_{3}}\text{, \ }\mathcal{D}_{1}=\partial_{\chi}%
^{k}\partial_{\lambda}^{k_{1}}\partial_{\beta}^{k_{3}},\text{ \ }%
\mathcal{D}_{2}=\partial_{\chi}^{k}\partial_{\lambda}^{k_{1}}\partial
_{y}^{k_{2}}.
\]
Also let%
\[
\mathbf{Y}=\inf_{0<\upsilon<1}\inf_{0<\delta<1}\left\{  C\left(
\upsilon,\delta\right)  \Theta^{\left(  1-\delta\right)  }\left(
\tfrac{\left\vert \chi\right\vert }{\lambda}\right)  e^{-\upsilon
\delta\left\vert y\right\vert }\right\}  ,\text{ }0<C\left(  \upsilon
,\delta\right)  <\infty.
\]
where $0<C\left(  \upsilon,\delta\right)  <\infty$ is such that $C\left(
\upsilon,\delta\right)  \leq C_{\upsilon_{0},\delta_{0}},$ for all $\left\vert
\upsilon\right\vert \leq\left\vert \upsilon_{0}\right\vert ,$ $\left\vert
\delta\right\vert \leq\left\vert \delta_{0}\right\vert $ and $C\left(
\upsilon,\delta\right)  \rightarrow\infty,$ as $\upsilon,\delta\rightarrow1.$
Set
\begin{equation}
p_{1}=\min\{p,2\}. \label{p}%
\end{equation}
We have the following.

\begin{lemma}
\label{Lemmaapp}\bigskip Let $\Xi=\left(  \chi,\beta,\lambda\right)
\in\mathbb{R}^{d}\times\mathbb{R}^{d}\times\mathbb{R}^{+}$ be a vector of
parameters with $\lambda\geq\lambda_{0}>0\ $such that $\lambda^{2}\sup
_{r\in\mathbb{R}}V\left(  r\right)  <1.$ There is $C\left(  V\right)  >0$ such
that for any $\left\vert \chi\right\vert \geq C\left(  V\right)  $ the
following holds. For any $n\geq1,$ there are $\boldsymbol{T}^{\left(
j\right)  },B_{j},M_{j}\in L^{\infty}\left(  \mathbb{R}^{d}\right)  ,$
$j=1,...,n,$ such that%
\begin{equation}
\left.
\begin{array}
[c]{c}%
\left\vert \mathcal{D}\boldsymbol{T}^{\left(  j\right)  }\right\vert \leq
C_{K}\left\langle \lambda\right\rangle ^{i_{K}}\left\langle \lambda
^{-1}\right\rangle ^{l_{K}}\left\langle \left\vert \beta\right\vert
\right\rangle ^{m_{K}}\mathbf{Y}\left(  \left\vert \beta\right\vert
+\Theta\left(  \tfrac{\left\vert \chi\right\vert }{\lambda}\right)  \right)
^{\left(  p_{1}-1\right)  \left(  j-1\right)  },\\
\left\vert \mathcal{D}_{1}B_{j}\right\vert +\left\vert \mathcal{D}_{1}%
M_{j}\right\vert \leq C_{K}\left\langle \lambda\right\rangle ^{i_{K}%
}\left\langle \lambda^{-1}\right\rangle ^{l_{K}}\left\langle \left\vert
\beta\right\vert \right\rangle ^{m_{K}}\Theta\left(  \tfrac{\left\vert
\chi\right\vert }{\lambda}\right)  \left(  \left\vert \beta\right\vert
+\Theta\left(  \tfrac{\left\vert \chi\right\vert }{\lambda}\right)  \right)
^{\left(  p_{1}-1\right)  \left(  j-1\right)  },
\end{array}
\right.  \label{ap14}%
\end{equation}
with $\left\vert K\right\vert \leq2$ and $\left\vert k\right\vert
+k_{1}+\left\vert k_{3}\right\vert \leq1,$ and some $i_{K},l_{K},m_{K}\geq0$
and $C_{K},\delta_{K}>0$. In addition, $B_{j}$ and $M_{j},$ $j=1,...,n,$
satisfy for all $\left\vert k\right\vert +k_{1}\leq1$
\begin{equation}
\left.
\begin{array}
[c]{c}%
\left\vert \partial_{\chi}^{k}\partial_{\lambda}^{k_{1}}B_{1}\right\vert
+\left\vert \partial_{\chi}^{k}\partial_{\lambda}^{k_{1}}M_{1}\right\vert \leq
C_{K}\left\langle \lambda\right\rangle ^{i_{K}}\left\langle \lambda
^{-1}\right\rangle ^{l_{K}}\left(  \left\vert \beta\right\vert \Theta
+\left\vert U_{V}\left(  \left\vert \tfrac{\chi}{\lambda}\right\vert \right)
\right\vert \right)  ,\\
\left\vert \partial_{\beta}M_{1}\right\vert \leq C_{K}\left\langle
\lambda\right\rangle ^{i_{K}}\left\langle \lambda^{-1}\right\rangle ^{l_{K}%
}\left\langle \left\vert \beta\right\vert \right\rangle ^{m_{K}}\left(
\left\vert U_{V}\left(  \left\vert \tfrac{\chi}{\lambda}\right\vert \right)
\right\vert ^{3/4}+\Theta^{3/2}\left(  \left\vert \tfrac{\chi}{\lambda
}\right\vert \right)  \right)  ,\\
\left\vert \partial_{\chi}^{k}\partial_{\lambda}^{k_{1}}B_{j}\right\vert
+\left\vert \partial_{\chi}^{k}\partial_{\lambda}^{k_{1}}M_{j}\right\vert \leq
C_{K}\left\langle \lambda\right\rangle ^{i_{K}}\left\langle \lambda
^{-1}\right\rangle ^{l_{K}}\left(  \left(  \left(  \left\vert \beta\right\vert
+\Theta\right)  \Theta\right)  ^{p_{1}}+\left\vert U_{V}\left(  \left\vert
\tfrac{\chi}{\lambda}\right\vert \right)  \right\vert ^{p_{1}}\right)  ,\text{
}j\geq2.
\end{array}
\right.  \label{BM}%
\end{equation}
For the approximation $Q+T,$ the error $\mathcal{E}_{\operatorname*{apr}%
}\left(  Q+T\right)  $ satisfies the estimate
\begin{equation}
\left.  \left\vert \mathcal{E}_{\operatorname*{apr}}\left(  Q+T\right)
\right\vert \leq C\left\langle \lambda\right\rangle ^{i}\left\langle
\lambda^{-1}\right\rangle ^{l}\left\langle \left\vert \beta\right\vert
\right\rangle ^{m}\mathbf{Y}\left(  \left(  \left\vert \beta\right\vert
+\Theta\left(  \tfrac{\left\vert \chi\right\vert }{\lambda}\right)  \right)
^{A\left(  n\right)  }+\left\vert U_{V}\left(  \left\vert \tfrac{\chi}%
{\lambda}\right\vert \right)  \right\vert \right)  ,\right.  \label{eps}%
\end{equation}
for some $i,l,m\geq0$ and $C>0,$ where $A\left(  n\right)  =\min\{n\left(
p_{1}-1\right)  ,2\}.$
\end{lemma}

\begin{proof}
First, we solve equation (\ref{eq5}) given by
\begin{equation}
\left(  L_{+}-\lambda^{2}V\left(  \left\vert y+\frac{\chi}{\lambda}\right\vert
\right)  \right)  T_{1}=f_{1} \label{eq7}%
\end{equation}
with%
\begin{equation}
f_{1}=-\lambda^{3}\left(  B\cdot y\right)  Q+\lambda^{2}V\left(  \left\vert
y+\frac{\chi}{\lambda}\right\vert \right)  Q. \label{f1}%
\end{equation}
Taking $B_{1}=B\left(  \chi,\lambda\right)  ,$ with $B\left(  \chi
,\lambda\right)  $ defined by (\ref{ap17}) we assure that the right-hand side
of (\ref{eq7}) is orthogonal to $\nabla Q.$ Noting that
\begin{equation}
\left.
\begin{array}
[c]{c}%
\partial_{\chi}^{k}V\left(  \left\vert y+\frac{\chi}{\lambda}\right\vert
\right)  =\lambda^{-\left\vert k\right\vert }\partial_{y}^{k}V\left(
\left\vert y+\frac{\chi}{\lambda}\right\vert \right) \\
\\
\partial_{\lambda}V\left(  \left\vert y+\frac{\chi}{\lambda}\right\vert
\right)  =\partial_{\lambda}\mathcal{V}\left(  \left\vert \lambda
y+\chi\right\vert \right)  =\left(  \lambda y\cdot\nabla_{y}\mathcal{V}%
\right)  \left(  \left\vert \lambda y+\chi\right\vert \right)  =y\cdot
\nabla_{y}V\left(  \left\vert y+\frac{\chi}{\lambda}\right\vert \right)
\end{array}
\right.  \label{ap265}%
\end{equation}
(recall that $\mathcal{V}$ is defined by (\ref{vcall})) we have
\begin{equation}
\left.  \left\vert \partial_{\chi}^{k}\partial_{\lambda}^{k_{1}}%
{\displaystyle\int}
V\left(  \left\vert y+\frac{\chi}{\lambda}\right\vert \right)  \nabla
Q^{2}\left(  y\right)  dy\right\vert \leq C_{k,k_{1}}\lambda^{-\left\vert
k\right\vert }\left\vert
{\displaystyle\int}
V\left(  \left\vert y+\frac{\chi}{\lambda}\right\vert \right)  \left(
\nabla_{y}\cdot y\right)  ^{k_{1}}\partial_{y}^{k}\left(  \nabla Q^{2}\left(
y\right)  \right)  dy\right\vert .\right.  \label{ap263}%
\end{equation}
Then, using (\ref{l2}) we estimate%
\[
\left\vert \partial_{\chi}^{k}\partial_{\lambda}^{k_{1}}\int V\left(
\left\vert y+\frac{\chi}{\lambda}\right\vert \right)  \nabla Q^{2}\left(
y\right)  dy\right\vert \leq C_{k,k_{1}}\lambda^{-\left\vert k\right\vert
}\left\vert U_{V}^{\prime}\left(  \left\vert \frac{\chi}{\lambda}\right\vert
\right)  \right\vert .
\]
Thus, by (\ref{estp}) we have
\begin{equation}
\left\vert \partial_{\chi}^{k}\partial_{\lambda}^{k_{1}}B_{1}\right\vert \leq
C_{k,k_{1}}\left\langle \lambda^{-1}\right\rangle ^{l_{k,k_{1}}}\left\vert
U_{V}^{\prime}\left(  \left\vert \frac{\chi}{\lambda}\right\vert \right)
\right\vert , \label{b1}%
\end{equation}
which imply (\ref{ap14}), (\ref{BM}) for $B_{1}.$ Also, by (\ref{estp}) we
get
\begin{align*}
\left\vert \mathcal{D}_{2}\left(  V\left(  \left\vert y+\frac{\chi}{\lambda
}\right\vert \right)  \right)  Q\left(  y\right)  \right\vert  &  \leq
C_{1,K}\left\langle \lambda\right\rangle ^{i_{1,K}}\left\langle \lambda
^{-1}\right\rangle ^{l_{1,K}}\inf_{0<\delta<1}\left\{  \left\vert V\left(
\left\vert y+\frac{\chi}{\lambda}\right\vert \right)  Q\left(  y\right)
\right\vert ^{\left(  1-\delta\right)  }Q^{\delta}\left(  y\right)  \right\}
\\
&  \leq C_{1,K}\left\langle \lambda\right\rangle ^{i_{1,K}}\left\langle
\lambda^{-1}\right\rangle ^{l_{1,K}}\mathbf{Y.}%
\end{align*}
Hence, from (\ref{f1}) we see that
\begin{equation}
\left.  \left\vert \mathcal{D}_{2}f_{1}\right\vert \leq C_{K}\left\langle
\lambda\right\rangle ^{i_{K}}\left\langle \lambda^{-1}\right\rangle ^{l_{K}%
}\mathbf{Y}.\right.  \label{ap234}%
\end{equation}
for all $\left\vert k\right\vert +k_{1}+\left\vert k_{2}\right\vert \geq0.$ By
Lemma \ref{Linv}, there exists a solution $T_{1}\in H^{1}$ satisfying%
\begin{equation}
\left\vert T_{1}\right\vert \leq C_{0}\left\langle \lambda\right\rangle
^{i_{0}}\left\langle \lambda^{-1}\right\rangle ^{l_{0}}\mathbf{Y}.
\label{ap235}%
\end{equation}
Differentiating equation (\ref{eq7}) with respect to $\lambda$ we get%
\begin{equation}
\left(  L_{+}-\lambda^{2}V\left(  \left\vert y+\frac{\chi}{\lambda}\right\vert
\right)  \right)  \partial_{\lambda}T_{1}=\partial_{\lambda}f_{1}%
+\partial_{\lambda}\left(  \lambda^{2}V\left(  \left\vert y+\frac{\chi
}{\lambda}\right\vert \right)  \right)  T_{1}. \label{ap236}%
\end{equation}
We write the above equation as
\begin{equation}
\partial_{\lambda}T_{1}=\left(  -\Delta+1\right)  ^{-1}\left(  \left(
pQ^{p-1}+\lambda^{2}V\left(  \left\vert \cdot+\frac{\chi}{\lambda}\right\vert
\right)  \right)  \partial_{\lambda}T_{1}+\partial_{\lambda}f_{1}%
+\partial_{\lambda}\left(  \lambda^{2}V\left(  \left\vert \cdot+\frac{\chi
}{\lambda}\right\vert \right)  \right)  T_{1}\right)  . \label{ap275}%
\end{equation}
Using the explicit expression for the kernel of the Helmholtz operator
$-\Delta+1,$ we estimate%
\[
\left\vert \partial_{\lambda}T_{1}\left(  y\right)  \right\vert \leq
C\Theta^{-\left(  1-\delta\right)  }e^{\delta\left\vert y\right\vert }%
\int_{\mathbb{R}^{d}}\frac{e^{-\left\vert y-z\right\vert }}{\left\vert
y-z\right\vert ^{\frac{d-1}{2}}}\left(  Q^{p-1}\left(  z\right)  \left\vert
\partial_{\lambda}T_{1}\left(  z\right)  \right\vert +\left\vert
\partial_{\lambda}f_{1}\left(  z\right)  \right\vert +\left\vert
\partial_{\lambda}\left(  \lambda^{2}V\left(  \left\vert z+\frac{\chi}%
{\lambda}\right\vert \right)  \right)  T_{1}\left(  z\right)  \right\vert
\right)  dz.
\]
Then, using (\ref{ineqpert}) to control the first term in the right-hand side
of the above equation, by (\ref{ap234}) and (\ref{ap235}) we have
\begin{equation}
\left\vert \partial_{\lambda}T_{1}\right\vert \leq C_{1}\left\langle
\lambda\right\rangle ^{i_{1}}\left\langle \lambda^{-1}\right\rangle ^{l_{1}%
}\mathbf{Y}. \label{ap237}%
\end{equation}
Moreover, from (\ref{ap275}), by the regularity properties of $-\Delta+1$ we
show that (\ref{ap237}) is valid for the derivatives $\partial_{\lambda
}\partial_{y}^{k_{2}}T_{1}$ for all multi-indices $k_{2}.$ Similarly we
estimate $\partial_{\chi}^{k}\partial_{y}^{k_{2}}T_{1}$ for $\left\vert
k\right\vert =1.$ By induction in $k,k_{1}$ we prove
\begin{equation}
\left\vert \mathcal{D}_{2}T_{1}\right\vert \leq C_{K}\left\langle
\lambda\right\rangle ^{i_{K}}\left\langle \lambda^{-1}\right\rangle ^{l_{K}%
}\mathbf{Y}, \label{ap238}%
\end{equation}
for all $\left\vert k\right\vert +k_{1}+\left\vert k_{2}\right\vert
\geq0.\qquad\qquad\qquad\qquad\qquad$

Next, we consider equation (\ref{eq6})%
\begin{equation}
\left(  L_{-}-\lambda^{2}V\left(  \left\vert y+\frac{\chi}{\lambda}\right\vert
\right)  \right)  T_{2}=f_{2} \label{eq8}%
\end{equation}
with%
\begin{equation}
f_{2}=-\lambda M\Lambda Q+2\lambda^{2}\beta\cdot\nabla_{\chi}T_{1} \label{f2}%
\end{equation}
We fix $M_{1}$ in such way that
\[
\left(  \lambda M_{1}\Lambda Q-2\lambda^{2}\beta\cdot\nabla_{\chi}%
T_{1},Q\right)  =0.
\]
That is
\begin{equation}
M_{1}=\frac{2\lambda}{\left(  \Lambda Q,Q\right)  }\left(  \beta\cdot
\nabla_{\chi}T_{1},Q\right)  . \label{m1}%
\end{equation}
Observe that $\left(  \Lambda Q,Q\right)  =\left(  \frac{2}{p-1}-\frac{d}%
{2}\right)  \left(  Q,Q\right)  \neq0.$ By (\ref{ap238}) we have%
\begin{equation}
\left\vert \mathcal{D}_{1}M_{1}\right\vert \leq C_{K}\left\langle
\lambda\right\rangle ^{i_{K}}\left\langle \lambda^{-1}\right\rangle ^{l_{K}%
}\omega_{K}\left(  \beta\right)  \Theta, \label{ap239}%
\end{equation}
for all $\left\vert K\right\vert \geq0,$ where%
\[
\omega_{K}\left(  \beta\right)  =\left\{
\begin{array}
[c]{c}%
\left\vert \beta\right\vert ,\text{ }\left\vert k_{3}\right\vert =0,\\
C_{K},\text{ }\left\vert k_{3}\right\vert \geq1.
\end{array}
\right.
\]
Then, (\ref{ap14}) and the relation in (\ref{BM}) for $M_{1}$ follow from
(\ref{ap239}). Observe by (\ref{estp}) that $f_{1}$ can be estimated as
\begin{equation}
\left\vert \mathcal{D}_{2}f_{1}\right\vert \leq C\left(  \inf_{0<\delta
<1}\left(  e^{\delta\left\vert y\right\vert }\left\vert U_{V}\left(
\left\vert \tfrac{\chi}{\lambda}\right\vert \right)  \right\vert
^{\frac{1+\delta}{2}}\right)  \right)  ,\text{ }\left\vert k\right\vert
+k_{1}+\left\vert k_{2}\right\vert \geq0. \label{ap276}%
\end{equation}
Then, using Lemma \ref{L8} and (\ref{ap238}) we get%
\begin{equation}
\left\Vert e^{-\delta\left\vert \cdot\right\vert }\mathcal{D}_{2}T_{1}\left(
\cdot\right)  \right\Vert _{L^{\infty}}\leq C_{K}\left(  \delta,\delta
^{\prime}\right)  \left(  \left\vert U_{V}\left(  \left\vert \tfrac{\chi
}{\lambda}\right\vert \right)  \right\vert ^{\frac{1+\delta}{2}}+\Theta\left(
\left\vert \tfrac{\chi}{\lambda}\right\vert \right)  ^{1+\delta^{\prime}%
}\right)  , \label{ap278}%
\end{equation}
with $0<\delta^{\prime}<\delta<1$, for all $\left\vert k\right\vert
+k_{1}+\left\vert k_{2}\right\vert \geq0.$ Using the last estimate with
$\delta>\delta^{\prime}=\frac{1}{2}$ in (\ref{m1}) we get the second relation
in (\ref{BM}). Similarly to (\ref{ap234}), from (\ref{f2}) we deduce
\begin{equation}
\left.  \left\vert \mathcal{D}f_{2}\right\vert \leq C_{K}\left\langle
\lambda\right\rangle ^{i_{K}}\left\langle \lambda^{-1}\right\rangle ^{l_{K}%
}\omega_{K}\left(  \beta\right)  \mathbf{Y}.\right.  \label{ap264}%
\end{equation}
for all $\left\vert K\right\vert \geq0.$ By Lemma \ref{Linv}, there exists a
solution $T_{2}\in H^{1}$ to (\ref{eq8}). Moreover, similarly to (\ref{ap238})
using (\ref{ap264}) we prove that
\begin{equation}
\left.  \left\vert \mathcal{D}T_{2}\right\vert \leq C_{K}\left\langle
\lambda\right\rangle ^{i_{K}}\left\langle \lambda^{-1}\right\rangle ^{l_{K}%
}\omega_{K}\left(  \beta\right)  \mathbf{Y}.\right.  \label{ap240}%
\end{equation}
for all $\left\vert K\right\vert \geq0.$ We put
\[
\boldsymbol{T}=T_{1}+iT_{2}.
\]
By (\ref{ap238}) and (\ref{ap240}) we get (\ref{ap14}) for $\boldsymbol{T}.$
Using (\ref{b1}), (\ref{ap238}), (\ref{ap239}), (\ref{ap240}) and
\begin{equation}
\left\vert \mathcal{N}\left(  T\right)  \right\vert \leq C\left\vert
T\right\vert ^{p_{1}}, \label{non2}%
\end{equation}
from (\ref{eq3}) we get (\ref{eps}) with $n=1.$ Therefore, we constructed the
first improved approximation $\boldsymbol{T}^{\left(  1\right)  }.$ Let us
construct $\boldsymbol{T}^{\left(  2\right)  }.$ We write $\boldsymbol{T}%
^{\left(  2\right)  }=T_{1}^{\left(  2\right)  }+iT_{2}^{\left(  2\right)  },$
and introduce $T=\boldsymbol{T}^{\left(  1\right)  }+\boldsymbol{T}^{\left(
2\right)  },$ $B=B_{1}+B_{2},$ $M=M_{1}+M_{2},$ into (\ref{E7}). Then, using
(\ref{eq7}) and (\ref{eq8}) and setting
\begin{equation}
\left.  \left(  L_{+}-\lambda^{2}V\left(  \left\vert y+\frac{\chi}{\lambda
}\right\vert \right)  \right)  T_{1}^{\left(  2\right)  }=-\lambda^{3}\left(
B_{2}\cdot y\right)  Q+f_{3}\right.  \label{eq9}%
\end{equation}
with%
\begin{equation}
\left.  f_{3}=\operatorname{Re}\mathcal{N}\left(  \boldsymbol{T}^{\left(
1\right)  }\right)  \right.  \label{ap279}%
\end{equation}
and%
\begin{equation}
\left(  L_{-}-\lambda^{2}V\left(  \left\vert y+\frac{\chi}{\lambda}\right\vert
\right)  \right)  T_{2}^{\left(  2\right)  }=-\lambda M_{2}\Lambda Q+f_{4}
\label{eq10}%
\end{equation}
with%
\begin{equation}
f_{4}=\operatorname{Im}\mathcal{N}\left(  \boldsymbol{T}^{\left(  1\right)
}\right)  , \label{ap280}%
\end{equation}
we derive%
\begin{equation}
\left.
\begin{array}
[c]{c}%
\mathcal{E}_{\operatorname*{apr}}\left(  Q+T\right)  =-i\lambda M\Lambda
T-\lambda^{3}\left(  B\cdot y\right)  T\\
+i\lambda^{2}\left(  2\beta\cdot\nabla_{\chi}\left(  iT_{2}+\boldsymbol{T}%
^{\left(  2\right)  }\right)  +B\cdot\nabla_{\beta}T+M\frac{\partial
T}{\partial\lambda}\right)  +\mathcal{N}\left(  \boldsymbol{T}^{\left(
1\right)  }+\boldsymbol{T}^{\left(  2\right)  }\right)  -\mathcal{N}\left(
\boldsymbol{T}^{\left(  1\right)  }\right)  .
\end{array}
\right.  \label{er1}%
\end{equation}
From (\ref{ap279}) and (\ref{ap280}), via (\ref{ap238}), (\ref{ap240}) and
(\ref{non2}) we get%
\begin{equation}
\left\vert \mathcal{D}f_{j}\right\vert \leq C\left\langle \lambda\right\rangle
^{i_{j,K}}\left\langle \lambda^{-1}\right\rangle ^{l_{j,K}}\mathbf{Y}^{p_{1}%
},\text{ }j=3,4, \label{ap241}%
\end{equation}
for $\left\vert K\right\vert \leq1$. Note that the restriction on the order of
derivatives $K$ is due to the nonlinear term in the definition of $f_{j}.$ We
put
\begin{equation}
B_{2}=-\frac{\left(  f_{3},\nabla Q\right)  }{\lambda^{3}\left\Vert
Q\right\Vert _{L^{2}}^{2}} \label{B2}%
\end{equation}
and%
\begin{equation}
M_{2}=\frac{\left(  f_{4},Q\right)  }{\lambda\left(  \Lambda Q,Q\right)  }.
\label{M2}%
\end{equation}
so that by Lemma \ref{Linv} there exist solution $T_{1}^{\left(  2\right)  }$
to equation (\ref{eq9}). By (\ref{ap241}) we derive (\ref{ap14}) with $j=2$
for $B_{2}$, $M_{2}$ and hence, for $\boldsymbol{T}^{\left(  2\right)  }$
satisfies (\ref{ap14})$.$ Observe that thanks to the regularity of $\left(
-\Delta+1\right)  ^{-1},$ we can estimate the second derivative on the $y$
variable of $\boldsymbol{T}^{\left(  2\right)  }$ by using (\ref{ap241})
only$.$ Let us prove (\ref{BM}). Using (\ref{ap240}), (\ref{ap278}) with
$\delta=\frac{3-p_{1}}{2p_{1}}$ and $\delta^{\prime}$ close to $\delta,$ and
(\ref{non2}) we estimate
\[
\left\vert \partial_{\chi}^{k}\left(  \operatorname{Re}\mathcal{N}\left(
\boldsymbol{T}^{\left(  1\right)  }\right)  ,\nabla Q\right)  \right\vert
+\left\vert \partial_{\chi}^{k}\left(  \operatorname{Im}\mathcal{N}\left(
\boldsymbol{T}^{\left(  1\right)  }\right)  ,Q\right)  \right\vert \leq
C\left(  \left\vert \beta\right\vert \Theta\right)  ^{p_{1}}+C\left\vert
U_{V}\left(  \left\vert \frac{\chi}{\lambda}\right\vert \right)  \right\vert
^{1+\frac{p_{1}-1}{4}},\text{ }k\geq0.
\]
Hence, from (\ref{B2}) and (\ref{M2}) we prove (\ref{BM}) for $j=2.$ Finally,
as
\begin{equation}
\left\vert \mathcal{N}\left(  a+b\right)  -\mathcal{N}\left(  a\right)
\right\vert \leq C\left\vert b\right\vert \left(  \left\vert a\right\vert
^{p_{1}-1}+\left\vert b\right\vert ^{p_{1}-1}\right)  , \label{nonlinear1}%
\end{equation}
using (\ref{ap14}) and (\ref{BM}) with $j=1,2$, in (\ref{er1}) we obtain
(\ref{eps}) with $n=2.$

We now proceed by induction. Suppose that we have constructed $\boldsymbol{T}%
^{\left(  j\right)  },B_{j},M_{j},$ $j=1,...,n,$ for some $n\geq3,$ satisfying
estimates (\ref{ap14}), (\ref{BM}) and%
\begin{equation}
\left\Vert e^{-\frac{3-p_{1}}{2p_{1}}\left\vert \cdot\right\vert }%
\mathcal{D}_{2}\boldsymbol{T}^{\left(  j\right)  }\left(  \cdot\right)
\right\Vert _{L^{\infty}}\leq C\left(  \left(  \left\vert \beta\right\vert
+\Theta\right)  \Theta+\left\vert U_{V}\left(  \left\vert \tfrac{\chi}%
{\lambda}\right\vert \right)  \right\vert \right)  , \label{ap283}%
\end{equation}
for $\left\vert K\right\vert \leq2$ and $\left\vert k\right\vert
+k_{1}+\left\vert k_{3}\right\vert \leq1.$ Denote
\[
\Upsilon^{\left(  n\right)  }=\mathcal{T}_{1}^{\left(  n\right)
}+i\mathcal{T}_{2}^{\left(  n\right)  },\text{ with\ }\mathcal{T}_{1}^{\left(
n\right)  }=\sum_{j=1}^{n}T_{1}^{\left(  j\right)  },\text{\ }\mathcal{T}%
_{2}^{\left(  n\right)  }=\sum_{j=1}^{n}T_{2}^{\left(  j\right)  },
\]
and
\[
\mathcal{B}^{\left(  n\right)  }=\sum_{j=1}^{n}B_{j},\text{ \ }\mathcal{M}%
^{\left(  n\right)  }=\sum_{j=1}^{n}M_{j}.
\]
Let us consider the equations%
\begin{equation}
\left(  L_{+}-\lambda^{2}V\left(  \left\vert y+\frac{\chi}{\lambda}\right\vert
\right)  \right)  T_{1}^{\left(  n+1\right)  }=f_{2n+1}=-\lambda^{3}\left(
B_{n+1}\cdot y\right)  Q+\operatorname{Re}\left(  \mathcal{N}\left(
\Upsilon^{\left(  n\right)  }\right)  -\mathcal{N}\left(  \Upsilon^{\left(
n-1\right)  }\right)  \right)  \label{eq11}%
\end{equation}
and
\begin{equation}
\left(  L_{-}-\lambda^{2}V\left(  \left\vert y+\frac{\chi}{\lambda}\right\vert
\right)  \right)  T_{2}^{\left(  n+1\right)  }=f_{2n+2}=-\lambda
M_{n+1}\Lambda Q+\operatorname{Im}\left(  \mathcal{N}\left(  \Upsilon^{\left(
n\right)  }\right)  -\mathcal{N}\left(  \Upsilon^{\left(  n-1\right)
}\right)  \right)  . \label{eq12}%
\end{equation}
We put
\begin{equation}
B_{n+1}=-\frac{\left(  \operatorname{Re}\left(  \mathcal{N}\left(
\Upsilon^{\left(  n\right)  }\right)  -\mathcal{N}\left(  \Upsilon^{\left(
n-1\right)  }\right)  \right)  ,\nabla Q\right)  }{\lambda^{3}\left\Vert
Q\right\Vert _{L^{2}}^{2}}\text{ and }M_{n+1}=\frac{\left(  \operatorname{Im}%
\left(  \mathcal{N}\left(  \Upsilon^{\left(  n\right)  }\right)
-\mathcal{N}\left(  \Upsilon^{\left(  n-1\right)  }\right)  \right)
,Q\right)  }{\lambda\left(  \Lambda Q,Q\right)  }. \label{Bn+1}%
\end{equation}
Using (\ref{nonlinear1}) we have
\begin{equation}
\left\vert \mathcal{N}\left(  \left(  \Upsilon^{\left(  n\right)  }\right)
\right)  -\mathcal{N}\left(  \Upsilon^{\left(  n-1\right)  }\right)
\right\vert \ \leq C\left\vert \boldsymbol{T}^{\left(  n\right)  }\right\vert
\left\vert \Upsilon^{\left(  n-1\right)  }\right\vert ^{p_{1}-1}.
\label{nonlinear2}%
\end{equation}
Then, by (\ref{ap14}) we show
\[
\left\vert \mathcal{D}f_{i}\right\vert \leq C\left\langle \lambda\right\rangle
^{k}\left\langle \lambda^{-1}\right\rangle ^{l}\left\langle \left\vert
\beta\right\vert \right\rangle ^{m}\mathbf{Y}\left(  \left\vert \beta
\right\vert +\Theta\left(  \tfrac{\left\vert \chi\right\vert }{\lambda
}\right)  \right)  ^{\left(  p_{1}-1\right)  n},\text{ }i=2n+1,2n+2,
\]
for $\left\vert K\right\vert \leq1$ and some $k,l,m\geq0.$ Then, $B_{n+1}$ and
$M_{n+1}$ satisfy (\ref{ap14}) with $j=n+1.$ By Lemma \ref{Linv} we can solve
(\ref{eq11}) and (\ref{eq12}) with $T_{1}^{\left(  n+1\right)  }$ and
$T_{2}^{\left(  n+1\right)  }$ satisfying (\ref{ap14}) with $j=n+1.$ We put%
\begin{equation}
\boldsymbol{T}^{\left(  n+1\right)  }=T_{1}^{\left(  n+1\right)  }%
+iT_{2}^{\left(  n+1\right)  }. \label{Tn+1}%
\end{equation}
Then, from the equations (\ref{Bn+1}), using (\ref{nonlinear2}) (with $n$
replaced by $n-1$) and (\ref{ap283}) we deduce (\ref{BM}) for $j=n+1.$
Moreover, we prove $\left\vert e^{-\frac{3-p_{1}}{2}\left\vert \cdot
\right\vert }\mathcal{D}_{2}f_{2n+1}\right\vert +\left\vert e^{-\frac{3-p_{1}%
}{2}\left\vert \cdot\right\vert }\mathcal{D}_{2}f_{2n+2}\right\vert \leq
C\left(  \left(  \left(  \left\vert \beta\right\vert +\Theta\right)
\Theta\right)  ^{p_{1}}+\left\vert U_{V}\left(  \left\vert \tfrac{\chi
}{\lambda}\right\vert \right)  \right\vert ^{p_{1}}\right)  \mathbf{,}$ for
$\left\vert K\right\vert \leq1.$ Then, by Lemma \ref{L8} and (\ref{ap14}) we
show that (\ref{ap283}) is true for $\boldsymbol{T}^{\left(  n+1\right)  }.$
Introducing $T=\Upsilon^{\left(  n\right)  }+\boldsymbol{T}^{\left(
n+1\right)  },$ $B=\mathcal{B}^{\left(  n+1\right)  }$, $M=\mathcal{M}%
^{\left(  n+1\right)  },$ into (\ref{E7}) we get%
\[
\left.
\begin{array}
[c]{c}%
\mathcal{E}_{\operatorname*{apr}}\left(  Q+T\right)  =-i\lambda M\Lambda
T-\lambda^{3}\left(  B\cdot y\right)  T\\
+i\lambda^{2}\left(  2\beta\cdot\nabla_{\chi}\left(  T-T_{1}\right)
+B\cdot\nabla_{\beta}T+M\frac{\partial T}{\partial\lambda}\right)
+\mathcal{N}\left(  \left(  \Upsilon^{\left(  n+1\right)  }\right)  \right)
-\mathcal{N}\left(  \Upsilon^{\left(  n\right)  }\right)  .
\end{array}
\right.
\]
Hence, using (\ref{ap14}) and (\ref{BM}) we prove (\ref{eps}) for $n=N+1.$
\end{proof}

In the case when the potential does not decay fast enough, that is,
$V=V^{\left(  2\right)  }$, with $V^{\left(  2\right)  }$ given by (\ref{v2}),
the construction of Lemma \ref{Lemmaapp} is not good enough to obtain a priori
estimates on the modulation parameters of Section \ref{ModPar} (see Remark
\ref{Rem1}). In order to cover also this case, we present a different
construction in the following Lemma. We denote%
\[
\Psi=\min\left\{  \left\vert V^{\prime}\left(  \left\vert \frac{\chi}{\lambda
}\right\vert \right)  \right\vert +\frac{\left\vert V\left(  \left\vert
\frac{\chi}{\lambda}\right\vert \right)  \right\vert }{\left\vert
\chi\right\vert }+e^{-\frac{\left\vert \chi\right\vert }{2\lambda}}%
+e^{-\frac{\left\vert \chi\right\vert }{2}},\left\vert V\left(  \left\vert
\frac{\chi}{\lambda}\right\vert \right)  \right\vert \right\}  ,
\]%
\[
\mathbf{Z=}\left(  \left\vert \beta\right\vert +\left\vert V\left(  \left\vert
\tfrac{\chi}{\lambda}\right\vert \right)  \right\vert \right)
\]
and%
\[
\mathbf{e=e}\left(  y\right)  =\inf_{0<\delta<1}\left\{  C\left(
\delta\right)  e^{-\delta\left\vert y\right\vert }\right\}  ,\text{
}0<C\left(  \delta\right)  <\infty.
\]

\begin{lemma}
\label{Lemmaapp1}Suppose that $V=V^{\left(  2\right)  }$. Let $\Xi=\left(
\chi,\beta,\lambda\right)  \in\mathbb{R}^{d}\times\mathbb{R}^{d}%
\times\mathbb{R}^{+}$ be a vector of parameters with $\lambda\geq\lambda
_{0}>0\ $such that $\lambda^{2}\sup_{r\in\mathbb{R}}V\left(  r\right)  <1.$
There is $C\left(  V\right)  >0$ such that for any $\left\vert \chi\right\vert
\geq C\left(  V\right)  $ the following holds. Given $n\geq0,$ there exist
$\boldsymbol{\tilde{T}}^{\left(  j\right)  },\tilde{B}_{j},\tilde{M}_{j}\in
L^{\infty}\left(  \mathbb{R}^{d}\right)  ,$ $j=0,...,n$ satisfying%
\begin{equation}
\left.
\begin{array}
[c]{c}%
\left\vert \partial_{\lambda}^{k_{1}}\partial_{y}^{k_{2}}\boldsymbol{\tilde
{T}}^{\left(  0\right)  }\right\vert \leq CC_{k,k_{1},j}\left(  \lambda
\right)  \left\vert V\left(  \left\vert \frac{\chi}{\lambda}\right\vert
\right)  \right\vert \mathbf{e},\\
\left\vert \mathcal{D}_{2}\boldsymbol{\tilde{T}}^{\left(  0\right)
}\right\vert \leq C_{K,j}\left(  \lambda\right)  \Psi\mathbf{e},\text{ }%
k\geq1,\\
\left\vert \mathcal{D}\boldsymbol{\tilde{T}}^{\left(  1\right)  }\right\vert
\leq C_{j,K}\left\langle \lambda\right\rangle ^{i_{j,K}}\left\langle
\lambda^{-1}\right\rangle ^{l_{j,K}}\left\langle \left\vert \beta\right\vert
\right\rangle ^{m_{j,K}}\Psi\mathbf{e},\\
\left\vert \mathcal{D}\boldsymbol{\tilde{T}}^{\left(  j\right)  }\right\vert
\leq C_{j,K}\left\langle \lambda\right\rangle ^{i_{j,K}}\left\langle
\lambda^{-1}\right\rangle ^{l_{j,K}}\left\langle \left\vert \beta\right\vert
\right\rangle ^{m_{j,K}}\left(  \Psi^{2}+\left\vert \beta\right\vert
\left\vert V^{\prime\prime}\left(  \left\vert \tfrac{\chi}{\lambda}\right\vert
\right)  \right\vert \right)  \mathbf{Z}^{j-1}\mathbf{e},\text{ }j\geq2,
\end{array}
\right.  \label{B}%
\end{equation}
and%
\begin{equation}%
\begin{array}
[c]{c}%
\left\vert \partial_{\lambda}^{k_{1}}\partial_{\beta}^{k_{3}}\tilde{B}%
_{1}\right\vert +\left\vert \partial_{\lambda}^{k_{1}}\partial_{\beta}^{k_{3}%
}\tilde{M}_{1}\right\vert \leq C_{j,K}\left\langle \lambda\right\rangle
^{i_{j,K}}\left\langle \lambda^{-1}\right\rangle ^{l_{j,K}}\left\langle
\left\vert \beta\right\vert \right\rangle ^{m_{j,K}}\Psi,\text{ }\\
\\
\left\vert \mathcal{D}_{1}\tilde{B}_{1}\right\vert +\left\vert \mathcal{D}%
_{1}\tilde{M}_{1}\right\vert \leq C_{j,K}\left\langle \lambda\right\rangle
^{i_{j,K}}\left\langle \lambda^{-1}\right\rangle ^{l_{j,K}}\left\langle
\left\vert \beta\right\vert \right\rangle ^{m_{j,K}}\left(  \Psi
^{2}+\left\vert V^{\prime\prime}\left(  \left\vert \tfrac{\chi}{\lambda
}\right\vert \right)  \right\vert \right)  ,\text{ }k\geq1,\\
\\
\left\vert \mathcal{D}_{1}\tilde{B}_{j}\right\vert +\left\vert \mathcal{D}%
_{1}\tilde{M}_{j}\right\vert \leq C_{j,K}\left\langle \lambda\right\rangle
^{i_{j,K}}\left\langle \lambda^{-1}\right\rangle ^{l_{j,K}}\left\langle
\left\vert \beta\right\vert \right\rangle ^{m_{j,K}}\left(  \Psi
^{2}+\left\vert \beta\right\vert \left\vert V^{\prime\prime}\left(  \left\vert
\tfrac{\chi}{\lambda}\right\vert \right)  \right\vert \right)  \mathbf{Z}%
^{j-1},\text{ }j\geq2,
\end{array}
\label{B'}%
\end{equation}
for all $\left\vert K\right\vert \geq0,$ with some $i_{j,K},l_{j,K}%
,m_{j,K}\geq0$ and $C_{j,K},\eta,\delta_{j,K}>0$. For the approximation $Q+T,$
with $T=\sum_{j=0}^{n}\boldsymbol{\tilde{T}}^{\left(  j\right)  },$ the error
$\mathcal{E}_{\operatorname*{apr}}\left(  Q+T\right)  $ satisfies the
estimate
\begin{equation}
\left.  \left\vert \mathcal{E}_{\operatorname*{apr}}\left(  Q+T\right)
\right\vert \leq C_{n}\left\langle \lambda\right\rangle ^{i_{n}}\left\langle
\lambda^{-1}\right\rangle ^{l_{n}}\left\langle \left\vert \beta\right\vert
\right\rangle ^{m_{n}}\Psi\mathbf{Z}^{n}\mathbf{e}.\text{ }\right.
\label{eps2}%
\end{equation}

\end{lemma}

\begin{proof}
In the recursive construction of $\boldsymbol{T}^{\left(  j\right)  }$
presented in Lemma \ref{Lemmaapp} we do not consider the terms in (\ref{E7})
containing derivatives of $T,$ treating them as an error. It is possible
thanks to the correct decay of these terms. This is not the case now. We need
to take into account the derivatives of $\boldsymbol{T}^{\left(  j\right)  }$
in their recurrent equations$.$ As consequence, we must control high order
derivatives $\mathcal{D}\boldsymbol{T}^{\left(  j\right)  }$. This requires
to control the derivatives of the nonlinear term, which may be not smooth
enough. To avoid any issue, we approximate the nonlinear term by polynomials
in the following way (see step 7 in the proof of Proposition 3.1 of
\cite{MerleRaphael}). For $z\in\mathbb{C}$, we consider the function
$\mathbf{n}\left(  z\right)  =\left\vert 1+z\right\vert ^{p-1}\left(
1+z\right)  .$ For $j\geq0,$ let $P_{j}\left(  z\right)  $ be the homogeneous
term of order $j$ in the Taylor approximation of $\mathbf{n}\left(  z\right)
$ for $\left\vert z\right\vert \leq\frac{1}{2}.$ Then as $\left\vert
\mathbf{n}\left(  z\right)  \right\vert \leq C\left\vert z\right\vert ^{p},$
for all $z\in\mathbb{C}$, we get%
\[
\left\vert \mathbf{n}\left(  z\right)  -\sum_{j=0}^{m}P_{j}\left(  z\right)
\right\vert \leq C_{m}\left\vert z\right\vert ^{m+1},\text{ }z\in
\mathbb{C}\text{, }m\geq0.
\]
Using the last relation we get%
\begin{equation}
\left\vert \left\vert Q+T\right\vert ^{p-1}\left(  Q+T\right)  -Q^{p}%
\sum_{j=0}^{m}P_{j}\left(  \frac{T}{Q}\right)  \right\vert \leq C_{m}%
Q^{p-m-1}\left\vert T\right\vert ^{m+1}. \label{nonlinear0}%
\end{equation}
In particular%
\begin{equation}
\left\vert \mathcal{N}\left(  T\right)  -Q^{p}\mathbf{P}_{m}\left(  \frac
{T}{Q}\right)  \right\vert \leq C_{m}Q^{p-m-1}\left\vert T\right\vert ^{m+1}.
\label{nonlinear}%
\end{equation}
where we denote
\[
\mathbf{P}_{m}=\sum_{j=2}^{m}P_{j},\text{ }m\geq2.
\]

Let us consider the equation
\begin{equation}
L_{+}T^{\left(  0\right)  }=\lambda^{2}V\left(  \left\vert \frac{\chi}%
{\lambda}\right\vert \right)  Q. \label{eq13}%
\end{equation}
Since the right-hand side is orthogonal to $\nabla Q,$ it follows from
(\ref{ineq2}) that there exists a solution $T^{\left(  0\right)  }$ to
(\ref{eq13}) which satisfies%
\begin{equation}
\left.  \left\vert \mathcal{D}_{2}T^{\left(  0\right)  }\right\vert \leq
C_{0}\left(  \lambda\right)  \left\vert V^{\left(  k\right)  }\left(
\left\vert \frac{\chi}{\lambda}\right\vert \right)  \right\vert \mathbf{e}%
,\text{ }k+k_{1}+k_{2}\geq0.\right.  \label{T0}%
\end{equation}
By the parity symmetry of equation (\ref{eq13}), $T^{\left(  0\right)  }$ may
be chosen even. Actually, since $L_{+}\left(  \Lambda Q\right)  =-2Q,$
$T^{\left(  0\right)  }$ is given explicitly by $T^{\left(  0\right)
}=-\lambda^{2}V\left(  \left\vert \frac{\chi}{\lambda}\right\vert \right)
\frac{\Lambda Q}{2}.$

For $j\geq1$ and $m\geq2,$ we define recursively $T_{j}$ as an even solution
of the equation%
\begin{equation}
L_{+}T^{\left(  j\right)  }=F_{j}=\lambda^{2}V\left(  \left\vert \frac{\chi
}{\lambda}\right\vert \right)  T^{\left(  j-1\right)  }+Q^{p}\left(
\mathbf{P}_{m}\left(  Q^{-1}%
{\displaystyle\sum\limits_{i=0}^{j-1}}
T^{\left(  i\right)  }\right)  -\mathbf{P}_{m}\left(  Q^{-1}%
{\displaystyle\sum\limits_{i=0}^{j-2}}
T^{\left(  i\right)  }\right)  \right)  . \label{ap262}%
\end{equation}
Due to the parity of $T^{\left(  0\right)  },$ for $j\geq1$ the right-hand
side of (\ref{ap262}) is orthogonal to $\nabla Q$ and hence, such $T^{\left(
j\right)  }$ exists. Observe that%
\begin{equation}
\left\vert \mathbf{P}_{m}\left(  a+b\right)  -\mathbf{P}_{m}\left(  a\right)
\right\vert \leq C_{m}\left\langle a\right\rangle ^{m}\left\langle
b\right\rangle ^{m}\left\vert b\right\vert \left(  \left\vert b\right\vert
+\left\vert a\right\vert \right)  . \label{Pm}%
\end{equation}
Then, as $\left\vert V^{\left(  k\right)  }\right\vert \leq C\left(  k\right)
\left\vert V\right\vert ,$ for all $k\geq0,$ by induction in $j$ we estimate%
\begin{equation}
\left.
\begin{array}
[c]{c}%
\left\vert \partial_{\lambda}^{k_{1}}\partial_{y}^{k_{2}}T^{\left(  j\right)
}\right\vert \leq C_{k,k_{1},j}\left(  \lambda\right)  \left\vert V\left(
\left\vert \frac{\chi}{\lambda}\right\vert \right)  \right\vert ^{1+j}%
\mathbf{e},\\
\left\vert \mathcal{D}_{2}T^{\left(  j\right)  }\right\vert \leq C_{k,k_{1}%
,j}\left(  \lambda\right)  \left\vert V^{\prime}\left(  \left\vert \frac{\chi
}{\lambda}\right\vert \right)  \right\vert \left\vert V\left(  \left\vert
\frac{\chi}{\lambda}\right\vert \right)  \right\vert ^{j}\mathbf{e},\text{
}k=1,\\
\left\vert \mathcal{D}_{2}T^{\left(  j\right)  }\right\vert \leq C_{k,k_{1}%
,j}\left(  \lambda\right)  \left(  \left\vert V^{\prime}\left(  \left\vert
\frac{\chi}{\lambda}\right\vert \right)  \right\vert ^{2}+\left\vert
V^{\prime\prime}\left(  \left\vert \frac{\chi}{\lambda}\right\vert \right)
\right\vert \left\vert V\left(  \left\vert \frac{\chi}{\lambda}\right\vert
\right)  \right\vert \right)  \left\vert V\left(  \left\vert \frac{\chi
}{\lambda}\right\vert \right)  \right\vert ^{j-1}\mathbf{e},\text{ }k\geq2.
\end{array}
\right.  \label{t01}%
\end{equation}
Let $\mathbf{j}$ be such that $\mathbf{j}\rho\geq1.$ Then, by (\ref{p1})%
\begin{equation}
\left\vert \partial_{\chi}^{k}\partial_{\lambda}^{k_{1}}T^{\left(
\mathbf{j}\right)  }\right\vert \leq C_{k,k_{1},\mathbf{j}}\left(
\lambda\right)  \left\vert V\left(  \left\vert \frac{\chi}{\lambda}\right\vert
\right)  \right\vert \left\vert \chi\right\vert ^{-\mathbf{j}\rho}%
\mathbf{e}\leq CC_{k,k_{1},\mathbf{j}}\left(  \lambda\right)  \frac{\left\vert
V\left(  \left\vert \frac{\chi}{\lambda}\right\vert \right)  \right\vert
}{\left\vert \chi\right\vert }\mathbf{e}. \label{t0j}%
\end{equation}
We put $\boldsymbol{\tilde{T}}^{\left(  0\right)  }=\sum_{i=0}^{\mathbf{j}%
}T^{\left(  i\right)  }$ and $B_{0}=M_{0}=0.$ By (\ref{t01})%
\begin{equation}
\left.
\begin{array}
[c]{c}%
\left\vert \partial_{\lambda}^{k_{1}}\partial_{y}^{k_{2}}\boldsymbol{\tilde
{T}}^{\left(  0\right)  }\right\vert \leq CC_{k,k_{1},\mathbf{j}}\left(
\lambda\right)  \left\vert V\left(  \left\vert \frac{\chi}{\lambda}\right\vert
\right)  \right\vert \mathbf{e},\\
\left\vert \mathcal{D}_{2}\boldsymbol{\tilde{T}}^{\left(  0\right)
}\right\vert \leq C_{K,\mathbf{j}}\left(  \lambda\right)  \left\vert
V^{\prime}\left(  \left\vert \frac{\chi}{\lambda}\right\vert \right)
\right\vert \mathbf{e},\text{ }k=1,\\
\left\vert \mathcal{D}_{2}\boldsymbol{\tilde{T}}^{\left(  0\right)
}\right\vert \leq C_{K,\mathbf{j}}\left(  \lambda\right)  \left(  \left\vert
V^{\prime}\left(  \left\vert \frac{\chi}{\lambda}\right\vert \right)
\right\vert ^{2}+\left\vert V^{\prime\prime}\left(  \left\vert \frac{\chi
}{\lambda}\right\vert \right)  \right\vert \right)  \mathbf{e},\text{ }k\geq2.
\end{array}
\right.  \label{T0bold}%
\end{equation}
In particular, we get (\ref{B}) for $\boldsymbol{\tilde{T}}^{\left(  0\right)
}$. From (\ref{E7}) we have
\begin{equation}
\left.
\begin{array}
[c]{c}%
\mathcal{E}_{\operatorname*{apr}}\left(  Q+\boldsymbol{\tilde{T}}^{\left(
0\right)  }\right)  =\lambda^{2}\left(  V\left(  \left\vert y+\frac{\chi
}{\lambda}\right\vert \right)  -V\left(  \left\vert \frac{\chi}{\lambda
}\right\vert \right)  \right)  \left(  Q+\boldsymbol{\tilde{T}}^{\left(
0\right)  }\right)  +\lambda^{2}V\left(  \left\vert \frac{\chi}{\lambda
}\right\vert \right)  T^{\left(  \mathbf{j}\right)  }\\
+i\lambda^{2}\left(  2\beta\cdot\nabla_{\chi}\boldsymbol{\tilde{T}}^{\left(
0\right)  }\right)  +Q^{p}\left(  \mathbf{P}_{m}\left(  Q^{-1}%
\boldsymbol{\tilde{T}}^{\left(  0\right)  }\right)  -\mathbf{P}_{m}\left(
Q^{-1}\sum_{i=0}^{\mathbf{j}-1}T^{\left(  i\right)  }\right)  \right)
+\mathcal{N}\left(  \boldsymbol{\tilde{T}}^{\left(  0\right)  }\right)
-Q^{p}\mathbf{P}_{m}\left(  Q^{-1}\boldsymbol{\tilde{T}}^{\left(  0\right)
}\right)  .
\end{array}
\right.  \label{tbold0}%
\end{equation}
Similarly to the proof of (\ref{estp1}), by using (\ref{ap248}) we show that
\begin{equation}
\left\vert V\left(  \left\vert y+\frac{\chi}{\lambda}\right\vert \right)
-V\left(  \left\vert \frac{\chi}{\lambda}\right\vert \right)  \right\vert
e^{-\frac{\left\vert y\right\vert }{2}}\leq C\left\langle \lambda\right\rangle
^{i}\left\langle \lambda^{-1}\right\rangle ^{l}\left(  \left\vert V^{\prime
}\left(  \left\vert \frac{\chi}{\lambda}\right\vert \right)  \right\vert
+\frac{\left\vert V\left(  \left\vert \frac{\chi}{\lambda}\right\vert \right)
\right\vert }{\left\vert \chi\right\vert }+e^{-\frac{\left\vert \chi
\right\vert }{2\lambda}}\right)  ,\text{ }i,l>0. \label{potentialprime}%
\end{equation}
Hence, by (\ref{nonlinear}), (\ref{Pm}), (\ref{t01}), (\ref{t0j}) and
(\ref{tbold0}) we obtain%
\begin{equation}
\left.  \left\vert \mathcal{E}_{\operatorname*{apr}}\left(
Q+\boldsymbol{\tilde{T}}^{\left(  0\right)  }\right)  \right\vert \leq
C\left(  \Psi+\left\vert V\left(  \left\vert \frac{\chi}{\lambda}\right\vert
\right)  \right\vert ^{m+1}\right)  \mathbf{e}.\right.  \label{ap281}%
\end{equation}
Taking%
\begin{equation}
m=j+n \label{m}%
\end{equation}
in (\ref{ap281}) and using (\ref{p1}) we get
\[
\left\vert \mathcal{E}_{\operatorname*{apr}}\left(  Q+\boldsymbol{\tilde{T}%
}^{\left(  0\right)  }\right)  \right\vert \leq C\Psi\left(  1+\left\vert
V\left(  \left\vert \frac{\chi}{\lambda}\right\vert \right)  \right\vert
^{n}\right)  \mathbf{e.}%
\]
In particular, we attain (\ref{eps2}) for $n=0.$

Let us consider now the equation%
\begin{equation}
\left.
\begin{array}
[c]{c}%
\left(  L_{+}-\lambda^{2}V\left(  \left\vert y+\frac{\chi}{\lambda}\right\vert
\right)  \right)  \tilde{T}_{1}=-\lambda^{3}\left(  B\cdot y\right)  \left(
Q+\boldsymbol{\tilde{T}}^{\left(  0\right)  }\right)  +\lambda^{2}\left(
V\left(  \left\vert y+\frac{\chi}{\lambda}\right\vert \right)  -V\left(
\left\vert \frac{\chi}{\lambda}\right\vert \right)  \right)  \left(
Q+\boldsymbol{\tilde{T}}^{\left(  0\right)  }\right) \\
+\lambda^{2}V\left(  \left\vert \frac{\chi}{\lambda}\right\vert \right)
T^{\left(  \mathbf{j}\right)  }+Q^{p}\left(  \mathbf{P}_{m}\left(
Q^{-1}\boldsymbol{\tilde{T}}^{\left(  0\right)  }\right)  -\mathbf{P}%
_{m}\left(  Q^{-1}\sum_{i=0}^{\mathbf{j}-1}T^{\left(  i\right)  }\right)
\right)  .
\end{array}
\right.  \label{eq14}%
\end{equation}
Observe that%
\[
\left(  \left(  B\cdot y\right)  \left(  Q+\boldsymbol{\tilde{T}}^{\left(
0\right)  }\right)  ,\nabla Q\right)  =-B\left(  \left\Vert Q\right\Vert
_{L^{2}}^{2}-\left(  y_{1}\boldsymbol{\tilde{T}}^{\left(  0\right)
},q^{\prime}\right)  \right)  .
\]
Moreover, by (\ref{T0bold}), for all $\left\vert \chi\right\vert \geq C>0,$
with $C$ large enough,%
\[
\left\Vert Q\right\Vert _{L^{2}}^{2}-\left(  y_{1}\boldsymbol{\tilde{T}%
}^{\left(  0\right)  },q^{\prime}\right)  \geq\frac{\left\Vert Q\right\Vert
_{L^{2}}^{2}}{2}.
\]
Then, since $\boldsymbol{\tilde{T}}^{\left(  0\right)  }$ is even, taking%
\[
\tilde{B}_{1}=-\frac{\left(  \left(  V\left(  \left\vert y+\frac{\chi}%
{\lambda}\right\vert \right)  -V\left(  \left\vert \frac{\chi}{\lambda
}\right\vert \right)  \right)  \left(  Q+\boldsymbol{\tilde{T}}^{\left(
0\right)  }\right)  ,\nabla Q\right)  }{\lambda\left(  \left\Vert Q\right\Vert
_{L^{2}}^{2}-\left(  y_{1}\boldsymbol{\tilde{T}}^{\left(  0\right)
},q^{\prime}\right)  \right)  }%
\]
we assure that the right-hand side of (\ref{eq14}) is orthogonal to $\nabla
Q.$ Using (\ref{estp1}), (\ref{T0bold}), (\ref{potentialprime}) we deduce%
\begin{align*}
\left\vert \partial_{\lambda}^{k_{1}}\tilde{B}_{1}\right\vert  &  \leq
C_{1,K}\left\langle \lambda\right\rangle ^{i_{1,K}}\left\langle \lambda
^{-1}\right\rangle ^{l_{1,K}}\Psi\\
\left\vert \partial_{\chi}^{k}\partial_{\lambda}^{k_{1}}\tilde{B}%
_{1}\right\vert  &  \leq C_{1,K}\left\langle \lambda\right\rangle ^{i_{1,K}%
}\left\langle \lambda^{-1}\right\rangle ^{l_{1,K}}\left(  \Psi^{2}+\left\vert
V^{\prime\prime}\left(  \left\vert \tfrac{\chi}{\lambda}\right\vert \right)
\right\vert \right)  ,\text{ }k\geq1.
\end{align*}
Thus, (\ref{B'}) for $\tilde{B}_{1}$ follows. Moreover, similarly to
(\ref{ap238}), via (\ref{estp1}) and (\ref{Pm}), we estimate
\begin{equation}
\left.
\begin{array}
[c]{c}%
\left\vert \partial_{\lambda}^{k_{1}}\partial_{y}^{k_{2}}\tilde{T}%
_{1}\right\vert \leq CC_{1,K}\left\langle \lambda\right\rangle ^{i_{1,K}%
}\left\langle \lambda^{-1}\right\rangle ^{l_{1,K}}\Psi\mathbf{e}\\
\left\vert \mathcal{D}_{2}\tilde{T}_{1}\right\vert \leq CC_{1,K}\left\langle
\lambda\right\rangle ^{i_{1,K}}\left\langle \lambda^{-1}\right\rangle
^{l_{1,K}}\left(  \Psi^{2}+\left\vert V^{\prime\prime}\left(  \left\vert
\tfrac{\chi}{\lambda}\right\vert \right)  \right\vert \right)  \mathbf{e}%
,\text{ }k\geq1.\text{ }%
\end{array}
\right.  \label{t1p}%
\end{equation}
We now consider the equation%
\begin{equation}
\left.  \left(  L_{-}-\lambda^{2}V\left(  \left\vert y+\frac{\chi}{\lambda
}\right\vert \right)  \right)  \tilde{T}_{2}=-\lambda M\Lambda\left(
Q+\boldsymbol{\tilde{T}}^{\left(  0\right)  }\right)  +\lambda^{2}\left(
\tilde{M}_{1}\frac{\partial\boldsymbol{\tilde{T}}^{\left(  0\right)  }%
}{\partial\lambda}\right)  +2\lambda^{2}\beta\cdot\nabla_{\chi}\left(
\boldsymbol{T}^{\left(  0\right)  }+\tilde{T}_{1}\right)  .\right.
\label{eq15}%
\end{equation}
Since $\left(  \Lambda Q,Q\right)  \neq0,$ by (\ref{T0bold}) and (\ref{t1p}),
for all $\left\vert \chi\right\vert \geq C>0,$ with $C$ large enough,%
\[
\left\vert \left(  \Lambda\left(  Q+\boldsymbol{\tilde{T}}^{\left(  0\right)
}\right)  -\lambda\frac{\partial\boldsymbol{\tilde{T}}^{\left(  0\right)  }%
}{\partial\lambda},Q\right)  \right\vert \geq\frac{\left\vert \left(  \Lambda
Q,Q\right)  \right\vert }{2}.
\]
Then, we define $\tilde{M}_{1}$ by
\begin{equation}
\tilde{M}_{1}=\frac{2\lambda}{\left(  \Lambda\left(  Q+\boldsymbol{\tilde{T}%
}^{\left(  0\right)  }\right)  -\lambda\frac{\partial\boldsymbol{\tilde{T}%
}^{\left(  0\right)  }}{\partial\lambda},Q\right)  }\left(  \beta\cdot
\nabla_{\chi}\left(  \boldsymbol{T}^{\left(  0\right)  }+\tilde{T}_{1}\right)
,Q\right)  . \label{ap311}%
\end{equation}
Using (\ref{T0bold}) and (\ref{t1p}) we see that
\begin{equation}
\left.
\begin{array}
[c]{c}%
\left\vert \partial_{\lambda}^{k_{1}}\partial_{y}^{k_{2}}\tilde{M}%
_{1}\right\vert \leq CC_{1,K}\left\langle \lambda\right\rangle ^{i_{1,K}%
}\left\langle \lambda^{-1}\right\rangle ^{l_{1,K}}\left\langle \left\vert
\beta\right\vert \right\rangle ^{m_{j,K}}\Psi,\\
\\
\left\vert \mathcal{D}_{2}\tilde{M}_{1}\right\vert \leq CC_{1,K}\left\langle
\lambda\right\rangle ^{i_{1,K}}\left\langle \lambda^{-1}\right\rangle
^{l_{1,K}}\left\langle \left\vert \beta\right\vert \right\rangle ^{m_{j,K}%
}\left(  \Psi^{2}+\left\vert V^{\prime\prime}\left(  \left\vert \tfrac{\chi
}{\lambda}\right\vert \right)  \right\vert \right)  ,\text{ }k\geq1,
\end{array}
\right.  \label{ap312}%
\end{equation}
and thus, we get (\ref{B'}) for $\tilde{M}_{1}.$ Moreover, we have%
\begin{equation}
\left\vert \mathcal{D}\tilde{T}_{2}\right\vert \leq CC_{1,K}\left\langle
\lambda\right\rangle ^{i_{1,K}}\left\langle \lambda^{-1}\right\rangle
^{l_{1,K}}\Psi\mathbf{e}. \label{t2p}%
\end{equation}
We set $\boldsymbol{\tilde{T}}^{\left(  1\right)  }=\tilde{T}_{1}+i\tilde
{T}_{2}.$ By (\ref{t1p}) and (\ref{t2p}) we deduce (\ref{B}) for
$\boldsymbol{\tilde{T}}^{\left(  1\right)  }$. Using (\ref{ap262}),
(\ref{eq14}) and (\ref{eq15}) in (\ref{E7}) we get%
\[
\left.
\begin{array}
[c]{c}%
\mathcal{E}_{\operatorname*{apr}}\left(  T\right)  =\mathcal{E}%
_{\operatorname*{apr}}\left(  Q+\boldsymbol{\tilde{T}}^{\left(  0\right)
}+\boldsymbol{\tilde{T}}^{\left(  1\right)  }\right)  =-i\lambda\tilde{M}%
_{1}\Lambda\boldsymbol{\tilde{T}}^{\left(  1\right)  }-\lambda^{3}\left(
\tilde{B}_{1}\cdot y\right)  \boldsymbol{\tilde{T}}^{\left(  1\right)  }\\
+i\lambda^{2}\left(  2i\beta\cdot\nabla_{\chi}\tilde{T}_{2}+i\tilde{B}%
_{1}\cdot\nabla_{\beta}\tilde{T}_{2}+\tilde{M}_{1}\frac{\partial
\boldsymbol{\tilde{T}}^{\left(  1\right)  }}{\partial\lambda}\right)
+\mathcal{N}\left(  \boldsymbol{\tilde{T}}^{\left(  0\right)  }%
+\boldsymbol{\tilde{T}}^{\left(  1\right)  }\right)  -Q^{p}\left(
\mathbf{P}_{m}\left(  Q^{-1}\left(  \boldsymbol{\tilde{T}}^{\left(  0\right)
}+\boldsymbol{\tilde{T}}^{\left(  1\right)  }\right)  \right)  \right) \\
+Q^{p}\left(  \mathbf{P}_{m}\left(  Q^{-1}\left(  \boldsymbol{\tilde{T}%
}^{\left(  0\right)  }+\boldsymbol{\tilde{T}}^{\left(  1\right)  }\right)
\right)  -\mathbf{P}_{m}\left(  Q^{-1}\boldsymbol{\tilde{T}}^{\left(
0\right)  }\right)  \right)  .
\end{array}
\right.
\]
Therefore, as $m$ is given by (\ref{m}), from (\ref{B}), (\ref{nonlinear}) and
(\ref{Pm}) we get
\[
\left\vert \mathcal{E}_{\operatorname*{apr}}\left(  T\right)  \right\vert \leq
C_{1}\left\langle \lambda\right\rangle ^{i_{1}}\left\langle \lambda
^{-1}\right\rangle ^{l_{1}}\left\langle \left\vert \beta\right\vert
\right\rangle ^{m_{1}}\Psi\left(  \left\vert \beta\right\vert +\left\vert
V\left(  \left\vert \tfrac{\chi}{\lambda}\right\vert \right)  \right\vert
\right)  \left(  1+\left\vert V\left(  \left\vert \tfrac{\chi}{\lambda
}\right\vert \right)  \right\vert ^{n-1}\right)  ,
\]
and hence (\ref{eps2}) with $n=2.$

We now proceed by induction. Suppose that we have constructed
$\boldsymbol{\tilde{T}}^{\left(  j\right)  },\tilde{B}_{j},\tilde{M}_{j},$
$j=0,...,n,$ for some $n\geq3,$ in such a way that (\ref{B}), (\ref{B'}) and
(\ref{eps2}) hold for all $j=0,...,n.$ Denote
\[
\tilde{\Upsilon}^{\left(  n\right)  }=\mathcal{\tilde{T}}_{1}^{\left(
n\right)  }+i\mathcal{\tilde{T}}_{2}^{\left(  n\right)  },\text{
with\ }\mathcal{\tilde{T}}_{i}^{\left(  n\right)  }=\sum_{j=1}^{n}\tilde
{T}_{i}^{\left(  j\right)  },\text{ }i=1,2,
\]
and
\[
\mathcal{\tilde{B}}^{\left(  n\right)  }=\sum_{j=1}^{n}\tilde{B}_{j},\text{
\ }\mathcal{\tilde{M}}^{\left(  n\right)  }=\sum_{j=1}^{n}\tilde{M}_{j}.
\]
Let us consider the equations%
\begin{equation}
\left(  L_{+}-\lambda^{2}V\left(  \left\vert y+\frac{\chi}{\lambda}\right\vert
\right)  \right)  \tilde{T}_{1}^{\left(  n+1\right)  }=-\lambda^{3}\left(
\tilde{B}_{n+1}\cdot y\right)  \left(  Q+\boldsymbol{\tilde{T}}^{\left(
0\right)  }\right)  +\tilde{f}_{2n+1} \label{eq16}%
\end{equation}
with%
\begin{equation}
\left.
\begin{array}
[c]{c}%
\tilde{f}_{2n+1}=-\lambda^{3}\left(  \mathcal{\tilde{B}}^{\left(  n\right)
}\cdot y\right)  \tilde{T}_{1}^{\left(  n\right)  }-\lambda^{3}\left(
\tilde{B}_{n}\cdot y\right)  \mathcal{\tilde{T}}_{1}^{\left(  n-1\right)
}+\lambda\mathcal{\tilde{M}}^{\left(  n\right)  }\Lambda\tilde{T}_{2}^{\left(
n\right)  }+\lambda\tilde{M}_{n}\Lambda\mathcal{\tilde{T}}_{2}^{\left(
n-1\right)  }\\
\\
-2\lambda^{2}\left(  \beta\cdot\nabla_{\chi}\tilde{T}_{2}^{\left(  n\right)
}\right)  -\lambda^{2}\left(  \mathcal{\tilde{B}}^{\left(  n\right)  }%
\cdot\nabla_{\beta}\tilde{T}_{2}^{\left(  n\right)  }+\mathcal{\tilde{M}%
}^{\left(  n\right)  }\frac{\partial\tilde{T}_{2}^{\left(  n\right)  }%
}{\partial\lambda}\right)  -\lambda^{2}\left(  \tilde{B}_{n}\cdot\nabla
_{\beta}\mathcal{\tilde{T}}_{2}^{\left(  n-1\right)  }+\tilde{M}_{n}%
\frac{\partial\mathcal{T}_{2}^{\left(  n-1\right)  }}{\partial\lambda}\right)
\\
\\
+Q^{p}\operatorname{Re}\left(  \mathbf{P}_{m}\left(  Q^{-1}\tilde{\Upsilon
}^{\left(  n\right)  }\right)  -\mathbf{P}_{m}\left(  Q^{-1}\tilde{\Upsilon
}^{\left(  n-1\right)  }\right)  \right)  ,
\end{array}
\right.  \label{f2n+1}%
\end{equation}
and
\begin{equation}
\left(  L_{-}-\lambda^{2}V\left(  \left\vert y+\frac{\chi}{\lambda}\right\vert
\right)  \right)  \tilde{T}_{2}^{\left(  n+1\right)  }=-\lambda\tilde{M}%
_{n+1}\Lambda\left(  Q+\boldsymbol{\tilde{T}}^{\left(  0\right)  }\right)
+\lambda^{2}\left(  \tilde{M}_{n+1}\frac{\partial\boldsymbol{\tilde{T}%
}^{\left(  0\right)  }}{\partial\lambda}\right)  +\tilde{f}_{2n+2}
\label{eq17}%
\end{equation}
with%
\begin{equation}
\left.
\begin{array}
[c]{c}%
\tilde{f}_{2n+2}=-\lambda\mathcal{\tilde{M}}^{\left(  n\right)  }\Lambda
\tilde{T}_{1}^{\left(  n\right)  }-\lambda\tilde{M}_{n}\Lambda\mathcal{\tilde
{T}}_{1}^{\left(  n-1\right)  }-\lambda^{3}\left(  \mathcal{\tilde{B}%
}^{\left(  n\right)  }\cdot y\right)  \tilde{T}_{2}^{\left(  n\right)
}-\lambda^{3}\left(  \tilde{B}_{n}\cdot y\right)  \mathcal{\tilde{T}}%
_{2}^{\left(  n-1\right)  }\\
\\
+2\lambda^{2}\left(  \beta\cdot\nabla_{\chi}\tilde{T}_{1}^{\left(  n\right)
}\right)  +\lambda^{2}\left(  \mathcal{\tilde{B}}^{\left(  n\right)  }%
\cdot\nabla_{\beta}\tilde{T}_{1}^{\left(  n\right)  }+\mathcal{\tilde{M}%
}^{\left(  n\right)  }\frac{\partial\tilde{T}_{1}^{\left(  n\right)  }%
}{\partial\lambda}\right)  +\lambda^{2}\left(  \tilde{B}_{n}\cdot\nabla
_{\beta}\mathcal{\tilde{T}}_{1}^{\left(  n-1\right)  }+\tilde{M}_{n}%
\frac{\partial\mathcal{\tilde{T}}_{1}^{\left(  n-1\right)  }}{\partial\lambda
}\right) \\
\\
+Q^{p}\operatorname{Im}\left(  \mathbf{P}_{m}\left(  Q^{-1}\tilde{\Upsilon
}^{\left(  n\right)  }\right)  -\mathbf{P}_{m}\left(  Q^{-1}\tilde{\Upsilon
}^{\left(  n-1\right)  }\right)  \right)  .
\end{array}
\right.  \label{f2n+2}%
\end{equation}
We put
\begin{equation}
\tilde{B}_{n+1}=-\frac{\left(  \tilde{f}_{2n+1},\nabla Q\right)  }{\lambda
^{3}\left(  \left\Vert Q\right\Vert _{L^{2}}^{2}-\left(  y_{1}%
\boldsymbol{\tilde{T}}^{\left(  0\right)  },q^{\prime}\right)  \right)
}\text{ and }\tilde{M}_{n+1}=\frac{\left(  \tilde{f}_{2n+2},Q\right)
}{\lambda\left(  \Lambda\left(  Q+\boldsymbol{\tilde{T}}^{\left(  0\right)
}\right)  -\lambda\frac{\partial\boldsymbol{\tilde{T}}^{\left(  0\right)  }%
}{\partial\lambda},Q\right)  }. \label{B3}%
\end{equation}
Using the hypothesis of induction we see that $B_{n+1}$ and $M_{n+1}$ satisfy
(\ref{B'}) with $j=n+1.$ We recursively define $T_{1}^{\left(  n+1\right)  }$
and $T_{2}^{\left(  n+1\right)  }$ as solutions to equations (\ref{eq11}) and
(\ref{eq12}), which exist thanks to (\ref{B3}) and Lemma \ref{Linv}$.$ We put%
\[
\boldsymbol{T}^{\left(  n+1\right)  }=T_{1}^{\left(  n+1\right)  }%
+iT_{2}^{\left(  n+1\right)  }.
\]
As (\ref{B}) holds for all $j=0,...,n,$ and (\ref{B'}) is true for all
$j=0,...,n+1,$ from (\ref{f2n+1}) and (\ref{f2n+2}) we deduce (\ref{B}) for
$j=n+1.$ Introducing $T=\Upsilon^{\left(  n\right)  }+\boldsymbol{T}^{\left(
n+1\right)  },$ $B=\mathcal{B}^{\left(  n+1\right)  }$, $M=\mathcal{M}%
^{\left(  n+1\right)  },$ into (\ref{E7}) we get%
\[
\left.
\begin{array}
[c]{c}%
\mathcal{E}_{\operatorname*{apr}}\left(  Q+T\right)  =-i\lambda\mathcal{\tilde
{M}}^{\left(  n+1\right)  }\Lambda\boldsymbol{\tilde{T}}^{\left(  n+1\right)
}-i\lambda\tilde{M}_{n+1}\Lambda\tilde{\Upsilon}^{\left(  n\right)  }%
-\lambda^{3}\left(  \mathcal{\tilde{B}}^{\left(  n+1\right)  }\cdot y\right)
\boldsymbol{\tilde{T}}^{\left(  n+1\right)  }-\lambda^{3}\left(  \tilde
{B}_{n+1}\cdot y\right)  \tilde{\Upsilon}^{\left(  n\right)  }\\
+i\lambda^{2}\left(  2\beta\cdot\nabla_{\chi}\boldsymbol{\tilde{T}}^{\left(
n+1\right)  }+\mathcal{\tilde{B}}^{\left(  n+1\right)  }\cdot\nabla_{\beta
}\boldsymbol{\tilde{T}}^{\left(  n+1\right)  }+\tilde{B}_{n+1}\cdot
\nabla_{\beta}\tilde{\Upsilon}^{\left(  n\right)  }+\mathcal{\tilde{M}%
}^{\left(  n+1\right)  }\frac{\partial\boldsymbol{\tilde{T}}^{\left(
n+1\right)  }}{\partial\lambda}+\tilde{M}_{n+1}\frac{\partial\tilde{\Upsilon
}^{\left(  n\right)  }}{\partial\lambda}\right) \\
+\mathcal{N}\left(  \left(  \Upsilon^{\left(  n+1\right)  }\right)  \right)
-Q^{p}\left(  \mathbf{P}_{m}\left(  Q^{-1}\tilde{\Upsilon}^{\left(  n\right)
}\right)  \right)  .
\end{array}
\right.
\]
From the validity of (\ref{B}) and (\ref{B'}), for any $j=0,...,n+1,$ using
(\ref{nonlinear}), we prove (\ref{eps2}) with $n$ replaced by $n+1.$
\end{proof}

Let us formulate the approximation result. Given a vector of parameters
$\Xi=\left(  \chi,\beta,\lambda\right)  \in\mathbb{R}^{d}\times\mathbb{R}%
^{d}\times\mathbb{R}^{+}$ and $\gamma\in\mathbb{R}$ consider the approximate
solutions $\boldsymbol{T}^{\left(  j\right)  }$ and $\boldsymbol{\tilde{T}%
}^{\left(  j\right)  }$ of Lemmas \ref{Lemmaapp} and \ref{Lemmaapp1},
respectively. For $j\geq0,$ we denote these solutions by $\boldsymbol{T}%
^{\left(  j\right)  }$ independently of the case. Let the approximate soliton
solution to the perturbed NLS equation be%
\begin{equation}
\mathcal{W}^{\left(  N\right)  }\left(  t,x;\Xi\right)  =\lambda^{-\frac
{2}{p-1}}W\left(  \frac{x-\chi}{\lambda}\right)  e^{-i\gamma}e^{i\beta\cdot
x}. \label{w}%
\end{equation}
with%
\begin{equation}
W=Q+%
{\textstyle\sum_{j=0}^{N}}
\boldsymbol{T}^{\left(  j\right)  }. \label{w1}%
\end{equation}
We have the following.

\begin{lemma}
\label{L1}Let $\Xi\left(  t\right)  =\left(  \chi\left(  t\right)
,\beta\left(  t\right)  ,\lambda\left(  t\right)  \right)  $ and
$\gamma\left(  t\right)  $ be $C^{1}$ functions on a time interval
$I=[t_{0},t_{1}],$ $t_{1}\leq\infty.$ Suppose that for $t\in I$%
\begin{align*}
1  &  \leq\frac{\left\vert \chi\left(  t_{0}\right)  \right\vert }{2}%
\leq\left\vert \chi\left(  t\right)  \right\vert ,\text{ \ }\left\vert
\beta\left(  t\right)  \right\vert \leq2\left\vert \beta\left(  t_{0}\right)
\right\vert ,\\
0  &  <\frac{\lambda\left(  t_{0}\right)  }{2}\leq\lambda\left(  t\right)
\leq2\lambda\left(  t_{0}\right)  .
\end{align*}
Let $N\geq1.$ Then, $\mathcal{W=W}^{\left(  N\right)  }$ satisfies
\begin{equation}
i\partial_{t}\mathcal{W}+\Delta\mathcal{W}+\left\vert \mathcal{W}\right\vert
^{p-1}\mathcal{W}+V\mathcal{W}=\lambda^{-\frac{2}{p-1}}\left(  \mathcal{E}%
^{\left(  N\right)  }+R\left(  W\right)  \right)  \left(  \frac{x-\chi\left(
t\right)  }{\lambda\left(  t\right)  }\right)  e^{-i\gamma\left(  t\right)
}e^{i\beta\left(  t\right)  \cdot x}, \label{ap182}%
\end{equation}
where $R\left(  W\right)  $ is given by (\ref{E9}) and the error term
$\mathcal{E}^{\left(  N\right)  }$ satisfies the estimate%
\begin{equation}
\left\vert \mathcal{E}^{\left(  N\right)  }\right\vert \leq C_{N,\lambda
,\beta}\mathbf{Y}\left(  \left(  \left\vert \beta\right\vert +\Theta\left(
\tfrac{\left\vert \chi\right\vert }{\lambda}\right)  \right)  ^{A\left(
N\right)  }+\left\vert U_{V}\left(  \left\vert \tfrac{\chi}{\lambda
}\right\vert \right)  \right\vert \right)  , \label{eps1}%
\end{equation}
with $A\left(  N\right)  =\min\{N\left(  p_{1}-1\right)  ,2\}$ and
$C_{N,\lambda,\beta}>0.$ In addition, in the case when $V=V^{\left(  2\right)
},$ the estimate%
\begin{equation}
\left\vert \mathcal{E}^{\left(  N\right)  }\right\vert \leq C_{N,\lambda
,\beta}\Psi\mathbf{Z}^{N}\mathbf{e} \label{eps1bis}%
\end{equation}
is true.
\end{lemma}

\begin{proof}
The result follows from (\ref{calc1}), (\ref{E6}) and Lemmas \ref{Lemmaapp},
\ref{Lemmaapp1}.
\end{proof}

\subsection{Approximate modulation parameters.\label{Sec1}}

We want to construct approximate modulation equations for $\Xi\left(
t\right)  $ and $\gamma\left(  t\right)  .$ For $\lambda^{\infty}\in
\mathbb{R}^{+}$ let us consider the problem of the motion in the central
field
\begin{equation}
\left\{
\begin{array}
[c]{c}%
\dot{\chi}^{\infty}=2\beta^{\infty},\\
\dot{\beta}^{\infty}=\dfrac{1}{2\lambda\left\Vert Q\right\Vert _{L^{2}}^{2}%
}\dfrac{\chi^{\infty}}{\left\vert \chi^{\infty}\right\vert }U_{V_{\lambda
^{\infty}}}^{\prime}\left(  \frac{\left\vert \chi^{\infty}\right\vert
}{\lambda^{\infty}}\right)
\end{array}
\right.  \label{syst1}%
\end{equation}
where we emphasize the dependence of $V$ on $\lambda$ (recall (\ref{vcall})).
The energy of the system is given by
\begin{equation}
E_{0}=\frac{\left\vert \dot{\chi}^{\infty}\right\vert ^{2}}{2}-\frac
{U_{V_{\lambda^{\infty}}}\left(  \frac{r^{\infty}}{\lambda^{\infty}}\right)
}{\left\Vert Q\right\Vert _{L^{2}}^{2}} \label{energy}%
\end{equation}
where $r^{\infty}=\left\vert \chi^{\infty}\right\vert .$ The unbounded
solutions $\chi^{\infty}$ of (\ref{syst}) have the following behaviour for
large $t$ depending on the regime. If $E_{0}>0$, for some
$C_{\operatorname*{hyp}},C_{\operatorname*{hyp}}^{\prime}>0,$ we have
\begin{equation}
r^{\infty}=C_{\operatorname*{hyp}}t+o\left(  t\right)  \text{ and }\left\vert
\beta^{\infty}\right\vert =C_{\operatorname*{hyp}}^{\prime}+o\left(  1\right)
. \label{ap200}%
\end{equation}
In the case $E_{0}=0,$ the unbounded solutions, if they exist behave as%
\begin{equation}
t=\frac{\left\Vert Q\right\Vert _{L^{2}}}{\sqrt{2}}\int_{r_{0}}^{\left\vert
\chi^{\infty}\right\vert }\frac{dr}{\sqrt{U_{V_{\lambda^{\infty}}}\left(
\frac{r}{\lambda^{\infty}}\right)  -\frac{\mu^{2}}{r^{2}}}}+t_{0}
\label{ap192}%
\end{equation}
where $\mu\geq0$ is the angular momentum$,$ and
\begin{equation}
\left\vert \beta^{\infty}\right\vert =\frac{\sqrt{U_{V_{\lambda^{\infty}}%
}\left(  \frac{r^{\infty}}{\lambda^{\infty}}\right)  -\frac{\mu^{2}}{\left(
r^{\infty}\right)  ^{2}}}}{\sqrt{2}\left\Vert Q\right\Vert _{L^{2}}}.
\label{ap193}%
\end{equation}
Let $\lambda^{\infty}\in\mathbb{R}^{+}$ and $\Xi^{\infty}\left(  t\right)
=\left(  \chi^{\infty}\left(  t\right)  ,\beta^{\infty}\left(  t\right)
,\lambda^{\infty}\right)  $ be a solution to (\ref{syst}) and $\gamma^{\infty
}\left(  t\right)  $ be given by (\ref{gamma}).

\begin{remark}
\label{Rem2}Observe that in the case $E_{0}=0$, if the potential
$U_{V_{\lambda^{\infty}}}\left(  \frac{\left\vert \chi^{\infty}\right\vert
}{\lambda^{\infty}}\right)  $ decays faster than the centrifugal energy term
$\frac{\mu^{2}}{\left\vert \chi^{\infty}\right\vert ^{2}},$ with $\mu>0$, all
the solutions are bounded. Therefore, in that case, the unbounded solutions
are possible only if $\mu=0.$ If the potential decays slower than $r^{2},$ for
instance if
\begin{equation}
\left\vert V_{\lambda^{\infty}}\left(  r\right)  \right\vert \geq
c\left\langle r\right\rangle ^{-2+\nu}>0, \label{cond2}%
\end{equation}
for some $0<\nu<2,$ then the solution to (\ref{syst1}) with $E_{0}=0$ is given
by
\begin{equation}
\chi^{\infty}\left(  t\right)  =r^{\infty}\left(  t\right)  \theta^{\infty
}\left(  t\right)  , \label{chiinf}%
\end{equation}
with $r^{\infty}\left(  t\right)  $ solving
\begin{equation}
\dot{r}^{\infty}=\sqrt{\frac{2}{\left\Vert Q\right\Vert _{L^{2}}^{2}%
}U_{V_{\lambda^{\infty}}}\left(  \frac{r^{\infty}}{\lambda^{\infty}}\right)
-\frac{\mu^{2}}{\left(  r^{\infty}\right)  ^{2}}}. \label{ap299}%
\end{equation}
and $\theta^{\infty}\left(  t\right)  $ satisfying
\begin{equation}
\left\vert \dot{\theta}^{\infty}\left(  t\right)  \right\vert =\frac{\mu
}{\left(  r^{\infty}\right)  ^{2}}. \label{angular}%
\end{equation}
By (\ref{cond2}), for any $\theta_{0}^{\infty}\in\mathbb{S}^{d-1},$ there is a
solution $\theta^{\infty}\left(  t\right)  $ for (\ref{angular}) such that
\begin{equation}
\left\vert \theta^{\infty}\left(  t\right)  -\theta_{0}^{\infty}\right\vert
\leq\int_{t}^{\infty}\frac{\mu}{\left(  r^{\infty}\right)  ^{2}}d\tau\leq
C\int_{t}^{\infty}\frac{\dot{r}^{\infty}d\tau}{\sqrt{V\left(  \frac{r^{\infty
}}{\lambda^{\infty}}\right)  }\left(  r^{\infty}\right)  ^{2}}\leq
C\int_{r^{\infty}}^{\infty}\frac{dr^{\infty}}{\left(  r^{\infty}\right)
^{1+\nu/2}}\leq C\left(  r^{\infty}\right)  ^{-\nu/2}. \label{angular1}%
\end{equation}

\end{remark}

We now define the approximate system of modulation parameters. The form of
$R\left(  W\right)  $ in (\ref{ap182}) suggests to define the approximate
system as follows. For $j\geq1$, let $B_{j}$, $M_{j}$ and $\tilde{B}_{j}$,
$\tilde{M}_{j}$ be defined by Lemmas \ref{Lemmaapp} and \ref{Lemmaapp1},
respectively. We omit the tilde and write $B_{j}$, $M_{j}$ in both cases.
Denote
\[
\mathbf{B}^{\left(  N\right)  }=\sum_{j=1}^{N}B_{j}\text{ and }\mathbf{M}%
^{\left(  N\right)  }=\sum_{j=1}^{N}M_{j}%
\]
Consider the system
\begin{equation}
\left\{
\begin{array}
[c]{c}%
\dot{\chi}^{\left(  N\right)  }=2\beta^{\left(  N\right)  },\\
\dot{\beta}^{\left(  N\right)  }=\mathbf{B}^{\left(  N\right)  },\\
\dot{\lambda}^{\left(  N\right)  }=\mathbf{M}^{\left(  N\right)  }.
\end{array}
\right.  \label{syst3}%
\end{equation}
Let us show that (\ref{syst3}) may be solved from infinity with asymptotic
behaviour given by $\Xi^{\infty}\left(  t\right)  .$ Observe that by Condition
\ref{ConidtionPotential} the potential is given either by
\begin{equation}
V\left(  r\right)  =V^{\left(  1\right)  }\left(  r\right)  , \label{cond8}%
\end{equation}
where $\left\vert V^{\left(  1\right)  }\left(  r\right)  \right\vert \leq
Ce^{-cr},$ for some $c>0,$ or else $V\left(  r\right)  =V^{\left(  2\right)
}\left(  r\right)  .$ Also by Condition \ref{ConidtionPotential}%
\begin{equation}
\left\vert U_{V^{\prime}}^{\prime}\left(  r\right)  \right\vert \leq
C\left\vert U_{V}^{\prime}\left(  r\right)  \right\vert . \label{Cond1}%
\end{equation}
By assumption (\ref{condpotin})\ there are constants $\mathcal{R}%
_{1},\mathcal{R}_{2}\in\mathbb{R}$ such that
\begin{equation}
U_{V_{\lambda}^{\prime}}^{\prime}\left(  r\right)  \mathcal{Y}\left(
r,\lambda\right)  =\mathcal{R}_{1}+o\left(  1\right)  \text{ and }%
\frac{U_{V_{\lambda}}^{\prime}\left(  r\right)  }{r}\mathcal{Y}\left(
r\right)  =\mathcal{R}_{2}+o\left(  1\right)  , \label{ap302}%
\end{equation}
as $r\rightarrow\infty,$ where we denote
\[
\mathcal{Y}\left(  r,\lambda\right)  =\left(  \int_{r_{0}}^{r}\frac{d\tau
}{\sqrt{U_{V_{\lambda}}\left(  \frac{\tau}{\lambda}\right)  }}\right)  ^{2}.
\]

\begin{lemma}
\label{approxsyst}Let $V$ satisfy by Condition \ref{ConidtionPotential}. For
any $N\geq1$ there is a solution $\Xi^{\left(  N\right)  }\left(  t\right)
=\left(  \chi^{\left(  N\right)  }\left(  t\right)  ,\beta^{\left(  N\right)
}\left(  t\right)  ,\lambda^{\left(  N\right)  }\left(  t\right)  \right)  $
to (\ref{syst3}) on $[T_{0},\infty),$ with $T_{0}$ large enough. This solution
satisfies the following asymptotics depending on the energy $E_{0}$ given by
(\ref{energy}). If $E_{0}>0,$ then
\begin{equation}
\left\vert \chi^{\left(  N\right)  }\left(  t\right)  -\chi^{\infty}\left(
t\right)  \right\vert +\left\vert \beta^{\infty}\left(  t\right)  \right\vert
^{-1}\left\vert \beta^{\left(  N\right)  }\left(  t\right)  -\beta^{\infty
}\left(  t\right)  \right\vert +\left\vert \lambda^{\left(  N\right)  }\left(
t\right)  -\lambda^{\infty}\right\vert =o\left(  1\right)  . \label{beh1}%
\end{equation}
If $E_{0}=0,$ suppose in addition that $V\left(  r\right)  \geq0,$ for all $r$
large enough. Then
\begin{equation}
\left\vert \frac{\left\vert \chi^{\left(  N\right)  }\left(  t\right)
\right\vert }{\left\vert \chi^{\infty}\left(  t\right)  \right\vert
}-1\right\vert +\left\vert \beta^{\infty}\left(  t\right)  \right\vert
^{-1}\left\vert \beta^{\left(  N\right)  }\left(  t\right)  -\beta^{\infty
}\left(  t\right)  \right\vert +\left\vert \lambda^{\left(  N\right)  }\left(
t\right)  -\lambda^{\infty}\right\vert =o\left(  1\right)  . \label{beh2}%
\end{equation}

\end{lemma}

\begin{proof}
The proof of (\ref{beh1}) follows by a fixed point argument similarly to Lemma
A.1 in \cite{Krieger}. We omit the proof.

We turn to relation (\ref{beh2})$.$ As we are interested in unbounded
solutions, by Remark \ref{Rem2} we put
\begin{equation}
\mu=0,\text{ if \ }\left\vert V\left(  r\right)  \right\vert \leq
C\left\langle r\right\rangle ^{-2}. \label{mu=0}%
\end{equation}
Suppose first that $V=V^{\left(  1\right)  }.$ In this case we can precise the
behaviour of $\left\vert \chi^{\infty}\right\vert .$ By the definition
(\ref{ap190}) and (\ref{ap191}) of $U_{V}\left(  \xi\right)  $, integrating by
parts, for $r_{0}$ large enough we show that%
\[
C^{\prime}\frac{e^{r^{\infty}+\frac{\left(  d-1\right)  }{2}\ln r^{\infty}}%
}{\sqrt{C_{\pm}\left(  r^{\infty}\right)  }}\leq\int_{r_{0}}^{r^{\infty}}%
\frac{dr}{\sqrt{U_{V_{\lambda^{\infty}}}\left(  r\right)  }}\leq
C\frac{e^{r^{\infty}+\frac{\left(  d-1\right)  }{2}\ln r^{\infty}}}%
{\sqrt{C_{\pm}\left(  r^{\infty}\right)  }},\text{ with some }C,C^{\prime}>0,
\]
if $V\left(  r\right)  =V_{-}\left(  r\right)  \ $or $V\left(  r\right)
=V_{+}\left(  r\right)  ,$ with $H\left(  r\right)  =o\left(  r\right)  .$
Moreover, if $V\left(  r\right)  =V_{+}\left(  r\right)  ,$ with $0<cr\leq
H\left(  r\right)  <2r,$ we estimate%
\[
C^{\prime}e^{r^{\infty}+\frac{\left(  d-1\right)  }{2}\ln r^{\infty}%
-\frac{H\left(  r^{\infty}\right)  }{2}}\leq\int_{r_{0}}^{r^{\infty}}\frac
{dr}{\sqrt{U_{V_{\lambda^{\infty}}}\left(  r\right)  }}\leq Ce^{r^{\infty
}+\frac{\left(  d-1\right)  }{2}\ln r^{\infty}-\frac{H\left(  r^{\infty
}\right)  }{2}},
\]
Then, (\ref{mu=0}) and (\ref{ap192}) imply%
\begin{equation}
r^{\infty}=K\left(  V\right)  \left(  1+o\left(  1\right)  \right)  \ln
t,\text{ as }t\rightarrow\infty, \label{ap199}%
\end{equation}
with
\begin{equation}
K\left(  V\right)  =\left\{
\begin{array}
[c]{c}%
1,\text{ if }V=V_{-}\ \text{ or }V=V_{+},\text{ with }H\left(  r\right)
=o\left(  r\right)  ,\\
\left(  \lim_{\tau\rightarrow\infty}\left(  1-\frac{H\left(  \tau\right)
}{2\tau}\right)  \right)  ^{-1},\text{ }V=V_{+},\text{ with }0<cr\leq H\left(
r\right)  <2r.
\end{array}
\right.  \label{K(V)}%
\end{equation}
Let us write $\chi^{\infty}=r^{\infty}\theta^{\infty}\left(  t\right)  ,$ with
$r^{\infty}>0$ and $\theta^{\infty}\left(  t\right)  \in\mathbb{S}^{d-1}.$
Then, as in this case we put $\mu=0,$ we have $\theta^{\infty}\left(
t\right)  =\theta_{0}^{\infty}=$constant and $\left\vert \dot{\chi}^{\infty
}\right\vert =\dot{r}^{\infty}.$ By Lemma \ref{Lemmaapp} we can rewrite
(\ref{syst3}) as follows%
\begin{equation}
\left\{
\begin{array}
[c]{c}%
\dot{\chi}=2\beta,\\
\dot{\beta}=B\left(  \chi,\lambda\right)  +f_{1}\left(  \chi,\beta
,\lambda\right)  ,\\
\dot{\lambda}=g_{1}\left(  \chi,\beta,\lambda\right)  ,
\end{array}
\right.  \label{syst4}%
\end{equation}
(we omit the index $N$ in $\chi,\beta,\lambda$) where
\begin{equation}
\left\vert f_{1}\left(  \chi,\beta,\lambda\right)  \right\vert \leq C\left(
\chi,\beta,\lambda\right)  \left(  \left(  \left(  \left\vert \beta\right\vert
+\Theta\left(  \left\vert \tfrac{\chi}{\lambda}\right\vert \right)  \right)
\Theta\left(  \left\vert \tfrac{\chi}{\lambda}\right\vert \right)  \right)
^{p_{1}}+\left\vert U_{V}\left(  \left\vert \tfrac{\chi}{\lambda}\right\vert
\right)  \right\vert ^{p_{1}}\right)  , \label{f}%
\end{equation}
and%
\begin{equation}
\left\vert g_{1}\left(  \chi,\beta,\lambda\right)  \right\vert \leq C\left(
\chi,\beta,\lambda\right)  \left(  \left(  \left(  \left\vert \beta\right\vert
+\Theta\left(  \left\vert \tfrac{\chi}{\lambda}\right\vert \right)  \right)
\Theta\left(  \left\vert \tfrac{\chi}{\lambda}\right\vert \right)  \right)
+\left\vert U_{V}\left(  \left\vert \tfrac{\chi}{\lambda}\right\vert \right)
\right\vert \right)  . \label{g}%
\end{equation}
Let us consider the intermediate system
\begin{equation}
\left\{
\begin{array}
[c]{c}%
\dot{\chi}_{\operatorname*{app}}=2\beta_{\operatorname*{app}},\\
\dot{\beta}_{\operatorname*{app}}=B\left(  \chi_{\operatorname*{app}}%
,\lambda^{\infty}\right)  .
\end{array}
\right.  \label{ap297}%
\end{equation}
Using Lemma \ref{L2 1} we have%
\[
B\left(  \chi_{\operatorname*{app}},\lambda^{\infty}\right)  =\frac
{\chi_{\operatorname*{app}}}{\left\vert \chi_{\operatorname*{app}}\right\vert
}\frac{c_{0}}{4\lambda^{\infty}}\left(  1+o\left(  1\right)  \right)
U_{V_{\lambda^{\infty}}}^{\prime}\left(  \left\vert \tfrac{\chi
_{\operatorname*{app}}}{\lambda^{\infty}}\right\vert \right)
\]
with $c_{0}=2\left\Vert Q\right\Vert _{L^{2}}^{-2}.$ Then, from (\ref{ap297})
we deduce%
\begin{equation}
\frac{d}{dt}\left\vert \dot{\chi}_{\operatorname*{app}}\right\vert ^{2}%
=c_{0}\left(  1+o\left(  1\right)  \right)  \frac{d}{dt}U_{V_{\lambda^{\infty
}}}\left(  \tfrac{r_{\operatorname*{app}}}{\lambda^{\infty}}\right)  .
\label{ap300}%
\end{equation}
We take $\chi_{\operatorname*{app}}=r_{\operatorname*{app}}\theta_{0}^{\infty
}$ of angular momentum $\mu=0$ such that $r_{\operatorname*{app}}%
\rightarrow\infty$, $\dot{r}_{\operatorname*{app}}\rightarrow0$ as
$t\rightarrow\infty.$ Integrating (\ref{ap300}) on $[t,\infty)$ and using
Lemma \ref{L2 1} we get%
\[
\dot{r}_{\operatorname*{app}}=\left(  c_{0}\left(  \left(  1+o\left(
1\right)  \right)  U_{V_{\lambda^{\infty}}}\left(  \tfrac
{r_{\operatorname*{app}}}{\lambda^{\infty}}\right)  \right)  \right)  ^{1/2},
\]
Hence
\begin{equation}
t=c_{0}^{-1/2}\int_{r_{0}}^{r_{\operatorname*{app}}}\frac{dr}{\sqrt{\left(
1+o\left(  1\right)  \right)  U_{V_{\lambda^{\infty}}}\left(  \tfrac
{r}{\lambda^{\infty}}\right)  }}+t_{0}. \label{ap303}%
\end{equation}
Similarly to (\ref{ap199}) we show that $r_{\operatorname*{app}}=K\left(
V\right)  \left(  1+o\left(  1\right)  \right)  \ln t,$ as $t\rightarrow
\infty.$ Then,
\begin{equation}
\left\vert \frac{r_{\operatorname*{app}}}{r^{\infty}}-1\right\vert =o\left(
1\right)  \text{ and }\left\vert \beta_{\operatorname*{app}}-\beta^{\infty
}\right\vert =o\left(  1\right)  \left\vert \beta^{\infty}\right\vert .
\label{rapp}%
\end{equation}
Let us decompose $\chi\left(  t\right)  =\chi_{_{\operatorname*{app}}}\left(
t\right)  +\delta\left(  t\right)  $ and $\lambda\left(  t\right)
=\lambda^{\infty}+\mu\left(  t\right)  .$ We aim to prove that for some
$T_{0}>0$ sufficiently big the following a priori estimates are true
\begin{equation}
\left\vert \delta\left(  t\right)  \right\vert \leq t^{-\frac{p_{1}}{4}%
},\text{ }\left\vert \dot{\delta}\left(  t\right)  \right\vert \leq
t^{-\left(  1+\frac{p_{1}}{4}\right)  },\text{ }\left\vert \mu\left(
t\right)  \right\vert \leq t^{-\frac{3}{4}} \label{boot1}%
\end{equation}
for all $t\in\lbrack T_{0},\infty).$ In view of relation (\ref{l2}) we write
$B\left(  \chi,\lambda\right)  =\frac{\chi}{\left\vert \chi\right\vert
}b\left(  \left\vert \chi\right\vert ,\lambda\right)  .$ Linearizing $B\left(
\chi,\lambda\right)  $ around $\left(  \chi_{\operatorname*{app}}%
,\lambda^{\infty}\right)  $ we get%
\[
B\left(  \chi,\lambda\right)  =B\left(  \chi_{\operatorname*{app}}%
,\lambda^{\infty}\right)  +\left(  \theta_{0}^{\infty}\cdot\delta\right)
\theta_{0}^{\infty}b^{\prime}\left(  \left\vert \chi_{\operatorname*{app}%
}\right\vert ,\lambda^{\infty}\right)  +f_{2}%
\]
where via (\ref{Cond1})
\begin{align*}
f_{2}  &  =O\left(  \left\vert \partial_{\lambda}B\left(  \chi
_{\operatorname*{app}},\lambda^{\infty}\right)  \right\vert \left\vert
\mu\right\vert +\frac{\left\vert \delta\right\vert }{\left\vert \chi
_{\operatorname*{app}}\right\vert }\left\vert b\left(  \left\vert
\chi_{\operatorname*{app}}\right\vert ,\lambda^{\infty}\right)  \right\vert
+\left\vert \delta\right\vert ^{2}\left\vert b^{\prime}\left(  \left\vert
\chi_{\operatorname*{app}}\right\vert ,\lambda^{\infty}\right)  \right\vert
\right) \\
&  =O\left(  \left\vert U_{V}^{\prime}\left(  r\right)  \right\vert \left(
\left\vert \mu\right\vert +\frac{\left\vert \delta\right\vert }{\left\vert
\chi_{\operatorname*{app}}\right\vert }+\left\vert \delta\right\vert
^{2}\right)  \right)  .
\end{align*}
We use the Cartesian coordinates with the $x_{1}$-axis directed along the
vector $\theta_{0}^{\infty}.$ We decompose $\delta=\sum_{j=1}^{d}\delta
_{j}\vec{e}_{j},$ and $\vec{e}_{j},$ $j=1,...,d$ is the canonical basis in
these coordinates. Then $\left(  \theta_{0}^{\infty}\cdot\delta\right)
\theta_{0}^{\infty}=\delta_{1}$ and by (\ref{syst4}) we obtain
\begin{equation}
\left\{
\begin{array}
[c]{c}%
\ddot{\delta}=b^{\prime}\left(  \left\vert \chi_{\operatorname*{app}%
}\right\vert ,\lambda^{\infty}\right)  \delta_{1}\vec{e}_{1}+2f_{1}\left(
\chi,\beta,\lambda\right)  +2f_{2},\\
\dot{\lambda}=g_{1}\left(  \chi,\beta,\lambda\right)  .
\end{array}
\right.  \label{syst5}%
\end{equation}
It follows from (\ref{ap299}) that $\dot{r}^{\infty}=\mathcal{X},$ where
\begin{equation}
\mathcal{X=}\sqrt{U_{V_{\lambda^{\infty}}}\left(  \frac{r^{\infty}}%
{\lambda^{\infty}}\right)  }. \label{x}%
\end{equation}
Using (\ref{f}), (\ref{g}), (\ref{boot1}) and (\ref{ap199}) we estimate%
\[
\left\vert f_{1}\left(  \chi,\beta,\lambda\right)  \right\vert \leq C\left(
\chi,\beta,\lambda\right)  \mathcal{X}^{p_{1}}\leq Ct^{-2-p_{1}},
\]%
\[
\left\vert f_{2}\right\vert \leq\frac{C}{t^{2}}\left(  t^{-\frac{3}{4}%
}+t^{-\frac{p_{1}}{4}}\left(  \ln t\right)  ^{-1}+t^{-p_{1}}\right)
\]
and
\[
\left\vert g_{1}\left(  \chi,\beta,\lambda\right)  \right\vert \leq C\left(
\chi,\beta,\lambda\right)  \mathcal{X}\leq Ct^{-2}.
\]
Then, from (\ref{syst5}), for some $T_{0}>0$ big enough we get%
\begin{equation}
\left.  \left\vert \delta_{j}\right\vert \leq\frac{1}{2t^{\frac{p_{1}}{4}}%
},\text{ }\left\vert \delta_{j}\right\vert \leq\frac{1}{2t^{1+\frac{p_{1}}{4}%
}},\text{ }j=2,...,d,\right.  \label{delta}%
\end{equation}
and
\begin{equation}
\left\vert \mu\left(  t\right)  \right\vert \leq\frac{1}{2t^{\frac{3}{4}}},
\label{mu}%
\end{equation}
for $t\in\lbrack T_{0},\infty).$ Let us estimate $\delta_{1}.$ Using Lemma
\ref{L2 1} with $V^{\prime}$ instead of $V,$ we have%
\[
b^{\prime}\left(  \left\vert \chi_{\operatorname*{app}}\right\vert
,\lambda^{\infty}\right)  =c_{0}\left(  \lambda^{\infty}\right)  \left(
1+o\left(  1\right)  \right)  U_{V^{\prime}}^{\prime}\left(  \tfrac
{r_{\operatorname*{app}}}{\lambda^{\infty}}\right)  .
\]
Then
\[
\ddot{\delta}_{1}=c_{0}\left(  1+o\left(  1\right)  \right)  U_{V^{\prime}%
}^{\prime}\left(  \tfrac{r_{\operatorname*{app}}}{\lambda^{\infty}}\right)
\delta_{1}+O\left(  t^{-2-p_{1}}\right)  .
\]
By (\ref{ap302}) and (\ref{ap303}) we deduce $U_{V^{\prime}}^{\prime}\left(
\tfrac{r_{\operatorname*{app}}}{\lambda^{\infty}}\right)  =\left(
\mathcal{R}_{1}+o\left(  1\right)  \right)  c_{0}^{-1}\left(  \lambda^{\infty
}\right)  t^{-2}.$ Thus%
\begin{equation}
\ddot{\delta}_{1}=\mathcal{R}_{1}t^{-2}\delta_{1}+O\left(  o\left(  1\right)
t^{-2}\left\vert \delta_{1}\right\vert +t^{-2-p_{1}}\right)  . \label{ap304}%
\end{equation}
If $\mathcal{R}_{1}=0,$ we improve the estimate on $\delta_{1}$ and
$\dot{\delta}_{1}$ (\ref{boot1}) directly by integrating (\ref{ap304}). If
$\mathcal{R}_{1}\neq0,$ the linear equation $\ddot{\delta}_{1}=\mathcal{R}%
_{1}t^{-2}\delta_{1}$ has two linearly independent solutions $\delta
_{1}^{\left(  1\right)  },\delta_{1}^{\left(  2\right)  }$ such that
$\left\vert \delta_{1}^{\left(  1\right)  }\delta_{1}^{\left(  2\right)
}\right\vert \leq Ct.$ By variation of parameters we solve (\ref{ap304}) and
via (\ref{boot1}) we estimate the solution as
\[
\left\vert \delta_{1}\right\vert \leq Co\left(  1\right)  \int_{t}^{\infty
}\left(  \tau^{-1}\left\vert \delta_{1}\right\vert +\tau^{-1-p_{1}}\right)
\text{ and }\left\vert \dot{\delta}_{1}\right\vert \leq Co\left(  1\right)
\left(  t^{-1}\left\vert \delta_{1}\right\vert +t^{-1-p_{1}}\right)
\]
Hence, for $T_{0}>0$ big enough we obtain the bound $\left\vert \delta
_{1}\right\vert \leq\left(  2t^{\frac{p_{1}}{4}}\right)  ^{-1}$ and
$\left\vert \dot{\delta}_{1}\right\vert \leq\left(  2t^{1+\frac{p_{1}}{4}%
}\right)  ^{-1}$ for $t\in\lbrack T_{0},\infty).$ Combining the last relation
with (\ref{delta}) and (\ref{mu}) we strictly improve (\ref{boot1}). Then, by
a contraction argument, via (\ref{rapp}), we prove the existence of a solution
$\Xi^{\left(  N\right)  }\left(  t\right)  $ for (\ref{syst4}) with the
asymptotics (\ref{beh2}).

Let now $V=V^{\left(  2\right)  }.$ First we suppose that the potential $V$
decays faster than $\left\vert x\right\vert ^{-2}.$ Namely,%
\begin{equation}
\left\vert V\left(  r\right)  \right\vert \leq C\left\langle r\right\rangle
^{-2}. \label{cond4}%
\end{equation}
We consider the intermediate system%
\begin{equation}
\left\{
\begin{array}
[c]{c}%
\dot{\chi}_{\operatorname*{app}}=2\beta_{\operatorname*{app}},\\
\dot{\beta}_{\operatorname*{app}}=\frac{\chi_{\operatorname*{app}}}{\left\vert
\chi_{\operatorname*{app}}\right\vert }\left(  \frac{1}{2\lambda^{\infty}%
}V^{\prime}\left(  \tfrac{\left\vert \chi_{\operatorname*{app}}\right\vert
}{\lambda^{\infty}}\right)  +\mathbf{r}\left(  \left\vert \chi
_{\operatorname*{app}}\right\vert \right)  \right)  .
\end{array}
\right.  \label{syst9}%
\end{equation}
Recall that
\begin{equation}
\left\vert h_{1}^{\left(  k\right)  }\left(  r\right)  \right\vert \leq
Cr^{-1}h_{1}^{\left(  k-1\right)  }\left(  r\right)  ,\text{ }k\geq1.
\label{cond3}%
\end{equation}
Then%
\[
\left\vert \mathbf{r}\left(  \tfrac{\left\vert \chi_{\operatorname*{app}%
}\right\vert }{\lambda^{\infty}}\right)  \right\vert \leq C\left\vert
h^{\prime}\left(  \tfrac{\left\vert \chi_{\operatorname*{app}}\right\vert
}{\lambda^{\infty}}\right)  \right\vert \left(  \left(  \left\vert h^{\prime
}\left(  \tfrac{\left\vert \chi_{\operatorname*{app}}\right\vert }%
{\lambda^{\infty}}\right)  \right\vert +\tfrac{1}{\left\vert \chi
_{\operatorname*{app}}\right\vert }\right)  V^{\prime}\left(  \tfrac
{\left\vert \chi_{\operatorname*{app}}\right\vert }{\lambda^{\infty}}\right)
+\left\vert \chi_{\operatorname*{app}}\right\vert ^{-2}V\left(  \tfrac
{\left\vert \chi_{\operatorname*{app}}\right\vert }{\lambda^{\infty}}\right)
\right)  .
\]
Thus, it follows from (\ref{syst9}) that%
\begin{equation}
\frac{d}{dt}\left\vert \dot{\chi}_{\operatorname*{app}}\right\vert ^{2}%
=\frac{d}{dt}\left(  2V\left(  \tfrac{\left\vert \chi_{\operatorname*{app}%
}\right\vert }{\lambda^{\infty}}\right)  \left(  1+O\left(  h^{\prime}\left(
\left\vert \chi_{\operatorname*{app}}\right\vert \right)  \right)  \right)
\right)  . \label{ap314}%
\end{equation}
We take $\chi_{\operatorname*{app}}=r_{\operatorname*{app}}\theta_{0}^{\infty
}$ of angular momentum $\mu=0$ such that $r_{\operatorname*{app}}%
\rightarrow\infty$, $\dot{r}_{\operatorname*{app}}\rightarrow0$ as
$t\rightarrow\infty.$ Integrating (\ref{ap314}) we get%
\[
\dot{r}_{\operatorname*{app}}=\sqrt{2V\left(  \tfrac{r_{\operatorname*{app}}%
}{\lambda^{\infty}}\right)  \left(  1+O\left(  h^{\prime}\left(
r_{\operatorname*{app}}\right)  \right)  \right)  }.
\]
Then,
\begin{equation}
t=\int_{r_{0}}^{r_{\operatorname*{app}}}\frac{dr}{\sqrt{2V\left(  \tfrac
{r}{\lambda^{\infty}}\right)  \left(  1+O\left(  h^{\prime}\left(  r\right)
\right)  \right)  }}+t_{0} \label{ap315}%
\end{equation}
Let $z\left(  t\right)  =\frac{r_{\operatorname*{app1}}\left(  t\right)
}{r^{\infty}\left(  t\right)  }.$ We introduce
\[
F\left(  z\right)  =\int_{r_{0}}^{zr^{\infty}}\frac{dr}{\sqrt{2V\left(
\tfrac{r}{\lambda^{\infty}}\right)  \left(  1+O\left(  h^{\prime}\left(
r\right)  \right)  \right)  }}.
\]
and
\[
F_{0}=\int_{r_{0}}^{r^{\infty}}\frac{dr}{\sqrt{2V\left(  \tfrac{r}%
{\lambda^{\infty}}\right)  }}.
\]
Expanding $F\left(  z\right)  $ around $z=1$ we have%
\[
F\left(  z\right)  =F_{0}+O\left(  \left(  V\left(  \tfrac{r^{\infty}}%
{\lambda^{\infty}}\right)  \right)  ^{-\frac{1}{2}}\right)  +F^{\prime}\left(
1\right)  \left(  z-1\right)  +O\left(  F^{\prime\prime}\left(  1\right)
\left(  z-1\right)  ^{2}\right)  .
\]
By (\ref{ap192}) $t=F_{0}+t_{0}.$ Then, using (\ref{ap315}) we deduce%
\[
\frac{F\left(  z\right)  }{F_{0}}=1=1+\frac{O\left(  \left(  V\left(
\tfrac{r^{\infty}}{\lambda^{\infty}}\right)  \right)  ^{-\frac{1}{2}}\right)
+r^{\infty}F^{\prime}\left(  1\right)  \left(  z-1\right)  +O\left(  \left(
r^{\infty}\right)  ^{2}F^{\prime\prime}\left(  1\right)  \left(  z-1\right)
^{2}\right)  }{F_{0}}.
\]
Hence,
\[
\left\vert z-1\right\vert =\frac{O\left(  \left(  V\left(  \tfrac{r^{\infty}%
}{\lambda^{\infty}}\right)  \right)  ^{-\frac{1}{2}}\right)  }{r^{\infty
}F^{\prime}\left(  1\right)  }=O\left(  \frac{1}{r^{\infty}}\right)  .
\]
Thus, we obtain
\[
\left\vert \frac{r_{\operatorname*{app}}\left(  t\right)  }{r^{\infty}\left(
t\right)  }-1\right\vert =o\left(  1\right)  \text{ and }\left\vert
\beta_{\operatorname*{app}}-\beta^{\infty}\right\vert =o\left(  1\right)
\left\vert \beta^{\infty}\right\vert .
\]
Then, arguing similarly to (\ref{boot1}), we prove the existence of a solution
$\Xi^{\left(  N\right)  }\left(  t\right)  $ for (\ref{syst4}) with the
asymptotics (\ref{beh2}) in the case when $V$ decay faster than $\left\vert
x\right\vert ^{-2}$.

Now, we consider the case of potentials that decay slower than $\left\vert
x\right\vert ^{-2}.$ We suppose that (\ref{cond2}) holds. Let $\chi^{\infty
}=r^{\infty}\theta^{\infty}$ and $\theta_{0}^{\infty}=\lim_{t\rightarrow
\infty}\theta^{\infty}\left(  t\right)  .$ Recall the definition (\ref{ap311})
of $M_{1}$ and estimate (\ref{ap312}). Let us consider the intermediate system%
\begin{equation}
\left\{
\begin{array}
[c]{c}%
\dot{\chi}_{\operatorname*{app1}}=2\beta_{\operatorname*{app1}},\\
\dot{\beta}_{\operatorname*{app1}}=\frac{1}{2\lambda_{\operatorname*{app1}}%
}\frac{\chi_{\operatorname*{app1}}}{\left\vert \chi_{\operatorname*{app1}%
}\right\vert }V^{\prime}\left(  \tfrac{\left\vert \chi_{\operatorname*{app1}%
}\right\vert }{\lambda_{\operatorname*{app1}}}\right)  ,\\
\dot{\lambda}_{\operatorname*{app1}}=M_{1}\left(  \lambda
_{\operatorname*{app1}}\right)  =O\left(  \left\vert \beta
_{\operatorname*{app1}}\right\vert \Psi\right)  .
\end{array}
\right.  \label{syst6}%
\end{equation}
Denote%
\[
F_{\mu}\left(  \xi\right)  =\sqrt{2V\left(  \tfrac{\xi}{\lambda^{\infty}%
}\right)  -\frac{\mu^{2}}{\xi^{2}}}.
\]
We search a solution to (\ref{syst6}) in the form%
\[
\chi_{\operatorname*{app1}}=r_{\operatorname*{app1}}\theta_{0}^{\infty}.
\]
We denote $\mu_{\operatorname*{app1}}=\lambda_{\operatorname*{app1}}%
-\lambda^{\infty}.$ Using (\ref{cond3}) we get $O\left(  \mu
_{\operatorname*{app1}}\left(  V^{\prime}\left(  \tfrac{r_{\operatorname*{app}%
}}{\lambda_{\operatorname*{app1}}}\right)  +r_{\operatorname*{app}}%
V^{\prime\prime}\left(  \tfrac{r_{\operatorname*{app}}}{\lambda
_{\operatorname*{app1}}}\right)  \right)  \right)  =O\left(  \mu
_{\operatorname*{app1}}V^{\prime}\left(  \tfrac{r_{\operatorname*{app}}%
}{\lambda_{\operatorname*{app1}}}\right)  \right)  $ and system (\ref{syst6})
reads
\begin{equation}
\left\{
\begin{array}
[c]{c}%
\ddot{r}_{\operatorname*{app1}}=\frac{1}{\lambda^{\infty}}V^{\prime}\left(
\tfrac{r_{\operatorname*{app1}}}{\lambda^{\infty}}\right)  +O\left(
\mu_{\operatorname*{app1}}V^{\prime}\left(  \tfrac{r_{\operatorname*{app1}}%
}{\lambda_{\operatorname*{app1}}}\right)  \right)  ,\\
\dot{\lambda}_{\operatorname*{app1}}=M_{1}\left(  \lambda
_{\operatorname*{app1}}\right)  =O\left(  \dot{r}_{\operatorname*{app1}}%
\Psi\right)  ,
\end{array}
\right.  \label{syst8}%
\end{equation}
Let us prove that for some $r_{\operatorname*{app1}}$ there is $T_{0}>0$ such
that
\begin{equation}
\left\vert z\left(  t\right)  -1\right\vert \leq\frac{1}{\left(  r^{\infty
}\right)  ^{\frac{\rho}{4}}}\text{ and }\left\vert \mu_{\operatorname*{app1}%
}\right\vert \leq\left(  r^{\infty}\right)  ^{-\frac{\rho}{2}},\text{ for
}t\in\lbrack T_{0},\infty), \label{boot2}%
\end{equation}
where we denote $z\left(  t\right)  =\frac{r_{\operatorname*{app1}}\left(
t\right)  }{r^{\infty}\left(  t\right)  }.$ Integrating the first equation we
get%
\[
\dot{r}_{\operatorname*{app1}}=F_{0}\left(  r_{\operatorname*{app1}}\right)
+O\left(  \mu_{\operatorname*{app1}}\sqrt{V\left(  \tfrac
{r_{\operatorname*{app1}}}{\lambda_{\operatorname*{app1}}}\right)  }\right)
\]
From (\ref{ap299}) we get%
\[
\dot{r}^{\infty}=F_{\mu}\left(  r^{\infty}\right)  .
\]
Then%
\begin{align*}
\dot{z}  &  =\frac{1}{r^{\infty}}\left(  \dot{r}_{\operatorname*{app1}}%
-z\dot{r}^{\infty}\right)  =\frac{1}{r^{\infty}}\left(  F_{0}\left(
zr^{\infty}\right)  -zF_{\mu}\left(  r^{\infty}\right)  \right) \\
&  =\frac{1}{r^{\infty}}\left(  \left(  F_{0}\left(  zr^{\infty}\right)
-zF_{0}\left(  r^{\infty}\right)  \right)  +z\left(  F_{0}\left(  r^{\infty
}\right)  -F_{\mu}\left(  r^{\infty}\right)  \right)  +O\left(  \mu
_{\operatorname*{app1}}\sqrt{V\left(  \tfrac{r_{\operatorname*{app1}}}%
{\lambda_{\operatorname*{app1}}}\right)  }\right)  \right)  .
\end{align*}
We put $w=z-1.$ Linearizing $F_{0}\left(  zr^{\infty}\right)  $ around $z=1$
we have%
\begin{equation}
\dot{w}=\left(  F_{0}^{\prime}\left(  r^{\infty}\right)  -\frac{F_{0}\left(
r^{\infty}\right)  }{r^{\infty}}\right)  w+\mathbf{w}\left(  r^{\infty
}\right)  \label{ap305}%
\end{equation}
with%
\[
\mathbf{w}\left(  r^{\infty}\right)  =O\left(  r^{\infty}F_{0}^{\prime\prime
}\left(  r^{\infty}\right)  w^{2}+\left(  r^{\infty}\right)  ^{-3}%
F_{0}^{\prime}\left(  r^{\infty}\right)  +\frac{\sqrt{V\left(  \tfrac
{r^{\infty}}{\lambda^{\infty}}\right)  }}{\left(  r^{\infty}\right)
^{1+\frac{\rho}{2}}}\right)  .
\]
Using (\ref{cond2}) we see that $\left\vert h_{1}\left(  r^{\infty}\right)
\right\vert \leq c\ln r^{\infty}.$ Then, by (\ref{boot2}), by (\ref{cond3}) we
get
\begin{equation}
\mathbf{w}\left(  r^{\infty}\right)  =O\left(  \frac{\dot{r}^{\infty}}{\left(
r^{\infty}\right)  ^{1+\frac{\rho}{2}}}\right)  . \label{ap310}%
\end{equation}
By variation of parameters we get%
\[
w\left(  t\right)  =w_{0}^{-1}\left(  t\right)  \int_{t}^{\infty}w_{0}\left(
\tau\right)  \mathbf{w}\left(  r^{\infty}\left(  \tau\right)  \right)  d\tau
\]
with%
\[
w_{0}\left(  t\right)  =e^{\int_{T_{0}}^{t}\left(  F_{0}^{\prime}\left(
r^{\infty}\right)  -\frac{F_{0}\left(  r^{\infty}\right)  }{r^{\infty}%
}\right)  }.
\]
For any $0<\delta<1,$ there exists $T_{0}>0$ sufficiently big, such that
\begin{equation}
\ln\left(  \frac{\left(  r^{\infty}\right)  ^{1-\delta}}{V^{1+\delta}}\right)
\leq-\int_{T_{0}}^{t}\left(  F_{0}^{\prime}\left(  r^{\infty}\right)
-\frac{F_{0}\left(  r^{\infty}\right)  }{r^{\infty}}\right)  \leq\ln\left(
\frac{\left(  r^{\infty}\right)  ^{1+\delta}}{V^{1-\delta}}\right)  .
\label{ap308}%
\end{equation}
Then, using (\ref{ap310}) and (\ref{p1}), and taking $\delta$ small enough we
obtain
\begin{equation}
\left\vert w\left(  t\right)  \right\vert \leq C\left(  Vr^{\infty}\right)
^{2\delta}\int_{t}^{\infty}\frac{\dot{r}^{\infty}}{\left(  r^{\infty}\right)
^{1+\frac{\rho}{2}}}d\tau\leq C\left(  r^{\infty}\right)  ^{2\delta\left(
1-\rho\right)  -\frac{\rho}{2}}\leq\frac{1}{2\left(  r^{\infty}\right)
^{\frac{\rho}{4}}}. \label{ap324}%
\end{equation}
Integrating the second equation in (\ref{syst8}) we deduce $\left\vert
\mu_{\operatorname*{app1}}\right\vert \leq\frac{1}{2}\left(  r^{\infty
}\right)  ^{-\frac{\rho}{2}}.$ Therefore, we strictly improve (\ref{boot2}).
Thus, we show that there is a solution $\chi_{\operatorname*{app1}}$ to
(\ref{syst6}) such that
\begin{equation}
\left\vert \frac{r_{\operatorname*{app1}}\left(  t\right)  }{r^{\infty}\left(
t\right)  }-1\right\vert \leq\frac{1}{\left(  r^{\infty}\right)  ^{\frac{\rho
}{4}}},\text{ }\left\vert \beta_{\operatorname*{app1}}-\beta^{\infty
}\right\vert =o\left(  1\right)  \left\vert \beta^{\infty}\right\vert \text{
and }\left\vert \mu_{\operatorname*{app1}}\right\vert \leq\left(  r^{\infty
}\right)  ^{-\frac{\rho}{2}},\text{ for }t\in\lbrack T_{0},\infty)\text{.}
\label{ap313}%
\end{equation}
Let us now return to the full system (\ref{syst3}). By Lemmas \ref{L2 1} and
\ref{Lemmaapp1} we have%
\begin{equation}
\left\{
\begin{array}
[c]{c}%
\dot{\chi}=2\beta,\\
\dot{\beta}=\frac{\chi}{\left\vert \chi\right\vert }\left(  \frac{1}{2\lambda
}V^{\prime}\left(  \tfrac{\left\vert \chi\right\vert }{\lambda}\right)
\right)  +f_{2}\left(  \chi,\beta,\lambda\right)  ,\\
\dot{\lambda}=M_{1}\left(  \lambda\right)  +g_{2}\left(  \chi,\beta
,\lambda\right)  ,
\end{array}
\right.  \label{syst7}%
\end{equation}
(we omit the index $N$ in $\chi,\beta,\lambda$) where
\[
\left\vert f_{2}\left(  \chi,\beta,\lambda\right)  \right\vert \leq C\left(
\chi,\beta,\lambda\right)  \left(  \Psi^{2}+\mathbf{r}\left(  \left\vert
\tfrac{\chi}{\lambda}\right\vert \right)  +\left\vert \beta\right\vert
\left\vert V^{\prime\prime}\left(  \left\vert \tfrac{\chi}{\lambda}\right\vert
\right)  \right\vert +e^{-\frac{\left\vert \chi\right\vert }{2}}\right)  ,
\]
and%
\[
\left\vert g_{2}\left(  \chi,\beta,\lambda\right)  \right\vert \leq C\left(
\chi,\beta,\lambda\right)  \Psi^{2}.
\]
We write $\chi=\chi_{\operatorname*{app1}}+\tilde{\delta}$ and $\lambda
=\lambda_{\operatorname*{app1}}+\mu.$ Then, as $\chi_{\operatorname*{app1}}$
and $\lambda_{\operatorname*{app1}}$ solve (\ref{syst6}), from (\ref{syst7})
we deduce%
\[
\left\{
\begin{array}
[c]{c}%
\tilde{\delta}^{\prime\prime}=-\tfrac{1}{\lambda_{\operatorname*{app1}}}%
\frac{1}{r_{\operatorname*{app1}}}\left(  \left(  \delta\cdot\theta
_{0}^{\infty}\right)  \theta_{0}^{\infty}+\delta\right)  V^{\prime}\left(
\tfrac{r_{\operatorname*{app1}}}{\lambda_{\operatorname*{app1}}}\right)
+\left(  \delta\cdot\theta_{0}^{\infty}\right)  \theta_{0}^{\infty}\left(
\lambda_{\operatorname*{app1}}\right)  ^{-2}V^{\prime\prime}\left(
\tfrac{r_{\operatorname*{app1}}}{\lambda_{\operatorname*{app1}}}\right)
+f_{3}\left(  \chi,\beta,\lambda\right) \\
\\
\dot{\mu}=M_{1}^{\prime}\left(  \lambda_{\operatorname*{app1}}\right)
\mu+O\left(  \Psi^{2}\right)  .
\end{array}
\right.
\]
where%
\begin{align*}
f_{3}\left(  \chi,\beta,\lambda\right)   &  =f_{2}\left(  \chi,\beta
,\lambda\right)  +O\left(  \left(  V^{\prime\prime}\left(  \tfrac
{r_{\operatorname*{app1}}}{\lambda^{\infty}}\right)  r_{\operatorname*{app1}%
}+V^{\prime}\left(  \tfrac{r_{\operatorname*{app1}}}{\lambda^{\infty}}\right)
\right)  \left\vert \mu\right\vert \right) \\
&  +O\left(  V^{\prime\prime\prime}\left(  \tfrac{r_{\operatorname*{app1}}%
}{\lambda^{\infty}}\right)  +\left(  r_{\operatorname*{app1}}\right)
^{-1}V^{\prime\prime}\left(  \tfrac{r_{\operatorname*{app1}}}{\lambda^{\infty
}}\right)  +\left(  r_{\operatorname*{app1}}\right)  ^{-2}V^{\prime}\left(
\tfrac{r_{\operatorname*{app1}}}{\lambda^{\infty}}\right)  \right)  \left\vert
\delta\right\vert ^{2}.
\end{align*}
We claim that for some $T_{0}>0$ sufficiently big,
\begin{equation}
\left\vert \delta\left(  t\right)  \right\vert \leq r_{\operatorname*{app1}%
}^{-\rho/4},\text{ }\left\vert \dot{\delta}\left(  t\right)  \right\vert \leq
r_{\operatorname*{app1}}^{-\left(  1+\rho/4\right)  },\text{ }\left\vert
\mu\left(  t\right)  \right\vert \leq r_{\operatorname*{app1}}^{-\left(
1+\rho/4\right)  }. \label{boot4}%
\end{equation}
Observe that by (\ref{cond2}) $\left\vert h_{1}\left(  r\right)  \right\vert
\leq C\ln r.$ Then, using (\ref{p1}), (\ref{cond3}) and (\ref{boot4}) we
estimate%
\[
\Psi\leq C\frac{\ln r_{\operatorname*{app1}}}{r_{\operatorname*{app1}}%
}\left\vert V\left(  \tfrac{r_{\operatorname*{app1}}}{\lambda^{\infty}%
}\right)  \right\vert
\]
and%
\[
\left\vert f_{3}\left(  \chi,\beta,\lambda\right)  \right\vert \leq
C\frac{\left(  \ln^{3}r_{\operatorname*{app1}}\right)  \left\vert V\left(
r_{\operatorname*{app1}}\right)  \right\vert }{r_{\operatorname*{app1}}^{2}%
}\left(  \frac{1}{r_{\operatorname*{app1}}}+\left\vert V\left(  \tfrac
{r_{\operatorname*{app1}}}{\lambda^{\infty}}\right)  \right\vert
^{1/2}\right)  .
\]
Arguing similarly to the case of (\ref{syst5}) we show (\ref{boot4}).
Therefore, using (\ref{ap313}), we prove the existence of a solution
$\Xi^{\left(  N\right)  }\left(  t\right)  $ for (\ref{syst3}) with the
asymptotics (\ref{beh2}).
\end{proof}

\section{Proof of Theorem \ref{T1}.\bigskip\label{Sec2}}

For a solution $\Xi^{\left(  N\right)  }\left(  t\right)  $ given by Lemma
\ref{approxsyst} and
\[
\dot{\gamma}^{\left(  N\right)  }\left(  t\right)  =-\frac{1}{\left(
\lambda^{\left(  N\right)  }\left(  t\right)  \right)  ^{2}}+\left\vert
\beta^{\left(  N\right)  }\left(  t\right)  \right\vert ^{2}+\dot{\beta
}^{\left(  N\right)  }\left(  t\right)  \cdot\chi^{\left(  N\right)  }\left(
t\right)  ,
\]
we define $\mathcal{W}^{\left(  N\right)  }\left(  t,x;\Xi^{\left(  N\right)
}\left(  t\right)  \right)  $ by (\ref{w}) and\ (\ref{w1}). For $n\in
\mathbb{N}$ let\ a sequence $T_{n}\rightarrow\infty,$ as $n\rightarrow\infty$
and $u_{n}\in H^{1}$ be the solution to the NLS\ equation with initial data
$u_{n}\left(  T_{n},x\right)  =\mathcal{W}^{\left(  N\right)  }\left(
T_{n},x;\Xi^{\left(  N\right)  }\left(  T_{n}\right)  \right)  $. That is
\begin{equation}
\left\{
\begin{array}
[c]{c}%
i\partial_{t}u_{n}+\Delta u_{n}+\left\vert u_{n}\right\vert ^{p-1}%
u_{n}+V\left(  x\right)  u_{n}=0\\
u_{n}\left(  T_{n},x\right)  =\mathcal{W}^{\left(  N\right)  }\left(
T_{n},x;\Xi^{\left(  N\right)  }\left(  T_{n}\right)  \right)  .
\end{array}
\right.  \label{appsys}%
\end{equation}

Recall that for a given vector $\Xi=\left(  \chi,\beta,\lambda\right)
\in\mathbb{R}^{d}\times\mathbb{R}^{d}\times\mathbb{R}^{+}$ and $\gamma
\in\mathbb{R}$ (see (\ref{w}) and\ (\ref{w1}))%
\begin{equation}
\mathcal{W}\left(  t,x;\Xi\right)  =\mathcal{W}^{\left(  N\right)  }\left(
t,x;\Xi\right)  =\lambda^{-\frac{2}{p-1}}W\left(  \frac{x-\chi}{\lambda
}\right)  e^{-i\gamma}e^{i\beta\cdot x}. \label{ap350}%
\end{equation}
By using the implicit function theorem in the following lemma we show that as
long as the solution evolves close to the solitary-wave solution for the free
NLS equation, it may be decomposed as%
\begin{equation}
u_{n}\left(  t,x\right)  =\mathcal{W}^{\left(  N\right)  }\left(
t,x;\tilde{\Xi}\left(  t\right)  \right)  +\varepsilon\left(  t,x\right)  .
\label{ap223}%
\end{equation}
where the modulation parameters $\Xi\left(  t\right)  \in C^{1}\left(
T_{n}-T,T_{n}+T\right)  ,$ for some $T>0,$ are chosen in such a way that the
remainder $\varepsilon$ satisfies the orthogonal conditions
\begin{equation}
\left(  \zeta,W\right)  =\left(  \zeta,yW\right)  =\left(  \zeta,i\Lambda
W\right)  =\left(  \zeta,i\nabla W\right)  =0 \label{ap208}%
\end{equation}
with
\begin{equation}
\zeta\left(  t,y\right)  =\left(  \lambda\left(  t\right)  \right)  ^{\frac
{2}{p-1}}\varepsilon\left(  t,\lambda\left(  t\right)  y+\chi\left(  t\right)
\right)  e^{i\gamma\left(  t\right)  }e^{-i\beta\left(  t\right)  \cdot\left(
\lambda\left(  t\right)  y+\chi\left(  t\right)  \right)  }.
\label{epsilonbar}%
\end{equation}

\begin{lemma}
\label{L5}Given $\Xi_{0}=\left(  \chi_{0},\beta_{0},\lambda_{0}\right)
\in\mathbb{R}^{d}\times\mathbb{R}^{d}\times\mathbb{R}^{+}$ and $\gamma_{0}%
\in\mathbb{R}$, let $u\left(  t,x\right)  $ be a solution for the
NLS\ equation with a potential (\ref{NLS}) defined on the interval
$[T,T_{\operatorname*{in}}]$, with some $0<T<T_{\operatorname*{in}},$ with the
initial value $u\left(  T_{\operatorname*{in}},x\right)  \in H^{2}$
satisfying
\begin{equation}
\left\Vert u\left(  T_{\operatorname*{in}},x\right)  -\mathcal{W}\left(
T_{\operatorname*{in}},x;\Xi_{0}\right)  \right\Vert \leq\delta,
\label{initial}%
\end{equation}
with $\delta>0$. Suppose that $\varepsilon\left(  T_{\operatorname*{in}%
},x\right)  =u\left(  T_{\operatorname*{in}},x\right)  -\mathcal{W}\left(
t,x;\Xi_{0}\right)  $ satisfies (\ref{ap208}). Then, there exist $A>0$
sufficiently big and $\delta_{0}>0$ small enough such that for all $\left\vert
\chi_{0}\right\vert \geq A$ and $\delta\leq\delta_{0}$ the following
affirmation is true. There is an open interval $I\left(  \delta\right)  \ni
T_{\operatorname*{in}}$, a unique vector $\Xi\left(  t\right)  =\left(
\chi\left(  t\right)  ,\beta\left(  t\right)  ,\lambda\left(  t\right)
\right)  \in C^{1}\left(  I\left(  \delta\right)  \right)  $ and
$\gamma\left(  t\right)  \in C^{1}\left(  I\left(  \delta\right)  \right)  $
satisfying $\Xi\left(  T_{\operatorname*{in}}\right)  =\Xi_{0}$ and
$\gamma\left(  T_{\operatorname*{in}}\right)  =\gamma_{0},$ such that $u$
decomposes as
\begin{equation}
u\left(  t,x\right)  =\mathcal{W}\left(  t,x;\Xi\left(  t\right)  \right)
+\varepsilon\left(  t,x\right)  , \label{decomp}%
\end{equation}
where the error $\varepsilon$ satisfies (\ref{ap208}) for any $t\in I\left(
\delta\right)  $.
\end{lemma}

\begin{proof}
Lemma \ref{L5} is proved similarly to Lemma 3 of \cite{MartelMerle} or Lemma 3
of \cite{Nguyen}. We omit the details.
\end{proof}

We now aim to compare the solution $\left(  \Xi\left(  t\right)
,\gamma\left(  t\right)  \right)  $ given by Lemma \ref{L5} with the solution
$\left(  \Xi^{\left(  N\right)  }\left(  t\right)  ,\gamma^{\left(  N\right)
}\left(  t\right)  \right)  $ of the approximate system (\ref{syst3}). We have
the following a priori estimates.

\begin{lemma}
\label{L6}Let $\lambda^{\infty}\geq\lambda_{0}>0$ satisfy $\left(
\lambda^{\infty}\right)  ^{2}\sup_{r\in\mathbb{R}}V\left(  r\right)  <1.$
Suppose that $\left\vert V\left(  r\right)  \right\vert \leq Ce^{-cr},$ for
some $c>0.$ There exists $T_{0}>0$ large enough such that for all $t\in\lbrack
T_{0},T_{n}],$ the estimates
\begin{equation}
\left\{
\begin{array}
[c]{c}%
\left\Vert \varepsilon\left(  t\right)  \right\Vert _{H^{1}}\leq t^{-2},\\
\left\vert \beta\left(  t\right)  -\beta^{\left(  N\right)  }\left(  t\right)
\right\vert +\left\vert \lambda\left(  t\right)  -\lambda^{\left(  N\right)
}\left(  t\right)  \right\vert \leq t^{-\left(  1+\frac{p_{1}-1}{4}\right)
}\\
\left\vert \chi\left(  t\right)  -\chi^{\left(  N\right)  }\left(  t\right)
\right\vert \leq t^{-\frac{p_{1}-1}{4}}\\
\left\vert \gamma\left(  t\right)  -\gamma^{\left(  N\right)  }\left(
t\right)  \right\vert \leq t^{-\frac{p_{1}-1}{8}}.
\end{array}
\right.  \label{boot}%
\end{equation}
are satisfied. If $V=V^{\left(  2\right)  },$ for some $T_{0}>0$ large enough
the following a priori estimates are true%
\begin{equation}
\left\{
\begin{array}
[c]{c}%
\left\Vert \varepsilon\left(  t\right)  \right\Vert _{H^{1}}\leq
\mathcal{X}^{N},\\
\\
\left\vert \beta\left(  t\right)  -\beta^{\left(  N\right)  }\left(  t\right)
\right\vert +\left\vert \lambda\left(  t\right)  -\lambda^{\left(  N\right)
}\left(  t\right)  \right\vert \leq\left(  \Psi^{2}+\left\vert V^{\prime
\prime}\left(  \left\vert \tfrac{\chi}{\lambda}\right\vert \right)
\right\vert \right)  \mathcal{X}^{\frac{N}{2}+2},\\
\\
\left\vert \chi\left(  t\right)  -\chi^{\left(  N\right)  }\left(  t\right)
\right\vert +\left\vert \gamma\left(  t\right)  -\gamma^{\left(  N\right)
}\left(  t\right)  \right\vert \leq\Psi\mathcal{X}^{\frac{N}{2}+1}%
\end{array}
\right.  \label{bootbis}%
\end{equation}
for all $t\in\lbrack T_{0},T_{n}].$
\end{lemma}

\begin{remark}
The condition on $\lambda^{\infty}$ in Lemma \ref{L6} is made in order to
assure, via Lemma \ref{approxsyst}, that the assumptions of Lemmas \ref{Linv},
\ref{Lemmaapp} and \ref{Lemmaapp1} are satisfied.
\end{remark}

Lemma \ref{L6} is proved in Section \ref{Secproof} below. We now use the
following corollary to prove Theorem \ref{T1}.

\begin{corollary}
\label{L7}For $n\in\mathbb{N}$ let\ a sequence $T_{n}\rightarrow\infty,$ as
$n\rightarrow\infty$ and $u_{n}\in H^{1}$ be a solution to (\ref{appsys}).
Then there exists $T_{0}$ independent of $n$ such that for all $n\in
\mathbb{N}$ and $t\in\lbrack T_{0},T_{n}],$
\begin{equation}
\left\Vert u_{n}\left(  t\right)  -\mathcal{W}\left(  t,x;\Xi^{\left(
N\right)  }\left(  t\right)  \right)  \right\Vert _{H^{1}}\leq C\Theta
_{\infty} \label{l7}%
\end{equation}
where%
\[
\Theta_{\infty}=\left\{
\begin{array}
[c]{c}%
t^{-\frac{p_{1}-1}{8}},\text{ if }\left\vert V\left(  r\right)  \right\vert
\leq Ce^{-cr},\text{ }c>0,\\
\mathcal{X}^{\frac{N}{2}+1},\text{ if }V=V^{\left(  2\right)  }.
\end{array}
\right.
\]

\end{corollary}

\begin{proof}
We have
\[
\left\Vert u_{n}\left(  t\right)  -\mathcal{W}\left(  t,x;\Xi^{\left(
N\right)  }\left(  t\right)  \right)  \right\Vert _{H^{1}}\leq\left\Vert
u_{n}\left(  t\right)  -\mathcal{W}\left(  t,x;\Xi\left(  t\right)  \right)
\right\Vert _{H^{1}}+\left\Vert \mathcal{W}\left(  t,x;\Xi\left(  t\right)
\right)  -\mathcal{W}\left(  t,x;\Xi^{\left(  N\right)  }\left(  t\right)
\right)  \right\Vert _{H^{1}}\leq C\Theta_{\infty}.
\]

\end{proof}

By a compactness argument (see page 1525 of \cite{Krieger}) we now show that
Theorem \ref{T1} is a direct consequence of the uniform backward estimate for
the sequence $u_{n}$ presented in Corollary \ref{L7}. By (\ref{l7}) there is
$C>0$ such that for any $n\in\mathbb{N}$,%
\[
\left\Vert u_{n}\left(  t\right)  \right\Vert _{H^{1}}\leq C,
\]
for $t\in\lbrack T_{0},T_{n}].$ Then, by Lemma 3.4 of \cite{Krieger} there
exists $U_{0}\in H^{1}\left(  \mathbb{R}^{d}\right)  $ and a subsequence
$\{u_{n_{k}}\}$ of $\{u_{n}\}$ such that $u_{n_{k}}\left(  T_{0}\right)
\rightarrow U_{0}$ in $L^{2}\left(  \mathbb{R}^{d}\right)  ,$ as
$n_{k}\rightarrow\infty.$ We consider the solution $u$ to the initial value
problem
\[
\left\{
\begin{array}
[c]{c}%
\left.  i\partial_{t}u+\Delta u+\left\vert u\right\vert ^{p-1}u+\mathcal{V}%
\left(  x\right)  u=0,\right. \\
\text{ }(t,x)\in\mathbb{R\times R}^{d},\text{ }u\left(  T_{0}\right)  =U_{0}.
\end{array}
\right.
\]
Fix $t\geq T_{0}.$ There is $n_{0}$ large enough such that $T_{n}\geq t,$ for
$n\geq n_{0}.$ By the continuous dependence of the solution of (\ref{NLS}) on
the initial data in $L^{2}\left(  \mathbb{R}^{d}\right)  $ (see Theorem 4.6.1
of \cite{Cazenave}) $u\left(  t\right)  $ is global and
\[
u_{n_{k}}\left(  t\right)  \rightarrow u\left(  t\right)  \text{ in }%
L^{2}\left(  \mathbb{R}^{d}\right)  ,\text{ as }n\rightarrow\infty.
\]
Since $u_{n_{k}}\left(  t\right)  -\mathcal{W}\left(  t,x;\Xi^{\left(
N\right)  }\left(  t\right)  \right)  $ is uniformly bounded in $H^{1}\left(
\mathbb{R}^{d}\right)  ,$ it converges weakly to $u\left(  t\right)
-\mathcal{W}\left(  t,x;\Xi^{\left(  N\right)  }\left(  t\right)  \right)  $
in $H^{1}\left(  \mathbb{R}^{d}\right)  ,$ as $n\rightarrow\infty.$ Thus, by
(\ref{l7}) we prove that
\[
\left\Vert u_{n_{k}}\left(  t\right)  -\mathcal{W}\left(  t,x;\Xi^{\left(
N\right)  }\left(  t\right)  \right)  \right\Vert _{H^{1}}\leq C\Theta
_{\infty},
\]
for all $t\geq T_{0}.$ Therefore, Theorem \ref{T1} follows from the definition
(\ref{w}) of $\mathcal{W}$, Lemmas \ref{Lemmaapp} and \ref{Lemmaapp1}, the
properties of $\Xi^{\left(  N\right)  }\left(  t\right)  $ described by Lemma
\ref{approxsyst} and the relation $U_{V_{\lambda^{\infty}}}=\mathbf{U}%
_{\lambda^{\infty},\mathcal{V}},$ that follows from Definition \ref{Def1} and
(\ref{vcall}).

\section{A priori estimates. Proof of Lemma \ref{L6}.\label{Secproof}}

\subsection{Control of the Modulation Parameters.\label{ModPar}}

\subsubsection*{Case I: Fast decaying potentials.}

Let the potential $\left\vert V\left(  r\right)  \right\vert \leq Ce^{-cr},$
for some $c>0.$ Suppose that (\ref{boot}) is true for all $t\in\lbrack
T^{\ast},T_{n}],$ with $T_{0}\leq T^{\ast}<T_{n}.$ Observe that by
(\ref{boot}), (\ref{beh1}) and (\ref{beh2})%
\begin{equation}
\frac{\left\vert \chi\left(  t\right)  \right\vert }{\left\vert \chi^{\infty
}\left(  t\right)  \right\vert }=1+o\left(  1\right)  \text{, }\frac
{\left\vert \beta\left(  t\right)  \right\vert }{\left\vert \beta^{\infty
}\left(  t\right)  \right\vert }=1+o\left(  1\right)  \text{ and }%
\lambda\left(  t\right)  =\lambda^{\infty}+o\left(  1\right)  , \label{beh3}%
\end{equation}
as $t\rightarrow\infty.$ Let $\Phi\in C^{\infty}\left(  \mathbb{R}^{d}\right)
$ satisfy the estimate $\left\vert \Phi\left(  y\right)  \right\vert \leq
C\left(  \left\vert y\right\vert ^{m}Q\left(  y\right)  +\mathbf{Y}\right)  $
with some $m\geq$ $0.$\ We denote
\[
\Phi_{1}\left(  t,x\right)  =\lambda\left(  t\right)  ^{-\frac{2}{p-1}}%
\Phi\left(  \frac{x-\chi\left(  t\right)  }{\lambda\left(  t\right)  }\right)
e^{-i\gamma\left(  t\right)  }e^{i\beta\left(  t\right)  \cdot x}.
\]
Let $N\geq1.$ For $1\leq j\leq N$ and $\Xi\left(  t\right)  ,$ let
$B_{j}=B_{j}\left(  \Xi\left(  t\right)  \right)  $, $M_{j}=M_{j}\left(
\Xi\left(  t\right)  \right)  $ be the approximated modulation equations given
by Lemma \ref{Lemmaapp} corresponding to $\Xi\left(  t\right)  .$ Let
\[
\mathbf{B}^{\left(  N\right)  }=\mathbf{B}^{\left(  N\right)  }\left(
\Xi\left(  t\right)  \right)  =\sum_{j=1}^{N}B_{j}\left(  \Xi\left(  t\right)
\right)  \text{ \ and \ }\mathbf{M}^{\left(  N\right)  }=\mathbf{M}^{\left(
N\right)  }\left(  \Xi\left(  t\right)  \right)  =\sum_{j=1}^{N}M_{j}\left(
\Xi\left(  t\right)  \right)  .
\]
We put%
\[
\mathcal{M}\left(  t\right)  =\left\vert \dot{\chi}\left(  t\right)
-2\beta\left(  t\right)  \right\vert +\left\vert \dot{\gamma}\left(  t\right)
+\frac{1}{\lambda^{2}\left(  t\right)  }-\left\vert \beta\left(  t\right)
\right\vert ^{2}-\dot{\beta}\left(  t\right)  \cdot\chi\left(  t\right)
\right\vert +\left\vert \dot{\beta}\left(  t\right)  -\mathbf{B}^{\left(
N\right)  }\left(  \Xi\left(  t\right)  \right)  \right\vert +\left\vert
\dot{\lambda}\left(  t\right)  -\mathbf{M}^{\left(  N\right)  }\left(
\Xi\left(  t\right)  \right)  \right\vert .
\]
Let us write the equation for $\varepsilon.$ Introducing the decomposition
(\ref{decomp}) into (\ref{NLS}) and using (\ref{ap182}) we obtain%
\begin{equation}
i\partial_{t}\varepsilon=-\Delta\varepsilon-\mathcal{V}\left(  x\right)
\varepsilon-\left(  \lambda^{-\frac{2}{p-1}}\left(  t\right)  \left(
\lambda^{-2}\left(  t\right)  \mathcal{N}_{0}\left(  W,\zeta\right)
+\mathcal{E}^{\left(  N\right)  }+R\left(  W\right)  \right)  \left(
t,\frac{x-\chi\left(  t\right)  }{\lambda\left(  t\right)  }\right)
e^{-i\gamma\left(  t\right)  }e^{i\beta\left(  t\right)  \cdot x}\right)
\label{eqeps}%
\end{equation}
where we denote
\begin{equation}
\mathcal{N}_{0}\left(  W,\zeta\right)  =\left\vert W+\zeta\right\vert
^{p-1}\left(  W+\zeta\right)  -\left\vert W\right\vert ^{p-1}W\text{.}
\label{n0}%
\end{equation}
with $W$ given by (\ref{w1}) and $\zeta$ defined by (\ref{epsilonbar}). Using
(\ref{eqeps}) we have%
\begin{equation}
\left.
\begin{array}
[c]{c}%
\frac{d}{dt}\operatorname{Im}\int\varepsilon\overline{\Phi_{1}}%
=-\operatorname{Re}\int i\partial_{t}\varepsilon\overline{\Phi_{1}%
}+\operatorname{Re}\int\varepsilon\overline{\left(  i\partial_{t}\Phi
_{1}\right)  }\\
=\operatorname{Re}\int\varepsilon\left(  \overline{i\partial_{t}\Phi
_{1}+\Delta\Phi_{1}+\mathcal{V}\left(  x\right)  \Phi_{1}}\right)
+\operatorname{Re}\int\mathcal{N}_{0}\left(  \mathcal{W},\varepsilon\right)
\overline{\Phi_{1}\left(  t,x\right)  }+\lambda^{-\frac{4}{p-1}}%
\operatorname{Re}\int\left(  \mathcal{E}^{\left(  N\right)  }+R\left(
W\right)  \right)  \overline{\Phi}.
\end{array}
\right.  \label{ap243}%
\end{equation}
Similarly to (\ref{calc1}) we calculate%
\[
\left.
\begin{array}
[c]{c}%
\left(  i\partial_{t}+\Delta+\mathcal{V}\left(  x\right)  \right)  \Phi
_{1}=\lambda^{-\frac{2p}{p-1}}\left(  -\mathcal{\tilde{L}}\Phi+\lambda
^{2}V\left(  \left\vert y+\frac{\chi}{\lambda}\right\vert \right)
\Phi-i\lambda\mathbf{M}^{\left(  N\right)  }\Lambda\Phi-\lambda^{3}\left(
\mathbf{B}^{\left(  N\right)  }\cdot y\right)  \Phi\right. \\
-\frac{p+1}{2}\left\vert W\right\vert ^{p-1}\Phi-\frac{p-1}{2}\left\vert
W\right\vert ^{p-3}W^{2}\Phi-i\lambda\left(  \dot{\lambda}-\mathbf{M}^{\left(
N\right)  }\right)  \Lambda\Phi-\lambda^{3}\left(  \left(  \dot{\beta
}-\mathbf{B}^{\left(  N\right)  }\right)  \cdot y\right)  \Phi\\
\left.  +\lambda^{2}\left(  \dot{\gamma}+\frac{1}{\lambda^{2}}-\left\vert
\beta\right\vert ^{2}-\dot{\beta}\cdot\chi\right)  \Phi\right)  \left(
t,\frac{x-\chi}{\lambda}\right)  e^{-i\gamma}e^{i\beta\cdot x}.
\end{array}
\right.
\]
where we define%
\begin{equation}
\mathcal{\tilde{L}}f:=-\Delta f+f-\frac{p+1}{2}\left\vert W\right\vert
^{p-1}f-\frac{p-1}{2}\left\vert W\right\vert ^{p-3}W^{2}{\overline{f}},\text{
\ \ }f\in H^{1}. \label{Ltilde}%
\end{equation}
Then, using (\ref{BM}) to estimate the modulation equations $\mathbf{B}%
^{\left(  N\right)  },\mathbf{M}^{\left(  N\right)  }$ and relation
(\ref{eps1}) from (\ref{ap243}) we get%
\begin{equation}
\left.  \frac{d}{dt}\operatorname{Im}%
{\displaystyle\int}
\varepsilon\overline{\Phi_{1}}=\lambda^{-\frac{4}{p-1}}\operatorname{Re}%
{\displaystyle\int}
R\left(  W\right)  \overline{\Phi}-\operatorname{Re}%
{\displaystyle\int}
\varepsilon\lambda^{-\frac{2p}{p-1}}\left(  \left(  \mathcal{\tilde{L}%
}-\lambda^{2}V\left(  \left\vert y+\frac{\chi}{\lambda}\right\vert \right)
\right)  \Phi\right)  \left(  t,\tfrac{x-\chi}{\lambda}\right)  e^{-i\gamma
}e^{i\beta\cdot x}+E_{\mathcal{N}_{1}}+E_{0}\right.  \label{ap271}%
\end{equation}
where%
\begin{equation}
\mathcal{N}_{1}\left(  W,\zeta\right)  =\mathcal{N}_{0}\left(  W,\zeta\right)
-\dfrac{p+1}{2}\left\vert W\right\vert ^{p-1}\zeta-\dfrac{p-1}{2}\left\vert
W\right\vert ^{p-3}W^{2}\overline{\zeta} \label{n1}%
\end{equation}%
\[
E_{\mathcal{N}_{1}}=\operatorname{Re}%
{\displaystyle\int}
\lambda^{-\frac{2p}{p-1}}\left(  \mathcal{N}_{1}\left(  W,\zeta\right)
\Phi\right)  \left(  t,\frac{x-\chi}{\lambda}\right)  e^{-i\gamma}%
e^{i\beta\cdot x}%
\]
and%
\[
E_{0}=O\left(  \Theta\left(  \left(  \left\vert \beta\right\vert
+\Theta\right)  ^{A\left(  N\right)  }+U\right)  +\left(  \left(  \left(
\left\vert \beta\right\vert +\Theta\right)  ^{2}+U\right)  +\mathcal{M}\left(
t\right)  \right)  \left\Vert \varepsilon\right\Vert _{H^{1}}\right)
\]
with $\Theta=\Theta\left(  \tfrac{\left\vert \chi\right\vert }{\lambda
}\right)  $ and $U=\left\vert U_{V}\left(  \left\vert \tfrac{\chi}{\lambda
}\right\vert \right)  \right\vert .$ Let us estimate $E_{\mathcal{N}_{1}}$.
Observe that
\begin{equation}
\left\vert \mathcal{N}_{1}\left(  W,\zeta\right)  \right\vert =O\left(
\left\vert \zeta\right\vert ^{p}+\left\vert \zeta\right\vert ^{p-\delta
}+\left\vert \zeta\right\vert ^{2}\right)  ,\text{ }\delta>0, \label{ap352}%
\end{equation}
for $p\neq2$ ($\mathcal{N}_{1}\left(  W,\zeta\right)  =O\left(  \left\vert
\zeta\right\vert ^{2}\right)  $, $p=2$). If we estimate $E_{\mathcal{N}_{1}}$
by using (\ref{ap352}) we have $E_{\mathcal{N}_{1}}=O\left(  \left\Vert
\varepsilon\right\Vert _{H^{1}}^{p-\delta}+\left\Vert \varepsilon\right\Vert
_{H^{1}}^{2}\right)  .$ By (\ref{boot}) this gives the decay $E_{\mathcal{N}%
_{1}}\sim t^{-p+\delta}+t^{-2}.$ In order to obtain the a priori estimates on
modulation parameters in (\ref{boot}), we need $\mathcal{M}\left(  t\right)  $
to be integrable twice on $[T_{0},\infty).$ But, in the case when $p<2,$ we
only have $E_{\mathcal{N}_{1}}\sim t^{-p+\delta}.\ $Hence, we need to obtain a
better estimate on $E_{\mathcal{N}_{1}}.$ We use an argument of \cite{Nguyen}
(see the proof of Proposition 10 on page 41). Let
\begin{equation}
\Omega=\left\{  y\in\mathbb{R}^{d}:\max_{1\leq j\leq N}\left\vert
\boldsymbol{T}^{\left(  j\right)  }\left(  y\right)  \right\vert \geq\frac
{1}{2N}Q\left(  y\right)  \right\}  . \label{ap244}%
\end{equation}
Then, as by (\ref{ap14}) $\left\vert \boldsymbol{T}^{\left(  j\right)
}\right\vert \leq C\mathbf{Y,}$ $1\leq j\leq N,$ we estimate
\[
\left\vert \Phi\left(  y\right)  \right\vert \leq C\left(  \left\vert
y\right\vert ^{M}Q\left(  y\right)  +\mathbf{Y}\right)  \leq C_{N}%
\inf_{0<\delta<1}\left(  C\left(  \delta\right)  Q^{1-\delta}\left(  y\right)
\left\vert y\right\vert ^{M}e^{-\delta\left\vert y\right\vert }\right)
+C\mathbf{Y}\leq C\mathbf{Y},\text{ }y\in\Omega,
\]
Thus, using the relation (\ref{ap352}) and Sobolev embedding theorem to
control the $L^{p}$ norm of $\zeta,$ we estimate
\begin{equation}%
{\displaystyle\int_{y\in\Omega}}
\left\vert \left(  \mathcal{N}_{1}\left(  W,\zeta\right)  \Phi\right)  \left(
y\right)  \right\vert dy\leq C_{N}C\left(  \delta\right)  \left\Vert
\zeta\right\Vert _{H^{1}}^{p_{0}}\Theta^{1-\delta},\text{ }p_{0}%
=\min\{p-\delta,2\}, \label{ap292}%
\end{equation}
for all $0<\delta<1.$ Suppose now that $y\notin\Omega$. Using (\ref{ap352}) we
estimate%
\begin{equation}
\left\vert \mathcal{N}_{1}\left(  W,\zeta\right)  \right\vert =O\left(
\left\vert W\right\vert ^{p-2}\left\vert \zeta\right\vert ^{2}\right)  .
\label{ap301}%
\end{equation}
Thus, by Sobolev embedding theorem, we get
\begin{equation}%
{\displaystyle\int_{y\notin\Omega}}
\left\vert \left(  \mathcal{N}_{1}\left(  W,\zeta\right)  \Phi\right)  \left(
y\right)  \right\vert dy\leq C\left\Vert \zeta\right\Vert _{H^{1}}^{2}.
\label{ap293}%
\end{equation}
Hence, from (\ref{ap292}) and (\ref{ap293}) we estimate%
\begin{equation}
\left\vert \operatorname{Re}%
{\displaystyle\int}
\lambda^{-\frac{2p}{p-1}}\left(  \mathcal{N}_{1}\left(  W,\zeta\right)
\Phi\right)  \left(  t,\frac{x-\chi}{\lambda}\right)  e^{-i\gamma}%
e^{i\beta\cdot x}\right\vert \leq C_{N}C\left(  \delta\right)  \left\Vert
\zeta\right\Vert _{H^{1}}^{p_{0}}\Theta^{1-\delta}+C\left\Vert \zeta
\right\Vert _{H^{1}}^{2} \label{ap294}%
\end{equation}
for all $0<\delta<1.$

Let us now calculate $\mathcal{\tilde{L}}\Phi$ for $\Phi=$ $iW,$ $iyW,$
$\Lambda W,$ $\nabla W.$ We use (\ref{E7}) and Lemma \ref{Lemmaapp} to derive%
\begin{equation}
\Delta W-W+\left\vert W\right\vert ^{p-1}W+\lambda^{2}V\left(  \left\vert
y+\tfrac{\chi}{\lambda}\right\vert \right)  W=O\left(  E_{1}\right)
\label{ap284}%
\end{equation}
where%
\[
E_{1}=\left(  \left\vert \beta\right\vert +\Theta+U\right)  e^{-\delta
\left\vert y\right\vert },\text{ }\delta>0.
\]
Letting act the group of symmetries of the free NLS\ equation on (\ref{ap284})
and differentiating with respect to the symmetry parameters the resulting
relation we calculate%
\begin{equation}
\left(  \mathcal{\tilde{L}-}\lambda^{2}\left(  V\left(  \left\vert
y+\tfrac{\chi}{\lambda}\right\vert \right)  \right)  \right)  \left(
iW\right)  =O\left(  E_{1}\right)  , \label{lb1}%
\end{equation}%
\begin{equation}
\left(  \mathcal{\tilde{L}-}\lambda^{2}V\left(  \left\vert y+\tfrac{\chi
}{\lambda}\right\vert \right)  \right)  \left(  iyW\right)  =-2\nabla
W+O\left(  E_{1}\right)  \label{lb2}%
\end{equation}%
\begin{equation}
\left(  \mathcal{\tilde{L}-}\lambda^{2}V\left(  \left\vert y+\tfrac{\chi
}{\lambda}\right\vert \right)  \right)  \left(  \Lambda W\right)
=-2W+2\lambda^{2}V\left(  \left\vert y+\tfrac{\chi}{\lambda}\right\vert
\right)  W+O\left(  \left\vert V^{\prime}\left(  \left\vert y+\tfrac{\chi
}{\lambda}\right\vert \right)  \left(  yW\left(  y\right)  \right)
\right\vert \right)  +O\left(  E_{1}\right)  , \label{lb3}%
\end{equation}%
\begin{equation}
\left(  \mathcal{\tilde{L}-}\lambda^{2}V\left(  \left\vert y+\tfrac{\chi
}{\lambda}\right\vert \right)  \right)  \left(  \nabla W\right)  =O\left(
\left\vert V^{\prime}\left(  \left\vert y+\tfrac{\chi}{\lambda}\right\vert
\right)  W\right\vert \right)  +O\left(  E_{1}\right)  . \label{lb4}%
\end{equation}
We now use (\ref{ap271}) with $\Phi=$ $iW,$ $iyW,$ $\Lambda W,$ $\nabla W.$ By
the orthogonal conditions (\ref{ap208}), using (\ref{lb1})-(\ref{lb4}) and
(\ref{ap294}) we obtain
\begin{equation}
\left.
\begin{array}
[c]{c}%
\left\vert \mathcal{M}\left(  t\right)  \right\vert \leq C\Theta\left(
\left(  \left\vert \beta\right\vert +\Theta\right)  ^{A\left(  N\right)
}+U\right)  +C\left(  \delta\right)  \left(  \left\vert \beta\right\vert
+\Theta+U+\mathcal{M}\left(  t\right)  +\left\Vert \zeta\right\Vert _{H^{1}%
}^{p_{0}}\Theta^{1-\delta}+\left\Vert \zeta\right\Vert _{H^{1}}\right)
\left\Vert \zeta\right\Vert _{H^{1}}\\
+C\left(  \left\Vert \left(  V\left(  \left\vert y+\tfrac{\chi}{\lambda
}\right\vert \right)  -V\left(  \left\vert \tfrac{\chi}{\lambda}\right\vert
\right)  \right)  W\right\Vert _{L^{2}}+\left\Vert V^{\prime}\left(
\left\vert y+\tfrac{\chi}{\lambda}\right\vert \right)  \left(  yW\left(
y\right)  \right)  \right\Vert _{L^{2}}\right)  \left\Vert \zeta\right\Vert
_{H^{1}},
\end{array}
\right.  \label{ap273}%
\end{equation}
for all $0<\delta<1.$ By (\ref{estp}) we estimate%
\begin{equation}
\left\Vert \left(  V\left(  \left\vert y+\tfrac{\chi}{\lambda}\right\vert
\right)  -V\left(  \left\vert \tfrac{\chi}{\lambda}\right\vert \right)
\right)  W\right\Vert _{L^{2}}+\left\Vert V^{\prime}\left(  \left\vert
y+\tfrac{\chi}{\lambda}\right\vert \right)  yW\left(  y\right)  \right\Vert
_{L^{2}}\leq Cr^{\infty}\sqrt{U}. \label{ap274}%
\end{equation}
Thus, from (\ref{ap273}) we get%
\begin{equation}
\left\vert \mathcal{M}\left(  t\right)  \right\vert \leq C\Theta\left(
\left(  \left\vert \beta\right\vert +\Theta\right)  ^{A\left(  N\right)
}+U\right)  +C\left(  \delta\right)  \left(  \left\vert \beta\right\vert
+\Theta+r^{\infty}\sqrt{U}+\left\Vert \zeta\right\Vert _{H^{1}}^{p_{0}}%
\Theta^{1-\delta}+\left\Vert \zeta\right\Vert _{H^{1}}\right)  \left\Vert
\zeta\right\Vert _{H^{1}}. \label{Mt}%
\end{equation}

By assumption, (\ref{boot}) is true for all $t\in\lbrack T^{\ast},T_{n}].$
From (\ref{syst1}) $\left\vert \beta^{\infty}\right\vert \leq C\mathcal{X},$
with $\mathcal{X}=\sqrt{U_{V}\left(  \frac{\left\vert \chi^{\infty}\right\vert
}{\lambda^{\infty}}\right)  }.$ Let $N\geq\frac{2}{\left(  p_{1}-1\right)  },$
so that $A\left(  N\right)  \geq2.$ Relation (\ref{ap302}) implies that
$\mathcal{X}\leq Ct^{-1}.$ Then, as $\Theta\leq C\mathcal{X}$, from
(\ref{beh3}) and (\ref{Mt}) with $\delta<\frac{p_{1}-1}{2}$
\begin{equation}
\left\vert \mathcal{M}\left(  t\right)  \right\vert \leq Ct^{-5/2}
\label{estmod1}%
\end{equation}
for all $t\in\lbrack T^{\ast},T_{n}].$ We are now\ in position to improve
(\ref{boot}) for the modulation parameters by using (\ref{estmod1}). By
(\ref{estmod1}) and (\ref{BM}) the vector $\Xi\left(  t\right)  $ solves%
\begin{equation}
\left\{
\begin{array}
[c]{c}%
\dot{\chi}\left(  t\right)  =2\beta\left(  t\right)  +O\left(  t^{-5/2}%
\right)  ,\\
\dot{\beta}\left(  t\right)  =B_{1}\left(  \Xi\left(  t\right)  \right)
+O\left(  t^{-5/2}+t^{-2p_{1}}\right)  ,\\
\dot{\lambda}\left(  t\right)  =M_{1}\left(  \Xi\left(  t\right)  \right)
+O\left(  t^{-5/2}+t^{-2p_{1}}\right)  ,\\
\dot{\gamma}\left(  t\right)  =-\frac{1}{\lambda^{2}\left(  t\right)
}+\left\vert \beta\left(  t\right)  \right\vert ^{2}+\dot{\beta}\left(
t\right)  \cdot\chi\left(  t\right)  +O\left(  t^{-5/2}\right)  .
\end{array}
\right.  \label{syst10}%
\end{equation}
As $\Xi\left(  t\right)  $ is only $C^{1},$ we cannot proceed directly as when
we considered (\ref{boot1}). We argue slightly different. Using (\ref{l2}) we
write $B_{1}\left(  \Xi\left(  t\right)  \right)  =\frac{\chi}{\left\vert
\chi\right\vert }b\left(  \left\vert \chi\right\vert ,\lambda\right)  .$ By
(\ref{boot1}) we deduce $\left\vert \chi^{\left(  N\right)  }\left(  t\right)
-r_{\operatorname*{app}}\theta_{0}^{\infty}\right\vert \leq t^{-\frac{p_{1}%
}{4}},$ for some constant vector $\theta_{0}^{\infty}\in\mathbb{S}^{d-1}.$
Using (\ref{BM}) and (\ref{beh3}) we estimate $\left\vert B_{1}\right\vert
+\left\vert \nabla_{\chi}B_{1}\right\vert +\left\vert \partial_{\lambda}%
B_{1}\right\vert \leq C\mathcal{X}^{2}\leq Ct^{-2}.$ Then, the equation for
$\dot{\beta}\left(  t\right)  $ in (\ref{syst10}) takes the form%
\begin{equation}
\dot{\beta}\left(  t\right)  =\theta_{0}^{\infty}b\left(  \left\vert
\chi\right\vert ,\lambda^{\left(  N\right)  }\right)  +O\left(  \left(
\left\vert \lambda-\lambda^{\left(  N\right)  }\right\vert +\frac{\left\vert
\chi-\chi^{\left(  N\right)  }\right\vert }{\left\vert \chi^{\left(  N\right)
}\right\vert }+t^{-\frac{p_{1}}{4}}\right)  t^{-2}\right)  +O\left(
t^{-5/2}+t^{-2p_{1}}\right)  . \label{ap318}%
\end{equation}
Similarly, we have%
\begin{equation}
\dot{\beta}^{\left(  N\right)  }\left(  t\right)  =\theta_{0}^{\infty}b\left(
\left\vert \chi^{\left(  N\right)  }\right\vert ,\lambda^{\left(  N\right)
}\right)  +O\left(  t^{-2-\frac{p_{1}}{4}}\right)  . \label{ap320}%
\end{equation}
As in (\ref{syst5}), we use the Cartesian coordinates with the $x_{1}$-axis
directed along the vector $\theta_{0}^{\infty}\ $and for a vector
$\mathbf{A}\in\mathbb{R}^{d}$ we decompose $\mathbf{A}=\sum_{j=1}^{d}A_{j}%
\vec{e}_{j},$ and $\vec{e}_{j},$ $j=1,...,d$ is the canonical basis in these
coordinates. Then, from (\ref{ap318}) and (\ref{ap320}) we deduce%
\begin{equation}
\dot{\beta}_{1}=b\left(  \left\vert \chi\right\vert ,\lambda^{\left(
N\right)  }\right)  +O\left(  \left(  \left\vert \mu\right\vert +\frac
{\left\vert \delta\right\vert }{r^{\infty}}+t^{-\frac{p_{1}}{4}}\right)
t^{-2}+t^{-5/2}+t^{-2p_{1}}\right)  , \label{beta1}%
\end{equation}%
\begin{equation}
\dot{\beta}_{1}^{\left(  N\right)  }=b\left(  \left\vert \chi^{\left(
N\right)  }\right\vert ,\lambda^{\left(  N\right)  }\right)  +O\left(  \left(
\frac{\left\vert \delta\right\vert }{r^{\infty}}+t^{-\frac{p_{1}}{4}}\right)
t^{-2}+t^{-5/2}+t^{-2p_{1}}\right)  \label{beta1N}%
\end{equation}
and%
\begin{equation}
\dot{\beta}_{j}-\dot{\beta}_{j}^{\left(  N\right)  }=O\left(  \left(
\frac{\left\vert \delta\right\vert }{r^{\infty}}+t^{-\frac{p_{1}}{4}}\right)
t^{-2}+t^{-5/2}+t^{-2p_{1}}\right)  ,\text{ }j=2,...,d, \label{betaj}%
\end{equation}
where we denote by $\delta\left(  t\right)  =\chi\left(  t\right)
-\chi^{\left(  N\right)  }\left(  t\right)  $ and $\mu\left(  t\right)
=\lambda\left(  t\right)  -\lambda^{\left(  N\right)  }\left(  t\right)  .$
Then, integrating (\ref{betaj}) and using (\ref{boot}) we show that for any
constant $A>0$ there is $T_{0}>0$ such that
\begin{equation}
\left\vert \delta_{j}\left(  t\right)  \right\vert \leq\frac{1}{At^{\frac
{p_{1}-1}{4}}}\text{ and }\left\vert \dot{\delta}_{j}\left(  t\right)
\right\vert \leq\frac{1}{At^{1+\frac{p_{1}-1}{4}}},\text{ }t\in\lbrack
T^{\ast},T_{n}],\text{ }T^{\ast}\geq T_{0}, \label{ap322}%
\end{equation}
for $j=2,...,d.$ From (\ref{beta1}) and (\ref{beta1N}), by using the equation
for $\dot{\chi}\left(  t\right)  $ in (\ref{syst10}) and integrating, via
(\ref{boot}) and (\ref{ap322}) we deduce%
\begin{equation}
\dot{\delta}_{1}=\left(  1+o\left(  1\right)  \right)  \left(  2\dot
{r}^{\infty}\right)  ^{-1}b\left(  r^{\infty},\lambda^{\infty}\right)
\delta_{1}+O\left(  \frac{1}{At^{1+\frac{p_{1}-1}{4}}}\right)  . \label{ap321}%
\end{equation}
By (\ref{l2}) and (\ref{ap302}) $\left(  \dot{r}^{\infty}\right)
^{-1}b\left(  r^{\infty},\lambda^{\infty}\right)  =\frac{K_{1}}{t}\left(
1+o\left(  1\right)  \right)  ,$ with $K_{1}>1.$ Then%
\begin{equation}
\dot{\delta}_{1}=\frac{K_{2}}{t}\delta_{1}+\frac{o\left(  1\right)
}{t^{1+\frac{p_{1}-1}{4}}}+O\left(  \frac{1}{At^{1+\frac{p_{1}-1}{4}}}\right)
. \label{ap326}%
\end{equation}
where $K_{2}>1.$ Thus, similarly to (\ref{ap324}) we deduce that for some
$T_{0}>0$
\[
\left\vert \delta_{1}\right\vert \leq\frac{K_{3}}{A}t^{-\frac{p_{1}-1}{4}%
},\text{ }t\in\lbrack T^{\ast},T_{n}],\text{ }K_{3}>1.
\]
Moreover, from (\ref{ap326}) we get%
\[
\left\vert \dot{\delta}_{1}\right\vert \leq\frac{K_{4}}{A}t^{-\left(
1+\frac{p_{1}-1}{4}\right)  },\text{ }K_{4}>1.
\]
Taking $A\geq2K_{3}K_{4}$ we deduce
\[
\left\vert \delta_{1}\right\vert \leq\frac{1}{2}t^{-\frac{p_{1}-1}{4}}\text{
and }\left\vert \dot{\delta}_{1}\right\vert \leq\frac{1}{2}t^{-\left(
1+\frac{p_{1}-1}{4}\right)  },\text{ }t\in\lbrack T^{\ast},T_{n}].
\]
Combining the last inequalities with (\ref{ap322}) we prove%
\begin{equation}
\left\vert \delta\right\vert \leq\frac{1}{2}t^{-\frac{p_{1}-1}{4}}\text{ and
}\left\vert \dot{\delta}\right\vert \leq\frac{1}{2}t^{-\left(  1+\frac
{p_{1}-1}{4}\right)  },\text{ }t\in\lbrack T^{\ast},T_{n}]. \label{imp1}%
\end{equation}
By (\ref{BM}) and (\ref{beh3}) we estimate $\left\vert M_{1}\right\vert
+\left\vert \nabla_{\chi}M_{1}\right\vert +\left\vert \partial_{\lambda}%
M_{1}\right\vert \leq C\mathcal{X}^{2}\leq Ct^{-2}.$ Moreover, by (\ref{BM})%
\[
\dot{\lambda}^{\left(  N\right)  }\left(  t\right)  =M_{1}\left(  \Xi^{\left(
N\right)  }\left(  t\right)  \right)  +O\left(  t^{-2p_{1}}\right)  .
\]
Then, by (\ref{boot})
\begin{align*}
\left\vert \dot{\lambda}-\dot{\lambda}^{\left(  N\right)  }\right\vert  &
\leq\left\vert M_{1}\left(  \Xi\left(  t\right)  \right)  -M_{1}^{\left(
N\right)  }\left(  \Xi^{\left(  N\right)  }\left(  t\right)  \right)
\right\vert +O\left(  t^{-2p_{1}}\right) \\
&  \leq\left\vert \nabla_{\chi}M_{1}\right\vert \left\vert \delta\right\vert
+\left\vert \partial_{\lambda}M_{1}\right\vert \left\vert \mu\right\vert
+O\left(  t^{-2p_{1}}\right)  \leq Ct^{-2-\frac{p_{1}-1}{4}}.
\end{align*}
Finally, using (\ref{boot}), as $\left\vert \dot{\beta}^{\left(  N\right)
}\left(  t\right)  \right\vert \leq C\mathcal{X}^{2}\leq Ct^{-2}$ and by
(\ref{beh2}), (\ref{ap199}) $\left\vert \chi^{\left(  N\right)  }\left(
t\right)  \right\vert \leq C\ln t,$ we have
\begin{align*}
\left\vert \dot{\gamma}\left(  t\right)  -\dot{\gamma}^{\left(  N\right)
}\left(  t\right)  \right\vert  &  \leq C\left\vert \lambda\left(  t\right)
-\lambda^{\left(  N\right)  }\left(  t\right)  \right\vert +C\left\vert
\beta\left(  t\right)  -\beta^{\left(  N\right)  }\left(  t\right)
\right\vert +C\left\vert \dot{\beta}^{\left(  N\right)  }\left(  t\right)
\right\vert \left\vert \chi\left(  t\right)  -\chi^{\left(  N\right)  }\left(
t\right)  \right\vert \\
&  +C\left\vert \chi^{\left(  N\right)  }\left(  t\right)  \right\vert \left(
\left\vert \dot{\beta}^{\left(  N\right)  }\left(  t\right)  \right\vert
+\left\vert \dot{\beta}\left(  t\right)  \right\vert \right)  +Ct^{-5/2}\leq
Ct^{-\left(  1+\frac{p_{1}-1}{4}\right)  }.
\end{align*}
Hence, there is $T_{0}>0$ such that
\begin{equation}
\left\vert \lambda-\lambda^{\left(  N\right)  }\right\vert \leq\frac
{t^{-\left(  1+\frac{p_{1}-1}{4}\right)  }}{2}\text{ and }\left\vert
\dot{\gamma}\left(  t\right)  -\dot{\gamma}^{\left(  N\right)  }\left(
t\right)  \right\vert \leq\frac{t^{-\frac{p_{1}-1}{8}}}{2}, \label{imp2}%
\end{equation}
for $t\in\lbrack T^{\ast},T_{n}].$ Therefore, from (\ref{imp1}) and
(\ref{imp2}) we improve (\ref{boot}) for the modulation parameters $\left(
\Xi\left(  t\right)  ,\gamma\left(  t\right)  \right)  $ on $[T^{\ast}%
,T_{n}].$ By a continuity argument, we show that (\ref{boot}) for $\left(
\Xi\left(  t\right)  ,\gamma\left(  t\right)  \right)  $ is satisfied on
$[T_{0},T_{n}].$\

\begin{remark}
\label{Rem1}In the case when $V=V^{\left(  2\right)  }$ the best bound on the
error term in (\ref{eps1}) that we are able to obtain by following the
construction provided by Lemma \ref{Lemmaapp} is $\mathcal{E}^{\left(
N\right)  }\sim V^{\prime}\dot{r}^{\infty}.$ For potentials that behave as
$V\left(  r\right)  \sim r^{-\rho}$, $0<\rho\leq2,$ we cannot\ even integrate
$V^{\prime}\dot{r}^{\infty}$ two times on $[t,\infty)$. Then, the proof of
(\ref{boot}) fails for the approximate modulations equations given in Lemma
\ref{Lemmaapp}. Therefore, we use a different approximated modulation
equations constructed in Lemma \ref{Lemmaapp1} to cover this case.
\end{remark}

\subsubsection*{Case II: Slow decaying potentials.}

We let now the potential $V=V^{\left(  2\right)  }.$ In this case we suppose
that (\ref{bootbis}) is true for all $t\in\lbrack T^{\ast},T_{n}],$ with
$T_{0}\leq T^{\ast}<T_{n}.$ By (\ref{bootbis}), (\ref{beh1}) and (\ref{beh2})
we see that (\ref{beh3}) is true. Let $N\geq1.$ For $1\leq j\leq N$ and
$\Xi\left(  t\right)  ,$ let $\tilde{B}_{j}=\tilde{B}_{j}\left(  \Xi\left(
t\right)  \right)  $, $\tilde{M}_{j}=\tilde{M}_{j}\left(  \Xi\left(  t\right)
\right)  $ be the approximated modulation equations given by Lemma
\ref{Lemmaapp1} corresponding to $\Xi\left(  t\right)  .$ We\ omit the tilde
and denote these equations by $B_{j}=B_{j}\left(  \Xi\left(  t\right)
\right)  $, $M_{j}=M_{j}\left(  \Xi\left(  t\right)  \right)  $. Let $\Phi\in
C^{\infty}\left(  \mathbb{R}^{d}\right)  $ satisfy the estimate $\left\vert
\Phi\left(  y\right)  \right\vert \leq C\left(  \left\vert y\right\vert
^{m}Q\left(  y\right)  +\Psi\mathbf{e}\right)  $ with some $m\geq$ $0.$\ We
denote
\[
\Phi_{1}\left(  t,x\right)  =\lambda\left(  t\right)  ^{-\frac{2}{p-1}}%
\Phi\left(  \frac{x-\chi\left(  t\right)  }{\lambda\left(  t\right)  }\right)
e^{-i\gamma\left(  t\right)  }e^{i\beta\left(  t\right)  \cdot x}.
\]
By using (\ref{B'}) and (\ref{eps1bis}) instead of (\ref{BM}), (\ref{eps1}),
relation (\ref{ap271}) takes the form
\begin{equation}
\left.
\begin{array}
[c]{c}%
\frac{d}{dt}\operatorname{Im}%
{\displaystyle\int}
\varepsilon\overline{\Phi_{1}}=\lambda^{-\frac{4}{p-1}}\operatorname{Re}%
{\displaystyle\int}
\mathcal{R}\overline{\Phi}-\operatorname{Re}%
{\displaystyle\int}
\varepsilon\lambda^{-\frac{2p}{p-1}}\left(  \left(  \mathcal{\tilde{L}%
}-\lambda^{2}V\left(  \left\vert y+\frac{\chi}{\lambda}\right\vert \right)
\right)  \Phi\right)  \left(  t,\frac{x-\chi}{\lambda}\right)  e^{-i\gamma
}e^{i\beta\cdot x}\\
\\
+\operatorname{Re}%
{\displaystyle\int}
\lambda^{-\frac{2p}{p-1}}\left(  \mathcal{N}_{1}\left(  W,\zeta\right)
\Phi\right)  \left(  t,\frac{x-\chi}{\lambda}\right)  e^{-i\gamma}%
e^{i\beta\cdot x}+O\left(  \Psi\mathbf{Z}^{N}+\Psi\left\Vert \varepsilon
\right\Vert _{H^{1}}\right)  .
\end{array}
\right.  \label{ap323}%
\end{equation}
By (\ref{ap352}) and Sobolev embedding theorem we get%
\begin{equation}
\left\vert
{\displaystyle\int}
\lambda^{-\frac{2p}{p-1}}\left(  \mathcal{N}_{1}\left(  W,\zeta\right)
\Phi\right)  \left(  t,\tfrac{x-\chi}{\lambda}\right)  e^{-i\gamma}%
e^{i\beta\cdot x}\right\vert \leq C\left\Vert \varepsilon\right\Vert _{H^{1}%
}^{p-\delta}. \label{ap325}%
\end{equation}
Using (\ref{E7}) and Lemma \ref{Lemmaapp1} we obtain (\ref{ap284}) with
$E_{1}$ replaced by $O\left(  \mathbf{Z}\Psi\right)  .$ Then, (\ref{lb1}%
)-(\ref{lb4}) are true with $O\left(  \mathbf{Z}\Psi\right)  $ instead of
$E_{1}.$ Therefore, similarly to (\ref{ap273}), by using (\ref{ap323}) with
$\Phi=$ $iW,$ $iyW,$ $\Lambda W,$ $\nabla W,$ via (\ref{ap325}) we get%
\[
\left.
\begin{array}
[c]{c}%
\left\vert \mathcal{M}\left(  t\right)  \right\vert \leq C\Psi\left\Vert
\varepsilon\right\Vert _{H^{1}}+C\left(  \delta\right)  \left(  \Psi
^{1-\delta}\left(  \Psi^{p-1}+\left\Vert \zeta\right\Vert _{H^{1}}%
^{p-1}\right)  +C\left\Vert \zeta\right\Vert _{H^{1}}\right)  \left\Vert
\zeta\right\Vert _{H^{1}}\\
\\
+C\left(  \left\Vert \left(  V\left(  \left\vert y+\tfrac{\chi}{\lambda
}\right\vert \right)  -V\left(  \left\vert \tfrac{\chi}{\lambda}\right\vert
\right)  \right)  W\right\Vert _{L^{2}}+\left\Vert V^{\prime}\left(
\left\vert y+\tfrac{\chi}{\lambda}\right\vert \right)  W\right\Vert _{L^{2}%
}\right)  \left\Vert \zeta\right\Vert _{H^{1}}.
\end{array}
\right.
\]
By (\ref{estp1}) and (\ref{potentialprime}) \
\[
\left\Vert \left(  V\left(  \left\vert y+\tfrac{\chi}{\lambda}\right\vert
\right)  -V\left(  \left\vert \tfrac{\chi}{\lambda}\right\vert \right)
\right)  W\right\Vert _{L^{2}}+\left\Vert V^{\prime}\left(  \left\vert
y+\tfrac{\chi}{\lambda}\right\vert \right)  W\right\Vert _{L^{2}}\leq C\Psi.
\]
Hence, we obtain%
\begin{equation}
\left\vert \mathcal{M}\left(  t\right)  \right\vert \leq C\Psi\mathbf{Z}%
^{N}+C\left(  \delta\right)  \left(  \Psi^{1-\delta}\left\Vert \zeta
\right\Vert _{H^{1}}^{p-1}+\Psi+\left\Vert \zeta\right\Vert _{H^{1}}\right)
\left\Vert \zeta\right\Vert _{H^{1}} \label{Mtbis}%
\end{equation}

\bigskip\ By assumption, (\ref{bootbis}) is true for all $t\in\lbrack T^{\ast
},T_{n}].$ Recall that $\left\vert \beta^{\infty}\right\vert \leq
C\mathcal{X}$, where $\mathcal{X}=\sqrt{U_{V}\left(  \frac{\left\vert
\chi^{\infty}\right\vert }{\lambda^{\infty}}\right)  }=\sqrt{V\left(
\frac{\left\vert \chi^{\infty}\right\vert }{\lambda^{\infty}}\right)  },$ and%
\[
\mathbf{Z}\leq C\mathcal{X}\text{.}%
\]
Taking $[N-8]\rho>4,$ we estimate
\begin{equation}
\mathcal{X}^{N/4}\leq\Psi. \label{Z}%
\end{equation}
Then%
\begin{equation}
\left\vert \mathcal{M}\left(  t\right)  \right\vert \leq C\left(
\Psi\mathcal{X}^{N}+\mathcal{X}^{2N}\right)  \leq C\Psi\mathcal{X}^{N}.
\label{estmod2}%
\end{equation}
We now use (\ref{estmod2}) to improve (\ref{bootbis}). By (\ref{B'}) we
estimate
\begin{equation}
\left\vert \nabla_{\chi}\mathbf{M}^{\left(  N\right)  }\right\vert +\left\vert
\nabla_{\chi}\mathbf{B}^{\left(  N\right)  }\right\vert \leq C\left(  \Psi
^{2}+\left\vert V^{\prime\prime}\left(  \left\vert \tfrac{\chi}{\lambda
}\right\vert \right)  \right\vert \right)  \label{ap307}%
\end{equation}
and
\begin{equation}
\left\vert \nabla_{\beta}\mathbf{M}^{\left(  N\right)  }\right\vert
+\left\vert \nabla_{\beta}\mathbf{B}^{\left(  N\right)  }\right\vert
+\left\vert \partial_{\lambda}\mathbf{M}^{\left(  N\right)  }\right\vert
+\left\vert \partial_{\lambda}\mathbf{B}^{\left(  N\right)  }\right\vert \leq
C\Psi. \label{ap316}%
\end{equation}
From the equations for $\beta^{\left(  N\right)  },\lambda^{\left(  N\right)
}$ we get
\begin{equation}
\left.
\begin{array}
[c]{c}%
\left\vert \dot{\lambda}-\dot{\lambda}^{\left(  N\right)  }\right\vert
+\left\vert \dot{\beta}-\dot{\beta}^{\left(  N\right)  }\right\vert
\leq\left\vert \mathcal{M}\left(  t\right)  \right\vert +\left\vert
\mathbf{M}^{\left(  N\right)  }\left(  \Xi\left(  t\right)  \right)
-\mathbf{M}^{\left(  N\right)  }\left(  \Xi^{\left(  N\right)  }\left(
t\right)  \right)  \right\vert +\left\vert \mathbf{B}^{\left(  N\right)
}\left(  \Xi\left(  t\right)  \right)  -\mathbf{B}^{\left(  N\right)  }\left(
\Xi^{\left(  N\right)  }\left(  t\right)  \right)  \right\vert \\
\\
\leq\left\vert \mathcal{M}\left(  t\right)  \right\vert +C\left(  \left\vert
\nabla_{\chi}\mathbf{M}^{\left(  N\right)  }\right\vert +\left\vert
\nabla_{\chi}\mathbf{B}^{\left(  N\right)  }\right\vert \right)  \left\vert
\chi-\chi^{\left(  N\right)  }\right\vert \\
\\
+C\left(  \left\vert \nabla_{\beta}\mathbf{M}^{\left(  N\right)  }\right\vert
+\left\vert \nabla_{\beta}\mathbf{B}^{\left(  N\right)  }\right\vert \right)
\left\vert \beta-\beta^{\left(  N\right)  }\right\vert +C\left(  \left\vert
\partial_{\lambda}\mathbf{M}^{\left(  N\right)  }\right\vert +\left\vert
\partial_{\lambda}\mathbf{B}^{\left(  N\right)  }\right\vert \right)
\left\vert \lambda-\lambda^{\left(  N\right)  }\right\vert
\end{array}
\right.  \label{ap347}%
\end{equation}
Recall that by (\ref{cond5}) there is $N_{0}>0$ such that $\mathcal{X}%
^{2N_{0}}\leq C\left\vert V^{\prime\prime}\left(  \tfrac{r^{\infty}}%
{\lambda^{\infty}}\right)  \right\vert .$ Taking $N>4\left(  N_{0}+1\right)  $
we get%
\begin{equation}
\mathcal{X}^{N}\leq\mathcal{X}^{2N_{0}}\mathcal{X}^{\frac{N}{2}+2}\leq
C\left\vert V^{\prime\prime}\left(  \tfrac{r^{\infty}}{\lambda^{\infty}%
}\right)  \right\vert \mathcal{X}^{\frac{N}{2}+2}. \label{ap344}%
\end{equation}
Then, by (\ref{estmod2})%
\begin{equation}
\left\vert \mathcal{M}\left(  t\right)  \right\vert \leq C\left(  \Psi
^{2}+\left\vert V^{\prime\prime}\left(  \left\vert \tfrac{r^{\infty}}%
{\lambda^{\infty}}\right\vert \right)  \right\vert \right)  \Psi
\mathcal{X}^{\frac{N}{2}+2}. \label{ap346}%
\end{equation}
Thus, using (\ref{ap307}) and (\ref{ap316}), by (\ref{bootbis}) and (\ref{Z})
from (\ref{ap347}) we have%
\begin{equation}
\left\vert \dot{\lambda}-\dot{\lambda}^{\left(  N\right)  }\right\vert
+\left\vert \dot{\beta}-\dot{\beta}^{\left(  N\right)  }\right\vert \leq
C\left(  \Psi^{2}+\left\vert V^{\prime\prime}\left(  \left\vert \tfrac
{r^{\infty}}{\lambda^{\infty}}\right\vert \right)  \right\vert \right)
\Psi\mathcal{X}^{\frac{N}{2}+1}. \label{ap209}%
\end{equation}
As $V,V^{\prime}$ are monotone, $V^{\prime},V^{\prime\prime}$ are of a
definite sign. Integrating (\ref{ap209}) on $[t,\infty)$ we deduce%
\begin{equation}
\left.
\begin{array}
[c]{c}%
\left\vert \lambda\left(  t\right)  -\lambda^{\left(  N\right)  }\left(
t\right)  \right\vert +\left\vert \beta\left(  t\right)  -\beta^{\left(
N\right)  }\left(  t\right)  \right\vert \\
\leq C\left(  \Psi^{2}+\left\vert V^{\prime\prime}\left(  \left\vert
\tfrac{\chi}{\lambda}\right\vert \right)  \right\vert \right)  \left(
\left\vert \int_{r^{\infty}}^{\infty}V^{\prime}\left(  r\right)  V^{\frac
{N}{4}}\left(  r\right)  dr\right\vert +\left(  r^{\infty}\right)  ^{\rho
\frac{N}{4}}\mathcal{X}^{\frac{N}{2}+2}\int_{r^{\infty}}^{\infty}\frac
{dr}{r^{1+\rho\frac{N}{4}}}\right) \\
\leq\frac{C}{\rho N}\left(  \Psi^{2}+\left\vert V^{\prime\prime}\left(
\tfrac{r^{\infty}}{\lambda^{\infty}}\right)  \right\vert \right)
\mathcal{X}^{\frac{N}{2}+2},
\end{array}
\right.  \label{ap342}%
\end{equation}
for all $t\in\lbrack T^{\ast},T_{n}].$ Using (\ref{ap346}) and (\ref{ap342})
we obtain%
\[
\left\vert \dot{\chi}-\dot{\chi}^{\left(  N\right)  }\right\vert \leq
C\Psi\mathcal{X}^{N}+C\left\vert \beta\left(  t\right)  -\beta^{\left(
N\right)  }\left(  t\right)  \right\vert \leq C\left(  \Psi^{2}+\left\vert
V^{\prime\prime}\left(  \tfrac{r^{\infty}}{\lambda^{\infty}}\right)
\right\vert \right)  \Psi\mathcal{X}^{\frac{N}{2}+2}+\frac{C}{\rho N}\left(
\Psi^{2}+\left\vert V^{\prime\prime}\left(  \tfrac{r^{\infty}}{\lambda
^{\infty}}\right)  \right\vert \right)  \mathcal{X}^{\frac{N}{2}+2},
\]
and then,%
\begin{equation}
\left\vert \chi\left(  t\right)  -\chi^{\left(  N\right)  }\left(  t\right)
\right\vert \leq\frac{C}{\rho N}\Psi\mathcal{X}^{\frac{N}{2}+1}. \label{ap348}%
\end{equation}
for all $t\in\lbrack T^{\ast},T_{n}].$ Finally, by (\ref{ap346}),
(\ref{ap209}), (\ref{ap342}) and (\ref{ap348}) we deduce%
\begin{align*}
\left\vert \dot{\gamma}\left(  t\right)  -\dot{\gamma}^{\left(  N\right)
}\left(  t\right)  \right\vert  &  \leq C\left(  \Psi^{2}+\left\vert
V^{\prime\prime}\left(  \left\vert \tfrac{r^{\infty}}{\lambda^{\infty}%
}\right\vert \right)  \right\vert \right)  \Psi\mathcal{X}^{\frac{N}{2}%
+2}+\frac{C}{\rho N}\left(  \Psi^{2}+\left\vert V^{\prime\prime}\left(
\tfrac{r^{\infty}}{\lambda^{\infty}}\right)  \right\vert \right)
\mathcal{X}^{\frac{N}{2}+2}+\frac{C}{\rho N}\Psi\mathcal{X}^{\frac{N}{2}%
+1}\left\vert \dot{\beta}^{\left(  N\right)  }\left(  t\right)  \right\vert \\
&  +C\left\vert \chi^{\left(  N\right)  }\left(  t\right)  \right\vert \left(
\Psi^{2}+\left\vert V^{\prime\prime}\left(  \left\vert \tfrac{r^{\infty}%
}{\lambda^{\infty}}\right\vert \right)  \right\vert \right)  \Psi
\mathcal{X}^{\frac{N}{2}+1}.
\end{align*}
Hence, \
\begin{equation}
\left\vert \gamma\left(  t\right)  -\gamma^{\left(  N\right)  }\left(
t\right)  \right\vert \leq\frac{C}{\rho N}\Psi\mathcal{X}^{\frac{N}{2}%
+2},\text{ }t\in\lbrack T^{\ast},T_{n}]. \label{ap349}%
\end{equation}
Therefore, by (\ref{ap342}), (\ref{ap348}) and (\ref{ap349}), taking $N$
sufficiently large we prove that
\[
\left.
\begin{array}
[c]{c}%
\left\vert \lambda\left(  t\right)  -\lambda^{\left(  N\right)  }\left(
t\right)  \right\vert +\left\vert \beta\left(  t\right)  -\beta^{\left(
N\right)  }\left(  t\right)  \right\vert \leq\frac{1}{2}\left(  \Psi
^{2}+\left\vert V^{\prime\prime}\left(  \tfrac{r^{\infty}}{\lambda^{\infty}%
}\right)  \right\vert \right)  \mathcal{X}^{\frac{N}{2}+2},\\
\left\vert \chi\left(  t\right)  -\chi^{\left(  N\right)  }\left(  t\right)
\right\vert +\left\vert \gamma\left(  t\right)  -\gamma^{\left(  N\right)
}\left(  t\right)  \right\vert \leq\frac{1}{2}\Psi\mathcal{X}^{\frac{N}{2}+1},
\end{array}
\right.
\]
on $t\in\lbrack T^{\ast},T_{n}],$ which strictly improve (\ref{bootbis}) for
the modulation parameters $\left(  \Xi\left(  t\right)  ,\gamma\left(
t\right)  \right)  $. Again, by a continuity argument, we show that
(\ref{bootbis})\ for $\left(  \Xi\left(  t\right)  ,\gamma\left(  t\right)
\right)  $ is true on $[T_{0},T_{n}].$

\subsection{Control of the error $\varepsilon$.\label{Sectionepsilon}}

Let us consider the energy, the mass and the momentum of $u_{n}$. By using the
orthogonal conditions (\ref{ap208})\ we have%
\begin{equation}
\left.
\begin{array}
[c]{c}%
E\left(  u_{n}\right)  =E\left(  \mathcal{W+\varepsilon}\right)  =E\left(
\mathcal{W}\right)  +\frac{1}{2}%
{\displaystyle\int}
\left\vert \mathcal{\nabla}\varepsilon\right\vert ^{2}-\frac{1}{2}%
{\displaystyle\int}
\mathcal{V}\left\vert \varepsilon\right\vert ^{2}-\dfrac{1}{1+p}\left(
{\displaystyle\int}
\left\vert \mathcal{W+\varepsilon}\right\vert ^{p+1}-%
{\displaystyle\int}
\left\vert \mathcal{W}\right\vert ^{p+1}\right) \\
-\operatorname{Re}%
{\displaystyle\int}
\mathcal{\Delta W}\overline{\varepsilon}-\operatorname{Re}%
{\displaystyle\int}
\mathcal{VW}\overline{\varepsilon},
\end{array}
\right.  \label{ap210}%
\end{equation}%
\begin{equation}
M\left(  u_{n}\right)  =\operatorname{Im}\int\mathcal{\nabla}\left(
\mathcal{W+}\varepsilon\right)  \left(  \overline{\mathcal{W+}\varepsilon
}\right)  =\operatorname{Im}\int\mathcal{\nabla W}\overline{\mathcal{W}%
}+\operatorname{Im}\int\mathcal{\nabla}\varepsilon\overline{\varepsilon}
\label{ap211}%
\end{equation}
and%
\begin{equation}
\int\left\vert u_{n}\right\vert ^{2}=\int\left\vert \mathcal{W}\right\vert
^{2}+\int\left\vert \varepsilon\right\vert ^{2}. \label{ap212}%
\end{equation}
Moreover, by (\ref{ap182}) and (\ref{ap208}) we have%
\[
\operatorname{Re}%
{\displaystyle\int}
\left(  \mathcal{\Delta W+V}\left(  \cdot\right)  \mathcal{W}\right)
\overline{\varepsilon}=-\operatorname{Re}%
{\displaystyle\int}
\left(  \left\vert \mathcal{W}\right\vert ^{p-1}\mathcal{W}+\lambda^{-\frac
{2}{p-1}}\mathcal{E}^{\left(  N\right)  }\left(  \frac{x-\chi\left(  t\right)
}{\lambda\left(  t\right)  }\right)  e^{-i\gamma\left(  t\right)  }%
e^{i\beta\left(  t\right)  \cdot x}\right)  \overline{\varepsilon}.
\]
Then, from (\ref{ap210})%
\begin{equation}
\left.
\begin{array}
[c]{c}%
E\left(  u_{n}\right)  =E\left(  \mathcal{W+\varepsilon}\right)  =E\left(
\mathcal{W}\right)  +\frac{1}{2}%
{\displaystyle\int}
\left\vert \mathcal{\nabla}\varepsilon\right\vert ^{2}-\frac{1}{2}%
{\displaystyle\int}
\mathcal{V}\left\vert \varepsilon\right\vert ^{2}\\
-\dfrac{1}{1+p}\left(
{\displaystyle\int}
\left\vert \mathcal{W+\varepsilon}\right\vert ^{p+1}-%
{\displaystyle\int}
\left\vert \mathcal{W}\right\vert ^{p+1}-\left(  p+1\right)  \left\vert
\mathcal{W}\right\vert ^{p-1}\operatorname{Re}%
{\displaystyle\int}
\mathcal{W}\overline{\varepsilon}\right)  +\operatorname*{Er},
\end{array}
\right.  \label{ap319}%
\end{equation}
with%
\[
\operatorname*{Er}=\operatorname{Re}%
{\displaystyle\int}
\left(  \lambda^{-\frac{2}{p-1}}\mathcal{E}^{\left(  N\right)  }\left(
\frac{x-\chi\left(  t\right)  }{\lambda\left(  t\right)  }\right)
e^{-i\gamma\left(  t\right)  }e^{i\beta\left(  t\right)  \cdot x}\right)
\overline{\varepsilon}.
\]
Note that by Lemma \ref{L1} $\operatorname*{Er}$ is small.

\subsubsection*{\bigskip Case I: Fast decaying potentials.}

Let the potential $\left\vert V\left(  r\right)  \right\vert \leq Ce^{-cr},$
for some $c>0.$ Suppose that (\ref{boot}) is true for all $t\in\lbrack
T^{\ast},T_{n}],$ with $T_{0}\leq T^{\ast}<T_{n}.$ Let $\psi_{\lambda^{\infty
}}\in C^{\infty}\left(  \mathbb{R}^{d}\right)  $ be such that $0\leq\psi
\leq1,$ $\psi\left(  x\right)  =1$ for $\left\vert x\right\vert \leq\frac
{1}{2\lambda^{\infty}}$ and $\psi\left(  x\right)  =0$ for $\left\vert
x\right\vert \geq\frac{1}{\lambda^{\infty}}.$ Set
\[
\psi_{\chi}\left(  x\right)  =\psi_{\lambda^{\infty}}\left(  \frac{8\left(
x-\chi\left(  t\right)  \right)  }{\lambda\left(  t\right)  \ln t}\right)  .
\]
We introduce the following conserved quantity of $u_{n}$ which is a
combination of the three conservation laws (\ref{ap210})-(\ref{ap212}) for the
NLS equation with a potential (\ref{NLS})%
\[
K_{\operatorname*{tot}}=E\left(  u_{n}\right)  -\beta\left(  t\right)
\cdot\operatorname{Im}\int\psi_{\chi}\mathcal{\nabla}u_{n}\overline{u_{n}%
}+\frac{1}{2}\left(  \lambda^{-2}\left(  t\right)  +\left\vert \beta\left(
t\right)  \right\vert ^{2}\right)
{\displaystyle\int}
\left\vert u_{n}\right\vert ^{2}.
\]
Observe that we localize the momentum in a neighborhood of the potential. We
compare $K_{\operatorname*{tot}}$ with
\[
K_{\operatorname*{sol}}=E\left(  \mathcal{W}\right)  -\beta\left(  t\right)
\cdot\operatorname{Im}\int\psi_{\chi}\mathcal{\nabla W}\overline{\mathcal{W}%
}+\frac{1}{2}\left(  \lambda^{-2}\left(  t\right)  +\left\vert \beta\left(
t\right)  \right\vert ^{2}\right)  \int\left\vert \mathcal{W}\right\vert
^{2}.
\]
Then, from (\ref{ap211}), (\ref{ap212}),\ (\ref{ap319}) we deduce%
\[
\left.  K_{\operatorname*{tot}}-K_{\operatorname*{sol}}=\mathcal{G}\left(
t\right)  +\operatorname*{Er}\right.
\]
where%
\begin{align*}
\mathcal{G}\left(  \varepsilon\left(  t\right)  \right)  =\mathcal{G}%
_{\mathcal{W}}\left(  \varepsilon\left(  t\right)  \right)  =  &  \frac{1}{2}%
{\displaystyle\int}
\left\vert \mathcal{\nabla}\varepsilon\right\vert ^{2}+\frac{1}{2}\left(
\lambda^{-2}\left(  t\right)  +\left\vert \beta\left(  t\right)  \right\vert
^{2}\right)  \int\left\vert \varepsilon\right\vert ^{2}-\frac{1}{2}%
{\displaystyle\int}
\mathcal{V}\left\vert \varepsilon\right\vert ^{2}\\
&  -\dfrac{1}{1+p}%
{\displaystyle\int}
\left(  \left\vert \mathcal{W+\varepsilon}\right\vert ^{p+1}-\left\vert
\mathcal{W}\right\vert ^{p+1}-\left(  1+p\right)  \left\vert \mathcal{W}%
\right\vert ^{p-1}\operatorname{Re}\left(  \mathcal{W}\overline{\varepsilon
}\right)  \right) \\
&  -\beta\left(  t\right)  \cdot\operatorname{Im}\int\psi_{\chi}%
\mathcal{\nabla}\varepsilon\overline{\varepsilon}.
\end{align*}
($\mathcal{V}$ is related to $V$ by (\ref{vcall})). We now study the
properties of $\mathcal{G}\left(  \varepsilon\left(  t\right)  \right)  .$ We
have the following result.

\begin{lemma}
\label{Lemmaenergy}\bigskip There exist $T_{0}$ large enough such that for
$t\in\lbrack T^{\ast},T_{n}],$ $T^{\ast}\geq T_{0},$ the following hold.

i) (Coercivity of the linearized energy) There is a constant $c_{0}>0$ such
that%
\begin{equation}
\mathcal{G}\left(  \varepsilon\left(  t\right)  \right)  \geq c_{0}\left\Vert
\varepsilon\right\Vert _{H^{1}}^{2}. \label{estbelow1}%
\end{equation}

ii) (Energy estimate on $\varepsilon$) For some $N$ large enough%
\begin{equation}
\left\vert \frac{d}{dt}\mathcal{G}\left(  \varepsilon\left(  t\right)
\right)  \right\vert \leq\frac{C}{t^{3}\ln t}. \label{estbelow}%
\end{equation}

\end{lemma}

Before proving Lemma \ref{Lemmaenergy} we improve (\ref{boot}) for
$\varepsilon\left(  t\right)  .$ Integrating (\ref{estbelow}) we have%
\[
\mathcal{G}\left(  \varepsilon\left(  t\right)  \right)  \leq\frac{C}{t^{2}\ln
t}.
\]
Then, using (\ref{estbelow1}) we prove
\begin{equation}
\left\Vert \varepsilon\right\Vert _{H^{1}}^{2}\leq\frac{C}{c_{0}t^{2}\ln t}.
\label{ap388}%
\end{equation}
Thus, taking $T_{0}>e^{2C/c_{0}}$ we strictly improve (\ref{boot}) for
$\varepsilon\left(  t\right)  .$ By a continuity argument, we conclude that
(\ref{bootbis})\ for $\varepsilon\left(  t\right)  $ is true on $[T_{0}%
,T_{n}].$

\begin{proof}
[Proof of Lemma \ref{Lemmaenergy}]\textit{Proof of item i. }We decompose
$\mathcal{G}\left(  \varepsilon\left(  t\right)  \right)  $ as
\begin{equation}
\mathcal{G}\left(  \varepsilon\left(  t\right)  \right)  =\mathcal{G}%
_{1}+\mathcal{G}_{2}+\mathcal{G}_{3}, \label{ap383}%
\end{equation}
where%
\[
\mathcal{G}_{1}=\tfrac{1}{2}%
{\displaystyle\int}
\left(  \left\vert \mathcal{\nabla}\varepsilon\right\vert ^{2}+\lambda
^{-2}\left(  t\right)  \left\vert \varepsilon\right\vert ^{2}-\mathcal{V}%
\left\vert \varepsilon\right\vert ^{2}-\tfrac{p+1}{2}\left\vert \mathcal{W}%
\right\vert ^{p-1}\left\vert \varepsilon\right\vert ^{2}-\tfrac{p-1}%
{2}\left\vert \mathcal{W}\right\vert ^{p-3}\operatorname{Re}\left(
\mathcal{W}\overline{\varepsilon}\right)  ^{2}\right)  {,}%
\]%
\[
\mathcal{G}_{2}=-\dfrac{1}{1+p}\left(
{\displaystyle\int}
\left(  \left\vert \mathcal{W+\varepsilon}\right\vert ^{p+1}-\left\vert
\mathcal{W}\right\vert ^{p+1}-\left(  1+p\right)  \left\vert \mathcal{W}%
\right\vert ^{p-1}\operatorname{Re}\left(  \mathcal{W}\overline{\varepsilon
}\right)  \right)  -\frac{1+p}{2}\operatorname{Re}\left(  \tfrac{p+1}%
{2}\left\vert \mathcal{W}\right\vert ^{p-1}\left\vert \varepsilon\right\vert
^{2}+\tfrac{p-1}{2}\left\vert \mathcal{W}\right\vert ^{p-3}\left(
\mathcal{W}\overline{\varepsilon}\right)  ^{2}\right)  \right)
\]
and%
\begin{equation}
\mathcal{G}_{3}=-\beta\left(  t\right)  \cdot\operatorname{Im}\int\psi_{\chi
}\mathcal{\nabla}\varepsilon\overline{\varepsilon}+\tfrac{1}{2}\left\vert
\beta\left(  t\right)  \right\vert ^{2}\int\left\vert \varepsilon\right\vert
^{2}=O\left(  \left\vert \beta\right\vert \left\Vert \varepsilon\right\Vert
_{H^{1}}^{2}\right)  . \label{ap384}%
\end{equation}
By Sobolev embedding theorem we estimate%
\begin{equation}
\mathcal{G}_{2}=O\left(  \left\Vert \varepsilon\right\Vert _{H^{1}%
}^{p+1-\delta}\right)  ,\text{ }0<\delta<1. \label{ap385}%
\end{equation}
Using (\ref{epsilonbar}) and (\ref{w}) we have%
\[
\mathcal{G}_{1}=\tfrac{1}{2}\left(  \lambda\left(  t\right)  \right)
^{d-\frac{2\left(  p+1\right)  }{p-1}}\left(  \mathcal{L}_{V}\zeta
,\zeta\right)  +\mathcal{G}_{11}%
\]
with $\mathcal{L}_{V}$ defined by (\ref{Lv}) and%
\[
\mathcal{G}_{11}=-\left(  \lambda\left(  t\right)  \right)  ^{d-\frac{2\left(
p+1\right)  }{p-1}}%
{\displaystyle\int}
\left(  \tfrac{p+1}{2}\left(  \left\vert W\right\vert ^{p-1}-Q^{p-1}\right)
\left\vert \zeta\right\vert ^{2}+\tfrac{p-1}{2}\operatorname{Re}\left(
\left(  \left\vert W\right\vert ^{p-3}W^{2}-Q^{p-1}\right)  \overline{\zeta
}^{2}\right)  \right)  .
\]
By Lemma \ref{Lemmaapp} we show that $\mathcal{G}_{11}=O\left(  \Theta
^{1-\delta}\left\Vert \varepsilon\right\Vert _{H^{1}}^{2}\right)  $ and%
\[
\left\vert \left(  \zeta,W-Q\right)  \right\vert +\left\vert \left(
\zeta,x\left(  W-Q\right)  \right)  \right\vert +\left\vert \left(
\zeta,i\Lambda\left(  W-Q\right)  \right)  \right\vert =O\left(
\Theta^{1-\delta}\left\Vert \varepsilon\right\Vert _{H^{1}}\right)  .
\]
Then, from (\ref{ineqpert1}), by using (\ref{ap208}) we conclude that there is
$c>0$ such that%
\[
\mathcal{G}_{1}\geq c\left\Vert \varepsilon\right\Vert _{H^{1}}^{2}+O\left(
\Theta^{1-\delta}\left\Vert \varepsilon\right\Vert _{H^{1}}^{2}\right)  .
\]
Hence, from (\ref{ap383}), (\ref{ap384}), (\ref{ap385}), via (\ref{boot}), we
conclude that for $T_{0}$ large enough there is $c_{0}>0$ such that for
$t\in\lbrack T^{\ast},T_{n}],$ $T^{\ast}\geq T_{0}$ (\ref{estbelow1}) is true.

\textit{Proof of item ii. }Let us make the change of variables $\varepsilon
_{1}=e^{-i\gamma_{1}\left(  t\right)  }\varepsilon$ and $\mathcal{W}%
_{1}=e^{-i\gamma_{1}\left(  t\right)  }\mathcal{W}$, with
\[
\gamma_{1}\left(  t\right)  =\int_{T_{0}}^{t}\left(  \lambda^{-2}\left(
\tau\right)  +\left\vert \beta\left(  t\right)  \right\vert ^{2}\right)
d\tau.
\]
Then $\mathcal{G}_{\mathcal{W}}\left(  \varepsilon\left(  t\right)  \right)
=\mathcal{G}_{\mathcal{W}_{1}}\left(  \varepsilon_{1}\left(  t\right)
\right)  .$ \ Differentiating $\mathcal{G}_{\mathcal{W}_{1}}\left(
\varepsilon_{1}\left(  t\right)  \right)  $ we have%
\begin{equation}
\frac{d}{dt}\mathcal{G}_{\mathcal{W}_{1}}\left(  \varepsilon_{1}\left(
t\right)  \right)  =\mathcal{G}_{11}\left(  t\right)  +\mathcal{G}_{12}\left(
t\right)  +\mathcal{G}_{13}\left(  t\right)  , \label{dtG}%
\end{equation}
where%
\[
\mathcal{G}_{11}\left(  t\right)  =-\left(  \mathcal{\dot{W}}_{1}%
,\mathcal{N}_{1}\left(  \mathcal{W}_{1},\varepsilon_{1}\right)  \right)  ,
\]%
\[
\mathcal{G}_{12}\left(  t\right)  =\left(  \dot{\varepsilon}_{1}%
,-\Delta\varepsilon_{1}+\left(  \lambda^{-2}\left(  t\right)  +\left\vert
\beta\left(  t\right)  \right\vert ^{2}\right)  \varepsilon_{1}-\mathcal{V}%
\varepsilon_{1}-\mathcal{N}_{0}\left(  \mathcal{W}_{1},\varepsilon_{1}\right)
\right)  .
\]
and%
\[
\mathcal{G}_{13}\left(  t\right)  =-\dot{\beta}\cdot\operatorname{Im}\int
\psi_{\chi}\mathcal{\nabla}\varepsilon_{1}\overline{\varepsilon}_{1}%
-\beta\cdot\operatorname{Im}\int\dot{\psi}_{\chi}\mathcal{\nabla}%
\varepsilon_{1}\overline{\varepsilon}_{1}+\beta\cdot\operatorname{Im}%
\int\mathcal{\nabla}\psi_{\chi}\dot{\varepsilon}_{1}\overline{\varepsilon}%
_{1}+2\beta\cdot\operatorname{Im}\int\psi_{\chi}\mathcal{\nabla}%
\overline{\varepsilon_{1}}\dot{\varepsilon}_{1}.
\]
Let us consider $\mathcal{G}_{11}\left(  t\right)  .$ By (\ref{w}) we have%
\begin{equation}
\left.
\begin{array}
[c]{c}%
\mathcal{\dot{W}}_{1}=-\lambda^{-\frac{2}{p-1}}\left(  \frac{\dot{\lambda}%
}{\lambda}\Lambda W+\left(  \frac{\dot{\chi}}{\lambda}\cdot\nabla
W+2i\left\vert \beta\right\vert ^{2}W\right)  \right)  \left(  \frac{x-\chi
}{\lambda}\right)  e^{-i\left(  \gamma\left(  t\right)  +\gamma_{1}\left(
t\right)  \right)  }e^{i\beta\left(  t\right)  \cdot x}\\
-\lambda^{-\frac{2}{p-1}}\left(  i\left(  \dot{\gamma}+\lambda^{-2}-\left\vert
\beta\right\vert ^{2}-\dot{\beta}\cdot\chi\right)  W-i\lambda\left(
\dot{\beta}\cdot y\right)  W+\dot{W}\right)  \left(  \frac{x-\chi}{\lambda
}\right)  e^{-i\left(  \gamma\left(  t\right)  +\gamma_{1}\left(  t\right)
\right)  }e^{i\beta\left(  t\right)  \cdot x}.
\end{array}
\right.  \label{ap357}%
\end{equation}
Let $\phi$ be the indicator function of the complement of the set $\Omega$
defined by (\ref{ap244}). In particular, $\left\vert T\left(  y\right)
\right\vert \leq\frac{1}{2}Q\left(  y\right)  $ when $\phi=1.$ Moreover, as by
(\ref{ap14}) $\left\vert \boldsymbol{T}^{\left(  j\right)  }\right\vert \leq
C\mathbf{Y,}$ $Q\left(  y\right)  \leq C\mathbf{Y}$ whereas $\phi=0.$ Using
(\ref{ap357}) we decompose%
\begin{equation}
\mathcal{G}_{11}\left(  t\right)  =\operatorname{Re}\int\mathbf{W}%
\mathcal{N}_{1}\left(  \mathcal{W}_{1},\varepsilon_{1}\right)  dx+\mathcal{G}%
_{11}^{\left(  0\right)  }\left(  t\right)  +\mathcal{G}_{11}^{\left(
1\right)  }\left(  t\right)  , \label{g11'}%
\end{equation}
with%
\begin{equation}
\mathbf{W=}\lambda^{-\frac{2}{p-1}}\left(  \phi\left(  \frac{\dot{\chi}%
}{\lambda}\cdot\nabla W+2i\left\vert \beta\right\vert ^{2}W\right)  \right)
\left(  \frac{x-\chi}{\lambda}\right)  e^{-i\left(  \gamma\left(  t\right)
+\gamma_{1}\left(  t\right)  \right)  }e^{i\beta\cdot x}, \label{wbold}%
\end{equation}%
\[
\mathcal{G}_{11}^{\left(  0\right)  }\left(  t\right)  =-\lambda^{-\frac
{2}{p-1}}\operatorname{Re}\int\left(  1-\phi\right)  \mathcal{\dot{W}}%
_{1}\mathcal{N}_{1}\left(  \mathcal{W}_{1},\varepsilon_{1}\right)  dx
\]
and%
\[
\mathcal{G}_{11}^{\left(  1\right)  }\left(  t\right)  =\lambda^{-\frac
{2}{p-1}}\operatorname{Re}\int\mathcal{N}_{1}\left(  \mathcal{W}%
_{1},\varepsilon_{1}\right)  \mathbf{W}_{1}dx,
\]
with%
\[
\mathbf{W}_{1}=\left(  \phi\left(  \frac{\dot{\lambda}}{\lambda}\Lambda
W+i\left(  \dot{\gamma}+\lambda^{-2}-\left\vert \beta\right\vert ^{2}%
-\dot{\beta}\cdot\chi\right)  W-i\lambda\left(  \dot{\beta}\cdot y\right)
W+\dot{W}\right)  \right)  \left(  \tfrac{x-\chi}{\lambda}\right)
e^{-i\left(  \gamma\left(  t\right)  +\gamma_{1}\left(  t\right)  \right)
}e^{i\beta\cdot x}.
\]
Note that
\begin{equation}
\dot{W}=2\beta\cdot\nabla_{\chi}W+B\cdot\nabla_{\beta}W+M\frac{\partial
W}{\partial\lambda}. \label{wdot}%
\end{equation}
Therefore, using (\ref{ap357}), and Lemma \ref{Lemmaapp} to control $\Lambda
W,W,yW,\dot{W},$ we get
\begin{equation}
\left\vert \mathcal{\dot{W}}_{1}\left(  y\right)  \right\vert \leq C\left(
\delta_{1}\right)  \left(  \left\vert \beta\right\vert +\mathcal{M}\left(
t\right)  +\left\vert \mathbf{B}^{\left(  N\right)  }\right\vert +\left\vert
\mathbf{M}^{\left(  N\right)  }\right\vert \right)  \left(  Q^{1-\delta_{1}%
}\left(  y\right)  +\mathbf{Y}\right)  ,\text{ }y\in\mathbb{R}^{d},
\label{ap353}%
\end{equation}
for $0<\delta_{1}<1.$ Thus, by (\ref{ap352}), (\ref{ap353}), Sobolev embedding
theorem and (\ref{boot}), as $Q\left(  y\right)  \leq C\mathbf{Y}$ on the
support of $1-\phi,$ taking $\delta_{1}$ and $\delta$ small enough, we
estimate
\begin{equation}
\left\vert \mathcal{G}_{11}^{\left(  0\right)  }\left(  t\right)  \right\vert
\leq C\left(  \left\vert \beta\right\vert +\mathcal{M}\left(  t\right)
+\left\vert \mathbf{B}^{\left(  N\right)  }\right\vert +\left\vert
\mathbf{M}^{\left(  N\right)  }\right\vert \right)  \mathbf{Y}^{1-\delta_{1}%
}\left(  \left\Vert \varepsilon_{1}\right\Vert _{H^{1}}^{p-\delta}+\left\Vert
\varepsilon_{1}\right\Vert _{H^{1}}^{2}\right)  \leq Ct^{-3-\frac{\left(
p_{1}-1\right)  }{2}}, \label{g110'}%
\end{equation}
for $t\in\lbrack T^{\ast},T_{n}].$ On the support of $\phi$ we have
\begin{equation}
\left\vert W\left(  y\right)  \right\vert \geq2^{-1}Q\left(  y\right)  .
\label{ap317}%
\end{equation}
Using (\ref{ap301}) and (\ref{ap317}), and taking $\delta<\frac{p_{1}-1}{2}$
we estimate%
\begin{equation}
\phi e^{-\left(  1-\delta\right)  \left\vert y\right\vert }\left\vert
\mathcal{N}_{1}\left(  W,\zeta\right)  \right\vert \leq Ce^{-\left(
1-\delta\right)  \left\vert y\right\vert }\left\vert W\right\vert
^{p-2}\left\vert \zeta\right\vert ^{2}\leq Ce^{-\left(  1-\delta\right)
\left\vert y\right\vert }Q^{p_{1}-2}\left(  y\right)  \left\vert
\zeta\right\vert ^{2}\leq Ce^{-\frac{p_{1}-1}{2}\left\vert y\right\vert
}\left\vert \zeta\right\vert ^{2}. \label{ap270}%
\end{equation}
(Here $\zeta$ is related to $\varepsilon_{1}$ by (\ref{epsilonbar})). By Lemma
\ref{Lemmaapp} we get
\[
\left\vert \frac{\dot{\lambda}}{\lambda}\Lambda W+i\left(  \dot{\gamma
}+\lambda^{-2}-\left\vert \beta\right\vert ^{2}-\dot{\beta}\cdot\chi\right)
W-i\lambda\left(  \dot{\beta}\cdot y\right)  W+\dot{W}\right\vert \leq
C\left(  \left\vert \beta\right\vert \Theta^{\nu}+\mathcal{M}\left(  t\right)
+\left\vert \mathbf{B}^{\left(  N\right)  }\right\vert +\left\vert
\mathbf{M}^{\left(  N\right)  }\right\vert \right)  e^{-\left(  1-\nu
-\delta_{1}\right)  \left\vert y\right\vert },
\]
with $0<\nu+\delta_{1}<1.$ Then, from (\ref{ap270}) and (\ref{boot}), for
$\nu+\delta_{1}<\frac{p_{1}-1}{2}$ we get
\begin{equation}
\left\vert \mathcal{G}_{11}^{\left(  1\right)  }\left(  t\right)  \right\vert
\leq C\left(  \left\vert \beta\right\vert \Theta^{\nu}+\mathcal{M}\left(
t\right)  +\left\vert \mathbf{B}^{\left(  N\right)  }\right\vert +\left\vert
\mathbf{M}^{\left(  N\right)  }\right\vert \right)  \left\Vert \varepsilon
_{1}\right\Vert _{H^{1}}^{2}\leq Ct^{-3-\nu}, \label{g111'}%
\end{equation}
for $t\in\lbrack T^{\ast},T_{n}],$ with $0<\nu<\frac{p_{1}-1}{2}.$ If
$\left\vert \varepsilon_{1}\right\vert \leq\frac{\left\vert \mathcal{W}%
_{1}\right\vert }{2}$ we expand
\begin{equation}
\left.  \mathcal{N}_{1}\left(  \mathcal{W}_{1},\varepsilon_{1}\right)
=\mathbf{N}\left(  \mathcal{W}_{1},\varepsilon_{1}\right)  +O\left(
\left\vert \mathcal{W}_{1}\right\vert ^{-1+\delta}\left\vert \varepsilon
_{1}\right\vert ^{p+1-\delta}\right)  ,\text{ }\delta>0,\right.
\label{expnon}%
\end{equation}
where%
\[
\mathbf{N}\left(  \mathcal{W}_{1},\varepsilon_{1}\right)  =\frac{p-1}%
{2}\left\vert \mathcal{W}_{1}\right\vert ^{p-3}\overline{\mathcal{W}_{1}%
}\varepsilon_{1}^{2}+\frac{p^{2}-1}{4}\left\vert \mathcal{W}_{1}\right\vert
^{p-3}\mathcal{W}_{1}\left\vert \varepsilon_{1}\right\vert ^{2}+\frac{\left(
p-1\right)  \left(  p-3\right)  }{2}\left\vert \mathcal{W}_{1}\right\vert
^{p-5}\mathcal{W}_{1}\operatorname{Re}\left(  \mathcal{W}_{1}\overline
{\varepsilon_{1}}\right)  ^{2}.
\]
If $\left\vert \varepsilon_{1}\right\vert \geq\frac{\left\vert \mathcal{W}%
_{1}\right\vert }{2}$ then $\left\vert \mathcal{N}_{1}\left(  \mathcal{W}%
_{1},\varepsilon_{1}\right)  \right\vert =O\left(  \left\vert \mathcal{W}%
_{1}\right\vert ^{-1+\delta}\left\vert \varepsilon_{1}\right\vert
^{p+1-\delta}\right)  $, $\delta>0.$ Hence
\begin{equation}
\mathcal{N}_{1}\left(  \mathcal{W}_{1},\varepsilon_{1}\right)  =\mathbf{N}%
\left(  \mathcal{W}_{1},\varepsilon_{1}\right)  +O\left(  \left\vert
\mathcal{W}_{1}\right\vert ^{-1+\delta}\left\vert \varepsilon_{1}\right\vert
^{p+1-\delta}\right)  ,\text{ }\delta>0. \label{ap309}%
\end{equation}
Using (\ref{ap309}) in (\ref{g11'}) we obtain%
\begin{equation}
\mathcal{G}_{11}\left(  t\right)  =\operatorname{Re}\int\mathbf{WN}\left(
\mathcal{W}_{1},\varepsilon_{1}\right)  +\mathcal{G}_{11}^{\left(  0\right)
}\left(  t\right)  +\mathcal{G}_{11}^{\left(  1\right)  }\left(  t\right)
+\mathcal{G}_{11}^{\left(  2\right)  }\left(  t\right)  , \label{g11'bis}%
\end{equation}
where%
\[
\mathcal{G}_{11}^{\left(  2\right)  }\left(  t\right)  =O\left(
\int\mathbf{W}\left(  \left\vert \mathcal{W}_{1}\right\vert ^{-1+\delta
}\left\vert \varepsilon_{1}\right\vert ^{p+1-\delta}\right)  dx\right)  .
\]
Using Lemma \ref{Lemmaapp} and (\ref{boot}) we control%
\[
\left\vert \mathbf{W}\right\vert \leq C\left(  \left\vert \dot{\chi
}\right\vert +\left\vert \beta\right\vert ^{2}\right)  e^{-\left(
1-\delta\right)  \left\vert y\right\vert }\leq Ct^{-1}e^{-\left(
1-\delta\right)  \left\vert y\right\vert },
\]
for $t\in\lbrack T^{\ast},T_{n}].$ Then, from (\ref{ap317}), via Sobolev
embedding theorem, taking $\delta$ sufficiently small we deduce%
\begin{equation}
\left\vert \mathcal{G}_{11}^{\left(  2\right)  }\left(  t\right)  \right\vert
\leq Ct^{-1}\left\Vert \varepsilon_{1}\right\Vert _{H^{1}}^{p+1-\delta}\leq
Ct^{-3-\frac{p-1}{2}}. \label{g112'}%
\end{equation}
Gathering (\ref{g110'}), (\ref{g111'}) and (\ref{g112'}), from (\ref{g11'bis})
we get%
\begin{equation}
\mathcal{G}_{11}\left(  t\right)  =\operatorname{Re}\int\mathbf{WN}\left(
\mathcal{W}_{1},\varepsilon_{1}\right)  +O\left(  t^{-3-\nu}\right)  ,\text{
}0<\nu<\frac{p_{1}-1}{2} \label{ap378}%
\end{equation}
for $t\in\lbrack T^{\ast},T_{n}].$

We turn now to $\mathcal{G}_{12}\left(  t\right)  .$ Using (\ref{eqeps}) we
obtain the equation for $\varepsilon_{1}$
\begin{equation}
i\partial_{t}\varepsilon_{1}=-\Delta\varepsilon_{1}+\left(  \lambda
^{-2}\left(  t\right)  +\left\vert \beta\left(  t\right)  \right\vert
^{2}\right)  \varepsilon_{1}-\mathcal{V}\varepsilon_{1}-\mathcal{N}_{0}\left(
\mathcal{W}_{1},\varepsilon_{1}\right)  -\left(  \lambda^{-\frac{2}{p-1}%
}\left(  t\right)  \left(  \mathcal{E}^{\left(  N\right)  }+R\left(  W\right)
\right)  \left(  t,\frac{x-\chi\left(  t\right)  }{\lambda\left(  t\right)
}\right)  e^{-i\left(  \gamma\left(  t\right)  +\gamma_{1}\left(  t\right)
\right)  }e^{i\beta\left(  t\right)  \cdot x}\right)  . \label{eqeps2}%
\end{equation}
Then%
\[
\mathcal{G}_{12}^{\left(  1\right)  }\left(  t\right)  =-\left(
\mathbf{R},i\left(  -\Delta\varepsilon_{1}+\left(  \lambda^{-2}\left(
t\right)  +\left\vert \beta\left(  t\right)  \right\vert ^{2}\right)
\varepsilon_{1}-\mathcal{V}\varepsilon_{1}-\mathcal{N}_{0}\left(
\mathcal{W}_{1},\varepsilon_{1}\right)  \right)  \right)  ,
\]
with
\[
\mathbf{R=}\lambda^{-\frac{2}{p-1}}\left(  t\right)  \left(  \mathcal{E}%
^{\left(  N\right)  }+R\left(  W\right)  \right)  \left(  t,\frac
{x-\chi\left(  t\right)  }{\lambda\left(  t\right)  }\right)  e^{-i\left(
\gamma\left(  t\right)  +\gamma_{1}\left(  t\right)  \right)  }e^{i\beta
\left(  t\right)  \cdot x}.
\]
Using (\ref{epsilonbar}) and (\ref{eps1}), as%
\begin{equation}
\left\vert R\left(  W\right)  \right\vert \leq Ce^{-\left(  1-\delta\right)
\left\vert y\right\vert }\mathcal{M}\left(  t\right)  , \label{ap360}%
\end{equation}
we get%
\[
\mathcal{G}_{12}\left(  t\right)  =\mathcal{G}_{12}^{\left(  1\right)
}\left(  t\right)  +\mathcal{G}_{12}^{\left(  2\right)  }\left(  t\right)
\]
where%
\[
\mathcal{G}_{12}^{\left(  1\right)  }\left(  t\right)  =-\lambda^{-2-\frac
{4}{p-1}}\left(  t\right)  \left(  R\left(  W\right)  ,i\left(  -\Delta
\zeta+\zeta-\lambda^{2}\left(  t\right)  V\left(  \left\vert \cdot+\frac{\chi
}{\lambda}\right\vert \right)  \zeta-\mathcal{N}_{0}\left(  W,\zeta\right)
\right)  \right)
\]
and%
\[
\mathcal{G}_{12}^{\left(  2\right)  }\left(  t\right)  =O\left(  \left(
\Theta^{3-\delta}\left(  \tfrac{\left\vert \chi\right\vert }{\lambda}\right)
+\left\vert \beta\right\vert \mathcal{M}\left(  t\right)  \right)  \left\Vert
\varepsilon_{1}\right\Vert _{H^{1}}\right)  .
\]
We decompose
\[
\mathcal{G}_{12}^{\left(  1\right)  }\left(  t\right)  =-\lambda^{-2-\frac
{4}{p-1}}\left(  t\right)  \left(  \left(  \mathcal{\tilde{L}}-\lambda
^{2}\left(  t\right)  V\left(  \left\vert \cdot+\frac{\chi}{\lambda
}\right\vert \right)  \right)  R\left(  W\right)  ,i\zeta\right)
+\lambda^{-2-\frac{4}{p-1}}\left(  t\right)  \left(  R\left(  W\right)
,i\mathcal{N}_{1}\left(  W,\zeta\right)  \right)
\]
where $\mathcal{\tilde{L}}$ is defined by (\ref{Ltilde}). Using (\ref{ap352})
and Sobolev embedding theorem, we estimate the second term in the right-hand
side of the last relation as $O\left(  \mathcal{M}\left(  t\right)  \left\Vert
\varepsilon_{1}\right\Vert _{H^{1}}^{p_{1}-\delta}\right)  .$ By the
orthogonal conditions (\ref{ap208}), using (\ref{lb1})-(\ref{lb4}) and
(\ref{ap274}) we control%
\[
\left\vert \mathcal{G}_{12}^{\left(  1\right)  }\left(  t\right)  \right\vert
\leq Cr^{\infty}\sqrt{U}\mathcal{M}\left(  t\right)  \left\Vert \varepsilon
_{1}\right\Vert _{H^{1}}.
\]
Hence, by (\ref{boot}), we get%
\begin{equation}
\left\vert \mathcal{G}_{12}\left(  t\right)  \right\vert \leq C\left(  \left(
\ln t\right)  t^{-1}\mathcal{M}\left(  t\right)  +t^{-3+\delta}\right)
\left\Vert \varepsilon_{1}\right\Vert _{H^{1}}. \label{ap379}%
\end{equation}

Next, we consider $\mathcal{G}_{13}\left(  t\right)  .$ The first term in the
right-hand side of $\mathcal{G}_{13}\left(  t\right)  $ is estimated by
Cauchy-Schwartz as
\begin{equation}
\left\vert \dot{\beta}\cdot\operatorname{Im}\int\psi_{\chi}\mathcal{\nabla
}\varepsilon_{1}\overline{\varepsilon}_{1}\right\vert \leq C\left\vert
\dot{\beta}\right\vert \left\Vert \varepsilon_{1}\right\Vert _{H^{1}}^{2}.
\label{ap377}%
\end{equation}
Since
\begin{equation}
\left\vert \dot{\psi}_{\chi}\left(  x\right)  \right\vert \leq\left(
1+\left\vert \tfrac{x-\chi\left(  t\right)  }{\left\vert \chi\left(  t\right)
\right\vert }\right\vert \right)  \tfrac{\left\vert \dot{\chi}\left(
t\right)  \right\vert }{\left\vert \chi\left(  t\right)  \right\vert
}\left\vert \psi^{\prime}\left(  \tfrac{x-\chi\left(  t\right)  }{\left\vert
\chi\left(  t\right)  \right\vert }\right)  \right\vert \label{ap267}%
\end{equation}
and
\begin{equation}
\left\vert \mathcal{\nabla}\psi_{\chi}\right\vert \leq\tfrac{1}{\left\vert
\chi\left(  t\right)  \right\vert }\left\vert \psi^{\prime}\left(
\tfrac{x-\chi\left(  t\right)  }{\left\vert \chi\left(  t\right)  \right\vert
}\right)  \right\vert \label{ap266}%
\end{equation}
we estimate%
\begin{equation}
\left\vert \beta\cdot\operatorname{Im}\int\dot{\psi}_{\chi}\mathcal{\nabla
}\varepsilon\overline{\varepsilon}\right\vert \leq C\left\vert \beta
\right\vert \tfrac{\left\vert \dot{\chi}\left(  t\right)  \right\vert
}{\left\vert \chi\left(  t\right)  \right\vert }\left\Vert \varepsilon
_{1}\right\Vert _{H^{1}}^{2} \label{ap374}%
\end{equation}
and%
\begin{equation}
\left\vert \beta\cdot\operatorname{Im}\int\mathcal{\nabla}\psi_{\chi}%
\dot{\varepsilon}_{1}\overline{\varepsilon}_{1}\right\vert \leq C\frac
{\left\vert \beta\right\vert }{\left\vert \chi\left(  t\right)  \right\vert
}\left\Vert \varepsilon_{1}\right\Vert _{H^{1}}^{2}. \label{ap375}%
\end{equation}
Using (\ref{eqeps2}), (\ref{ap360}), (\ref{eps1}) and integrating by parts we
have%
\begin{equation}
\operatorname{Im}\int\psi_{\chi}\mathcal{\nabla}\overline{\varepsilon_{1}}%
\dot{\varepsilon}_{1}=I_{1}+I_{2}+I_{3}+O\left(  \left(  \Theta^{3-\delta
}\left(  \tfrac{\left\vert \chi\right\vert }{\lambda}\right)  +\mathcal{M}%
\left(  t\right)  \right)  \left\Vert \varepsilon_{1}\right\Vert _{H^{1}%
}\right)  . \label{ap371}%
\end{equation}
where%
\[
I_{1}=\operatorname{Re}\int\mathcal{\nabla}\psi_{\chi}\overline{\varepsilon
_{1}}\left(  -\Delta\varepsilon_{1}+\left(  \lambda^{-2}\left(  t\right)
+\left\vert \beta\left(  t\right)  \right\vert ^{2}\right)  \varepsilon
_{1}-\mathcal{V}\varepsilon_{1}\right)  ,
\]%
\[
I_{2}=-\operatorname{Re}\int\psi_{\chi}\mathcal{\nabla V}\left\vert
\varepsilon_{1}\right\vert ^{2}%
\]
and%
\[
I_{3}=\operatorname{Re}\int\psi_{\chi}\mathcal{\nabla}\overline{\varepsilon
_{1}}\mathcal{N}_{0}\left(  \mathcal{W}_{1},\varepsilon_{1}\right)  .
\]
By (\ref{ap266}) we have%
\begin{equation}
\left\vert I_{1}\right\vert \leq\frac{C}{\left\vert \chi\left(  t\right)
\right\vert }\left\Vert \varepsilon_{1}\right\Vert _{H^{1}}^{2}. \label{I1'}%
\end{equation}
By (\ref{boot}), there is $T_{0}>0$ such that $\left\vert \lambda\left(
t\right)  \right\vert \leq2\lambda^{\infty},$ $t\in\lbrack T^{\ast},T_{n}],$
$T^{\ast}\geq T_{0}.$ Then $\left\Vert \psi\left(  \frac{8\left(  \left(
\cdot\right)  -\chi\left(  t\right)  \right)  }{\lambda\left(  t\right)  \ln
t}\right)  \mathcal{\nabla V}\left(  \cdot\right)  \right\Vert _{L^{\infty}%
}\leq C\left\vert \mathcal{V}^{\prime}\left(  \frac{3}{4}\ln t\right)
\right\vert \leq Ct^{-\frac{3}{2}K\left(  V\right)  }$ ($K\left(  V\right)  $
given by (\ref{K(V)})). Hence, by (\ref{boot}) we get%
\begin{equation}
\left\vert I_{2}\right\vert \leq Ct^{-\frac{3}{2}K\left(  V\right)
}\left\Vert \varepsilon_{1}\right\Vert _{H^{1}}^{2}\leq Ct^{-2-\frac{3}%
{2}K\left(  V\right)  }, \label{I2'}%
\end{equation}
$t\in\lbrack T^{\ast},T_{n}].$ By (\ref{n1}) we write%
\[
\mathcal{N}_{0}\left(  \mathcal{W}_{1},\varepsilon_{1}\right)  =\mathcal{N}%
_{L}\left(  \mathcal{W}_{1},\varepsilon_{1}\right)  +\mathcal{N}_{1}\left(
\mathcal{W}_{1},\varepsilon_{1}\right)  ,
\]
where
\[
\mathcal{N}_{L}\left(  \mathcal{W}_{1},\varepsilon_{1}\right)  =\tfrac{p+1}%
{2}\left\vert \mathcal{W}_{1}\right\vert ^{p-1}\varepsilon_{1}+\tfrac{p-1}%
{2}\left\vert \mathcal{W}_{1}\right\vert ^{p-3}\mathcal{W}_{1}^{2}%
\overline{\varepsilon_{1}}.
\]
We decompose%
\begin{equation}
I_{3}=I_{31}+I_{32} \label{I3'}%
\end{equation}
with
\[
I_{31}=\operatorname{Re}\int\psi_{\chi}\mathcal{\nabla}\overline
{\varepsilon_{1}}\mathcal{N}_{L}\left(  \mathcal{W}_{1},\varepsilon
_{1}\right)
\]
and%
\[
I_{32}=\operatorname{Re}\int\psi_{\chi}\mathcal{\nabla}\overline
{\varepsilon_{1}}\mathcal{N}_{1}\left(  \mathcal{W}_{1},\varepsilon
_{1}\right)  .
\]
Integrating by parts we have%
\begin{equation}
I_{31}=-\operatorname{Re}\int\psi_{\chi}\left(  \nabla\mathcal{W}_{1}\right)
\mathbf{N}\left(  \mathcal{W}_{1},\varepsilon_{1}\right)  +I_{31}^{\left(
1\right)  }, \label{ap372}%
\end{equation}
with%
\[
I_{31}^{\left(  1\right)  }=-\operatorname{Re}\int\left(  \mathcal{\nabla}%
\psi_{\chi}\right)  \left(  \tfrac{p+1}{4}\left\vert \mathcal{W}%
_{1}\right\vert ^{p-1}\left\vert \varepsilon_{1}\right\vert ^{2}+\tfrac
{p-1}{4}\operatorname{Re}\left(  \left\vert \mathcal{W}_{1}\right\vert
^{p-3}\mathcal{W}_{1}^{2}\left(  \overline{\varepsilon_{1}}\right)
^{2}\right)  \right)  .
\]
By (\ref{ap352}) we estimate%
\begin{equation}
\left\vert I_{32}\right\vert \leq C\left\Vert \varepsilon_{1}\right\Vert
_{H^{1}}^{p_{1}+1-\delta}. \label{I32}%
\end{equation}
Using that (\ref{ap266}) we control%
\begin{equation}
\left\vert I_{31}^{\left(  1\right)  }\right\vert \leq C\tfrac{1}{\left\vert
\chi\left(  t\right)  \right\vert }\left\Vert \varepsilon_{1}\right\Vert
_{H^{1}}^{2}. \label{ap373}%
\end{equation}
By Lemma \ref{Lemmaapp} we have
\[
Q\left(  y\right)  -\max_{1\leq j\leq N}\left\vert \boldsymbol{T}^{\left(
j\right)  }\left(  y\right)  \right\vert \geq Q\left(  y\right)  \left(
1-t^{-1/2}\right)  ,\text{ }\left\vert y\right\vert \leq\tfrac{\ln t}{4}.
\]
Then, there is $T_{0}>0$ such that
\begin{equation}
\max_{1\leq j\leq N}\left\vert \boldsymbol{T}^{\left(  j\right)  }\left(
y\right)  \right\vert <\left(  2N\right)  ^{-1}Q\left(  y\right)  ,
\label{ap370}%
\end{equation}
for all $\left\vert y\right\vert \leq\tfrac{\ln t}{4}$ and $t\in\lbrack
T^{\ast},T_{n}],$ $T^{\ast}\geq T_{0}.$ In particular, $\psi_{\chi}\phi
=\psi_{\chi}$ and then%
\[
\int\psi_{\chi}\left(  \nabla\mathcal{W}_{1}\right)  \mathbf{N}\left(
\mathcal{W}_{1},\varepsilon_{1}\right)  =\int\phi\left(  \nabla\mathcal{W}%
_{1}\right)  \mathbf{N}\left(  \mathcal{W}_{1},\varepsilon_{1}\right)
+\int\phi\left(  \psi_{\chi}-1\right)  \left(  \nabla\mathcal{W}_{1}\right)
\mathbf{N}\left(  \mathcal{W}_{1},\varepsilon_{1}\right)  .
\]
From (\ref{ap370}) we get $\left\vert W\left(  y\right)  \right\vert
\geq2^{-1}Q\left(  y\right)  .$ Then, as%
\begin{equation}
\left\vert \mathbf{N}\left(  \mathcal{W}_{1},\varepsilon_{1}\right)
\right\vert \leq C\left\vert \mathcal{W}_{1}\right\vert ^{p-2}\left\vert
\varepsilon_{1}\right\vert ^{2} \label{ap382}%
\end{equation}
we obtain
\[
\left\vert \phi\left(  \psi_{\chi}-1\right)  \left(  \nabla\mathcal{W}%
_{1}\right)  \mathbf{N}\left(  \mathcal{W}_{1},\varepsilon_{1}\right)
\right\vert \leq Ct^{-\frac{p-1}{8}}\left\vert \varepsilon_{1}\right\vert
^{2}.
\]
Hence%
\[
\int\psi_{\chi}\left(  \nabla\mathcal{W}_{1}\right)  \mathbf{N}\left(
\mathcal{W}_{1},\varepsilon_{1}\right)  =\int\phi\left(  \nabla\mathcal{W}%
_{1}\right)  \mathbf{N}\left(  \mathcal{W}_{1},\varepsilon_{1}\right)
+O\left(  t^{-\frac{p-1}{8}}\left\Vert \varepsilon_{1}\right\Vert _{H^{1}}%
^{2}\right)  .
\]
Using the last relation in (\ref{ap372}), from (\ref{ap373}) we deduce
\[
I_{31}=-\operatorname{Re}\int\phi\left(  \nabla\mathcal{W}_{1}\right)
\mathbf{N}\left(  \mathcal{W}_{1},\varepsilon_{1}\right)  +O\left(  \left(
t^{-\frac{p-1}{8}}+\tfrac{1}{\left\vert \chi\left(  t\right)  \right\vert
}\right)  \left\Vert \varepsilon_{1}\right\Vert _{H^{1}}^{2}\right)  .
\]
Then, by (\ref{I3'}) and (\ref{I32}) we get%
\[
I_{3}=-\operatorname{Re}\int\phi\left(  \nabla\mathcal{W}_{1}\right)
\mathbf{N}\left(  \mathcal{W}_{1},\varepsilon_{1}\right)  +O\left(  \left(
t^{-\frac{p-1}{8}}+\tfrac{1}{\left\vert \chi\left(  t\right)  \right\vert
}\right)  \left\Vert \varepsilon_{1}\right\Vert _{H^{1}}^{2}+\left\Vert
\varepsilon_{1}\right\Vert _{H^{1}}^{p_{1}+1-\delta}\right)  ,
\]
with $0<\delta<1.$ Thus, from (\ref{ap371}), (\ref{I1'}) and (\ref{I2'}), via
(\ref{boot}) it follows that%
\begin{equation}
\operatorname{Im}\int\psi_{\chi}\mathcal{\nabla}\overline{\varepsilon_{1}}%
\dot{\varepsilon}_{1}=-\operatorname{Re}\int\phi\left(  \nabla\mathcal{W}%
_{1}\right)  \mathbf{N}\left(  \mathcal{W}_{1},\varepsilon_{1}\right)
+O\left(  \frac{1}{t^{2}\ln t}\right)  . \label{ap376}%
\end{equation}
From (\ref{w}) we have%
\begin{equation}
\nabla\mathcal{W}_{1}=\lambda^{-\frac{2}{p-1}}\left(  \lambda^{-1}\nabla
W+i\beta W\right)  \left(  \frac{x-\chi}{\lambda}\right)  e^{-i\left(
\gamma\left(  t\right)  +\gamma_{1}\left(  t\right)  \right)  }e^{i\beta
\left(  t\right)  \cdot x}. \label{ap381}%
\end{equation}
Then, using (\ref{ap370}) and (\ref{ap382}) we get%
\[
2\beta\cdot\int\phi\left(  \nabla\mathcal{W}_{1}\right)  \mathbf{N}\left(
\mathcal{W}_{1},\varepsilon_{1}\right)  =\int\mathbf{WN}\left(  \mathcal{W}%
_{1},\varepsilon_{1}\right)  +O\left(  \mathcal{M}\left(  t\right)  \left\Vert
\varepsilon_{1}\right\Vert _{H^{1}}^{2}\right)
\]
Therefore, by (\ref{ap377}), (\ref{ap374}), (\ref{ap375}), (\ref{ap376}),
(\ref{boot}) taking into account (\ref{wbold}) we derive%
\begin{equation}
\mathcal{G}_{13}\left(  t\right)  =-\int\mathbf{WN}\left(  \mathcal{W}%
_{1},\varepsilon_{1}\right)  +O\left(  \frac{1}{t^{3}\ln t}\right)  .
\label{ap380}%
\end{equation}
for $t\in\lbrack T^{\ast},T_{n}].$ Finally, gathering together (\ref{ap378}),
(\ref{ap379}), (\ref{ap380}), from (\ref{dtG}) we attain (\ref{estbelow}).
\end{proof}

\subsubsection*{Case II: Slow decaying potentials.}

Let now the potential $V=V^{\left(  2\right)  }\ $and suppose that
(\ref{bootbis}) is true for all $t\in\lbrack T^{\ast},T_{n}],$ with $T_{0}\leq
T^{\ast}<T_{n}.$Let $\varphi\in C^{\infty}\left(  \mathbb{R}^{d}\right)  $ be
such that $0\leq\varphi\leq1,$ $\varphi\left(  x\right)  =1$ for $\left\vert
x\right\vert \leq\frac{1}{4\lambda^{\infty}}$ and $\varphi\left(  x\right)
=0$ for $\left\vert x\right\vert \geq\frac{1}{2\lambda^{\infty}}.$ Set
\[
\varphi_{\chi}\left(  x\right)  =\varphi\left(  \frac{x-\chi\left(  t\right)
}{\lambda\left(  t\right)  \left\vert \chi\left(  t\right)  \right\vert
}\right)  .
\]
We consider
\begin{align*}
\mathcal{\tilde{G}}\left(  \varepsilon\left(  t\right)  \right)
=\mathcal{\tilde{G}}_{\mathcal{W}}\left(  \varepsilon\left(  t\right)
\right)  =  &  \tfrac{1}{2}%
{\displaystyle\int}
\left\vert \mathcal{\nabla}\varepsilon\right\vert ^{2}+\tfrac{1}{2}\left(
\lambda^{-2}\left(  t\right)  +\left\vert \beta\left(  t\right)  \right\vert
^{2}\right)  \int\left\vert \varepsilon\right\vert ^{2}-\tfrac{1}{2}%
{\displaystyle\int}
\mathcal{V}\left\vert \varepsilon\right\vert ^{2}\\
&  -\tfrac{1}{1+p}%
{\displaystyle\int}
\left(  \left\vert \mathcal{W+\varepsilon}\right\vert ^{p+1}-\left\vert
\mathcal{W}\right\vert ^{p+1}-\left(  1+p\right)  \left\vert \mathcal{W}%
\right\vert ^{p-1}\operatorname{Re}\left(  \mathcal{W}\overline{\varepsilon
}\right)  \right) \\
&  -\beta\left(  t\right)  \cdot\operatorname{Im}\int\varphi_{\chi
}\mathcal{\nabla}\varepsilon\overline{\varepsilon}.
\end{align*}
As in the case of Lemma \ref{Lemmaenergy} we show that for some $T_{0}>0$
there is a constant $c_{0}>0$ such that%
\begin{equation}
\mathcal{\tilde{G}}\left(  \varepsilon\left(  t\right)  \right)  \geq
c_{0}\left\Vert \varepsilon\right\Vert _{H^{1}}^{2} \label{ap389}%
\end{equation}
for any $t\geq T_{0}.$ Suppose that for some $N$ large enough%
\begin{equation}
\left\vert \frac{d}{dt}\mathcal{\tilde{G}}\left(  \varepsilon\left(  t\right)
\right)  \right\vert \leq C\left(  \mathcal{X}\left(  \left(  r^{\infty
}\right)  ^{-1}+\left\vert \mathcal{V}^{\prime}\left(  \tfrac{r^{\infty}}%
{4}\right)  \right\vert \right)  +\Psi\right)  \mathcal{X}^{2N} \label{dtGbis}%
\end{equation}
for $t\in\lbrack T^{\ast},T_{n}],$ $T^{\ast}\geq T_{0},$ with constant $C$
independent on $N.$ Integrating (\ref{dtGbis}) and using (\ref{ap389}) we get%
\[
\left\Vert \varepsilon\right\Vert _{H^{1}}^{2}\leq\frac{C}{c_{0}N}%
\mathcal{X}^{2N}.
\]
with some constant $C$ independent on $N.$ Then for some $N$ big enough we
strictly improve (\ref{bootbis}) for $\varepsilon.$ Hence, by continuity we
conclude that (\ref{bootbis})\ for $\varepsilon\left(  t\right)  $ is true on
$[T_{0},T_{n}].$

Therefore we need to prove (\ref{dtGbis}). Similarly to (\ref{dtG}) we
decompose%
\begin{equation}
\frac{d}{dt}\mathcal{\tilde{G}}_{\mathcal{W}_{1}}\left(  \varepsilon
_{1}\left(  t\right)  \right)  =\mathcal{\tilde{G}}_{11}\left(  t\right)
+\mathcal{\tilde{G}}_{12}\left(  t\right)  +\mathcal{\tilde{G}}_{13}\left(
t\right)  , \label{dtG1}%
\end{equation}
where%
\[
\mathcal{\tilde{G}}_{11}\left(  t\right)  =-\left(  \mathcal{\dot{W}}%
_{1},\mathcal{N}_{1}\left(  \mathcal{W}_{1},\varepsilon_{1}\right)  \right)
,
\]%
\[
\mathcal{\tilde{G}}_{12}\left(  t\right)  =\left(  \dot{\varepsilon}%
_{1},-\Delta\varepsilon_{1}+\left(  \lambda^{-2}\left(  t\right)  +\left\vert
\beta\left(  t\right)  \right\vert ^{2}\right)  \varepsilon_{1}-\mathcal{V}%
\varepsilon_{1}-\mathcal{N}_{0}\left(  \mathcal{W}_{1},\varepsilon_{1}\right)
\right)  .
\]
and%
\begin{equation}
\mathcal{\tilde{G}}_{13}\left(  t\right)  =-\dot{\beta}\cdot\operatorname{Im}%
\int\varphi_{\chi}\mathcal{\nabla}\varepsilon_{1}\overline{\varepsilon}%
_{1}-\beta\cdot\operatorname{Im}\int\dot{\varphi}_{\chi}\mathcal{\nabla
}\varepsilon_{1}\overline{\varepsilon}_{1}+\beta\cdot\operatorname{Im}%
\int\mathcal{\nabla}\varphi_{\chi}\dot{\varepsilon}_{1}\overline{\varepsilon
}_{1}+2\beta\cdot\operatorname{Im}\int\varphi_{\chi}\mathcal{\nabla}%
\overline{\varepsilon_{1}}\dot{\varepsilon}_{1}. \label{g13tilde}%
\end{equation}
Let us consider $\mathcal{G}_{11}\left(  t\right)  .$ We define%
\begin{equation}
\tilde{\varphi}_{\mathcal{X}}\left(  y\right)  =\varphi\left(  \frac{y}%
{4\ln\mathcal{X}^{-2N}}\right)  . \label{psi}%
\end{equation}
Using (\ref{ap357}) we split%
\begin{equation}
\mathcal{\tilde{G}}_{11}\left(  t\right)  =\operatorname{Re}\int
\mathbf{\tilde{W}}\mathcal{N}_{1}\left(  \mathcal{W}_{1},\varepsilon
_{1}\right)  dx+\mathcal{\tilde{G}}_{11}^{\left(  0\right)  }\left(  t\right)
+\mathcal{\tilde{G}}_{11}^{\left(  1\right)  }\left(  t\right)  , \label{g11}%
\end{equation}
with%
\[
\mathbf{\tilde{W}=}\lambda^{-\frac{2}{p-1}}\left(  \tilde{\psi}_{\mathcal{X}%
}\left(  \frac{\dot{\chi}}{\lambda}\cdot\nabla W+2i\left\vert \beta\right\vert
^{2}W\right)  \right)  \left(  \frac{x-\chi}{\lambda}\right)  e^{-i\left(
\gamma\left(  t\right)  +\gamma_{1}\left(  t\right)  \right)  }e^{i\beta\cdot
x},
\]%
\[
\mathcal{\tilde{G}}_{11}^{\left(  0\right)  }\left(  t\right)  =-\lambda
^{-\frac{2}{p-1}}\operatorname{Re}\int\left(  1-\tilde{\varphi}_{\mathcal{X}%
}\right)  \mathcal{\dot{W}}_{1}\mathcal{N}_{1}\left(  \mathcal{W}%
_{1},\varepsilon_{1}\right)  dx
\]
and%
\[
\mathcal{\tilde{G}}_{11}^{\left(  1\right)  }\left(  t\right)  =\lambda
^{-\frac{2}{p-1}}\operatorname{Re}\int\mathcal{N}_{1}\left(  \mathcal{W}%
_{1},\varepsilon_{1}\right)  \mathbf{\tilde{W}}_{1}dx,
\]
where%
\[
\mathbf{\tilde{W}}_{1}=\left(  \tilde{\varphi}_{\mathcal{X}}\left(  \frac
{\dot{\lambda}}{\lambda}\Lambda W+i\left(  \dot{\gamma}+\lambda^{-2}%
-\left\vert \beta\right\vert ^{2}-\dot{\beta}\cdot\chi\right)  W-i\lambda
\left(  \dot{\beta}\cdot y\right)  W+\dot{W}\right)  \right)  \left(
\tfrac{x-\chi}{\lambda}\right)  e^{-i\left(  \gamma\left(  t\right)
+\gamma_{1}\left(  t\right)  \right)  }e^{i\beta\cdot x}.
\]
Using (\ref{ap357}) and Lemma \ref{Lemmaapp1} we obtain (\ref{ap353}). Thus,
by (\ref{ap352}), (\ref{ap353}), Sobolev embedding theorem and (\ref{bootbis})
we estimate
\begin{equation}
\left\vert \mathcal{\tilde{G}}_{11}^{\left(  0\right)  }\left(  t\right)
\right\vert \leq C\mathcal{X}^{N}\left(  \left\Vert \varepsilon_{1}\right\Vert
_{H^{1}}^{p-\delta}+\left\Vert \varepsilon_{1}\right\Vert _{H^{1}}^{2}\right)
\leq C\mathcal{X}^{N}\left(  \mathcal{X}^{\left(  p-\delta\right)
N}+\mathcal{X}^{2N}\right)  . \label{g110}%
\end{equation}
Using (\ref{ap309}) in (\ref{g11}) we obtain%
\begin{equation}
\mathcal{\tilde{G}}_{11}\left(  t\right)  =\operatorname{Re}\int
\mathbf{\tilde{W}N}\left(  \mathcal{W}_{1},\varepsilon_{1}\right)
+\mathcal{\tilde{G}}_{11}^{\left(  0\right)  }\left(  t\right)
+\mathcal{\tilde{G}}_{11}^{\left(  1\right)  }\left(  t\right)
+\mathcal{\tilde{G}}_{11}^{\left(  2\right)  }\left(  t\right)  ,
\label{ap358}%
\end{equation}
where%
\[
\mathcal{\tilde{G}}_{11}^{\left(  2\right)  }\left(  t\right)  =O\left(
\int\mathbf{\tilde{W}}\left(  \left\vert \mathcal{W}_{1}\right\vert
^{-1+\delta}\left\vert \varepsilon_{1}\right\vert ^{p+1-\delta}\right)
dx\right)  .
\]
By (\ref{bootbis}), there is $T_{0}>0$ such that $\left\vert \lambda\left(
t\right)  \right\vert \leq\frac{3}{2}\lambda^{\infty},$ $t\in\lbrack T^{\ast
},T_{n}],$ $T^{\ast}\geq T_{0}.$ From Lemma \ref{Lemmaapp1} it follows
\[
\left\vert W\left(  y\right)  \right\vert \geq Q\left(  y\right)  -\left\vert
T\left(  y\right)  \right\vert \geq Q\left(  y\right)  \left(  1-\left\vert
V\left(  \left\vert \tfrac{\chi}{\lambda}\right\vert \right)  \right\vert
\mathcal{X}^{-8\delta N}\right)  ,\text{ }\left\vert y\right\vert \leq
4\ln\mathcal{X}^{-2N}.
\]
From (\ref{bootbis}) it follows that $\left\vert V\left(  \left\vert
\tfrac{\chi}{\lambda}\right\vert \right)  \right\vert \mathcal{X}^{-8\delta
N}\leq C\mathcal{X}$, for $\delta<\left(  8N\right)  ^{-1}.$ Then, there is
$C_{0}>0$ such that
\begin{equation}
\left\vert W\left(  y\right)  \right\vert \geq2^{-1}Q\left(  y\right)  ,
\label{ap354}%
\end{equation}
for all $\left\vert y\right\vert \leq4\ln\mathcal{X}^{-2N}$ and $\left\vert
\chi\right\vert \geq C_{0}.$ Using (\ref{ap301}) and (\ref{ap354}) we estimate%
\[
e^{-\left(  1-\delta\right)  \left\vert y\right\vert }\left\vert
\mathcal{N}_{1}\left(  \mathcal{W}_{1},\varepsilon_{1}\right)  \right\vert
\leq Ce^{-\left(  1-\delta\right)  \left\vert y\right\vert }\left\vert
\mathcal{W}_{1}\right\vert ^{p-2}\left\vert \varepsilon_{1}\right\vert
^{2}\leq Ce^{-\left(  1-\delta\right)  \left\vert y\right\vert }Q^{p_{1}%
-2}\left(  y\right)  \left\vert \varepsilon_{1}\right\vert ^{2}\leq
C\left\vert \varepsilon_{1}\right\vert ^{2},
\]
for all $\left\vert y\right\vert \leq4\ln\mathcal{X}^{-2N}.$ By (\ref{wdot}),
using Lemma \ref{Lemmaapp1} to control $\Lambda W,W,yW,\dot{W},$ we get
\begin{equation}
\left\vert \mathcal{\tilde{G}}_{11}^{\left(  1\right)  }\left(  t\right)
\right\vert \leq C\left(  \mathcal{M}\left(  t\right)  +\left\vert
\mathbf{B}^{\left(  N\right)  }\right\vert +\left\vert \mathbf{M}^{\left(
N\right)  }\right\vert +\left\vert \beta\right\vert \Psi\right)  \left\Vert
\varepsilon_{1}\right\Vert _{H^{1}}^{2} \label{g111}%
\end{equation}
Moreover, via Sobolev embedding theorem we deduce%
\begin{equation}
\left\vert \mathcal{\tilde{G}}_{11}^{\left(  2\right)  }\left(  t\right)
\right\vert \leq C\left\vert \dot{\chi}\right\vert \left(  \left\Vert
\varepsilon_{1}\right\Vert _{H^{1}}^{2+\delta}+\left\Vert \varepsilon
_{1}\right\Vert _{H^{1}}^{p+1-\delta}\right)  , \label{g112}%
\end{equation}
with $0<\delta<p-1$. Using (\ref{g110}), (\ref{g111}) and (\ref{g112}) in
(\ref{ap358}) we get%
\begin{equation}
\left.
\begin{array}
[c]{c}%
\mathcal{\tilde{G}}_{11}\left(  t\right)  =\operatorname{Re}\int
\mathbf{\tilde{W}N}\left(  \mathcal{W}_{1},\varepsilon_{1}\right)
+C\mathcal{X}^{N}\left(  \mathcal{X}^{\left(  p-\delta\right)  N}%
+\mathcal{X}^{2N}\right) \\
+O\left(  \mathcal{M}\left(  t\right)  +\left\vert \mathbf{B}^{\left(
N\right)  }\right\vert +\left\vert \mathbf{M}^{\left(  N\right)  }\right\vert
+\left\vert \beta\right\vert \Psi\right)  \left\Vert \varepsilon
_{1}\right\Vert _{H^{1}}^{2}+\left\vert \dot{\chi}\right\vert \left(
\left\Vert \varepsilon_{1}\right\Vert _{H^{1}}^{2+\delta}+\left\Vert
\varepsilon_{1}\right\Vert _{H^{1}}^{p+1-\delta}\right)  .
\end{array}
\right.  \label{g11bis}%
\end{equation}
Using (\ref{ap360}), $\left\vert \mathcal{N}_{0}\left(  \mathcal{W}%
_{1},\varepsilon_{1}\right)  \right\vert \leq C\left\vert \varepsilon
_{1}\right\vert \ $and (\ref{eps1bis}), via Sobolev theorem we control%
\begin{equation}
\left\vert \mathcal{\tilde{G}}_{12}\left(  t\right)  \right\vert \leq C\left(
\Psi\mathcal{X}^{N}+\mathcal{M}\left(  t\right)  \right)  \left\Vert
\varepsilon_{1}\right\Vert _{H^{1}}. \label{g12}%
\end{equation}
Next, we consider $\mathcal{\tilde{G}}_{13}\left(  t\right)  .$ Let us
estimate the last term in the right-hand side of (\ref{g13tilde}). Using
(\ref{eqeps2}), (\ref{ap360}), (\ref{eps1bis}) and integrating by parts we
have%
\begin{equation}
\operatorname{Im}\int\varphi_{\chi}\mathcal{\nabla}\overline{\varepsilon_{1}%
}\dot{\varepsilon}_{1}=\tilde{I}_{1}+\tilde{I}_{2}+\tilde{I}_{3}+O\left(
\left(  \Psi\mathcal{X}^{N}+\mathcal{M}\left(  t\right)  \right)  \left\Vert
\varepsilon_{1}\right\Vert _{H^{1}}\right)  . \label{mom}%
\end{equation}
where%
\[
\tilde{I}_{1}=\operatorname{Re}\int\mathcal{\nabla}\varphi_{\chi}%
\overline{\varepsilon_{1}}\left(  -\Delta\varepsilon_{1}+\left(  \lambda
^{-2}\left(  t\right)  +\left\vert \beta\left(  t\right)  \right\vert
^{2}\right)  \varepsilon_{1}-\mathcal{V}\varepsilon_{1}\right)  ,
\]%
\[
\tilde{I}_{2}=-\operatorname{Re}\int\varphi_{\chi}\mathcal{\nabla V}\left\vert
\varepsilon_{1}\right\vert ^{2}%
\]
and%
\[
\tilde{I}_{3}=\operatorname{Re}\int\varphi_{\chi}\mathcal{\nabla}%
\overline{\varepsilon_{1}}\mathcal{N}_{0}\left(  \mathcal{W}_{1}%
,\varepsilon_{1}\right)  .
\]
By (\ref{ap266}) we have%
\begin{equation}
\left\vert \tilde{I}_{1}\right\vert \leq\frac{C}{\left\vert \chi\left(
t\right)  \right\vert }\left\Vert \varepsilon_{1}\right\Vert _{H^{1}}^{2}.
\label{I1}%
\end{equation}
Noting that $\left\Vert \varphi\left(  \frac{\left(  \cdot\right)
-\chi\left(  t\right)  }{\lambda\left(  t\right)  \left\vert \chi\left(
t\right)  \right\vert }\right)  \mathcal{\nabla V}\left(  \cdot\right)
\right\Vert _{L^{\infty}}\leq C\left\vert \mathcal{V}^{\prime}\left(
\frac{\left\vert \chi\left(  t\right)  \right\vert }{4}\right)  \right\vert ,$
we get%
\begin{equation}
\left\vert \tilde{I}_{2}\right\vert \leq C\left\vert \mathcal{V}^{\prime
}\left(  \frac{\left\vert \chi\left(  t\right)  \right\vert }{4}\right)
\right\vert \left\Vert \varepsilon_{1}\right\Vert _{H^{1}}^{2}. \label{I2}%
\end{equation}
Observe that there is $C_{0}>0$ such that $\varphi_{\chi}\tilde{\varphi
}_{\mathcal{X}}=\tilde{\varphi}_{\mathcal{X}},$ with $\tilde{\varphi
}_{\mathcal{X}}$ given by (\ref{psi}). By (\ref{n1}) we write%
\[
\mathcal{N}_{0}\left(  \mathcal{W}_{1},\varepsilon_{1}\right)  =\mathcal{N}%
_{L}\left(  \mathcal{W}_{1},\varepsilon_{1}\right)  +\mathcal{N}_{1}\left(
\mathcal{W}_{1},\varepsilon_{1}\right)  ,
\]
where
\[
\mathcal{N}_{L}\left(  \mathcal{W}_{1},\varepsilon_{1}\right)  =\tfrac{p+1}%
{2}\left\vert \mathcal{W}_{1}\right\vert ^{p-1}\varepsilon_{1}+\tfrac{p-1}%
{2}\left\vert \mathcal{W}_{1}\right\vert ^{p-3}\mathcal{W}_{1}^{2}%
\overline{\varepsilon_{1}}.
\]
We decompose%
\[
\tilde{I}_{3}=\tilde{I}_{31}+\tilde{I}_{32}%
\]
with%
\[
\tilde{I}_{31}=\operatorname{Re}\int\tilde{\varphi}_{\mathcal{X}%
}\mathcal{\nabla}\overline{\varepsilon_{1}}\mathcal{N}_{L}\left(
\mathcal{W}_{1},\varepsilon_{1}\right)
\]
and%
\[
\tilde{I}_{32}=\operatorname{Re}\int\varphi_{\chi}\left(  1-\tilde{\varphi
}_{\mathcal{X}}\right)  \mathcal{\nabla}\overline{\varepsilon_{1}}%
\mathcal{N}_{0}\left(  \mathcal{W}_{1},\varepsilon_{1}\right)
+\operatorname{Re}\int\tilde{\varphi}_{\mathcal{X}}\mathcal{\nabla}%
\overline{\varepsilon_{1}}\mathcal{N}_{1}\left(  \mathcal{W}_{1}%
,\varepsilon_{1}\right)  .
\]
Integrating by parts we have%
\[
\tilde{I}_{31}=-\operatorname{Re}\int\tilde{\varphi}_{\mathcal{X}}\left(
\nabla\mathcal{W}_{1}\right)  \mathbf{N}\left(  \mathcal{W}_{1},\varepsilon
_{1}\right)  +\tilde{I}_{31}^{\left(  1\right)  },
\]
with%
\[
\tilde{I}_{31}^{\left(  1\right)  }=-\operatorname{Re}\int\left(
\mathcal{\nabla}\tilde{\varphi}_{\mathcal{X}}\right)  \left(  \tfrac{p+1}%
{4}\left\vert \mathcal{W}_{1}\right\vert ^{p-1}\left\vert \varepsilon
_{1}\right\vert ^{2}+\tfrac{p-1}{4}\operatorname{Re}\left(  \left\vert
\mathcal{W}_{1}\right\vert ^{p-3}\mathcal{W}_{1}^{2}\left(  \overline
{\varepsilon_{1}}\right)  ^{2}\right)  \right)  .
\]
Using that $\left\vert \mathcal{N}_{0}\left(  \mathcal{W}_{1},\varepsilon
_{1}\right)  \right\vert \leq C\left(  \left\vert \mathcal{W}_{1}\right\vert
^{p}+\left\vert \varepsilon_{1}\right\vert ^{p}\right)  ,$ as $\left\vert
\left(  \mathcal{\nabla}\tilde{\psi}_{\mathcal{X}}\right)  \mathcal{W}%
_{1}\left(  y\right)  \right\vert +\left\vert \left(  1-\tilde{\psi
}_{\mathcal{X}}\right)  \mathcal{W}_{1}\left(  y\right)  \right\vert \leq
C\mathcal{X}^{2\left(  1-\delta\right)  N}$ and $\left\vert \mathcal{N}%
_{1}\left(  \mathcal{W}_{1},\varepsilon_{1}\right)  \right\vert \leq C\left(
\left\vert \varepsilon_{1}\right\vert ^{p_{1}}+\left\vert \varepsilon
_{1}\right\vert ^{p_{1}-\delta}\right)  ,$ $\delta>0,$ ($p_{1}=\min\{p,2\}$)
we get
\begin{align*}
&  \left\vert \int\left(  \mathcal{\nabla}\tilde{\varphi}_{\mathcal{X}%
}\right)  \left(  \tfrac{p+1}{4}\left\vert \mathcal{W}_{1}\right\vert
^{p-1}\left\vert \varepsilon_{1}\right\vert ^{2}+\tfrac{p-1}{4}%
\operatorname{Re}\left(  \left\vert \mathcal{W}_{1}\right\vert ^{p-3}%
\mathcal{W}_{1}^{2}\left(  \overline{\varepsilon_{1}}\right)  ^{2}\right)
\right)  \right\vert +\left\vert I_{32}\right\vert \\
&  \leq C\left(  \mathcal{X}^{2\left(  1-\delta\right)  Np}+\mathcal{X}%
^{2\left(  1-\delta\right)  N\left(  p-1\right)  }\left\Vert \varepsilon
_{1}\right\Vert _{H^{1}}^{2}+\left\Vert \varepsilon_{1}\right\Vert _{H^{1}%
}^{p_{1}-\delta}\right)  \left\Vert \varepsilon_{1}\right\Vert _{H^{1}}.
\end{align*}
Hence%
\begin{equation}
\tilde{I}_{3}=-\operatorname{Re}\int\tilde{\varphi}_{\mathcal{X}}\left(
\nabla\mathcal{W}_{1}\right)  \mathbf{N}\left(  \mathcal{W}_{1},\varepsilon
_{1}\right)  +O\left(  \left(  \mathcal{X}^{2\left(  1-\delta\right)
Np}+\mathcal{X}^{2\left(  1-\delta\right)  N\left(  p-1\right)  }\left\Vert
\varepsilon_{1}\right\Vert _{H^{1}}^{2}+\left\Vert \varepsilon_{1}\right\Vert
_{H^{1}}^{p_{1}-\delta}\right)  \left\Vert \varepsilon_{1}\right\Vert _{H^{1}%
}\right)  . \label{I3}%
\end{equation}
Using (\ref{I1}), (\ref{I2}) and (\ref{I3}) in (\ref{mom}) we get
\begin{equation}
2\beta\cdot\operatorname{Im}\int\varphi_{\chi}\mathcal{\nabla}\overline
{\varepsilon_{1}}\dot{\varepsilon}_{1}=-2\beta\cdot\operatorname{Re}\int
\tilde{\varphi}_{\mathcal{X}}\left(  \nabla\mathcal{W}_{1}\right)
\mathbf{N}\left(  \mathcal{W}_{1},\varepsilon_{1}\right)  +\operatorname*{Er}%
\nolimits_{1} \label{mom1}%
\end{equation}
with%
\begin{align*}
\operatorname*{Er}\nolimits_{1}  &  =O\left(  \left\vert \beta\right\vert
\left(  \left\vert \chi\left(  t\right)  \right\vert ^{-1}+\left\vert
\mathcal{V}^{\prime}\left(  \tfrac{\left\vert \chi\left(  t\right)
\right\vert }{4}\right)  \right\vert \right)  \left\Vert \varepsilon
_{1}\right\Vert _{H^{1}}^{2}\right) \\
&  +O\left(  \left\vert \beta\right\vert \left(  \mathcal{X}^{2\left(
1-\delta\right)  Np}+\Psi\mathcal{X}^{N}+\mathcal{M}\left(  t\right)
+\mathcal{X}^{2\left(  1-\delta\right)  N\left(  p-1\right)  }\left\Vert
\varepsilon_{1}\right\Vert _{H^{1}}^{2}+\left\Vert \varepsilon_{1}\right\Vert
_{H^{1}}^{p_{1}-\delta}\right)  \left\Vert \varepsilon_{1}\right\Vert _{H^{1}%
}\right)  .
\end{align*}
Similarly to (\ref{ap377}), (\ref{ap374}) and (\ref{ap375}) we estimate the
first three terms in (\ref{g13tilde}) by $O\left(  \left(  \left\vert
\dot{\beta}\right\vert +\frac{\left\vert \beta\right\vert }{\left\vert
\chi\left(  t\right)  \right\vert }\right)  \left\Vert \varepsilon
_{1}\right\Vert _{H^{1}}^{2}\right)  .$\ Therefore, from (\ref{mom1})\ we
obtain%
\[
\mathcal{\tilde{G}}_{13}\left(  t\right)  =-2\beta\cdot\operatorname{Re}%
\int\tilde{\varphi}_{\mathcal{X}}\left(  \nabla\mathcal{W}_{1}\right)
\mathbf{N}\left(  \mathcal{W}_{1},\varepsilon_{1}\right)  +\operatorname*{Er}%
\nolimits_{2}%
\]
with%
\[
\operatorname*{Er}\nolimits_{2}=\operatorname*{Er}\nolimits_{1}+O\left(
\left(  \left\vert \dot{\beta}\right\vert +\frac{\left\vert \beta\right\vert
}{\left\vert \chi\left(  t\right)  \right\vert }\right)  \left\Vert
\varepsilon_{1}\right\Vert _{H^{1}}^{2}\right)  .
\]
By (\ref{ap354}) $\left\vert \tilde{\varphi}_{\mathcal{X}}\mathbf{N}\left(
\mathcal{W}_{1},\varepsilon_{1}\right)  \right\vert \leq C\left\vert
\varepsilon_{1}\right\vert ^{2}.$ Then, using (\ref{ap381}) we get
\begin{equation}
\mathcal{\tilde{G}}_{13}\left(  t\right)  =-\operatorname{Re}\int
\mathbf{\tilde{W}N}\left(  \mathcal{W}_{1},\varepsilon_{1}\right)
+\operatorname*{Er}\nolimits_{2}+O\left(  \mathcal{M}\left(  t\right)
\left\Vert \varepsilon_{1}\right\Vert _{H^{1}}^{2}\right)  . \label{g13}%
\end{equation}
Using (\ref{g11bis}), (\ref{g12}), (\ref{g13}) in (\ref{dtG1}) we arrive to%
\begin{align*}
\frac{d}{dt}\mathcal{\tilde{G}}_{\mathcal{W}}\left(  \varepsilon\left(
t\right)  \right)   &  =\frac{d}{dt}\mathcal{\tilde{G}}_{\mathcal{W}_{1}%
}\left(  \varepsilon_{1}\left(  t\right)  \right)  =\operatorname*{Er}%
\nolimits_{2}+C\mathcal{X}^{N}\left(  \mathcal{X}^{\left(  p-\delta\right)
N}+\mathcal{X}^{2N}\right) \\
&  +\left(  \mathcal{M}\left(  t\right)  +\left\vert \mathbf{B}^{\left(
N\right)  }\right\vert +\left\vert \mathbf{M}^{\left(  N\right)  }\right\vert
+\left\vert \beta\right\vert \Psi\right)  \left\Vert \varepsilon
_{1}\right\Vert _{H^{1}}^{2}+\left\vert \dot{\chi}\right\vert \left(
\left\Vert \varepsilon_{1}\right\Vert _{H^{1}}^{2+\delta}+\left\Vert
\varepsilon_{1}\right\Vert _{H^{1}}^{p+1-\delta}\right)  .
\end{align*}
Finally, using (\ref{Mtbis}) and (\ref{bootbis}), for $N$ sufficiently big
such that $\mathcal{X}^{\frac{N\left(  p-1\right)  }{2}}\leq C\left(  \left(
r^{\infty}\right)  ^{-1}+\left\vert \mathcal{V}^{\prime}\left(  \tfrac
{r^{\infty}}{4}\right)  \right\vert +\Psi\right)  $ we prove (\ref{dtGbis}).

\section{Asymptotics of $\mathcal{J}\left(  \chi\right)  $.\label{AppendixA}}

Let us study the asymptotics as $\left\vert \chi\right\vert \rightarrow\infty$
of the integral
\begin{equation}
\mathcal{J}\left(  \chi\right)  =\int G\left(  \left\vert y+\chi\right\vert
\right)  \nabla Q^{2}\left(  y\right)  dy=%
{\displaystyle\int}
G\left(  \left\vert z\right\vert \right)  \nabla Q^{2}\left(  \chi-z\right)
dz \label{ap28}%
\end{equation}
where $G\in C^{\infty}$.\ We consider $G\in C^{\infty}$ of the form
(\ref{potential}). That is
\begin{equation}
G\left(  r\right)  =V_{+}\left(  r\right)  \text{ or }G\left(  r\right)
=V_{-}\left(  r\right)  . \label{ap171}%
\end{equation}
Recall that $\upsilon\left(  d\right)  $ and $C_{\pm}\left(  \lambda\right)  $
are defined by (\ref{ap142}) and (\ref{ap167}), respectively. First we study
the case when $\nabla Q^{2}$ determines the behavior of (\ref{ap28}), as
$\left\vert \chi\right\vert \rightarrow\infty.$ That is
\begin{equation}
G\left(  r\right)  =V_{-}\left(  r\right)  ,\text{ }H\geq0. \label{ap154}%
\end{equation}
We denote by $K$ the limit $e^{-H}\rightarrow K.$ We prove the following.

\begin{lemma}
\label{L3}Let $G\in C^{\infty}\ $be as in (\ref{ap154})$,$ where $H,H^{\prime
}$ are monotone. Then, the following is true. If $d\geq4,$ the asymptotics
\begin{equation}
\mathcal{J}\left(  \chi\right)  =\frac{\chi}{\left\vert \chi\right\vert
}e^{-2\left\vert \chi\right\vert }\left\vert \chi\right\vert ^{-\left(
d-1\right)  }\left(  \int\left(  \mathcal{A}^{2}G\left(  \left\vert
z\right\vert \right)  +K\kappa Q^{2}\left(  z\right)  \right)  e^{2\frac
{\chi\cdot z}{\left\vert \chi\right\vert }}dz+o\left(  1\right)  \right)
\label{ap150}%
\end{equation}
as $\left\vert \chi\right\vert \rightarrow\infty$ holds. Suppose that $d=2$ or
$d=3.$ If $r^{-\frac{d-1}{2}}e^{-H\left(  r\right)  }\in L^{1}\left(
[1,\infty)\right)  ,$ then%
\begin{equation}
\mathcal{J}\left(  \chi\right)  =\frac{\chi}{\left\vert \chi\right\vert
}e^{-2\left\vert \chi\right\vert }\left\vert \chi\right\vert ^{-\left(
d-1\right)  }\left(  \mathcal{A}^{2}\int G\left(  \left\vert z\right\vert
\right)  e^{2\frac{\chi\cdot z}{\left\vert \chi\right\vert }}dz+o\left(
1\right)  \right)  \label{ap151}%
\end{equation}
as $\left\vert \chi\right\vert \rightarrow\infty$. In the case when
$r^{-\frac{d-1}{2}}e^{-H\left(  r\right)  }\notin L^{1}\left(  [1,\infty
)\right)  ,$ the expansion
\begin{equation}
\left.  \mathcal{J}\left(  \chi\right)  =\frac{\chi}{\left\vert \chi
\right\vert }\kappa\mathcal{A}^{2}\upsilon\left(  d\right)  \left(  1+o\left(
1\right)  \right)  C_{-}\left(  \left\vert \chi\right\vert \right)
e^{-2\left\vert \chi\right\vert }\left\vert \chi\right\vert ^{-\left(
d-1\right)  },\right.  \label{ap152}%
\end{equation}
as $\left\vert \chi\right\vert \rightarrow\infty$ takes place.
\end{lemma}

\begin{proof}
First, we note that $\mathcal{J}\left(  \chi\right)  $ is directed along the
vector $\chi.$ Indeed, we introduce the polar coordinate system, where the
$x_{1}-$axis is directed along the vector $\chi.$ Then, $\chi=\left\vert
\chi\right\vert \left(  1,0,...,0\right)  $ and $y=\left\vert y\right\vert
\left(  \cos\theta,\sin\theta\cos\theta_{1},...,\sin\theta\sin\theta
_{1}...\cos\theta_{d-2}\right)  ,$ where $\theta$ is the angle between $\chi$
and $z.$ Thus
\begin{equation}
\left.
\begin{array}
[c]{c}%
\mathcal{J}\left(  \chi\right)  =\int_{0}^{\infty}\left(  \int_{0}^{\pi
}...\int_{0}^{\pi}\int_{0}^{2\pi}G\left(  \left\vert \sqrt{\left\vert
\chi\right\vert ^{2}+r^{2}+2\left\vert \chi\right\vert r\cos\theta}\right\vert
\right)  \left(  q^{2}\right)  ^{\prime}\left(  r\right)  r^{d-1}y_{\Theta
}d\Omega\right)  dr\\
\\
=C\frac{\chi}{\left\vert \chi\right\vert }\int_{0}^{\infty}\int_{0}^{\pi
}G\left(  \left\vert \sqrt{\left\vert \chi\right\vert ^{2}+r^{2}+2\left\vert
\chi\right\vert r\cos\theta}\right\vert \right)  \cos\theta\left(
q^{2}\right)  ^{\prime}\left(  r\right)  r^{d-1}d\theta dr,
\end{array}
\right.  \label{coord}%
\end{equation}
with $y_{\Theta}:=\left(  \cos\theta,\sin\theta\cos\theta_{1},...,\sin
\theta\sin\theta_{1}...\cos\theta_{d-2}\right)  $ and $d\Omega=\sin
^{d-2}\theta\sin^{d-3}\theta_{1}...\sin\theta_{d-3}d\theta d\theta
_{1}...d\theta_{d-2}.$

Using (\ref{ap29}) we estimate
\[
\left.
\begin{array}
[c]{c}%
\left\vert
{\displaystyle\int_{\left\vert z\right\vert \geq\frac{3\left\vert
\chi\right\vert }{2}}}
G\left(  \left\vert z\right\vert \right)  \nabla Q^{2}\left(  \chi-z\right)
dz\right\vert \leq C%
{\displaystyle\int_{\left\vert z\right\vert \geq\frac{3\left\vert
\chi\right\vert }{2}}}
e^{-2\left\vert z\right\vert -\left(  d-1\right)  \ln\left\vert z\right\vert
-H\left(  \left\vert z\right\vert \right)  }\left\vert \nabla Q^{2}\left(
\chi-z\right)  \right\vert dz\\
\\
\leq Ce^{-3\left\vert \chi\right\vert }\left(
{\displaystyle\int}
\left\vert \nabla Q^{2}\left(  \chi-z\right)  \right\vert dz\right)  .
\end{array}
\right.
\]
Therefore%
\begin{equation}
\left\vert
{\displaystyle\int_{\left\vert z\right\vert \geq\frac{3\left\vert
\chi\right\vert }{2}}}
G\left(  \left\vert z\right\vert \right)  \nabla Q^{2}\left(  \chi-z\right)
dz\right\vert \leq Ce^{-3\left\vert \chi\right\vert }. \label{ap36}%
\end{equation}
Next, we consider the region $\frac{\left\vert \chi\right\vert }{2}%
\leq\left\vert z\right\vert \leq\frac{3\left\vert \chi\right\vert }{2}.$ Since
for $\left\vert z\right\vert \leq\frac{\left\vert \chi\right\vert }{2},$ the
inequality $\frac{\left\vert \chi\right\vert }{2}\leq\left\vert z\right\vert
\leq\frac{3\left\vert \chi\right\vert }{2}$ holds, we write%
\begin{equation}
\left.
\begin{array}
[c]{c}%
{\displaystyle\int_{\frac{\left\vert \chi\right\vert }{2}\leq\left\vert
z\right\vert \leq\frac{3\left\vert \chi\right\vert }{2}}}
G\left(  \left\vert z\right\vert \right)  \nabla Q^{2}\left(  \chi-z\right)
dz=-%
{\displaystyle\int_{\frac{\left\vert \chi\right\vert }{2}\leq\left\vert
\chi-z\right\vert \leq\frac{3\left\vert \chi\right\vert }{2}}}
G\left(  \left\vert \chi-z\right\vert \right)  \nabla Q^{2}\left(  z\right)
dz\\
\\
=-\kappa e^{-2\left\vert \chi\right\vert }%
{\displaystyle\int_{\left\vert z\right\vert \leq\frac{\left\vert
\chi\right\vert }{2}}}
e^{-2\left(  \left\vert z-\chi\right\vert -\left\vert \chi\right\vert
+\left\vert z\right\vert \right)  }\left\vert \chi-z\right\vert ^{-\left(
d-1\right)  }e^{-H\left(  \left\vert \chi-z\right\vert \right)  }\left(
\frac{z}{\left\vert z\right\vert }e^{2\left\vert z\right\vert }\partial
_{\left\vert z\right\vert }q^{2}\left(  \left\vert z\right\vert \right)
\right)  dz\\
\\
-%
{\displaystyle\int_{\frac{\left\vert \chi\right\vert }{2}\leq\left\vert
\chi-z\right\vert \leq\frac{3\left\vert \chi\right\vert }{2};\frac{\left\vert
\chi\right\vert }{2}\leq\left\vert z\right\vert }}
G\left(  \left\vert \chi-z\right\vert \right)  \nabla Q^{2}\left(  z\right)
dz.
\end{array}
\right.  \label{ap122}%
\end{equation}
We now present the following estimates. \ Using the coordinate system in
(\ref{coord}) we have%
\begin{equation}
\left\vert \frac{z}{\left\vert z\right\vert }-\frac{\chi}{\left\vert
\chi\right\vert }\right\vert \leq C\left(  \left(  1-\cos\theta\right)
+\sin\theta\right)  . \label{ap123}%
\end{equation}
By
\begin{equation}
\left\vert y+\chi\right\vert -\left\vert \chi\right\vert +\left\vert
y\right\vert =\frac{2\left\vert \chi\right\vert \left\vert y\right\vert
\left(  1+\frac{\chi\cdot y}{\left\vert \chi\right\vert \left\vert
y\right\vert }\right)  }{\left\vert y+\chi\right\vert +\left\vert
\chi\right\vert -\left\vert y\right\vert } \label{ap118}%
\end{equation}
we get%
\begin{equation}
\left\vert e^{-2\left(  \left\vert z-\chi\right\vert -\left\vert
\chi\right\vert +\left\vert z\right\vert \right)  }-e^{-\frac{2\left\vert
\chi\right\vert \left\vert z\right\vert \left(  1-\cos\theta\right)
}{\left\vert \chi\right\vert -\left\vert z\right\vert }}\right\vert \leq
Ce^{-\left\vert z\right\vert \left(  1-\cos\theta\right)  }\frac{\left\vert
z\right\vert ^{2}\left(  1-\cos\theta\right)  ^{2}}{\left\vert \chi\right\vert
}. \label{ap115}%
\end{equation}
Also, note that%
\begin{equation}
\left\vert \left\vert \chi-z\right\vert ^{-1}-\left(  \left\vert
\chi\right\vert -\left\vert z\right\vert \right)  ^{-1}\right\vert \leq
C\frac{\left\vert z\right\vert \left(  1-\cos\theta\right)  }{\left\vert
\chi\right\vert ^{2}},\text{ for }\left\vert z\right\vert \leq\frac{\left\vert
\chi\right\vert }{2}. \label{ap116}%
\end{equation}
Suppose first that $d\geq4.$ Let $e^{-H}\rightarrow K\ $and $\psi\in
L^{\infty}\left(  \mathbb{R}\right)  $ such that $\psi\left(  r\right)  =1,$
for $0\leq r\leq1,$ and $\psi\left(  r\right)  =0,$ for $r>1.$ Using
(\ref{ap29}), (\ref{ap115}) and (\ref{ap116}) we have%
\begin{equation}
\left.  \left\vert
{\displaystyle\int_{\left\vert z\right\vert \leq\frac{\left\vert
\chi\right\vert }{2}}}
e^{-2\left(  \left\vert z-\chi\right\vert -\left\vert \chi\right\vert
+\left\vert z\right\vert \right)  }\left\vert \chi-z\right\vert ^{-\left(
d-1\right)  }e^{-H\left(  \left\vert \chi-z\right\vert \right)  }\left(
\frac{z}{\left\vert z\right\vert }e^{2\left\vert z\right\vert }\partial
_{\left\vert z\right\vert }q^{2}\left(  \left\vert z\right\vert \right)
\right)  dz-I_{1}\right\vert \leq r_{1}\right.  \label{ap117}%
\end{equation}
with%
\[
I_{1}=K%
{\displaystyle\int}
\psi\left(  \frac{2\left\vert z\right\vert }{\left\vert \chi\right\vert
}\right)  e^{-\frac{2\left\vert \chi\right\vert \left\vert z\right\vert
\left(  1-\cos\theta\right)  }{\left\vert \chi\right\vert -\left\vert
z\right\vert }}\left(  \left\vert \chi\right\vert -\left\vert z\right\vert
\right)  ^{-\left(  d-1\right)  }\left(  \frac{z}{\left\vert z\right\vert
}e^{2\left\vert z\right\vert }\partial_{\left\vert z\right\vert }q^{2}\left(
\left\vert z\right\vert \right)  \right)  dz
\]
and
\[
\left.  r_{1}=C\left\vert \chi\right\vert ^{-\left(  d-1\right)  }\left(
\sup_{\left\vert z\right\vert \leq\frac{\left\vert \chi\right\vert }{2}%
}\left\vert e^{-H\left(  \left\vert \chi-z\right\vert \right)  }-K\right\vert
+\left\vert \chi\right\vert ^{-1}\right)  \int_{\left\vert z\right\vert
\leq\frac{\left\vert \chi\right\vert }{2}}e^{-\frac{\left\vert z\right\vert
}{2}\left(  1-\cos\theta\right)  }\left\vert z\right\vert ^{-\left(
d-1\right)  }dz.\right.
\]
By (\ref{ap29})%
\[
\psi\left(  \frac{2\left\vert z\right\vert }{\left\vert \chi\right\vert
}\right)  e^{-\frac{2\left\vert \chi\right\vert \left\vert z\right\vert
\left(  1-\cos\theta\right)  }{\left\vert \chi\right\vert -\left\vert
z\right\vert }}\left(  \left\vert \chi\right\vert -\left\vert z\right\vert
\right)  ^{-\left(  d-1\right)  }e^{2\left\vert z\right\vert }\partial
_{\left\vert z\right\vert }q^{2}\left(  \left\vert z\right\vert \right)  \leq
C\left\vert \chi\right\vert ^{-\left(  d-1\right)  }e^{-\frac{\left\vert
z\right\vert }{2}\left(  1-\cos\theta\right)  }\left\vert z\right\vert
^{-\left(  d-1\right)  }.
\]
Let us show that $e^{-\frac{\left\vert z\right\vert }{2}\left(  1-\cos
\theta\right)  }\left\vert z\right\vert ^{-\left(  d-1\right)  }$ is
integrable if $d\geq4.$ Using the polar coordinate system in (\ref{coord}) we
have%
\begin{equation}
\left.
\begin{array}
[c]{c}%
{\displaystyle\int}
e^{-\frac{\left\vert z\right\vert }{2}\left(  1-\cos\theta\right)  }\left\vert
z\right\vert ^{-\left(  d-1\right)  }dz\leq C%
{\displaystyle\int_{0}^{\infty}}
\left(
{\displaystyle\int_{0}^{\pi}}
e^{-\frac{r}{2}\left(  1-\cos\theta\right)  }\theta^{d-2}d\theta\right)  dr\\
\\
\leq C%
{\displaystyle\int_{0}^{1}}
\left(
{\displaystyle\int_{0}^{\pi}}
e^{-\frac{r}{2}\left(  1-\cos\theta\right)  }\theta^{d-2}d\theta\right)  dr+C%
{\displaystyle\int_{1}^{\infty}}
r^{-\frac{d-1}{2}}\left(
{\displaystyle\int_{0}^{\pi\sqrt{r}}}
e^{-\frac{\theta^{2}}{2}}\theta^{d-2}d\theta\right)  dr\\
\\
\leq C+C\left(
{\displaystyle\int_{0}^{\infty}}
e^{-\frac{\theta^{2}}{2}}\theta^{d-2}d\theta\right)
{\displaystyle\int_{1}^{\infty}}
r^{-\frac{d-1}{2}}dr\leq C.
\end{array}
\right.  \label{polar}%
\end{equation}
Thus, by the dominated convergence theorem%
\[
I_{1}=K\left\vert \chi\right\vert ^{-\left(  d-1\right)  }%
{\displaystyle\int}
\frac{z}{\left\vert z\right\vert }e^{2\frac{\chi\cdot z}{\left\vert
\chi\right\vert }}\partial_{\left\vert z\right\vert }q^{2}\left(  \left\vert
z\right\vert \right)  dz+\left\vert \chi\right\vert ^{-\left(  d-1\right)
}o\left(  1\right)  ,
\]
as $\left\vert \chi\right\vert \rightarrow\infty$. Integrating by parts we get%
\begin{equation}
I_{1}=-K\frac{\chi}{\left\vert \chi\right\vert }\left\vert \chi\right\vert
^{-\left(  d-1\right)  }%
{\displaystyle\int}
e^{2\frac{\chi\cdot z}{\left\vert \chi\right\vert }}Q^{2}\left(  z\right)
dz+\left\vert \chi\right\vert ^{-\left(  d-1\right)  }o\left(  1\right)  .
\label{ap121}%
\end{equation}
Also from (\ref{polar}) it follows that%
\begin{equation}
r_{1}=\left\vert \chi\right\vert ^{-\left(  d-1\right)  }o\left(  1\right)  .
\label{ap120}%
\end{equation}
To estimate the last term in the right-hand side of (\ref{ap122}) we decompose%
\begin{align*}
&
{\displaystyle\int_{\frac{\left\vert \chi\right\vert }{2}\leq\left\vert
\chi-z\right\vert \leq\frac{3\left\vert \chi\right\vert }{2};\frac{\left\vert
\chi\right\vert }{2}\leq\left\vert z\right\vert }}
G\left(  \left\vert \chi-z\right\vert \right)  \nabla Q^{2}\left(  z\right)
dz\\
&  =%
{\displaystyle\int_{\frac{\left\vert \chi\right\vert }{2}\leq\left\vert
\chi-z\right\vert \leq\frac{3\left\vert \chi\right\vert }{2};\frac{\left\vert
\chi\right\vert }{2}\leq\left\vert z\right\vert \leq\frac{\left\vert
\chi\right\vert }{2}+\left\vert \chi\right\vert ^{\frac{1}{4}}}}
G\left(  \left\vert \chi-z\right\vert \right)  \nabla Q^{2}\left(  z\right)
dz\\
&  +%
{\displaystyle\int_{\frac{\left\vert \chi\right\vert }{2}\leq\left\vert
\chi-z\right\vert \leq\frac{3\left\vert \chi\right\vert }{2};\frac{\left\vert
\chi\right\vert }{2}+\left\vert \chi\right\vert ^{\frac{1}{4}}\leq\left\vert
z\right\vert }}
G\left(  \left\vert \chi-z\right\vert \right)  \nabla Q^{2}\left(  z\right)
dz.
\end{align*}
Using (\ref{ap29}) and passing to the polar system as in (\ref{coord}) we get%
\begin{equation}
\left.
\begin{array}
[c]{c}%
\left\vert
{\displaystyle\int_{\frac{\left\vert \chi\right\vert }{2}\leq\left\vert
\chi-z\right\vert \leq\frac{3\left\vert \chi\right\vert }{2};\frac{\left\vert
\chi\right\vert }{2}\leq\left\vert z\right\vert \leq\frac{\left\vert
\chi\right\vert }{2}+\left\vert \chi\right\vert ^{\frac{1}{4}}}}
G\left(  \left\vert \chi-z\right\vert \right)  \nabla Q^{2}\left(  z\right)
dz\right\vert \\
\leq Ce^{-2\left\vert \chi\right\vert }\left\vert \chi\right\vert ^{-\left(
d-1\right)  }e^{-H\left(  \frac{\left\vert \chi\right\vert }{2}\right)  }%
{\displaystyle\int_{\frac{\left\vert \chi\right\vert }{2}\leq\left\vert
z\right\vert \leq\frac{\left\vert \chi\right\vert }{2}+\left\vert
\chi\right\vert ^{\frac{1}{4}}}}
e^{-\frac{\left\vert z\right\vert }{2}\left(  1-\cos\theta\right)  }\left\vert
z\right\vert ^{-\left(  d-1\right)  }dz\\
\leq Ce^{-2\left\vert \chi\right\vert }\left\vert \chi\right\vert ^{-\left(
d-1\right)  }e^{-H\left(  \frac{\left\vert \chi\right\vert }{2}\right)  }%
{\displaystyle\int_{\frac{\left\vert \chi\right\vert }{2}}^{\frac{\left\vert
\chi\right\vert }{2}+\left\vert \chi\right\vert ^{\frac{1}{4}}}}
r^{-\frac{d-1}{2}}\left(
{\displaystyle\int_{0}^{\pi\sqrt{r}}}
e^{-\frac{1}{2}\theta^{2}}\theta^{-\left(  d-2\right)  }d\theta\right)  dr\leq
Ce^{-2\left\vert \chi\right\vert }\left\vert \chi\right\vert ^{-\left(
d-1\right)  }e^{-H\left(  \frac{\left\vert \chi\right\vert }{2}\right)
}\left\vert \chi\right\vert ^{\frac{1}{4}-\frac{d-1}{2}}%
\end{array}
\right.  \label{ap155}%
\end{equation}
and%
\[
\left\vert
{\displaystyle\int_{\frac{\left\vert \chi\right\vert }{2}\leq\left\vert
\chi-z\right\vert \leq\frac{3\left\vert \chi\right\vert }{2};\frac{\left\vert
\chi\right\vert }{2}+\left\vert \chi\right\vert ^{\frac{1}{4}}\leq\left\vert
z\right\vert }}
G\left(  \left\vert \chi-z\right\vert \right)  \nabla Q^{2}\left(  z\right)
dz\right\vert \leq Ce^{-2\left\vert \chi\right\vert }\left\vert \chi
\right\vert ^{-\left(  d-1\right)  }e^{-H\left(  \frac{\left\vert
\chi\right\vert }{2}\right)  }e^{-\left\vert \chi\right\vert ^{\frac{1}{4}}}.
\]
Hence
\begin{equation}
\left\vert
{\displaystyle\int_{\frac{\left\vert \chi\right\vert }{2}\leq\left\vert
\chi-z\right\vert \leq\frac{3\left\vert \chi\right\vert }{2};\frac{\left\vert
\chi\right\vert }{2}\leq\left\vert z\right\vert }}
G\left(  \left\vert \chi-z\right\vert \right)  \nabla Q^{2}\left(  z\right)
dz\right\vert \leq Ce^{-2\left\vert \chi\right\vert }\left\vert \chi
\right\vert ^{-\left(  d-1\right)  }e^{-H\left(  \frac{\left\vert
\chi\right\vert }{2}\right)  }\left\vert \chi\right\vert ^{\frac{1}{4}%
-\frac{d-1}{2}}. \label{ap119}%
\end{equation}
Making use of (\ref{ap117}), (\ref{ap121}), (\ref{ap120}) and (\ref{ap119}) in
(\ref{ap122}) we arrive to%
\begin{equation}
\left.
{\displaystyle\int_{\frac{\left\vert \chi\right\vert }{2}\leq\left\vert
z\right\vert \leq\frac{3\left\vert \chi\right\vert }{2}}}
G\left(  \left\vert z\right\vert \right)  \nabla Q^{2}\left(  \chi-z\right)
dz=K\kappa\frac{\chi}{\left\vert \chi\right\vert }e^{-2\left\vert
\chi\right\vert }\left\vert \chi\right\vert ^{-\left(  d-1\right)  }%
{\displaystyle\int}
e^{2\frac{\chi\cdot z}{\left\vert \chi\right\vert }}Q^{2}\left(  z\right)
dz+o\left(  e^{-2\left\vert \chi\right\vert }\left\vert \chi\right\vert
^{-\left(  d-1\right)  }\right)  .\right.  \label{ap143}%
\end{equation}
Next, we consider the cases $d=2$ and $d=3.$ If $H^{\prime}\left(  r\right)  $
does not tend to $0,$ as $r\rightarrow\infty,$ then as $H\geq0,$ we see that
$H\left(  r\right)  \geq cr,$ for some $c>0.$ In this case by (\ref{ap29}) we
have
\[
\left.
\begin{array}
[c]{c}%
\left\vert
{\displaystyle\int_{\frac{\left\vert \chi\right\vert }{2}\leq\left\vert
z\right\vert \leq\frac{3\left\vert \chi\right\vert }{2}}}
G\left(  \left\vert z\right\vert \right)  \nabla Q^{2}\left(  \chi-z\right)
dz\right\vert \\
\leq C\left\vert \chi\right\vert ^{-\left(  d-1\right)  }e^{-2\left\vert
\chi\right\vert }e^{-\frac{c\left\vert z\right\vert }{4}}%
{\displaystyle\int_{\frac{\left\vert \chi\right\vert }{2}\leq\left\vert
z\right\vert \leq\frac{3\left\vert \chi\right\vert }{2}}}
e^{-\frac{c\left\vert z\right\vert }{2}}\left\vert \chi-z\right\vert
^{-\left(  d-1\right)  }dz\leq C\left\vert \chi\right\vert ^{-\left(
d-1\right)  }e^{-2\left\vert \chi\right\vert }e^{-\frac{c\left\vert
\chi\right\vert }{4}}.
\end{array}
\right.
\]
Suppose now that $H^{\prime}\rightarrow0.$ We take the following estimate into
account%
\begin{equation}
\left\vert e^{-H\left(  \left\vert \chi-z\right\vert \right)  }-e^{-H\left(
\left\vert \chi\right\vert -\left\vert z\right\vert \right)  }\right\vert \leq
C\left\vert H^{\prime}\left(  \left\vert \chi\right\vert -\left\vert
z\right\vert \right)  \right\vert e^{-H\left(  \left\vert \chi\right\vert
-\left\vert z\right\vert \right)  }\left\vert z\right\vert \left(
1-\cos\theta\right)  . \label{ap134}%
\end{equation}
Using (\ref{ap2}), (\ref{ap123}), (\ref{ap115}), (\ref{ap116}) and
(\ref{ap134}) we estimate%
\begin{equation}
\left.  \left\vert
{\displaystyle\int_{\left\vert z\right\vert \leq\frac{\left\vert
\chi\right\vert }{2}}}
e^{-2\left(  \left\vert z-\chi\right\vert -\left\vert \chi\right\vert
+\left\vert z\right\vert \right)  }\left\vert \chi-z\right\vert ^{-\left(
d-1\right)  }e^{-H\left(  \left\vert \chi-z\right\vert \right)  }\left(
\frac{z}{\left\vert z\right\vert }e^{2\left\vert z\right\vert }\partial
_{\left\vert z\right\vert }q^{2}\left(  \left\vert z\right\vert \right)
\right)  dz+I_{2}\right\vert \leq r_{2}\right.  \label{ap124}%
\end{equation}
with%
\[
I_{2}=\mathcal{A}^{2}\frac{\chi}{\left\vert \chi\right\vert }%
{\displaystyle\int_{\left\vert z\right\vert \leq\frac{\left\vert
\chi\right\vert }{2}}}
e^{-\frac{2\left\vert \chi\right\vert \left\vert z\right\vert \left(
1-\cos\theta\right)  }{\left\vert \chi\right\vert -\left\vert z\right\vert }%
}\left(  \left\vert \chi\right\vert -\left\vert z\right\vert \right)
^{-\left(  d-1\right)  }e^{-H\left(  \left\vert \chi\right\vert -\left\vert
z\right\vert \right)  }\left\vert z\right\vert ^{-\left(  d-1\right)  }dz
\]
and
\[
\left.
\begin{array}
[c]{c}%
r_{2}=C\left\vert \chi\right\vert ^{-\left(  d-1\right)  }%
{\displaystyle\int_{\left\vert z\right\vert \leq\frac{\left\vert
\chi\right\vert }{2}}}
e^{-\frac{\left\vert z\right\vert }{2}\left(  1-\cos\theta\right)  }\left\vert
H^{\prime}\left(  \left\vert \chi\right\vert -\left\vert z\right\vert \right)
\right\vert e^{-H\left(  \left\vert \chi\right\vert -\left\vert z\right\vert
\right)  }\left\vert z\right\vert ^{-\left(  d-1\right)  }dz\\
+C\left\vert \chi\right\vert ^{-\left(  d-1\right)  }e^{-H\left(
\frac{\left\vert \chi\right\vert }{2}\right)  }%
{\displaystyle\int_{\left\vert z\right\vert \leq\frac{\left\vert
\chi\right\vert }{2}}}
e^{-\frac{\left\vert z\right\vert }{2}\left(  1-\cos\theta\right)  }\left(
\left(  1+\left\vert z\right\vert \right)  ^{-1}+\theta\right)  \left\vert
z\right\vert ^{-\left(  d-1\right)  }dz.
\end{array}
\right.
\]
Passing to the polar system as in (\ref{coord}) we get%
\begin{equation}
\left.
\begin{array}
[c]{c}%
\left\vert
{\displaystyle\int_{\left\vert z\right\vert \leq\frac{\left\vert
\chi\right\vert }{2}}}
e^{-\frac{2\left\vert \chi\right\vert \left\vert z\right\vert \left(
1-\cos\theta\right)  }{\left\vert \chi\right\vert -\left\vert z\right\vert }%
}\left(  \left\vert \chi\right\vert -\left\vert z\right\vert \right)
^{-\left(  d-1\right)  }e^{-H\left(  \left\vert \chi\right\vert -\left\vert
z\right\vert \right)  }\left\vert z\right\vert ^{-\left(  d-1\right)
}dz\right. \\
\left.  -\sigma%
{\displaystyle\int_{1}^{\frac{\left\vert \chi\right\vert }{2}}}
\left(  \left\vert \chi\right\vert -r\right)  ^{-\left(  d-1\right)
}e^{-H\left(  \left\vert \chi\right\vert -r\right)  }\left(
{\displaystyle\int_{0}^{\pi}}
e^{-\frac{2\left\vert \chi\right\vert r\theta^{2}}{\left\vert \chi\right\vert
-r}}\theta^{d-2}d\theta\right)  dr\right\vert \\
\leq C\left\vert \chi\right\vert ^{-\left(  d-1\right)  }e^{-H\left(
\frac{\left\vert \chi\right\vert }{2}\right)  }\left(  1+%
{\displaystyle\int_{\left\vert z\right\vert \leq\frac{\left\vert
\chi\right\vert }{2}}}
e^{-\frac{\left\vert z\right\vert }{2}\left(  1-\cos\theta\right)  }%
\theta\left\vert z\right\vert ^{-\left(  d-1\right)  }dz\right)  .
\end{array}
\right.  \label{ap125}%
\end{equation}%
\[
\sigma=\frac{2\pi^{\frac{d-1}{2}}}{\Gamma\left(  \frac{d-1}{2}\right)  },
\]
Making the change $\theta\rightarrow\left(  \frac{\left\vert \chi\right\vert
-r}{\left\vert \chi\right\vert r}\right)  ^{\frac{1}{2}}\eta$ we obtain%
\begin{equation}
\left.
\begin{array}
[c]{c}%
{\displaystyle\int_{1}^{\frac{\left\vert \chi\right\vert }{2}}}
\left(  \left\vert \chi\right\vert -r\right)  ^{-\left(  d-1\right)
}e^{-H\left(  \left\vert \chi\right\vert -r\right)  }\left(
{\displaystyle\int_{0}^{\pi}}
e^{-\frac{2\left\vert \chi\right\vert r\theta^{2}}{\left\vert \chi\right\vert
-r}}\theta^{d-2}d\theta\right)  dr=\upsilon\left(  d\right)  \left\vert
\chi\right\vert ^{-\frac{d-1}{2}}%
{\displaystyle\int_{\frac{\left\vert \chi\right\vert }{2}}^{\left\vert
\chi\right\vert -1}}
\left(  r\left(  \left\vert \chi\right\vert -r\right)  \right)  ^{-\frac
{d-1}{2}}e^{-H\left(  r\right)  }dr\\
-\left\vert \chi\right\vert ^{-\frac{d-1}{2}}%
{\displaystyle\int_{1}^{\frac{\left\vert \chi\right\vert }{2}}}
\left(  r\left(  \left\vert \chi\right\vert -r\right)  \right)  ^{-\frac
{d-1}{2}}e^{-H\left(  \left\vert \chi\right\vert -r\right)  }\left(
{\displaystyle\int_{\pi\sqrt{\frac{\left\vert \chi\right\vert r}{\left\vert
\chi\right\vert -r}}}^{\infty}}
e^{-2\eta^{2}}\eta^{d-2}d\eta\right)  dr.
\end{array}
\right.  \label{ap126}%
\end{equation}
where $\upsilon\left(  d\right)  $ is given by (\ref{ap142}). The second term
in the right-hand side of (\ref{ap126}) is estimated by using%
\begin{equation}
\left.
\begin{array}
[c]{c}%
\left\vert \chi\right\vert ^{-\frac{d-1}{2}}%
{\displaystyle\int_{1}^{\frac{\left\vert \chi\right\vert }{2}}}
\left(  r\left(  \left\vert \chi\right\vert -r\right)  \right)  ^{-\frac
{d-1}{2}}\left(
{\displaystyle\int_{\pi\sqrt{\frac{\left\vert \chi\right\vert r}{\left\vert
\chi\right\vert -r}}}^{\infty}}
e^{-2\eta^{2}}\eta^{d-2}d\eta\right)  dr\\
\leq C\left\vert \chi\right\vert ^{-\left(  d-1\right)  }%
{\displaystyle\int_{1}^{\infty}}
r^{-\frac{d-1}{2}}e^{-\frac{\pi^{2}}{2}r}dr\leq C\left\vert \chi\right\vert
^{-\left(  d-1\right)  }.
\end{array}
\right.  \label{ap129}%
\end{equation}
From (\ref{ap126}) and (\ref{ap129}) we get%
\begin{equation}
\left.
\begin{array}
[c]{c}%
{\displaystyle\int_{1}^{\frac{\left\vert \chi\right\vert }{2}}}
\left(  \left\vert \chi\right\vert -r\right)  ^{-\left(  d-1\right)
}e^{-H\left(  \left\vert \chi\right\vert -r\right)  }\left(
{\displaystyle\int_{0}^{\pi}}
e^{-\frac{2\left\vert \chi\right\vert r\theta^{2}}{\left\vert \chi\right\vert
-r}}\theta^{d-2}d\theta\right)  dr\\
=\upsilon\left(  d\right)  \left\vert \chi\right\vert ^{-\frac{d-1}{2}}%
{\displaystyle\int_{\frac{\left\vert \chi\right\vert }{2}}^{\left\vert
\chi\right\vert -1}}
\left(  r\left(  \left\vert \chi\right\vert -r\right)  \right)  ^{-\frac
{d-1}{2}}e^{-H\left(  r\right)  }dr+O\left(  \left\vert \chi\right\vert
^{-\left(  d-1\right)  }e^{-H\left(  \frac{\left\vert \chi\right\vert }%
{2}\right)  }\right)  .
\end{array}
\right.  \label{ap130}%
\end{equation}
Passing to the polar system as in (\ref{coord}) we estimate%
\begin{equation}
\left.
\begin{array}
[c]{c}%
{\displaystyle\int_{\left\vert z\right\vert \leq\frac{\left\vert
\chi\right\vert }{2}}}
e^{-\frac{\left\vert z\right\vert }{2}\left(  1-\cos\theta\right)  }\left\vert
H^{\prime}\left(  \left\vert \chi\right\vert -\left\vert z\right\vert \right)
\right\vert e^{-H\left(  \left\vert \chi\right\vert -\left\vert z\right\vert
\right)  }\left\vert z\right\vert ^{-\left(  d-1\right)  }dz\\
\leq C\left\vert H^{\prime}\left(  \frac{\left\vert \chi\right\vert }%
{2}\right)  \right\vert
{\displaystyle\int_{1}^{\frac{\left\vert \chi\right\vert }{2}}}
r^{-\frac{d-1}{2}}e^{-H\left(  \left\vert \chi\right\vert -r\right)  }\left(
{\displaystyle\int_{0}^{\pi\sqrt{r}}}
e^{-\frac{\theta^{2}}{4}}\theta^{d-2}d\theta\right)  dr\\
\leq C\left\vert H^{\prime}\left(  \frac{\left\vert \chi\right\vert }%
{2}\right)  \right\vert
{\displaystyle\int_{\frac{\left\vert \chi\right\vert }{2}}^{\left\vert
\chi\right\vert -1}}
\left(  \left\vert \chi\right\vert -r\right)  ^{-\frac{d-1}{2}}e^{-H\left(
r\right)  }dr
\end{array}
\right.  \label{ap135}%
\end{equation}
and%
\begin{equation}
\left.
\begin{array}
[c]{c}%
{\displaystyle\int_{\left\vert z\right\vert \leq\frac{\left\vert
\chi\right\vert }{2}}}
e^{-\frac{\left\vert z\right\vert }{2}\left(  1-\cos\theta\right)  }\left(
\left(  1+\left\vert z\right\vert \right)  ^{-1}+\theta\right)  \left\vert
z\right\vert ^{-\left(  d-1\right)  }dz\\
\leq C\left(  1+%
{\displaystyle\int_{1}^{\frac{\left\vert \chi\right\vert }{2}}}
r^{-\frac{d}{2}}\left(
{\displaystyle\int_{0}^{\pi\sqrt{r}}}
e^{-\frac{\theta^{2}}{4}}\left(  1+\theta\right)  ^{d-1}d\theta\right)
dr\right)  \leq C\left\{
\begin{array}
[c]{c}%
\ln\left\vert \chi\right\vert ,\text{ \ }d=2,\\
1,\text{ \ }d=3.
\end{array}
\right.
\end{array}
\right.  \label{ap136}%
\end{equation}
Thus, using (\ref{ap130}) and (\ref{ap136}) in (\ref{ap125}) we obtain%
\begin{equation}
\left.  \left\vert I_{2}-\mathcal{A}^{2}\frac{\chi}{\left\vert \chi\right\vert
}\upsilon\left(  d\right)  \left\vert \chi\right\vert ^{-\frac{d-1}{2}}%
{\displaystyle\int_{\frac{\left\vert \chi\right\vert }{2}}^{\left\vert
\chi\right\vert -1}}
\left(  r\left(  \left\vert \chi\right\vert -r\right)  \right)  ^{-\frac
{d-1}{2}}e^{-H\left(  r\right)  }dr\right\vert \leq C\left\vert \chi
\right\vert ^{-\left(  d-1\right)  }\allowbreak e^{-H\left(  \frac{\left\vert
\chi\right\vert }{2}\right)  }\left\{
\begin{array}
[c]{c}%
\ln\left\vert \chi\right\vert ,\text{ \ }d=2,\\
1,\text{ \ }d=3.
\end{array}
\right.  .\right.  \label{ap137}%
\end{equation}
Moreover (\ref{ap135}) and (\ref{ap136}) imply
\begin{equation}
\left.  r_{2}\leq C\left\vert \chi\right\vert ^{-\left(  d-1\right)  }\left(
\left\vert H^{\prime}\left(  \frac{\left\vert \chi\right\vert }{2}\right)
\right\vert
{\displaystyle\int_{\frac{\left\vert \chi\right\vert }{2}}^{\left\vert
\chi\right\vert -1}}
\left(  \left\vert \chi\right\vert -r\right)  ^{-\frac{d-1}{2}}e^{-H\left(
r\right)  }dr+e^{-H\left(  \frac{\left\vert \chi\right\vert }{2}\right)
}\left\{
\begin{array}
[c]{c}%
\ln\left\vert \chi\right\vert ,\text{ \ }d=2,\\
1,\text{ \ }d=3.
\end{array}
\right.  \right)  \right.  \label{ap131}%
\end{equation}
Therefore, since%
\begin{equation}
\left\vert \chi\right\vert ^{-\frac{d-1}{2}}%
{\displaystyle\int_{\frac{\left\vert \chi\right\vert }{2}}^{\left\vert
\chi\right\vert -1}}
\left(  r\left(  \left\vert \chi\right\vert -r\right)  \right)  ^{-\frac
{d-1}{2}}dr\geq C\left\vert \chi\right\vert ^{-\left(  d-1\right)  }\left\{
\begin{array}
[c]{c}%
\sqrt{\left\vert \chi\right\vert },\text{ \ }d=2,\\
\ln\left\vert \chi\right\vert ,\text{ \ }d=3,
\end{array}
\right.  \label{ap146}%
\end{equation}
via (\ref{ap122}), (\ref{ap119}), (\ref{ap124}), (\ref{ap137}) and
(\ref{ap131}) we arrive to%
\begin{equation}
\left.
\begin{array}
[c]{c}%
{\displaystyle\int_{\frac{\left\vert \chi\right\vert }{2}\leq\left\vert
z\right\vert \leq\frac{3\left\vert \chi\right\vert }{2}}}
G\left(  \left\vert z\right\vert \right)  \nabla Q^{2}\left(  \chi-z\right)
dz\\
=\kappa\mathcal{A}^{2}\upsilon\left(  d\right)  \left(  \frac{\chi}{\left\vert
\chi\right\vert }+o\left(  1\right)  \right)  e^{-2\left\vert \chi\right\vert
}\left\vert \chi\right\vert ^{-\frac{d-1}{2}}%
{\displaystyle\int_{\frac{\left\vert \chi\right\vert }{2}}^{\left\vert
\chi\right\vert -1}}
\left(  r\left(  \left\vert \chi\right\vert -r\right)  \right)  ^{-\frac
{d-1}{2}}e^{-H\left(  r\right)  }dr.
\end{array}
\right.  \label{ap138}%
\end{equation}
Let us consider now the region $\left\vert z\right\vert \leq\frac{\left\vert
\chi\right\vert }{2}.$ Note that%
\begin{equation}
\left\vert \frac{\chi-z}{\left\vert \chi-z\right\vert }-\frac{\chi}{\left\vert
\chi\right\vert }\right\vert \leq C\frac{\left\vert z\right\vert \left(
1-\cos\theta\right)  +\left\vert z\right\vert \sin\theta}{\left\vert
\chi\right\vert }. \label{ap156}%
\end{equation}
Then, using (\ref{ap2}), (\ref{ap115}) and (\ref{ap116}) we have%
\begin{equation}
\left.  \left\vert
{\displaystyle\int_{\left\vert z\right\vert \leq\frac{\left\vert
\chi\right\vert }{2}}}
G\left(  \left\vert z\right\vert \right)  \nabla Q^{2}\left(  \chi-z\right)
dz-\kappa\mathcal{A}^{2}\frac{\chi}{\left\vert \chi\right\vert }%
e^{-2\left\vert \chi\right\vert }%
{\displaystyle\int_{\left\vert z\right\vert \leq\frac{\left\vert
\chi\right\vert }{2}}}
e^{-\frac{2\left\vert \chi\right\vert \left\vert z\right\vert \left(
1-\cos\theta\right)  }{\left\vert \chi\right\vert -\left\vert z\right\vert }%
}\left\vert z\right\vert ^{-\left(  d-1\right)  }e^{-H\left(  \left\vert
z\right\vert \right)  }\left(  \left\vert \chi\right\vert -\left\vert
z\right\vert \right)  ^{-\left(  d-1\right)  }dz\right\vert \leq
r_{3},\right.  \label{ap140}%
\end{equation}
with%
\begin{equation}
r_{3}=Ce^{-2\left\vert \chi\right\vert }\left\vert \chi\right\vert
^{-d+\frac{1}{2}}%
{\displaystyle\int_{\left\vert z\right\vert \leq\frac{\left\vert
\chi\right\vert }{2}}}
e^{-\frac{\left\vert z\right\vert }{2}\left(  1-\cos\theta\right)  }\left\vert
z\right\vert ^{-\left(  d-1\right)  }e^{-H\left(  \left\vert z\right\vert
\right)  }dz. \label{ap139}%
\end{equation}
If $d\geq4,$ by (\ref{polar}) the integral in the right-hand side of
(\ref{ap139}) exists. Then by the dominated convergence theorem
\[%
{\displaystyle\int_{\left\vert z\right\vert \leq\frac{\left\vert
\chi\right\vert }{2}}}
e^{-\frac{2\left\vert \chi\right\vert \left\vert z\right\vert \left(
1-\cos\theta\right)  }{\left\vert \chi\right\vert -\left\vert z\right\vert }%
}\left\vert z\right\vert ^{-\left(  d-1\right)  }e^{-H\left(  \left\vert
z\right\vert \right)  }\left(  \left\vert \chi\right\vert -\left\vert
z\right\vert \right)  ^{-\left(  d-1\right)  }dz=\left\vert \chi\right\vert
^{-\left(  d-1\right)  }\int e^{-2\left\vert z\right\vert \left(  1-\cos
\theta\right)  }\left\vert z\right\vert ^{-\left(  d-1\right)  }e^{-H\left(
\left\vert z\right\vert \right)  }dz+o\left(  \left\vert \chi\right\vert
^{-\left(  d-1\right)  }\right)  .
\]
Thus, from (\ref{ap140}) it follows%
\begin{equation}%
{\displaystyle\int_{\left\vert z\right\vert \leq\frac{\left\vert
\chi\right\vert }{2}}}
G\left(  \left\vert z\right\vert \right)  \nabla Q^{2}\left(  \chi-z\right)
dz=\mathcal{A}^{2}\frac{\chi}{\left\vert \chi\right\vert }e^{-2\left\vert
\chi\right\vert }\left\vert \chi\right\vert ^{-\left(  d-1\right)  }\int
G\left(  \left\vert z\right\vert \right)  e^{2\frac{\chi\cdot z}{\left\vert
\chi\right\vert }}dz+o\left(  e^{-2\left\vert \chi\right\vert }\left\vert
\chi\right\vert ^{-\left(  d-1\right)  }\right)  . \label{ap141}%
\end{equation}
In the case of dimensions $d=2$ or $d=3.~$If $r^{-\frac{d-1}{2}}e^{-H\left(
r\right)  }\in L^{1}\left(  [1,\infty)\right)  ,$ $r_{3}=O\left(
e^{-2\left\vert \chi\right\vert }\left\vert \chi\right\vert ^{-d+\frac{1}{2}%
}\right)  ,$ similarly to the case $d\geq4$\ we obtain (\ref{ap141}). If
$r^{-\frac{d-1}{2}}e^{-H\left(  r\right)  }\notin L^{1}\left(  [1,\infty
)\right)  $ similarly to (\ref{ap137}) via (\ref{ap136}), (\ref{ap129}),
(\ref{ap142}) we have%
\begin{equation}
\left.
\begin{array}
[c]{c}%
\left\vert
{\displaystyle\int_{\left\vert z\right\vert \leq\frac{\left\vert
\chi\right\vert }{2}}}
e^{-\frac{2\left\vert \chi\right\vert \left\vert z\right\vert \left(
1-\cos\theta\right)  }{\left\vert \chi\right\vert -\left\vert z\right\vert }%
}\left\vert z\right\vert ^{-\left(  d-1\right)  }e^{-H\left(  \left\vert
z\right\vert \right)  }\left(  \left\vert \chi\right\vert -\left\vert
z\right\vert \right)  ^{-\left(  d-1\right)  }dz-\upsilon\left(  d\right)
\left\vert \chi\right\vert ^{-\frac{d-1}{2}}%
{\displaystyle\int_{1}^{\frac{\left\vert \chi\right\vert }{2}}}
\left(  r\left(  \left\vert \chi\right\vert -r\right)  \right)  ^{-\frac
{d-1}{2}}e^{-H\left(  r\right)  }dr\right\vert \\
\leq C\left\vert \chi\right\vert ^{-\left(  d-1\right)  }%
{\displaystyle\int_{1}^{\frac{\left\vert \chi\right\vert }{2}}}
r^{-\frac{d}{2}}e^{-H\left(  r\right)  }dr.
\end{array}
\right.  \label{ap157}%
\end{equation}
Moreover
\[
r_{3}\leq Ce^{-2\left\vert \chi\right\vert }\left\vert \chi\right\vert
^{-d+\frac{1}{2}}%
{\displaystyle\int_{1}^{\frac{\left\vert \chi\right\vert }{2}}}
r^{-\frac{d-1}{2}}e^{-H\left(  r\right)  }dr\leq Ce^{-2\left\vert
\chi\right\vert }\left\vert \chi\right\vert ^{-\left(  d-1\right)  }.
\]
Then as $\left\vert \chi\right\vert ^{-\frac{d-1}{2}}%
{\displaystyle\int_{1}^{\frac{\left\vert \chi\right\vert }{2}}}
\left(  r\left(  \left\vert \chi\right\vert -r\right)  \right)  ^{-\frac
{d-1}{2}}e^{-H\left(  r\right)  }dr\geq\left\vert \chi\right\vert ^{-\left(
d-1\right)  }%
{\displaystyle\int_{1}^{\frac{\left\vert \chi\right\vert }{2}}}
r^{-\frac{d-1}{2}}e^{-H\left(  r\right)  }dr,$ (\ref{ap140}) implies%
\begin{equation}
\left.
{\displaystyle\int_{\left\vert z\right\vert \leq\frac{\left\vert
\chi\right\vert }{2}}}
G\left(  \left\vert z\right\vert \right)  \nabla Q^{2}\left(  \chi-z\right)
dz=\kappa\mathcal{A}^{2}\upsilon\left(  d\right)  \left(  \frac{\chi
}{\left\vert \chi\right\vert }+o\left(  1\right)  \right)  e^{-2\left\vert
\chi\right\vert }\left\vert \chi\right\vert ^{-\frac{d-1}{2}}%
{\displaystyle\int_{1}^{\frac{\left\vert \chi\right\vert }{2}}}
\left(  r\left(  \left\vert \chi\right\vert -r\right)  \right)  ^{-\frac
{d-1}{2}}e^{-H\left(  r\right)  }dr.\right.  \label{ap144}%
\end{equation}

We now conclude as follows. If $d\geq4,$ gathering together (\ref{ap36}),
(\ref{ap143}) and (\ref{ap141}) we get (\ref{ap150}). In the case of
dimensions $d=2$ or $d=3,$ suppose first that $r^{-\frac{d-1}{2}}e^{-H\left(
r\right)  }\in L^{1}\left(  [1,\infty)\right)  .$ Then, (\ref{ap138}) implies
that
\[
\left.
{\displaystyle\int_{\frac{\left\vert \chi\right\vert }{2}\leq\left\vert
z\right\vert \leq\frac{3\left\vert \chi\right\vert }{2}}}
G\left(  \left\vert z\right\vert \right)  \nabla Q^{2}\left(  \chi-z\right)
dz=o\left(  e^{-2\left\vert \chi\right\vert }\left\vert \chi\right\vert
^{-\frac{d-1}{2}}\right)  .\right.
\]
Therefore, from (\ref{ap36}) and (\ref{ap141}) we attain (\ref{ap151}). If
$r^{-\frac{d-1}{2}}e^{-H\left(  r\right)  }\notin L^{1}\left(  [1,\infty
)\right)  ,$ by (\ref{ap36}), (\ref{ap138}) and (\ref{ap144}) we deduce
(\ref{ap152}).
\end{proof}

We now turn to the study of the case when $G$ determines the behavior of
(\ref{ap28})$.$ We consider bounded function $G\in C^{\infty}\left(
\mathbb{R}^{d}\right)  $ of the form
\begin{equation}
G\left(  r\right)  =V_{+}\left(  r\right)  ,\text{ }H\geq0. \label{ap166}%
\end{equation}
We prove the following.

\begin{lemma}
\label{L4}Let $G\in C^{\infty}\left(  \mathbb{R}^{d}\right)  $ be bounded and
of the form (\ref{ap166})$,$ where $H,H^{\prime}$ are monotone. If $H\left(
r\right)  =o\left(  r\right)  ,$ as $r\rightarrow\infty,$ the asymptotics
\begin{equation}
\mathcal{J}\left(  \chi\right)  =\frac{\chi}{\left\vert \chi\right\vert
}\kappa\mathcal{A}^{2}\upsilon\left(  d\right)  \left(  1+o\left(  1\right)
\right)  C_{+}\left(  \left\vert \chi\right\vert \right)  e^{-2\left\vert
\chi\right\vert }\left\vert \chi\right\vert ^{-\left(  d-1\right)  }
\label{ap168}%
\end{equation}
as $\left\vert \chi\right\vert \rightarrow\infty$ are valid with $C_{+}\left(
\left\vert \chi\right\vert \right)  $ given by (\ref{ap167}). If $0<cr\leq
H\left(  r\right)  <2r$ suppose in addition that $H^{\prime\prime}\left(
r\right)  $ is bounded. Then
\begin{equation}
\mathcal{J}\left(  \chi\right)  =\mathcal{I}\frac{\chi}{\left\vert
\chi\right\vert }\left(  1+O\left(  \tfrac{\left\vert K^{\prime}-H^{\prime
}\left(  \left\vert \chi\right\vert \right)  \right\vert }{\left\vert
\chi\right\vert }\right)  \right)  \left(  2-H^{\prime}\left(  \left\vert
\chi\right\vert \right)  \right)  G\left(  \left\vert \chi\right\vert \right)
, \label{ap169}%
\end{equation}
where $K^{\prime}=\lim_{r\rightarrow\infty}H^{\prime}\left(  r\right)  $ and
$\mathcal{I}$ is defined by (\ref{ap184}). Finally, if $V=V^{\left(  2\right)
},$
\begin{equation}
\mathcal{J}\left(  \chi\right)  =-\frac{\chi}{\left\vert \chi\right\vert
}\left(  \left(  \int Q^{2}\left(  z\right)  dz\right)  V^{\prime}\left(
\left\vert \chi\right\vert \right)  +\mathbf{r}\left(  \left\vert
\chi\right\vert \right)  +e^{-\frac{\left\vert \chi\right\vert }{2}}\right)  ,
\label{ap306}%
\end{equation}
with%
\begin{equation}
\mathbf{r}\left(  \left\vert \chi\right\vert \right)  =O\left(  \left(
\left\vert h^{\prime}\left(  \left\vert \chi\right\vert \right)  \right\vert
\left(  \left\vert h^{\prime}\left(  \left\vert \chi\right\vert \right)
\right\vert ^{2}+\tfrac{\left\vert h^{\prime}\left(  \left\vert \chi
\right\vert \right)  \right\vert }{\left\vert \chi\right\vert }+\left\vert
h^{\prime\prime}\left(  \left\vert \chi\right\vert \right)  \right\vert
+\left\vert \chi\right\vert ^{-2}\right)  +\tfrac{\left\vert h^{\prime\prime
}\left(  \left\vert \chi\right\vert \right)  \right\vert }{\left\vert
\chi\right\vert }+\left\vert h^{\prime\prime\prime}\left(  \left\vert
\chi\right\vert \right)  \right\vert \right)  V\left(  \left\vert
\chi\right\vert \right)  \right)  . \label{Rcal}%
\end{equation}

\end{lemma}

\begin{proof}
Suppose first that $H\left(  r\right)  =o\left(  r\right)  ,$ as
$r\rightarrow\infty.$ In this case, as $H^{\prime}$ is monotone, $H^{\prime
}\left(  r\right)  =o\left(  1\right)  ,$ as $r\rightarrow\infty.$ Similarly
to (\ref{ap36}) we estimate
\begin{equation}
\left.  \left\vert
{\displaystyle\int_{\left\vert z\right\vert \geq\frac{3\left\vert
\chi\right\vert }{2}}}
G\left(  \left\vert \chi-z\right\vert \right)  \nabla Q^{2}\left(  z\right)
dz\right\vert \leq Ce^{-3\left\vert \chi\right\vert }%
{\displaystyle\int}
\left\vert G\left(  \left\vert \chi-z\right\vert \right)  \right\vert dz\leq
Ce^{-3\left\vert \chi\right\vert }.\right.  \label{ap161}%
\end{equation}
Next, we decompose%
\begin{equation}
\left.
\begin{array}
[c]{c}%
{\displaystyle\int_{\frac{\left\vert \chi\right\vert }{2}\leq\left\vert
z\right\vert \leq\frac{3\left\vert \chi\right\vert }{2}}}
G\left(  \left\vert \chi-z\right\vert \right)  \nabla Q^{2}\left(  z\right)
dz=-%
{\displaystyle\int_{\frac{\left\vert \chi\right\vert }{2}\leq\left\vert
\chi-z\right\vert \leq\frac{3\left\vert \chi\right\vert }{2}}}
G\left(  \left\vert z\right\vert \right)  \nabla Q^{2}\left(  \chi-z\right)
dz\\
=-%
{\displaystyle\int_{\left\vert z\right\vert \leq\frac{\left\vert
\chi\right\vert }{2}}}
G\left(  \left\vert z\right\vert \right)  \nabla Q^{2}\left(  \chi-z\right)
dz-r_{4},
\end{array}
\right.  \label{ap158}%
\end{equation}
with%
\[
r_{4}=%
{\displaystyle\int_{\frac{\left\vert \chi\right\vert }{2}\leq\left\vert
\chi-z\right\vert \leq\frac{3\left\vert \chi\right\vert }{2};\frac{\left\vert
\chi\right\vert }{2}\leq\left\vert z\right\vert }}
G\left(  \left\vert z\right\vert \right)  \nabla Q^{2}\left(  \chi-z\right)
dz.
\]
To estimate $r_{4}$ we write%
\begin{align*}
r_{4}  &  =%
{\displaystyle\int_{\frac{\left\vert \chi\right\vert }{2}\leq\left\vert
\chi-z\right\vert \leq\frac{3\left\vert \chi\right\vert }{2};\frac{\left\vert
\chi\right\vert }{2}\leq\left\vert z\right\vert \leq\frac{\left\vert
\chi\right\vert }{2}+A}}
G\left(  \left\vert z\right\vert \right)  \nabla Q^{2}\left(  \chi-z\right)
dz\\
&  +%
{\displaystyle\int_{\frac{\left\vert \chi\right\vert }{2}\leq\left\vert
\chi-z\right\vert \leq\frac{3\left\vert \chi\right\vert }{2};\frac{\left\vert
\chi\right\vert }{2}+A\leq\left\vert z\right\vert \leq\left\vert
\chi\right\vert }}
G\left(  \left\vert z\right\vert \right)  \nabla Q^{2}\left(  \chi-z\right)
dz\\
&  +%
{\displaystyle\int_{\frac{\left\vert \chi\right\vert }{2}\leq\left\vert
\chi-z\right\vert \leq\frac{3\left\vert \chi\right\vert }{2};\left\vert
\chi\right\vert \leq\left\vert z\right\vert }}
G\left(  \left\vert z\right\vert \right)  \nabla Q^{2}\left(  \chi-z\right)
dz.
\end{align*}
where $A=\max\left\{  \left\vert \chi\right\vert ^{\frac{1}{4}},H\left(
\left\vert \chi\right\vert \right)  \right\}  .$ Using (\ref{ap29}) and
(\ref{ap118}), as $H\left(  r\right)  =o\left(  r\right)  ,$ similarly to
(\ref{ap155}) we get
\begin{equation}
\left\vert r_{4}\right\vert \leq C\left\vert \chi\right\vert ^{-\left(
d-1\right)  }e^{-2\left\vert \chi\right\vert }%
{\displaystyle\int_{\frac{\left\vert \chi\right\vert }{2}}^{\frac{\left\vert
\chi\right\vert }{2}+A}}
r^{-\frac{\left(  d-1\right)  }{2}}e^{H\left(  r\right)  }dr\leq C\left\vert
\chi\right\vert ^{-\left(  d-1\right)  }e^{-2\left\vert \chi\right\vert
}o\left(  1\right)
{\displaystyle\int_{1}^{\frac{\left\vert \chi\right\vert }{2}}}
r^{-\frac{\left(  d-1\right)  }{2}}e^{H\left(  r\right)  }dr. \label{ap159}%
\end{equation}
Introduce the polar coordinate system as in (\ref{coord})$.$ By (\ref{ap2}),
(\ref{ap115}), (\ref{ap116}) and (\ref{ap156}) we have%
\[
\left.  \left\vert
{\displaystyle\int_{\left\vert z\right\vert \leq\frac{\left\vert
\chi\right\vert }{2}}}
G\left(  \left\vert z\right\vert \right)  \nabla Q^{2}\left(  \chi-z\right)
dz-\mathcal{A}^{2}\kappa\frac{\chi}{\left\vert \chi\right\vert }%
e^{-2\left\vert \chi\right\vert }%
{\displaystyle\int_{\left\vert z\right\vert \leq\frac{\left\vert
\chi\right\vert }{2}}}
e^{-\frac{2\left\vert \chi\right\vert \left\vert z\right\vert \left(
1-\cos\theta\right)  }{\left\vert \chi\right\vert -\left\vert z\right\vert }%
}e^{H\left(  \left\vert z\right\vert \right)  }\left\vert z\right\vert
^{-\left(  d-1\right)  }\left(  \left\vert \chi\right\vert -\left\vert
z\right\vert \right)  ^{-\left(  d-1\right)  }dz\right\vert \leq r_{5}\right.
\]
with%
\[
r_{5}=C\left\vert \chi\right\vert ^{-\left(  d-1\right)  }e^{-2\left\vert
\chi\right\vert }%
{\displaystyle\int_{\left\vert z\right\vert \leq\frac{\left\vert
\chi\right\vert }{2}}}
e^{-\frac{\left\vert z\right\vert }{2}\left(  1-\cos\theta\right)
}e^{H\left(  \left\vert z\right\vert \right)  }\left(  1+\left\vert
z\right\vert \right)  ^{-\frac{1}{2}}\left\vert z\right\vert ^{-\left(
d-1\right)  }dz.
\]
Proceeding similarly to (\ref{ap157}) and (\ref{ap136}) we obtain%
\[
\left.
\begin{array}
[c]{c}%
\left\vert
{\displaystyle\int_{\left\vert z\right\vert \leq\frac{\left\vert
\chi\right\vert }{2}}}
e^{-\frac{2\left\vert \chi\right\vert \left\vert z\right\vert \left(
1-\cos\theta\right)  }{\left\vert \chi\right\vert -\left\vert z\right\vert }%
}e^{H\left(  \left\vert z\right\vert \right)  }\left\vert z\right\vert
^{-\left(  d-1\right)  }\left(  \left\vert \chi\right\vert -\left\vert
z\right\vert \right)  ^{-\left(  d-1\right)  }dz-\upsilon\left(  d\right)
\left\vert \chi\right\vert ^{-\frac{d-1}{2}}%
{\displaystyle\int_{1}^{\frac{\left\vert \chi\right\vert }{2}}}
\left(  r\left(  \left\vert \chi\right\vert -r\right)  \right)  ^{-\frac
{d-1}{2}}e^{H\left(  r\right)  }dr\right\vert \\
\leq C\left\vert \chi\right\vert ^{-\left(  d-1\right)  }%
{\displaystyle\int_{1}^{\frac{\left\vert \chi\right\vert }{2}}}
r^{-\frac{d}{2}}e^{H\left(  r\right)  }dr\leq Co\left(  1\right)  \left\vert
\chi\right\vert ^{-\left(  d-1\right)  }%
{\displaystyle\int_{1}^{\frac{\left\vert \chi\right\vert }{2}}}
r^{-\frac{d-1}{2}}e^{H\left(  r\right)  }dr
\end{array}
\right.
\]
and%
\[
r_{5}\leq Co\left(  1\right)  \left\vert \chi\right\vert ^{-\left(
d-1\right)  }%
{\displaystyle\int_{1}^{\frac{\left\vert \chi\right\vert }{2}}}
r^{-\frac{d-1}{2}}e^{H\left(  r\right)  }dr.
\]
Hence%
\begin{equation}%
{\displaystyle\int_{\left\vert z\right\vert \leq\frac{\left\vert
\chi\right\vert }{2}}}
G\left(  \left\vert z\right\vert \right)  \nabla Q^{2}\left(  \chi-z\right)
dz=\upsilon\left(  d\right)  \mathcal{A}^{2}\kappa\frac{\chi}{\left\vert
\chi\right\vert }\left(  1+o\left(  1\right)  \right)  e^{-2\left\vert
\chi\right\vert }\left\vert \chi\right\vert ^{-\frac{d-1}{2}}%
{\displaystyle\int_{1}^{\frac{\left\vert \chi\right\vert }{2}}}
\left(  r\left(  \left\vert \chi\right\vert -r\right)  \right)  ^{-\frac
{d-1}{2}}e^{H\left(  r\right)  }dr \label{ap160}%
\end{equation}
and thus from (\ref{ap158}) and (\ref{ap159}) we deduce%
\begin{equation}
\left.
{\displaystyle\int_{\frac{\left\vert \chi\right\vert }{2}\leq\left\vert
z\right\vert \leq\frac{3\left\vert \chi\right\vert }{2}}}
G\left(  \left\vert \chi-z\right\vert \right)  \nabla Q^{2}\left(  z\right)
dz=-\upsilon\left(  d\right)  \mathcal{A}^{2}\frac{\chi}{\left\vert
\chi\right\vert }\left(  1+o\left(  1\right)  \right)  e^{-2\left\vert
\chi\right\vert }\left\vert \chi\right\vert ^{-\frac{d-1}{2}}%
{\displaystyle\int_{1}^{\frac{\left\vert \chi\right\vert }{2}}}
\left(  r\left(  \left\vert \chi\right\vert -r\right)  \right)  ^{-\frac
{d-1}{2}}e^{H\left(  r\right)  }dr.\right.  \label{ap162}%
\end{equation}
Note that%
\[
\left\vert H\left(  \left\vert \chi-z\right\vert \right)  -H\left(  \left\vert
\chi\right\vert -\left\vert z\right\vert \right)  \right\vert \leq CH^{\prime
}\left(  \left\vert \chi\right\vert -\left\vert z\right\vert \right)
\left\vert z\right\vert \left(  1-\cos\theta\right)  .
\]
Then, similarly to (\ref{ap160}) (see also the proof of (\ref{ap138})) we show
that%
\begin{equation}%
{\displaystyle\int_{\left\vert z\right\vert \leq\frac{\left\vert
\chi\right\vert }{2}}}
G\left(  \left\vert \chi-z\right\vert \right)  \nabla Q^{2}\left(  z\right)
dz=-\upsilon\left(  d\right)  \mathcal{A}^{2}\kappa\frac{\chi}{\left\vert
\chi\right\vert }\left(  1+o\left(  1\right)  \right)  e^{-2\left\vert
\chi\right\vert }\left\vert \chi\right\vert ^{-\frac{d-1}{2}}%
{\displaystyle\int_{\frac{\left\vert \chi\right\vert }{2}}^{\left\vert
\chi\right\vert -1}}
\left(  r\left(  \left\vert \chi\right\vert -r\right)  \right)  ^{-\frac
{d-1}{2}}e^{H\left(  r\right)  }dr. \label{ap163}%
\end{equation}
Therefore, as
\begin{equation}
\mathcal{J}\left(  \chi\right)  =-%
{\displaystyle\int}
G\left(  \left\vert \chi-z\right\vert \right)  \nabla Q^{2}\left(  z\right)
dz \label{ap170}%
\end{equation}
by (\ref{ap161}), (\ref{ap162}), (\ref{ap163}) we prove (\ref{ap168}).

Suppose now $0<cr\leq H\left(  r\right)  \leq2r+o\left(  r\right)  $ for all
$r\ $sufficiently large$.$ First we note that%
\begin{equation}
\left\vert
{\displaystyle\int_{\left\vert z\right\vert \geq\delta\left\vert
\chi\right\vert }}
G\left(  \left\vert \chi-z\right\vert \right)  \nabla Q^{2}\left(  z\right)
dz\right\vert \leq C\left\vert \chi\right\vert ^{-\left(  d-1\right)
}e^{-\delta\left\vert \chi\right\vert }%
{\displaystyle\int_{\left\vert z\right\vert \geq\delta\left\vert
\chi\right\vert }}
e^{-\left\vert z\right\vert }dz\leq C\left\vert \chi\right\vert ^{-\left(
d-1\right)  }e^{-\delta\left\vert \chi\right\vert }, \label{ap165}%
\end{equation}
for some $1-\frac{c}{2}<\delta<1$. We consider the following estimates%
\begin{equation}
\left\vert \left\vert \chi+y\right\vert -\left\vert \chi\right\vert
-\frac{\chi\cdot y}{\left\vert \chi\right\vert }\right\vert \leq
C\frac{\left\vert y\right\vert ^{2}}{\left\vert \chi\right\vert },
\label{ap249}%
\end{equation}
and%
\begin{equation}
\left\vert \left\vert \chi+y\right\vert ^{-\left(  d-1\right)  }-\left\vert
\chi\right\vert ^{-\left(  d-1\right)  }\right\vert \leq C\frac{\left\vert
y\right\vert }{\left\vert \chi\right\vert ^{d}}. \label{ap250}%
\end{equation}
Since $0<cr\leq H\left(  r\right)  <2r$ and $H^{\prime}$ is monotone, we have
$c_{1}\leq H^{\prime}\left(  r\right)  <2,$ $0<c_{1}<c,$ for all
$r\ $sufficiently large$.$ Then, using that $H^{\prime\prime}$ is bounded we
get%
\begin{equation}
\left\vert H\left(  \left\vert \chi-z\right\vert \right)  -H\left(  \left\vert
\chi\right\vert \right)  +H^{\prime}\left(  \left\vert \chi\right\vert
\right)  \left(  \frac{\chi\cdot z}{\left\vert \chi\right\vert }\right)
\right\vert \leq C\frac{1+\left\vert z\right\vert ^{4}}{\left\vert
\chi\right\vert } \label{ap251}%
\end{equation}
for $\left\vert z\right\vert \leq\delta\left\vert \chi\right\vert .$ In
particular, (\ref{ap249})-(\ref{ap251}) imply
\[
\left\vert G\left(  \left\vert \chi-z\right\vert \right)  -G\left(  \left\vert
\chi\right\vert \right)  e^{\left(  2-H^{\prime}\left(  \left\vert
\chi\right\vert \right)  \right)  \frac{\chi\cdot z}{\left\vert \chi
\right\vert }}\right\vert \leq\left\vert G\left(  \left\vert \chi\right\vert
\right)  \right\vert e^{\left(  2-H^{\prime}\left(  \left\vert \chi\right\vert
\right)  \right)  \frac{\chi\cdot z}{\left\vert \chi\right\vert }}%
\frac{1+\left\vert z\right\vert ^{4}}{\left\vert \chi\right\vert },\text{ as
}\left\vert \chi\right\vert \rightarrow\infty.
\]
Using the last estimate we deduce
\begin{equation}
\left.  \left\vert
{\displaystyle\int_{\left\vert z\right\vert \leq\delta\left\vert
\chi\right\vert }}
G\left(  \left\vert \chi-z\right\vert \right)  \nabla Q^{2}\left(  z\right)
dz-G\left(  \left\vert \chi\right\vert \right)
{\displaystyle\int_{\left\vert z\right\vert \leq\delta\left\vert
\chi\right\vert }}
e^{\left(  2-H^{\prime}\left(  \left\vert \chi\right\vert \right)  \right)
\frac{\chi\cdot z}{\left\vert \chi\right\vert }}\nabla Q^{2}\left(  z\right)
dz\right\vert \leq C\frac{\left\vert G\left(  \left\vert \chi\right\vert
\right)  \right\vert }{\left\vert \chi\right\vert }r_{6}\right.  \label{ap164}%
\end{equation}
where%
\[
r_{6}=%
{\displaystyle\int}
e^{\left(  2-H^{\prime}\left(  \left\vert \chi\right\vert \right)  \right)
\frac{\chi\cdot z}{\left\vert \chi\right\vert }}\left(  1+\left\vert
z\right\vert ^{4}\right)  \left\vert \nabla Q^{2}\left(  z\right)  \right\vert
dz.
\]
Integrating by parts we have%
\[%
{\displaystyle\int}
e^{\left(  2-H^{\prime}\left(  \left\vert \chi\right\vert \right)  \right)
\frac{\chi\cdot z}{\left\vert \chi\right\vert }}\nabla Q^{2}\left(  z\right)
dz=-\left(  2-H^{\prime}\left(  \left\vert \chi\right\vert \right)  \right)
\frac{\chi}{\left\vert \chi\right\vert }%
{\displaystyle\int}
e^{\left(  2-H^{\prime}\left(  \left\vert \chi\right\vert \right)  \right)
\frac{\chi\cdot z}{\left\vert \chi\right\vert }}Q^{2}\left(  z\right)  dz.
\]
$\ $Since $\left\vert 2-H^{\prime}\left(  \left\vert \chi\right\vert \right)
\right\vert \leq\max\{2-c_{1},1\}<2,$ for all $\left\vert \chi\right\vert $
sufficiently large, it follows from (\ref{ap29}) that $r_{6}<\infty$,
\[%
{\displaystyle\int_{\left\vert z\right\vert \geq\delta\left\vert
\chi\right\vert }}
e^{\left(  2-H^{\prime}\left(  \left\vert \chi\right\vert \right)  \right)
\frac{\chi\cdot z}{\left\vert \chi\right\vert }}\nabla Q^{2}\left(  z\right)
dz\leq Ce^{-c_{1}\delta^{\prime}\left\vert \chi\right\vert },\text{
\ }0<\delta^{\prime}<\delta,
\]
and
\[
\left\vert
{\displaystyle\int}
\left(  e^{\left(  2-H^{\prime}\left(  \left\vert \chi\right\vert \right)
\right)  \frac{\chi\cdot z}{\left\vert \chi\right\vert }}-e^{\left(
2-K^{\prime}\right)  \frac{\chi\cdot z}{\left\vert \chi\right\vert }}\right)
Q^{2}\left(  z\right)  dz\right\vert \leq C\left\vert K^{\prime}-H^{\prime
}\left(  \left\vert \chi\right\vert \right)  \right\vert .
\]
Then, from (\ref{ap170}), (\ref{ap165}) and (\ref{ap164}) we attain
(\ref{ap169}).

Finally, we consider the case $V=V^{\left(  2\right)  }.$ Recall that in this
case $V\left(  r\right)  =e^{-h_{1}\left(  r\right)  },$ with $h_{1}\left(
r\right)  \geq0$, $h_{1}\left(  r\right)  =o\left(  r\right)  ,$
$h_{1}^{\left(  k\right)  }$ monotone for all $k\geq0$ and satisfying
\[
\left\vert h_{1}^{\left(  k\right)  }\left(  \frac{r}{2}\right)  \right\vert
\leq C_{k}\left\vert h_{1}^{\left(  k\right)  }\left(  r\right)  \right\vert
,\text{ }k\geq0,
\]
for all $r.$ Then, expanding the potential $V^{\prime}$ we have%
\[
V^{\prime}\left(  \left\vert \chi-z\right\vert \right)  =V^{\prime}\left(
\left\vert \chi\right\vert \right)  +V_{1}\left(  z,\chi\right)
+\mathbf{r}\left(  \left\vert \chi\right\vert \right)  \left\vert z\right\vert
^{2},
\]
for $\left\vert z\right\vert \leq\frac{\left\vert \chi\right\vert }{2},$ where
$V_{1}\left(  z,\chi\right)  $ denotes the linear in $z$ polynomial in the
Taylor expansion of $V^{\prime}.$ Since $Q^{2}\left(  z\right)  $ is even, we
get $\int V_{1}\left(  z,\chi\right)  Q^{2}\left(  z\right)  dz=0.$ Then,\ we
have
\[%
{\displaystyle\int_{\left\vert z\right\vert \leq\frac{\left\vert
\chi\right\vert }{2}}}
\nabla V\left(  \left\vert \chi-z\right\vert \right)  Q^{2}\left(  z\right)
dz=\frac{\chi}{\left\vert \chi\right\vert }\left(  V^{\prime}\left(
\left\vert \chi\right\vert \right)  \int Q^{2}\left(  z\right)  dz+\mathbf{r}%
\left(  \left\vert \chi\right\vert \right)  \right)  .
\]
Hence, as $\left\vert
{\textstyle\int_{\left\vert z\right\vert \geq\frac{\left\vert \chi\right\vert
}{2}}}
V\left(  \left\vert \chi-z\right\vert \right)  \nabla Q^{2}\left(  z\right)
dz\right\vert \leq Ce^{-\frac{\left\vert \chi\right\vert }{2}},$ we attain
(\ref{ap306}).
\end{proof}

\section{Invertibility of $\mathcal{L}_{V}.$ Proof of Lemma \ref{Linv}%
\label{Invertibility}}

The second statement of Lema \ref{Linv} is a direct consequence of
(\ref{ineqpert}). Indeed, assuming (\ref{ineqpert}), the invertibility of the
operators $L_{\pm}-\lambda^{2}V\left(  \left\vert y+\tilde{\chi}\right\vert
\right)  $ follows from Fredholm alternative applied to $\left(
-\Delta+1\right)  ^{-1}\left(  L_{\pm}-\lambda^{2}V\left(  \left\vert
y+\tilde{\chi}\right\vert \right)  \right)  $. The regularity and exponential
decay of the solution follow from the properties of the elliptic operators
(\cite{Agmon}). See Lemma 2.4 of \cite{Krieger} for the proof in the case of
the Hartree equation.

Therefore, we need to prove (\ref{ineqpert}). We recall the positivity
estimates for $L_{\pm},$
\[
\ker L_{+}=\operatorname*{span}\{\nabla Q\},\left.  L_{+}\right\vert
_{\left\{  Q\right\}  ^{\perp}}\geq0\text{ and }\left.  L_{-}\right\vert
_{\left\{  Q\right\}  ^{\perp}}>0
\]
(see, for example, Lemmas 2.1 and 2.2 of \cite{Nakanishi}). Then, from
(\ref{tw78}) we get%
\begin{equation}
\left(  \mathcal{L}f,f\right)  \geq c\left\Vert f\right\Vert _{H^{1}}%
^{2},\text{ }c>0, \label{ineq1}%
\end{equation}
for all $f\in H^{1}$ such that $\left\vert \left(  f,Q\right)  \right\vert
+\left\vert \left(  f,iQ\right)  \right\vert +\left\vert \left(  f,\nabla
Q\right)  \right\vert =0.$ We claim that (\ref{ineq1}) and identity
\begin{equation}
L_{+}Q=-\left(  p-1\right)  \left(  -\Delta+1\right)  Q \label{idL+}%
\end{equation}
imply%
\begin{equation}
\left\Vert f\right\Vert _{H^{1}}\leq C\left(  \left\Vert \mathcal{L}%
f\right\Vert _{H^{-1}}+\left\vert \left(  f,iQ\right)  \right\vert +\left\vert
\left(  f,\nabla Q\right)  \right\vert \right)  , \label{ineq2}%
\end{equation}
for all $f\in H^{1}$. Indeed, let $f=h+ig,$ with real $h\ $and $g.$ We write
$h=\left(  h,Q\right)  \left\Vert Q\right\Vert _{L^{2}}^{-1}Q+\left(  h,\nabla
Q\right)  \left\Vert \nabla Q\right\Vert _{L^{2}}^{-1}\nabla Q+h^{\perp}$ and
$g=\left(  g,Q\right)  \left\Vert Q\right\Vert _{L^{2}}^{-1}+g^{\perp},$
$\left(  h^{\perp},Q\right)  =\left(  h^{\perp},\nabla Q\right)  =0$ and
$\left(  g^{\perp},Q\right)  =0.$ Then%
\[
L_{+}h=\left(  h,Q\right)  \left\Vert Q\right\Vert _{L^{2}}^{-1}L_{+}%
Q+L_{+}h^{\perp}\text{ and }L_{-}g=L_{-}g^{\perp}.
\]
Since
\begin{equation}
\left(  Q,\nabla Q\right)  =\left(  h^{\perp},\nabla Q\right)  =0\text{ and
}\ker L_{+}=\operatorname*{span}\{\nabla Q\}, \label{cond}%
\end{equation}
$L_{+}Q$ and $L_{+}h^{\perp}$ are linearly independent. Moreover, there is
$0<\delta<1,$ such that
\begin{equation}
\left\vert \left(  L_{+}Q,L_{+}h^{\perp}\right)  _{H^{-1}}\right\vert
\leq\left(  1-\delta\right)  \left\Vert L_{+}Q\right\Vert _{H^{-1}}\left\Vert
L_{+}h^{\perp}\right\Vert _{H^{-1}}, \label{cauchy}%
\end{equation}
uniformly on $h.$ Otherwise, there is a sequence $\{h_{n}^{\perp}%
\}_{n\in\mathbb{N}},$ $\left\Vert h_{n}^{\perp}\right\Vert _{H^{1}}=1,$
satisfying $\left(  h_{n}^{\perp},Q\right)  =\left(  h_{n}^{\perp},\nabla
Q\right)  =0$, such that%
\[
\left\vert \left(  L_{+}Q,L_{+}h_{n}^{\perp}\right)  _{H^{-1}}\right\vert
=\mu_{n}\left\Vert L_{+}Q\right\Vert _{H^{-1}}\left\Vert L_{+}h_{n}^{\perp
}\right\Vert _{H^{-1}},
\]
with $\mu_{n}\rightarrow1,$ as $n\rightarrow\infty.$ Since $h_{n}^{\perp}$
converges weakly to a function $h_{\infty}^{\perp}\in H^{1},$ we have
$\lim_{n\rightarrow\infty}\left(  L_{+}Q,L_{+}h_{n}^{\perp}\right)  _{H^{-1}%
}=\left(  L_{+}Q,L_{+}h_{\infty}^{\perp}\right)  _{H^{-1}}$ and $\left\Vert
L_{+}h_{\infty}^{\perp}\right\Vert _{H^{-1}}\leq\lim_{n\rightarrow\infty
}\left\Vert L_{+}h_{n}^{\perp}\right\Vert _{H^{-1}}.$ Then%
\[
\left\vert \left(  L_{+}Q,L_{+}h_{\infty}^{\perp}\right)  _{H^{-1}}\right\vert
=\left\Vert L_{+}Q\right\Vert _{H^{-1}}\left(  \lim_{n\rightarrow\infty
}\left\Vert L_{+}h_{n}^{\perp}\right\Vert _{H^{-1}}\right)  .
\]
Hence, as $\left\vert \left(  L_{+}Q,L_{+}h_{\infty}^{\perp}\right)  _{H^{-1}%
}\right\vert \leq\left\Vert L_{+}Q\right\Vert _{H^{-1}}\left\Vert
L_{+}h_{\infty}^{\perp}\right\Vert _{H^{-1}}$, we get
\[
\left\vert \left(  L_{+}Q,L_{+}h_{\infty}^{\perp}\right)  _{H^{-1}}\right\vert
=\left\Vert L_{+}Q\right\Vert _{H^{-1}}\left\Vert L_{+}h_{\infty}^{\perp
}\right\Vert _{H^{-1}},
\]
which means that $L_{+}\left(  \gamma Q-h_{\infty}^{\perp}\right)  =0,$ for
some $\gamma\neq0.$ Then by (\ref{cond}), $\gamma Q-h_{\infty}^{\perp}=0.$
Multiplying the last equality by $Q,$ we get $\gamma=0,$ a contradiction.
Hence, (\ref{cauchy}) holds. Then, by (\ref{ineq1}) and (\ref{idL+})%
\begin{align*}
\left\Vert L_{+}h\right\Vert _{H^{-1}}^{2}  &  =\left(  h,Q\right)
^{2}\left\Vert L_{+}Q\right\Vert _{H^{-1}}^{2}+\left\Vert L_{+}h^{\perp
}\right\Vert _{H^{-1}}^{2}+2\left(  h,Q\right)  \left(  L_{+}Q,L_{+}h^{\perp
}\right)  _{H^{-1}}\\
&  \geq c\left(  \left(  h,Q\right)  ^{2}\left\Vert L_{+}Q\right\Vert
_{H^{-1}}^{2}+\left\Vert L_{+}h^{\perp}\right\Vert _{H^{-1}}^{2}\right)  \geq
b\left\Vert h\right\Vert _{H^{1}}^{2}-\frac{1}{b}\left\vert \left(  h,\nabla
Q\right)  \right\vert ^{2},
\end{align*}
and%
\[
\left\Vert L_{-}g\right\Vert _{H^{-1}}^{2}=\left\Vert L_{-}g^{\perp
}\right\Vert _{H^{-1}}^{2}\geq\left\Vert g^{\perp}\right\Vert _{H^{1}}^{2}\geq
b\left\Vert g\right\Vert _{H^{1}}^{2}-\frac{1}{b}\left(  g,Q\right)  ^{2}%
\]
for some $b>0.$ Therefore, (\ref{ineq2}) follows.

Let us show that (\ref{ineq2}) remains true for the perturbed operator
\[
\mathcal{L}_{V}=\mathcal{L}-\lambda^{2}V\left(  \left\vert y+\tilde{\chi
}\right\vert \right)
\]
for all $\lambda\geq\lambda_{0}>0\ $such that $\lambda^{2}\sup_{r\in
\mathbb{R}}V\left(  r\right)  <1$ and $\left\vert \chi\right\vert $
sufficiently big. Let $\rho\in C^{\infty}\left(  \mathbb{R}^{d}\right)  $ be
such that $0\leq\rho\leq1,$ $\rho\left(  x\right)  =1$ for $\left\vert
x\right\vert \leq\frac{1}{4}$ and $\rho\left(  x\right)  =0$ for $\left\vert
x\right\vert \geq\frac{1}{2}.$ For $f\in H^{1}$ and $\sigma>0,$ we decompose
$f=f_{1}+f_{2},$ where $f_{1}\left(  y\right)  =\rho_{1}\left(  \frac
{y}{\sigma}\right)  f\left(  y\right)  $, $f_{2}\left(  y\right)  =\rho
_{2}\left(  \frac{y}{\sigma}\right)  f\left(  y\right)  ,$ $\rho_{1}\left(
y\right)  =1-\rho\left(  y\right)  $ and $\rho_{2}\left(  y\right)
=\rho\left(  y\right)  .$ We estimate%
\begin{equation}
\left.
\begin{array}
[c]{c}%
\left\Vert \mathcal{L}_{V}f\right\Vert _{H^{-1}}^{2}=\left\Vert \mathcal{L}%
_{V}f_{1}\right\Vert _{H^{-1}}^{2}+\left\Vert \mathcal{L}_{V}f_{2}\right\Vert
_{H^{-1}}^{2}+2\left(  \mathcal{L}_{V}f_{1},\mathcal{L}_{V}f_{2}\right)
_{H^{-1}}\\
\geq\left\Vert \left(  -\Delta+1-\lambda^{2}V\left(  \left\vert y+\tilde{\chi
}\right\vert \right)  \right)  f_{1}\right\Vert _{H^{-1}}^{2}+\left\Vert
\mathcal{L}f_{2}\right\Vert _{H^{-1}}^{2}+2\left(  \mathcal{L}_{V}%
f_{1},\mathcal{L}_{V}f_{2}\right)  _{H^{-1}}+r,
\end{array}
\right.  \label{ap230}%
\end{equation}
with%
\[
\left.  r=-2\left\vert \left(  \left(  -\Delta+1-\lambda^{2}V\left(
\left\vert y+\tilde{\chi}\right\vert \right)  \right)  f_{1},-\frac{p+1}%
{2}Q^{p-1}f_{1}-\frac{p-1}{2}Q^{p-1}{\overline{f_{1}}}\right)  _{H^{-1}%
}\right\vert -2\lambda^{2}\left\vert \left(  \mathcal{L}f_{2},V\left(
\left\vert y+\tilde{\chi}\right\vert \right)  f_{2}\right)  _{H^{-1}%
}\right\vert \right.
\]
Observe that%
\begin{equation}
\int Q^{p-1}\left\vert f_{1}\right\vert \left\vert g\right\vert \leq\int
Q^{p-1}\left(  y\right)  \rho_{1}\left(  \frac{y}{\sigma}\right)  \left\vert
f\left(  y\right)  \right\vert \left\vert g\left(  y\right)  \right\vert
dy\leq q^{p-1}\left(  \frac{\sigma}{4}\right)  \left\Vert f\right\Vert
\left\Vert g\right\Vert \label{es1}%
\end{equation}
and%
\begin{equation}
\int\left\vert V\left(  y+\tilde{\chi}\right)  \right\vert \left\vert
f_{2}\right\vert \left\vert g\right\vert \leq\int\left\vert V\left(
y+\tilde{\chi}\right)  \right\vert \rho_{2}\left(  \frac{y}{\sigma}\right)
\left\vert f\left(  y\right)  \right\vert \left\vert g\right\vert
dy\leq\left\Vert V\left(  y+\tilde{\chi}\right)  \rho_{2}\left(  \frac
{y}{\sigma}\right)  \right\Vert _{L^{\infty}}\left\Vert f\right\Vert
\left\Vert g\right\Vert . \label{es2}%
\end{equation}
Then%
\begin{equation}
\left\vert r\right\vert \leq K_{1}\left(  V\right)  \left(  \left\Vert
V\left(  \left\vert y+\tilde{\chi}\right\vert \right)  \rho_{2}\left(
\frac{y}{\sigma}\right)  \right\Vert _{L^{\infty}}+q^{p-1}\left(  \frac
{\sigma}{4}\right)  \right)  \left\Vert f\right\Vert _{H^{1}}^{2},\text{
}K_{1}\left(  V\right)  >0. \label{r0}%
\end{equation}
Using that $\nabla f_{j}\left(  y\right)  =\frac{y}{\sigma\left\vert
y\right\vert }\rho_{j}^{\prime}\left(  \frac{y}{\sigma}\right)  f\left(
y\right)  +\rho_{j}\left(  \frac{y}{\sigma}\right)  \nabla f\left(  y\right)
$, $j=1,2\ $we have%
\begin{equation}
\left(  \mathcal{L}_{V}f_{1},\mathcal{L}_{V}f_{2}\right)  _{H^{-1}}=\left(
\rho_{1}\left(  \frac{\cdot}{\sigma}\right)  \nabla f,\rho_{2}\left(
\frac{\cdot}{\sigma}\right)  \nabla f\right)  +\left(  f_{1},f_{2}\right)
+r_{1}+r_{2}+r_{3} \label{ap231}%
\end{equation}
where%
\[
r_{1}=\frac{1}{\sigma^{2}}\left(  \rho_{1}^{\prime}\left(  \frac{y}{\sigma
}\right)  f\left(  y\right)  ,\rho_{2}^{\prime}\left(  \frac{y}{\sigma
}\right)  f\left(  y\right)  \right)  +\frac{1}{\sigma}\left(  \frac
{y}{\left\vert y\right\vert }\rho_{1}^{\prime}\left(  \frac{y}{\sigma}\right)
f\left(  y\right)  ,\rho_{2}\left(  \frac{y}{\sigma}\right)  \nabla f\left(
y\right)  \right)  -\frac{1}{\sigma}\left(  \rho_{1}\left(  \frac{y}{\sigma
}\right)  \nabla f\left(  y\right)  ,\frac{y}{\left\vert y\right\vert }%
\rho_{2}^{\prime}\left(  y\right)  f\left(  y\right)  \right)
\]%
\[
\left.
\begin{array}
[c]{c}%
r_{2}=-\left(  \left(  -\Delta+1\right)  f_{1},\frac{p+1}{2}Q^{p-1}f_{2}%
+\frac{p-1}{2}Q^{p-1}{\overline{f_{2}}}+\lambda^{2}V\left(  \left\vert
y+\tilde{\chi}\right\vert \right)  f_{2}\right)  _{H^{-1}}\\
-\left(  \frac{p+1}{2}Q^{p-1}f_{1}+\frac{p-1}{2}Q^{p-1}{\overline{f_{1}}%
}+\lambda^{2}V\left(  \left\vert y+\tilde{\chi}\right\vert \right)
f_{1},\left(  -\Delta+1\right)  f_{2}\right)  _{H^{-1}}\\
+\left(  \frac{p+1}{2}Q^{p-1}f_{1}+\frac{p-1}{2}Q^{p-1}{\overline{f_{1}}%
},\frac{p+1}{2}Q^{p-1}f_{2}+\frac{p-1}{2}Q^{p-1}{\overline{f_{2}}}+\lambda
^{2}V\left(  \left\vert y+\tilde{\chi}\right\vert \right)  f_{2}\right)
_{H^{-1}}\\
+\left(  \lambda^{2}V\left(  \left\vert y+\tilde{\chi}\right\vert \right)
f_{1},\lambda^{2}V\left(  \left\vert y+\tilde{\chi}\right\vert \right)
f_{2}\right)  _{H^{-1}}%
\end{array}
\right.
\]%
\[
r_{3}=\left(  \lambda^{2}V\left(  \left\vert y+\tilde{\chi}\right\vert
\right)  f_{1},\frac{p+1}{2}Q^{p-1}f_{2}+\frac{p-1}{2}Q^{p-1}{\overline{f_{2}%
}}\right)  _{H^{-1}}.
\]
We estimate $r_{1}$ as
\begin{equation}
\left\vert r_{1}\right\vert \leq\frac{1}{\sigma}\left(  2+\sigma^{-1}\right)
\left(  1+\left\Vert \rho^{\prime}\right\Vert _{L^{\infty}}^{2}\right)
\left\Vert f\right\Vert _{H^{1}}^{2}. \label{r1}%
\end{equation}
Using (\ref{es1}) and (\ref{es2}) we control $r_{2}$ by%
\begin{equation}
\left.
\begin{array}
[c]{c}%
\left\vert r_{2}\right\vert \leq2pq^{p-1}\left(  \frac{\sigma}{4}\right)
\left\Vert f\right\Vert _{H^{1}}^{2}+2\lambda^{2}\left\Vert V\left(
\left\vert y+\tilde{\chi}\right\vert \right)  \rho_{2}\left(  \frac{y}{\sigma
}\right)  \right\Vert _{L^{\infty}}\left\Vert f\right\Vert _{H^{1}}^{2}\\
+\frac{p+1}{2}q^{p-1}\left(  \frac{\sigma}{4}\right)  \left(  p\left\Vert
Q\right\Vert _{L^{\infty}}^{p-1}+\lambda^{2}\left\Vert V\right\Vert
_{L^{\infty}}\right)  \left\Vert f\right\Vert _{H^{1}}^{2}\\
+\lambda^{2}\left\Vert V\left(  \left\vert y+\tilde{\chi}\right\vert \right)
\rho_{2}\left(  \frac{y}{\sigma}\right)  \right\Vert _{L^{\infty}}\left\Vert
V\right\Vert _{L^{\infty}}\left\Vert f\right\Vert _{H^{1}}^{2}.
\end{array}
\right.  \label{r2}%
\end{equation}
Finally, we estimate $r_{3}.$ We have
\[
\left\vert r_{3}\right\vert \leq\lambda^{2}\left\Vert V\left(  \left\vert
\cdot+\tilde{\chi}\right\vert \right)  \left\langle \cdot\right\rangle
^{-2}\right\Vert _{L^{\infty}}\left\Vert f\right\Vert _{H^{1}}\left\Vert
\frac{\left\langle \cdot\right\rangle ^{2}}{1-\Delta}\left(  \frac{p+1}%
{2}Q^{p-1}f_{2}+\frac{p-1}{2}Q^{p-1}{\overline{f_{2}}}\right)  \right\Vert .
\]
Since
\[
\left\Vert V\left(  \left\vert y+\tilde{\chi}\right\vert \right)  \left\langle
\cdot\right\rangle ^{-2}\right\Vert _{L^{\infty}}\leq\left\Vert V\left(
\left\vert y+\tilde{\chi}\right\vert \right)  \right\Vert _{L^{\infty}\left(
\left\vert y\right\vert \leq\frac{\left\vert \tilde{\chi}\right\vert }%
{2}\right)  }+\left(  1+\frac{\left\vert \tilde{\chi}\right\vert ^{2}}%
{4}\right)  ^{-1}\left\Vert V\right\Vert _{L^{\infty}}%
\]
and%
\[
\left\Vert \frac{\left\langle \cdot\right\rangle ^{2}}{1-\Delta}\left(
\frac{p+1}{2}Q^{p-1}f_{2}+\frac{p-1}{2}Q^{p-1}{\overline{f_{2}}}\right)
\right\Vert \leq K_{2}\left\Vert f\right\Vert _{H^{1}},\text{ }K_{2}>0,
\]
we deduce%
\begin{equation}
\left\vert r_{3}\right\vert \leq K_{2}\lambda^{2}\left(  \left\Vert V\left(
\left\vert y+\tilde{\chi}\right\vert \right)  \right\Vert _{L^{\infty}\left(
\left\vert y\right\vert \leq\frac{\left\vert \tilde{\chi}\right\vert }%
{2}\right)  }+\left(  1+\frac{\left\vert \chi\right\vert ^{2}}{4}\right)
^{-1}\left\Vert V\right\Vert _{L^{\infty}}\right)  \left\Vert f\right\Vert
_{H^{1}}^{2}. \label{r3}%
\end{equation}
As by assumption $\lambda^{2}\sup_{r\in\mathbb{R}}V\left(  r\right)  <1,$ we
have%
\[
\left(  \left(  -\Delta+1-\lambda^{2}V\left(  \left\vert \cdot+\tilde{\chi
}\right\vert \right)  \right)  f_{1},f_{1}\right)  \geq\left(  \left(
-\Delta+c_{1}\right)  f_{1},f_{1}\right)  \geq b_{1}\left\Vert f_{1}%
\right\Vert _{H^{1}}^{2},
\]
with some $b_{1},c_{1}>0.$ Then, using (\ref{ineq2}) and (\ref{ap231}) in
(\ref{ap230}) we have%
\begin{equation}
\left\Vert \mathcal{L}_{V}f\right\Vert _{H^{-1}}^{2}\geq b_{3}\left\Vert
f\right\Vert _{H^{1}}^{2}-\frac{2}{b_{3}}\left(  \left\vert \left(
f_{2},iQ\right)  \right\vert +\left\vert \left(  f_{2},\nabla Q\right)
\right\vert \right)  +r+2\left(  r_{1}+r_{2}+r_{3}\right)  ,\text{ }b_{3}>0.
\label{in2}%
\end{equation}
Since%
\[
\left\vert \left(  f_{1},iQ\right)  \right\vert ^{2}+\left\vert \left(
f_{1},\nabla Q\right)  \right\vert ^{2}\leq2\left(  q+\left\vert q^{\prime
}\right\vert \right)  \left(  \frac{\sigma}{4}\right)  \left(  \int
Q+\left\vert \nabla Q\right\vert \right)  \left\Vert f\right\Vert _{H^{1}}^{2}%
\]
we estimate
\begin{equation}
\left\vert \left(  f_{2},iQ\right)  \right\vert ^{2}+\left\vert \left(
f_{2},\nabla Q\right)  \right\vert ^{2}\leq2\left\vert \left(  f,iQ\right)
\right\vert ^{2}+2\left\vert \left(  f,\nabla Q\right)  \right\vert ^{2}%
+r_{4}. \label{ap227}%
\end{equation}
with%
\[
r_{4}=4\left(  q+\left\vert q^{\prime}\right\vert \right)  \left(
\frac{\sigma}{4}\right)  \left(  \int Q+\left\vert \nabla Q\right\vert
\right)  \left\Vert f\right\Vert _{H^{1}}^{2}.
\]
Then, (\ref{in2}) takes the form
\begin{equation}
\left\Vert \mathcal{L}_{V}f\right\Vert _{H^{-1}}^{2}\geq b_{3}\left\Vert
f\right\Vert _{H^{1}}^{2}-\frac{4}{b_{3}}\left(  \left\vert \left(
f,iQ\right)  \right\vert ^{2}+\left\vert \left(  f,\nabla Q\right)
\right\vert ^{2}\right)  +r+2\left(  r_{1}+r_{2}+r_{3}\right)  -\frac{2}%
{b_{3}}r_{4}. \label{in3}%
\end{equation}
We choose $\sigma>0$ and $C\left(  V\right)  >0$ in (\ref{r0}), (\ref{r1}),
(\ref{r2}), (\ref{r3}) such that $\left\vert r\right\vert +2\left(  \left\vert
r_{1}\right\vert +\left\vert r_{2}\right\vert +\left\vert r_{3}\right\vert
\right)  +\frac{2}{b_{3}}\left\vert r_{4}\right\vert \leq\frac{3b_{3}}%
{4}\left\Vert f\right\Vert _{H^{1}}^{2},$ for all $\left\vert \chi\right\vert
\geq C\left(  V\right)  .$ Therefore, from (\ref{in2}) we deduce
(\ref{ineqpert}).

In order to complete the proof, we need to show that (\ref{ineqpert1}) holds.
We use the coercivity property of the unperturbed operator $\mathcal{L}$ that
follows, for example, from Lemma 2.2 of \cite{Nakanishi} (see also
\cite{Martel})
\[
\left(  \mathcal{L}f,f\right)  \geq c\left\Vert f\right\Vert _{H^{1}}%
^{2}-\frac{1}{c}\left(  \left(  f,Q\right)  ^{2}+\left\vert \left(
f,xQ\right)  \right\vert ^{2}+\left(  f,i\Lambda Q\right)  ^{2}\right)
,\text{ }f\in H^{1},
\]
for some $c>0$ independent of $f.$ Then, decomposing $\left(  \mathcal{L}%
_{V}f,f\right)  $ similarly to (\ref{ap230}) and arguing as in the proof of
(\ref{ineqpert}), we deduce (\ref{ineqpert1}). This completes the proof of
Lemma \ref{Linv}.

\section{Appendix.}

\subsubsection*{Proof of Lemma \ref{L8}.}

Recall that the kernel $G\left(  x\right)  $ of the Bessel potential $\left(
1-\Delta\right)  ^{-1}$ behaves asymptotically as (see pages 416-417 of
\cite{Aronszajn})
\begin{equation}
\left.  G\left(  x\right)  =2^{\frac{d+1}{2}}\pi^{\frac{d-1}{2}}\left\vert
x\right\vert ^{-\frac{d-1}{2}}e^{-\left\vert x\right\vert }\left(  1+o\left(
1\right)  \right)  ,\text{ as }\left\vert x\right\vert \rightarrow
\infty,\right.  \label{Bessel}%
\end{equation}
and%
\begin{equation}
G\left(  x\right)  =\left\{
\begin{array}
[c]{c}%
\frac{1}{2\pi}\ln\frac{1}{\left\vert x\right\vert }\left(  1+o\left(
1\right)  \right)  ,\text{ }d=2,\\
\frac{\Gamma\left(  \frac{d-2}{2}\right)  }{4\pi\left\vert x\right\vert
^{d-2}}\left(  1+o\left(  1\right)  \right)  ,\text{ }d\geq3,
\end{array}
\right.  \text{ as }\left\vert x\right\vert \rightarrow0, \label{Bessel1}%
\end{equation}
where $\Gamma$ denotes the gamma function. For $0<\delta<1,$ let $G_{\delta
}\left(  x\right)  =e^{\delta\left\vert x\right\vert }G\left(  x\right)  .$
Using (\ref{Bessel}) and (\ref{Bessel1}) we estimate%
\begin{equation}
\left\vert G_{\delta}\left(  x\right)  \right\vert \leq Ce^{-\left(
1-\delta\right)  \left\vert x\right\vert }\left\langle x\right\rangle
^{-\frac{d-1}{2}}\left\langle \left\vert x\right\vert ^{-\left(  d-2\right)
-\nu}\right\rangle ,\text{ }\nu>0. \label{Bessel0}%
\end{equation}
We decompose%
\begin{equation}
\left\vert e^{-\delta\left\vert y\right\vert }T\left(  y\right)  \right\vert
\leq C_{0}\left(  I_{1}+I_{2}+I_{3}\right)  , \label{ap277}%
\end{equation}
where
\[
I_{1}=\int_{\mathbb{R}^{d}}G_{\delta}\left(  y-z\right)  e^{-\delta\left\vert
z\right\vert }Q^{p-1}\left(  z\right)  \left\vert T\left(  z\right)
\right\vert dz,
\]%
\[
I_{2}=\int_{\mathbb{R}^{d}}G_{\delta}\left(  y-z\right)  e^{-\delta\left\vert
z\right\vert }V\left(  \left\vert z+\frac{\chi}{\lambda}\right\vert \right)
\left\vert T\left(  z\right)  \right\vert dz
\]
and%
\[
I_{3}=\int_{\mathbb{R}^{d}}G_{\delta}\left(  y-z\right)  e^{-\delta\left\vert
z\right\vert }\left\vert f\left(  z\right)  \right\vert dz,
\]
for some $C_{0}>0.$ Let $\psi\in C^{\infty}$ be such that $\left\Vert
\psi\right\Vert _{L\infty}\leq1,$ $\psi=1$ for $\left\vert x\right\vert \leq1$
and $\psi=0$ for $\left\vert x\right\vert \geq2.$ For $a>0$ we decompose%
\[
I_{1}=I_{11}+I_{22},
\]
with%
\[
I_{11}=\int_{\mathbb{R}^{d}}G_{\delta}\left(  y-z\right)  e^{-\delta\left\vert
z\right\vert }Q^{p-1}\left(  z\right)  \psi\left(  \frac{z}{a}\right)
\left\vert T\left(  z\right)  \right\vert dz
\]
and%
\[
I_{12}=\int_{\mathbb{R}^{d}}G_{\delta}\left(  y-z\right)  e^{-\delta\left\vert
z\right\vert }Q^{p-1}\left(  z\right)  \left(  1-\psi\left(  \frac{z}%
{a}\right)  \right)  \left\vert T\left(  z\right)  \right\vert dz.
\]
By (\ref{ineqpert})%
\[
\left\Vert \psi T\right\Vert _{H^{1}}^{2}\leq C_{1}\left(  \left\Vert
\mathcal{L}_{V}\left(  \psi T\right)  \right\Vert _{H^{-1}}^{2}+\left\vert
\left(  \psi T,\nabla Q\right)  \right\vert ^{2}\right)  ,\text{ }C_{1}>0.
\]
Noting that $[\mathcal{L}_{V},\psi\left(  \frac{\cdot}{a}\right)
]=-a^{-2}\left(  \Delta\psi\right)  \left(  \frac{\cdot}{a}\right)
-a^{-1}\left(  \nabla\psi\right)  \left(  \frac{\cdot}{a}\right)  \nabla,$ as
$\left(  T,\nabla Q\right)  =0,$ we have%
\begin{align*}
\left\Vert \psi T\right\Vert _{H^{1}}^{2}  &  \leq C_{1}\left(  \left\Vert
\psi f\right\Vert ^{2}+a^{-1}\left\Vert \left(  \nabla\psi\right)  \nabla
T\right\Vert ^{2}+a^{-2}\left\Vert \Delta\psi T\right\Vert ^{2}+\left\vert
\left(  \left(  1-\psi\right)  T,\nabla Q\right)  \right\vert ^{2}\right) \\
&  \leq C_{1}\left(  \left\Vert \psi f\right\Vert ^{2}+C_{2}a^{-1}\left\Vert
T\right\Vert _{H^{1}}^{2}\right)  ,\text{ }C_{2}>0.
\end{align*}
Then, taking $a\geq A_{0}^{-2}\left(  \left\vert \chi\right\vert \right)  +1$
and using (\ref{ap290}) we get%
\[
\left\Vert \psi T\right\Vert _{H^{1}}^{2}\leq C_{1}\left(  \left\Vert \psi
e^{\delta\left\vert z\right\vert }e^{-\delta\left\vert z\right\vert
}f\right\Vert ^{2}+C_{2}A_{0}^{2}\left(  \left\vert \chi\right\vert \right)
\left\Vert T\right\Vert _{H^{1}}^{2}\right)  \leq C_{3}A_{0}^{2}\left(
\left\vert \chi\right\vert \right)
\]
with $C_{3}>0.$ Thus, by (\ref{Bessel0}) and Young's inequality we get%
\[
\left\Vert I_{11}\right\Vert _{L^{2}}\leq C_{0}\left\Vert G_{\delta}\left(
\cdot\right)  \right\Vert _{L^{1}}\left\Vert \psi T_{1}\right\Vert _{L^{2}%
}\leq C_{4}\left(  \delta\right)  A_{0}\left(  \left\vert \chi\right\vert
\right)  ,\text{ }C_{4}\left(  \delta\right)  >0.
\]
Moreover, we have
\[
\left\Vert I_{12}\right\Vert _{L^{2}}\leq Ca^{-1}\left\Vert G_{\delta}\left(
\cdot\right)  \right\Vert _{L^{1}}\left\Vert \left\vert z\right\vert
Q^{p-1}\left(  z\right)  \right\Vert _{L^{\infty}}\left\Vert T\right\Vert
_{L^{2}}\leq Ca^{-1}\leq CC_{4}\left(  \delta\right)  A_{0}^{2}\left(
\left\vert \chi\right\vert \right)  .
\]
Hence, $\left\Vert I_{1}\right\Vert _{L^{2}}\leq CC_{4}\left(  \delta\right)
A_{0}\left(  \left\vert \chi\right\vert \right)  .$ By (\ref{estp}) and
(\ref{ap235})\
\begin{align*}
\left\Vert I_{2}\right\Vert _{L^{2}}  &  \leq C\left\Vert T\right\Vert
_{L^{\infty}}\left\Vert e^{-\delta\left\vert \cdot\right\vert }V\left(
\left\vert \cdot+\frac{\chi}{\lambda}\right\vert \right)  \right\Vert _{L^{2}%
}\left\Vert \left\vert \cdot\right\vert ^{\frac{d-1}{2}}e^{-\left(
1-\delta\right)  \left\vert \cdot\right\vert }\right\Vert _{L^{2}}\\
&  \leq C\left\Vert T\right\Vert _{L^{\infty}}\left\Vert V\left(  \left\vert
\cdot+\frac{\chi}{\lambda}\right\vert \right)  Q\left(  \cdot\right)
\right\Vert _{L^{\infty}}^{\delta^{\prime}}\left\Vert \left\vert
\cdot\right\vert ^{\delta^{\prime}\frac{d-1}{2}}e^{-\left(  \delta
-\delta^{\prime}\right)  \left\vert \cdot\right\vert }\right\Vert _{L^{2}%
}\left\Vert \left\vert \cdot\right\vert ^{\frac{d-1}{2}}e^{-\left(
1-\delta\right)  \left\vert \cdot\right\vert }\right\Vert _{L^{2}}\\
&  \leq CC_{5}\left(  \delta,\delta^{\prime}\right)  \left\Vert T\right\Vert
_{L^{\infty}}\Theta\left(  \left\vert \tilde{\chi}\right\vert \right)
^{\delta^{\prime}},\text{ }C_{5}\left(  \delta,\delta^{\prime}\right)  >0,
\end{align*}
for any $\delta^{\prime}<\delta.$ By using (\ref{ap290}) we control
$\left\Vert I_{3}\right\Vert _{L^{2}}\leq C_{6}\left(  \delta\right)
A_{0}\left(  \left\vert \chi\right\vert \right)  ,$ $C_{6}\left(
\delta\right)  >0.$ Using the estimates for $I_{1},I_{2},I_{3}$ in
(\ref{ap277}) we deduce%
\begin{equation}
\left\Vert e^{-\delta\left\vert \cdot\right\vert }T\left(  \cdot\right)
\right\Vert _{L^{2}}\leq C\left(  \delta,\delta^{\prime}\right)  \left(
A_{0}\left(  \left\vert \chi\right\vert \right)  +\left\Vert T\right\Vert
_{L^{\infty}}\Theta\left(  \left\vert \tilde{\chi}\right\vert \right)
^{\delta^{\prime}}\right)  . \label{ap330}%
\end{equation}
Since $G_{\delta}\in L^{p},$ for any $1<p<\frac{d}{d-2},$ from the equation
(\ref{eq18}), via Young's inequality, we control the $L^{\frac{2p}{2-p}}$ norm
of $e^{-\delta\left\vert \cdot\right\vert }T\left(  \cdot\right)  $ by the
right-hand side of (\ref{ap330}). Note that $\frac{2p}{2-p}>2.$ Iterating the
last argument a finite number of times, we attain (\ref{ap291}). Lemma
\ref{L8} is proved.

\end{document}